\pdfoutput=1
\documentclass[11pt]{article}

\usepackage[utf8]{inputenc}
\usepackage[pdftex]{graphicx,xcolor}
\usepackage{array,bookmark,fullpage,enumitem,microtype}
\usepackage[leqno]{amsmath}
\usepackage{amssymb,amsthm,mathtools,bm,stmaryrd,tikz-cd,mathdots}
\usepackage[nottoc]{tocbibind}
\usepackage{hyperref}
\hypersetup{pdftex,colorlinks=true,allcolors=blue,unicode,bookmarksnumbered,psdextra}
\usepackage[capitalize,nosort]{cleveref}

\usepackage{mymacros}

\usetikzlibrary{decorations.pathmorphing}

\let\defn\textbf
\let\displaystyle\textstyle

\def\op{\mathrm{op}}
\def\eq{\mathrm{eq}}
\def\ins{\mathrm{ins}}
\def\coeq{\mathrm{coeq}}
\def\coins{\mathrm{coins}}
\def\Sub{\mathrm{Sub}}
\def\OSub{\mathrm{OSub}}
\def\RSub{\mathrm{RSub}}
\def\oker{\mathrm{oker}}
\def\Sp{\mathrm{Sp}}
\def\Term{\mathrm{Term}}
\def\upk{\mathrm{upk}}
\def\dj{\mathrm{dj}}
\def\COMGR{\mathrm{COMGR}}
\def\Comgr{\mathrm{Comgr}}
\def\DESC{\mathrm{DESC}}
\def\qua{\text{ qua }}

\mathlig{<|}{\lhd}
\mathlig{|>}{\rhd}
\mathlig{|-}{\vdash}
\mathlig{|=}{\models}

\begin{document}

\title{Borel and analytic sets in locales}
\author{Ruiyuan Chen}
\date{}
\maketitle

\begin{abstract}
We systematically develop analogs of basic concepts from classical descriptive set theory in the context of pointless topology.
Our starting point is to take the elements of the free complete Boolean algebra generated by the frame $\@O(X)$ of opens to be the ``$\infty$-Borel sets'' in a locale $X$.
We show that several known results in locale theory may be interpreted in this framework as direct analogs of classical descriptive set-theoretic facts, including e.g., the Lusin separation, Lusin--Suslin, and Baire category theorems for locales; we also prove several extensions of these results, such as an ordered Novikov separation theorem.
We give a detailed analysis of various notions of image, and prove that a continuous map need not have an $\infty$-Borel image.
We introduce the category of ``analytic $\infty$-Borel locales'' as the regular completion under images of the category of locales and $\infty$-Borel maps (as a unary site), and prove analogs of several classical results about analytic sets, such as a boundedness theorem for well-founded analytic relations.
We also consider the ``positive $\infty$-Borel sets'' of a locale, formed from opens without using $\neg$.

In fact, we work throughout in the more refined context of $\kappa$-copresented $\kappa$-locales and $\kappa$-Borel sets for arbitrary regular $\omega_1 \le \kappa \le \infty$; taking $\kappa = \omega_1$ then recovers the classical context as a special case.
The basis for the aforementioned localic results is a detailed study of various known and new methods for presenting $\kappa$-frames, $\kappa$-Boolean algebras, and related algebraic structures.
In particular, we introduce a new type of ``two-sided posite'' for presenting $(\kappa, \kappa)$-frames $A$, meaning both $A$ and $A^\op$ are $\kappa$-frames, and use this to prove a general $\kappa$-ary interpolation theorem for $(\kappa, \kappa)$-frames, which dualizes to the aforementioned separation theorems.
\end{abstract}

\tableofcontents

\section{Introduction}

In this paper, we study the connection between \emph{descriptive set theory} and \emph{locale theory}, two areas which can each be thought of as providing a ``non-pathological'' version of point-set topology.

\defn{Descriptive set theory} can be broadly described as the study of simply-definable sets and functions.
The standard framework is to begin with a well-behaved topological space, namely a \defn{Polish space}, i.e., a separable, completely metrizable space.
One then declares open sets to be ``simply-definable'', and builds up more complicated sets by closing under simple set-theoretic operations, e.g., countable Boolean operations (yielding the \emph{Borel sets}), continuous images (yielding the \emph{analytic sets}), etc.
The definability restriction rules out many general point-set pathologies, yielding powerful structural and classification results which nonetheless apply to typical topological contexts encountered in mathematical practice.
This has led in recent decades to an explosion of fruitful connections and applications to many diverse branches of mathematics, such as dynamical systems, operator algebras, combinatorics, and model theory.
See \cite{Kcdst}, \cite{Knew}, \cite{Mdst}, \cite{Gidst}, \cite{KMgraph} for general background on descriptive set theory and its applications.

\defn{Locale theory}, also known as \emph{pointless} or \emph{point-free topology}, is the ``dual'' algebraic study of topological spaces and generalizations thereof via their lattices of open sets; see \cite{Jstone}, \cite{PPloc}.
A \defn{frame} is a complete lattice with finite meets distributing over arbitrary joins; the motivating example is the frame $\@O(X)$ of open sets of a topological space $X$.
In locale theory, one formally regards an \emph{arbitrary} frame $\@O(X)$ as the ``open sets'' of a generalized ``space'' $X$, called a \defn{locale}.
One then proceeds to study algebraically-defined analogs of various topological notions.
A key insight due to Isbell~\cite{Iloc} is that such analogies are not perfect, but that this is a feature, not a bug: where it differs from point-set topology, locale theory tends to be less pathological.

The starting point for this paper is an observation ``explaining'' this feature of locale theory: it seems to behave like a formal generalization of descriptive set theory, with countability restrictions removed.
This observation is supported by many results scattered throughout the literature, e.g.:
\begin{itemize}

\item
A fundamental construction of Isbell~\cite{Iloc} is to freely adjoin complements to an arbitrary frame $\@O(X)$, yielding a bigger frame $\@N(\@O(X))$, whose elements are in canonical bijection with the sublocales of $X$.
This is analogous to the $\*\Sigma^0_2$ (i.e., $F_\sigma$) sets in a Polish space $X$, the closure of the open and closed sets under finite intersections and countable unions; the complements of $\*\Sigma^0_2$ sets, the $\*\Pi^0_2$ (i.e., $G_\delta$) sets, are precisely the Polish subspaces of $X$ \cite[3.11]{Kcdst}.

\item
Isbell~\cite{Iloc} also showed that dense sublocales of an arbitrary locale are closed under arbitrary intersections, just as dense Polish subspaces of Polish spaces are closed under countable intersections, by the Baire category theorem \cite[8.4]{Kcdst}.

\item
Transfinitely iterating the $\@N(-)$ construction yields a ``Borel hierarchy'' for an arbitrary locale $X$, the union of which is the free complete Boolean algebra generated by the frame $\@O(X)$.
By a classical result of Gaifman~\cite{Gcbool} and Hales~\cite{Hcbool}, free complete Boolean algebras are generally proper classes, whence the ``Borel hierarchy of locales is proper''.

\item
Ball~\cite{Bbaire} developed a ``Baire hierarchy of real-valued functions'' on an arbitrary locale, via a transfinite construction closely related to the aforementioned $\@N$ functor.

\item
Isbell~\cite{Idst} proved an analog of the classical Hausdorff--Kuratowski theorem \cite[22.27]{Kcdst}: any complemented sublocale (i.e., ``$\*\Delta^0_2$ set'') can be written as $(((F_0 \setminus F_1) \cup F_2) \setminus F_3) \cup \dotsb$ for some transfinite decreasing sequence $F_0 \supseteq F_1 \supseteq \dotsb$ of closed sublocales.

\item
The closed subgroup theorem of Isbell--Kříž--Pultr--Rosický~\cite{IKPRgrp} states that every closed localic subgroup of a localic group is closed, just as every Polish subgroup of a Polish group is closed \cite[2.2.1]{Gidst}.
Moreover, the proof of the closed subgroup theorem by Johnstone~\cite{Jgpd} uses a localic version of Pettis's theorem \cite[2.3.2]{Gidst}, itself proved via an analog of the classical proof based on (what we are calling) localic Baire category.

\item
The dual formulation of measure theory is well-known, via measure algebras.
Explicit connections with general locales have been made by Simpson~\cite{Smeas} and Pavlov~\cite{Pmeasloc}.

\item
De~Brecht~\cite{dBqpol} introduced \defn{quasi-Polish spaces}, a well-behaved non-Hausdorff generalization of Polish spaces sharing many of the same descriptive set-theoretic properties.
Heckmann~\cite{Hqpol} proved that the category of these is equivalent to the category of locales $X$ whose corresponding frame $\@O(X)$ is countably presented.
Thus,
\begin{center}
``locale = (Polish space) $-$ (countability restrictions) $-$ (separation axioms)''.
\end{center}

\item
In \cite{Cscc}, we explicitly used this correspondence to transfer a known result \emph{from} locale theory (the Joyal--Tierney representation theorem \cite{JTgpd} for Grothendieck toposes via localic groupoids) \emph{to} the classical descriptive set-theoretic context (a Makkai-type strong conceptual completeness theorem for the infinitary logic $\@L_{\omega_1\omega}$ in terms of Borel groupoids of models).

\item
In \cite{Cborin}, we gave an intrinsic categorical characterization of the classical category $\!{SBor}$ of \defn{standard Borel spaces} (Polish spaces with their topologies forgotten, remembering only the Borel sets).
This characterization was proved for ``standard $\kappa$-Borel locales'' for all regular cardinals $\kappa \ge \omega$, using algebraic techniques; taking $\kappa = \omega_1$ then yielded the classical result.

\end{itemize}

In spite of these and other results, the known analogy between locale theory and descriptive set theory has been on a largely \emph{ad hoc} basis to date, without a complete or coherent locale-theoretic account of even basic descriptive set-theoretic notions such as Borel sets.
The goal of this paper is to provide the beginnings of such an account.

\subsection{Overview of localic results}
\label{sec:intro-loc}

Our main focus of study is on the two categories $\!{\kappa Loc}$ of \defn{$\kappa$-locales}, and $\!{\kappa BorLoc}$ of \defn{$\kappa$-Borel locales}, for each infinite regular cardinal $\kappa$ (or $\kappa = \infty$).
These are defined as the formal duals of the respective algebraic categories $\!{\kappa Frm}$ of \defn{$\kappa$-frames}, meaning frames but required only to have $\kappa$-ary joins, and $\!{\kappa Bool}$ of \defn{$\kappa$-(complete )Boolean algebras}.
Thus, a $\kappa$-locale $X$ is, formally, the same thing as a $\kappa$-frame $\@O_\kappa(X)$, whose elements we call the \defn{$\kappa$-open sets} in $X$; while a $\kappa$-Borel locale $X$ is formally a $\kappa$-Boolean algebra $\@B_\kappa(X)$, whose elements we call the \defn{$\kappa$-Borel sets} in $X$.%
\footnote{When $\kappa = \infty$, we allow $\@B_\infty(X)$ to be a \emph{large} (i.e., proper class) complete Boolean algebra, where ``complete'' means with respect to \emph{small} (i.e., set-sized) joins; but we require $\@B_\infty(X)$ to be \emph{small-presented}, i.e., it must have a presentation (as a complete Boolean algebra) $\@B_\infty(X) = \ang{G \mid R}$ with a \emph{set} of generators $G$ and of relations $R$, or equivalently, $\@B_\infty(X)$ must be freely generated by a (small) $\kappa$-Boolean algebra for some $\kappa < \infty$.}
A morphism $f : X -> Y$ in any of these categories is a homomorphism $f^*$ between the respective algebras, \emph{in the opposite direction}, thought of as taking preimage of sets under $f$.
When $\kappa = \omega_1$, we also write $\sigma$ in place of $\kappa$.
We define these categories in \cref{sec:loc-cat}, and recall their connection with the classical categories of spaces in \cref{sec:loc-sp}.

As $\kappa$ varies, these categories are related by canonical forgetful functors
\begin{equation*}
\begin{tikzcd}
\!{\omega BorLoc} \rar &
\!{\sigma BorLoc} \rar &
\dotsb \rar &
\!{\kappa BorLoc} \rar &
\dotsb \rar &
\!{\infty BorLoc} \\
\!{\omega Loc} \uar \rar &
\!{\sigma Loc} \uar \rar &
\dotsb \rar &
\!{\kappa Loc} \uar \rar &
\dotsb \rar &
\!{Loc} \uar
\end{tikzcd}
\end{equation*}
which are given by the \emph{free} functors between the respective algebraic categories.
Thus, for instance, given a $\sigma$-locale $X$, the underlying $\sigma$-Borel locale has $\sigma$-Borel sets $\@B_\sigma(X)$ given by repeatedly adjoining complements to the $\sigma$-open sets $\@O_\sigma(X)$ (see \eqref{eq:intro-borhier} below); while the $\infty$-Borel sets $\@B_\infty(X)$ are given by further freely completing under arbitrary ``intersections'' and ``unions''.

We are particularly interested in the (full) subcategories of the above categories consisting of the \emph{$\kappa$-copresented} objects $X$, meaning $\@O_\kappa(X)$ (resp., $\@B_\kappa(X)$) is $\kappa$-presented as an algebra of the respective type.
We also call such $X$ \defn{standard}, and denote these subcategories by $\!{\kappa Loc}_\kappa \subseteq \!{\kappa Loc}$ and $\!{\kappa BorLoc}_\kappa \subseteq \!{\kappa BorLoc}$.
When $\kappa = \omega_1$, by the aforementioned result of Heckmann~\cite{Hqpol} as well as a classical result of Loomis--Sikorski (see \cite[29.1]{Sbool}, \cite[4.1]{Cborin}), we have equivalences with the classical descriptive set-theoretic categories:
\begin{align}
\label{eq:intro-loc-ctbpres}
\begin{aligned}
\text{standard $\sigma$-locales} = \!{\sigma Loc}_\sigma &\simeq \!{QPol} := \text{quasi-Polish spaces}, \\
\text{standard $\sigma$-Borel locales} = \!{\sigma BorLoc}_\sigma &\simeq \!{SBor} := \text{standard Borel spaces}.
\end{aligned}
\end{align}
This is described in \cref{sec:loc-ctbpres} (see \cref{thm:loc-ctbpres}).
Thus, as $\kappa$ varies between $\omega_1$ and $\infty$, we interpolate between the contexts of classical descriptive set theory and classical locale theory.

As hinted above, for a $\kappa$-locale $X$, we may stratify the free $\kappa$-Boolean algebra $\@B_\kappa(X)$ generated by the $\kappa$-frame $\@O_\kappa(X)$ into the transfinite iterates $\@N_\kappa^\alpha(\@O_\kappa(X))$ of the functor $\@N_\kappa : \!{\kappa Frm} -> \!{\kappa Frm}$ which freely adjoins a complement for each pre-existing element; we call this the \defn{$\kappa$-Borel hierarchy of $X$}, described in \cref{sec:loc-bor}.
Each level $\@N_\kappa^\alpha(\@O_\kappa(X))$ consists of the \defn{$\kappa\Sigma^0_{1+\alpha}$-sets of $X$}, while their complements are the \defn{$\kappa\Pi^0_{1+\alpha}$-sets}:
\begin{align}
\label{eq:intro-borhier}
\begin{aligned}
\kappa\Sigma^0_{1+\alpha}(X) &:= \@N_\kappa^\alpha(\@O_\kappa(X)), \\
\kappa\Pi^0_{1+\alpha}(X) &:= \{\neg B \mid B \in \@N_\kappa^\alpha(\@O_\kappa(X))\}.
\end{aligned}
\end{align}
(We should note that, unlike much of the locale theory literature (e.g., \cite{Idst}), our point of view is fundamentally Boolean: we always regard open sets, sublocales, etc., as embedded in the complete Boolean algebra $\@B_\infty(X)$, with lattice operations taking place in there; see \cref{cvt:frm-cbool-incl}.
In fact, we will never refer to Heyting algebras or intuitionistic logic at all.)

$\kappa$-frames were introduced and studied by Madden~\cite{Mkfrm}; along with $\kappa$-Boolean algebras, they have since played a key role in much work on (what we are calling) the $\infty$-Borel hierarchy of locales, by e.g., Madden--Molitor~\cite{MMfrmepi}, Wilson~\cite{Wasm}, Plewe~\cite{Pdissolv}.
Especially, they tend to be relevant whenever considering notions of ``injectivity'', ``surjectivity'', ``image'', etc.
It is well-known that such notions are a subtle area of locale theory.
For instance, the usual notion of ``image sublocale'' of a continuous map $f : X -> Y \in \!{Loc}$ is not well-behaved; see e.g., \cite{Pquotloc}.

In \cref{sec:loc-im}, generalizing and extending these known results, we give a detailed analysis of ``injectivity'', ``surjectivity'', ``image'', etc., in our main categories of interest $\!{\kappa Loc}, \!{\kappa BorLoc}$ as well as their subcategories of standard objects.
This analysis shows that the subtlety in these notions is much clarified by the Boolean point of view, and ultimately arises from hidden interactions with topology when one insists on remaining in $\!{Loc}$.
The following is a summary of this analysis; see \cref{thm:loc-emb-im-cls}, \cref{thm:loc-sub-im}, \cref{thm:loc-im-factor}, \cref{thm:loc-bor-epi-pullback}, and \cref{thm:loc-bor-epi-invlim}.

\begin{theorem}[structure of subobjects in $\!{\kappa Loc}, \!{\kappa BorLoc}$]
\label{thm:intro-sub}
\leavevmode
\begin{enumerate}

\item[(a)]  \textup{(LaGrange)}  Monomorphisms (``injections'') in $\!{\kappa BorLoc}$ are the same as regular monomorphisms (``embeddings'').

\item[(b)]  \textup{(essentially Madden)}  Regular subobjects of an object $X$ (equivalence classes of ``embeddings'' into $X$) in $\!{\kappa Loc}$ or $\!{\kappa BorLoc}$ are canonically identified with certain $\infty$-Borel sets $B \in \@B_\infty(X)$: with arbitrary intersections of $\kappa$-Borel sets in the case of $\!{\kappa BorLoc}$, and intersections of implications between $\kappa$-open sets in the case of $\!{\kappa Loc}$.

\item[(c)]  In the subcategory $\!{\kappa BorLoc}_\kappa \subseteq \!{\kappa BorLoc}$ of standard $\kappa$-Borel locales, (regular) subobjects instead correspond to $\kappa$-Borel sets $B \in \@B_\kappa(X) \subseteq \@B_\infty(X)$.
Similarly, in $\!{\kappa Loc}_\kappa$, regular subobjects correspond to $\kappa\Pi^0_2$ sets.

\item[(d)]  For an arbitrary map $f : X -> Y$ in $\!{\kappa Loc}$ or $\!{\kappa BorLoc}$, its epi--regular mono (``image'') factorization $X ->> Z `-> Y$ in that category is given by the ``$\infty$-Borel image''%
\footnote{assuming the $\infty$-Borel image exists; otherwise, we need to take the $\infty$-analytic image as described below}
$f(X) \in \@B_\infty(Y)$ (the smallest $\infty$-Borel set whose $f$-pullback is all of $X$), followed by closure into the respective class of sets as in (b).

\item[(e)]  Epimorphisms (``surjections'') in $\!{\kappa BorLoc}$ are well-behaved: they are regular epimorphisms (``quotient maps''), pullback-stable (so that $\!{\kappa BorLoc}$ is a regular category), and closed under products and codirected inverse limits of diagrams in which each morphism is an epimorphism.

\end{enumerate}
\end{theorem}

This result includes the localic analogs of several basic results in classical descriptive set theory.
The combination of (a) and the first part of (c) yields a \defn{Lusin--Suslin theorem} for standard $\kappa$-Borel locales (every $\kappa$-Borel injection has a $\kappa$-Borel image; see \cite[15.1]{Kcdst}).
Likewise, as mentioned near the start of this Introduction, the second part of (c) generalizes the classical result that the Polish subspaces of a Polish space are precisely the $\*\Pi^0_2$ sets \cite[3.11]{Kcdst} as well as its quasi-Polish analog \cite[Th.~23]{dBqpol}.

Classically, a key ingredient in the Lusin--Suslin theorem is the fundamental \defn{Lusin separation theorem} \cite[14.7]{Kcdst}: if two Borel maps $f : Y -> X$ and $g : Z -> X$ between standard Borel spaces have disjoint images, then there is a Borel set $B \subseteq X$ separating those images.
Note that this can be stated as follows, avoiding mention of the images: if the fiber product $Y \times_X Z$ is empty, then there is $B$ such that $f^{-1}(B) = Y$ and $g^{-1}(B) = \emptyset$.
The $\kappa$-Borel localic version of this was essentially proved, in a dual algebraic form, by LaGrange~\cite{Lamalg} (and then used to prove \cref{thm:intro-sub}(a); see \cite[3.2]{Cborin} for an explanation of this).
The \defn{Novikov separation theorem} \cite[28.5]{Kcdst} is the generalization of the Lusin separation theorem to countably many Borel maps $f_i : Y_i -> X$ with $\bigcap_i f(Y_i) = \emptyset$.
We prove the $\kappa$-Borel localic generalization of this in \cref{thm:sigma11-sep}:

\begin{theorem}[Novikov separation theorem]
\label{thm:intro-sigma11-sep}
Let $f_i : Y_i -> X$ be $<\kappa$-many $\kappa$-Borel maps between standard $\kappa$-Borel locales.
If the fiber product $\prod_X (Y_i)_i$ of all of the $Y_i$ over $X$ is empty, then there are $\kappa$-Borel sets $B_i \in \@B_\kappa(X)$ with $\bigcap_i B_i = \emptyset$ and $f_i^*(B_i) = Y_i$ for each $i$.
\end{theorem}

In fact we also establish versions of most of the aforementioned results for a third category intermediate between $\!{\kappa Loc}$ and $\!{\kappa BorLoc}$: the category $\!{\kappa Bor^+Loc}$ of \defn{positive $\kappa$-Borel locales}, dual to the category $\!{\kappa\kappa Frm}$ of \defn{$(\kappa, \kappa)$-frames},%
\footnote{According to some literature, these would be called ``$\kappa$-biframes''; but that term also has a different meaning.}
meaning $\kappa$-frames whose order-duals are also $\kappa$-frames.
The motivation for these is that given a $\kappa$-locale $X$, we may close the $\kappa$-open sets under \emph{positive} $\kappa$-Boolean operations, i.e., not using $\neg$; the resulting $(\kappa, \kappa)$-frame does not remember all of the topological structure on $X$, but still encodes the specialization order, with respect to which all of the positive $\kappa$-Borel sets must be upward-closed.
In \cref{sec:poloc} (see \cref{thm:loc-pos-upper} and \cref{thm:loc-pos-poloc}), we make this precise by showing that

\begin{theorem}
There is a canonical full and faithful ``forgetful'' functor $\!{\kappa Bor^+Loc} -> \!{\kappa BorPOLoc}$ into the category of (internally) partially ordered $\kappa$-Borel locales.
In other words, positive $\kappa$-Borel locales may be faithfully identified with certain partially ordered $\kappa$-Borel locales.
\end{theorem}

This result is in turn based on a positive $\kappa$-localic analog of a classical ordered version of the Lusin separation theorem \cite[28.12]{Kcdst}; we prove a simultaneous generalization of this and the Novikov separation theorem in \cref{thm:bor+-sep}.

The aforementioned results clearly hint at a theory of ``$\kappa$-analytic'' or ``$\kappa\Sigma^1_1$'' sets, meaning ``images'' of arbitrary $\kappa$-Borel maps $f : X -> Y \in \!{\kappa BorLoc}_\kappa$ between standard $\kappa$-Borel locales.
In \cref{sec:sigma11}, we develop such a theory by ``closing the category $\!{\kappa BorLoc}_\kappa$ under non-existent images''.
This is done via two equivalent methods: externally, by closing under pre-existing images in the ambient supercategory $\!{\kappa BorLoc}$ of possibly nonstandard $\kappa$-Borel locales; or internally, by adjoining a ``formal image'' for each morphism $f : X -> Y \in \!{\kappa BorLoc}_\kappa$ (via an application of Shulman's~\cite{Sksite} theory of \emph{unary sites}).
We call the result the category of \defn{analytic $\kappa$-Borel locales}, denoted $\!{\kappa\Sigma^1_1BorLoc}$, with subobjects in it called \defn{$\kappa$-analytic sets}.
We prove in \cref{thm:sigma11-cat} that $\!{\kappa\Sigma^1_1BorLoc}$ shares many of the nice properties of $\!{\kappa BorLoc}$ and $\!{\kappa BorLoc}_\kappa$.
These categories are related as follows:
\begin{equation*}
\begin{tikzcd}
\!{\omega\Sigma^1_1BorLoc} \rar[hook] &
\!{\sigma\Sigma^1_1BorLoc} \rar[dashed,"?"] &
\dotsb \rar[dashed,"?"] &
\!{\kappa\Sigma^1_1BorLoc} \rar[dashed,"?"] &
\dotsb \rar[dashed,"?"] &
\!{\infty\Sigma^1_1BorLoc}
\\
\!{\omega BorLoc}_\omega \uar[equal] \rar[hook] &
\!{\sigma BorLoc}_\sigma \uar[hook] \rar[hook] &
\dotsb \rar[hook] &
\!{\kappa BorLoc}_\kappa \uar[hook] \rar[hook] &
\dotsb \rar[hook] &
\!{\infty BorLoc} \uar[hook]
\end{tikzcd}
\end{equation*}
Unlike in the standard case, we do not know if analytic $\kappa$-Borel locales are increasingly general as $\kappa$ increases, i.e., if there are canonical forgetful functors $\!{\kappa\Sigma^1_1BorLoc} -> \!{\lambda\Sigma^1_1BorLoc}$ for $\kappa \le \lambda$ (see \cref{rmk:sigma11-absolute}).
In particular, we do not know if the resulting theory of ``$\infty$-analytic sets'' generalizes the classical ($\kappa = \omega_1$) one, or is merely an analog of it; hence the importance of our considering all $\kappa \le \infty$ at once.

In \cref{sec:sigma11-invlim}, we generalize a fundamental aspect of classical analytic sets: their connection with well-foundedness.
Classically, given a Borel map $f : X -> Y$ between standard Borel spaces, by representing each fiber $f^{-1}(y)$ for $y \in Y$ as the space of branches through a countably branching tree $T_y$ in a uniform manner, we obtain a Borel family of trees $(T_y)_{y \in Y}$ such that the analytic set $\im(f) \subseteq Y$ consists of precisely those $y$ for which $T_y$ has an infinite branch.
Note that the space of infinite branches through a tree is the inverse limit of its (finite) levels.
Thus, the following (\cref{thm:sigma11-invlim-k}) can be seen as a generalization of the classical result just mentioned:

\begin{theorem}[``tree'' representation of $\kappa$-analytic sets]
\label{thm:intro-sigma11-invlim}
For any $\kappa$-Borel map $f : X -> Y$ between standard $\kappa$-Borel locales, there is a $\kappa$-ary codirected diagram $(X_i)_{i \in I}$ of standard $\kappa$-Borel locales $X_i -> Y$ over $Y$ and $\kappa$-ary-to-one $\kappa$-Borel maps between them, such that $X = \projlim_i X_i$ (over $Y$).

Thus, the $\kappa$-analytic set $\im(f) \subseteq Y$ in $\!{\kappa\Sigma^1_1BorLoc}$ is given by ``$\{y \in Y \mid \projlim_i (X_i)_y \ne \emptyset\}$''.
\end{theorem}

As we recall in \cref{sec:loc-intlog}, a well-known technique from categorical logic endows every (sufficiently nice) category with an \defn{internal logic} via which ``pointwise'' expressions like the above may be made precise.
This yields a systematic way of translating many classical definitions into locale-theoretic analogs, including in the following.

Classically, given a Borel family of countably branching trees $(T_y)_{y \in Y}$ representing an analytic set $\im(f) \subseteq Y$ as above, by repeatedly pruning the leaves of these trees, we obtain an $\omega_1$-length decreasing sequence of Borel sets $B_\alpha \subseteq Y$ whose intersection is $\im(f)$; see \cite[25.16]{Kcdst}.
Moreover, stabilization of this sequence below $\omega_1$ is closely related to Borelness of the analytic set $\im(f)$; see \cite[35.D]{Kcdst}.
In \cref{thm:sigma11-invlim-prune}, we prove the following localic generalization of these results:

\begin{theorem}[pruning of ``tree'' representations]
\label{thm:intro-sigma11-invlim-prune}
Let $f : X -> Y \in \!{\kappa BorLoc}_\kappa$ be the inverse limit, over a standard $\kappa$-Borel locale $Y$, of a $\kappa$-ary codirected diagram $(X_i)_{i \in I}$ of standard $\kappa$-Borel locales over $Y$ and $\kappa$-ary-to-one $\kappa$-Borel maps between them, as in \cref{thm:intro-sigma11-invlim}.
\begin{enumerate}

\item[(a)]  The $\kappa$-analytic set $\im(f) \subseteq Y$ is the $\kappa$-length decreasing meet (in the poset $\kappa\Sigma^1_1(Y)$) of the $\kappa$-Borel images of the $\alpha$th prunings of the $X_i$.

\item[(b)]  If $X = \emptyset$ (i.e., $\im(f) = \emptyset$), then the pruning stabilizes at the empty diagram by some $\alpha < \kappa$.

\item[(c)]  If $\im(f) \subseteq Y$ is $\kappa$-Borel, then the sequence in (a) stabilizes (at $\im(f)$) by some $\alpha < \kappa$.

\end{enumerate}
\end{theorem}

In fact, this result holds for \emph{any} $\kappa$-ary codirected diagram over $Y$, not necessarily with $\kappa$-ary-to-one maps; the only difference is that the images in (a) need no longer be $\kappa$-Borel.
Part (b) can then be seen as a generalization of the classical \defn{boundedness theorem for $\*\Sigma^1_1$}; see \cref{rmk:sigma11-invlim-bounded}.
In \cref{sec:sigma11-proper} (see \cref{thm:sigma11-proper}), using this boundedness theorem, we finally prove that

\begin{theorem}[Borel $\subsetneq$ analytic]
\label{thm:intro-sigma11-proper}
There is a continuous map $f : X -> Y$ between quasi-Polish spaces such that, regarded as a map in $\!{\sigma Loc}_\sigma$ via \eqref{eq:intro-loc-ctbpres}, there is no smallest $\kappa$-Borel set $\im(f) \in \@B_\kappa(Y)$ whose $f$-pullback is all of $X$, for any $\omega_1 \le \kappa \le \infty$.

In other words, $\im(f) \subseteq Y$ is a ``$\sigma$-analytic, non-$\infty$-Borel set''.
\end{theorem}

The map $f : X -> Y$ here is simply the usual projection to the space $Y$ of binary relations on $\#N$ from the space $X$ of binary relations equipped with a countable descending sequence, whose image would be the set of ill-founded relations.
The key point is that, when we pass to the underlying $\infty$-Borel locales, these same $X, Y$ also have nonempty $\infty$-Borel sets ``of relations of arbitrary cardinality''.
This is based on a well-known technique: a localic ``cardinal collapse'' (see \cite[C1.2.8--9]{Jeleph}, \cref{thm:collapse}), proved using (what we are calling) the localic Baire category theorem of Isbell~\cite{Iloc} mentioned near the beginning of this Introduction.
In \cref{sec:loc-baire}, we give a self-contained exposition of localic Baire category from a descriptive point of view, including e.g., the \defn{property of Baire} for $\infty$-Borel sets (\cref{thm:loc-bor-baireprop}), as well as an inverse limit generalization of the Baire category theorem (\cref{thm:loc-baire-invlim}) which is used in the proof of \cref{thm:intro-sigma11-invlim-prune}.

All of the aforementioned results reduce to their classical analogs when $\kappa = \omega_1$.
For many of these results, this yields, modulo the equivalences of categories \eqref{eq:intro-loc-ctbpres} (which are based on a simple Baire category argument), a purely lattice-theoretic proof of the classical result which is quite different from the classical proof (found in, e.g., \cite{Kcdst}).

\subsection{Overview of algebraic results}
\label{sec:intro-alg}

While our main goal in this paper is descriptive set theory in locales, the technical heart of this paper lies largely on the dual algebraic side, in \cref{sec:frm}.
Broadly speaking, we study there various methods for \emph{presenting} the types of lattice-theoretic algebras $A$ mentioned above (i.e., $\kappa$-frames, $(\lambda, \kappa)$-frames, and $\kappa$-Boolean algebras) via generators and relations:
\begin{align*}
A = \ang{G \mid R}.
\end{align*}
We are particularly interested in presentation methods where $G$ and/or $R$ are more structured than mere sets; usually, they already incorporate part of the structure of the algebra $A$ being presented.
Such structured presentation methods can be used to prove many important facts about free functors between our various algebraic categories of interest (e.g., $\!{\kappa Frm} -> \!{\kappa Bool}$), hence about the \emph{forgetful} functors between the corresponding categories of locales (e.g., $\!{\kappa Loc} -> \!{\kappa BorLoc}$).

For example, a standard method of presenting a frame $A$, due to Johnstone~\cite[II~2.11]{Jstone} (although based on older ideas from sheaf theory), is via a \defn{posite} $(G, {<|})$, where $G$ is a (finitary) meet-semilattice, while $<|$ (called a \defn{coverage} on $G$) is a set of relations of the form ``$a \le \bigvee B$'' where $B \subseteq G$.
The presented frame $A = \ang{G \mid {<|}}$ is then freely generated by $G$ \emph{as a meet-semilattice}, subject to the relations specified by $<|$ becoming true.
The key point is that, since $G$ already has finite meets, it is enough for $<|$ to impose join relations.
This can then be used to prove, for example, that a direct limit of frames is the direct limit of the underlying complete join-semilattices (see \cref{thm:frm-colim-dir-vlat}), from which one then deduces the results on inverse limits of locales mentioned in the preceding subsection, such as \cref{thm:intro-sub}(e) and the localic Baire category theorem (see \cref{thm:frm-colim-dir-ker}, \cref{thm:bool-colim-dir-ker}, \cref{thm:loc-bor-epi-invlim}, \cref{thm:loc-baire-invlim}).

In the first parts of \cref{sec:frm}, after defining the basic algebraic categories in \cref{sec:frm-cat} and some general discussion of categorical properties of presentations in \cref{sec:pres-adj}, we give a self-contained review of various standard constructions of free and presented frames, including posites in \cref{sec:frm-post}, as well as their generalizations to $\kappa$-frames.
In many of these cases, the $\kappa$-ary generalization does not appear explicitly in the literature, but is largely routine.
Our reason for the level of detail given is threefold: first, for the sake of completeness, especially since there are some occasional subtleties involved in the $\kappa$-ary generalizations.
Second, we wish to make this paper accessible without assuming familiarity with standard locale theory.
And third, this serves as motivation for some further generalizations of these constructions that we give later (such as $(\lambda, \kappa)$-distributive polyposets; see below).

We also take this opportunity to prove, via reasoning about presentations, some simple facts about $\kappa$-ary products in the categories $\!{\kappa Frm}, \!{\kappa\kappa Frm}, \!{\kappa Bool}$ (see \cref{sec:frm-prod}), which are quite possibly folklore, although we could not find them explicitly stated in this generality in the literature.

In \cref{sec:frm-neg}, we consider the functor $\@N_\kappa : \!{\kappa Frm} -> \!{\kappa Frm}$ that adjoins complements to a $\kappa$-frame, as well as its transfinite iterates $\@N_\kappa^\alpha$.
For $\kappa = \infty$, this is Isbell's~\cite{Iloc} $\@N$ functor we mentioned near the start of this Introduction; for $\kappa < \infty$, this functor was studied by Madden~\cite{Mkfrm}.
We recall several of his results, but using a new approach based on an explicit posite presentation for $\@N_\kappa(A)$ (see \cref{thm:frm-neg-pres}).
We also prove, in \cref{thm:frm-bool-pres,thm:frm-bool-pres-cof}, the algebraic results underlying the generalizations (see \cref{thm:loc-bor-loc,thm:loc-bor-dissolv}) of the classical change of topology results for (quasi-)Polish spaces (see \cite[13.1--5]{Kcdst}, \cite[Th.~73]{dBqpol}):

\begin{theorem}[change of topology]
\label{thm:intro-dissolv}
\leavevmode
\begin{enumerate}
\item[(a)]  Every $\kappa$-presented $\kappa$-Boolean algebra $B$ is isomorphic to the free $\kappa$-Boolean algebra $\@N_\kappa^\infty(A)$ generated by some $\kappa$-presented $\kappa$-frame $A$.
\item[(b)]  Given any such $A$, we may enlarge it to some $A \subseteq A' \subseteq \@N_\kappa^\alpha(A)$ containing any $<\kappa$-many prescribed elements of $\@N_\kappa^\alpha(A)$, such that $A'$ still freely generates $B$.
\end{enumerate}
Dually:
\begin{enumerate}
\item[(a)]  Every standard $\kappa$-Borel locale $X$ is the underlying $\kappa$-Borel locale of some standard $\kappa$-locale.
\item[(b)]  Given any standard $\kappa$-locale $X$ and $<\kappa$-many $\kappa\Sigma^0_\alpha$-sets $C_i \in \kappa\Sigma^0_\alpha(X)$, we may find a standard $\kappa$-locale $X'$, such that $\@O_\kappa(X) \cup \{C_i\}_i \subseteq \@O_\kappa(X') \subseteq \kappa\Sigma^0_\alpha(X)$ and the canonical map $X' -> X$ is a $\kappa$-Borel isomorphism.
In other words, $X'$ is ``$X$ with a finer topology making each $C_i$ open''.
\end{enumerate}
\end{theorem}

While this result can easily be proved in the same way as the classical proof (in \cite[13.1--5]{Kcdst}), it is interesting to note that our proof (in \cref{thm:frm-bool-pres-cof}) derives it from completely abstract universal-algebraic principles that have nothing to do with $\kappa$-frames or $\kappa$-Boolean algebras \emph{per se}.

In \cref{sec:upkzfrm}, we consider $\kappa$-frames which are dual to \defn{zero-dimensional} $\kappa$-locales, hence which are generated by complemented elements.
Here, there is an extra subtlety which is invisible in the classical case $\kappa = \omega_1$: it is more natural to consider the \defn{ultraparacompact} zero-dimensional $\kappa$-locales, meaning that there are ``enough clopen partitions'' (see \cref{thm:upkzfrm}).
This is because ultraparacompactness, which is automatic for zero-dimensional Polish spaces, is precisely what allows every such Polish space to be represented as the space of infinite branches through a tree; the $\kappa = \infty$ version of this fact was proved by Paseka~\cite{Pupk}.
Classically, one may always change to a zero-dimensional Polish topology by making basic open sets clopen.
When $\kappa > \omega_1$, that ultraparacompactness can be obtained via change of topology (in the sense of \cref{thm:intro-dissolv}) is much less obvious, and was shown by Plewe~\cite{Psubloc} when $\kappa = \infty$.
We again adopt an approach based on presentations of ultraparacompact zero-dimensional $\kappa$-frames (via Boolean algebras and pairwise disjoint join relations), in order to give completely different proofs of the two results just cited for arbitrary $\kappa$ (see \cref{thm:upkzfrm-cabool-dircolim}, \cref{thm:frm-neg-upkz}, \cref{thm:loc-upkz-invlim}, and \cref{thm:loc-bor-dissolv-upkz}):

\begin{theorem}[ultraparacompactness and change of topology]
\leavevmode
\begin{enumerate}
\item[(a)]  Every ultraparacompact zero-dimensional ($\kappa$-presented) $\kappa$-frame is a ($\kappa$-ary) direct limit, in $\!{\kappa Frm}$, of powersets of $\kappa$-ary sets.
\item[(b)]  Every $\kappa$-presented $\kappa$-frame $A$ may be enlarged to an ultraparacompact zero-dimensional $\kappa$-presented $\kappa$-frame $A' \supseteq A$ freely generating the same $\kappa$-Boolean algebra $\@N_\kappa^\infty(A') \cong \@N_\kappa^\infty(A)$.
\end{enumerate}
Dually:
\begin{enumerate}
\item[(a)]  Every ultraparacompact zero-dimensional (standard) $\kappa$-locale is a ($\kappa$-ary) inverse limit of discrete $\kappa$-ary sets.
\item[(b)]  Given any standard $\kappa$-locale $X$, we may find an ultraparacompact zero-dimensional standard $\kappa$-locale $X'$ such that $\@O_\kappa(X') \supseteq \@O_\kappa(X)$ and the canonical map $X' -> X$ is a $\kappa$-Borel isomorphism.
\end{enumerate}
\end{theorem}

These results play a crucial role in the proof of the ``tree'' representation of $\kappa$-analytic sets (\cref{thm:intro-sigma11-invlim}).

In \cref{sec:dpoly}, we introduce a new method for presenting \defn{$(\lambda, \kappa)$-frames}, meaning $\kappa$-frames whose order-duals are $\lambda$-frames, for arbitrary $\lambda, \kappa$.
A \defn{$(\lambda, \kappa)$-distributive polyposet} $(A, {<|})$ is a ``symmetric'' generalization of a posite, consisting of an underlying set $A$ equipped with a binary relation $<|$ between $\lambda$-ary subsets and $\kappa$-ary subsets of $A$, thought of as a presentation of a $(\lambda, \kappa)$-frame $\ang{A \mid {<|}}$ with generating set $A$ and relations ``$\bigwedge B \le \bigvee C$'' for $B <| C$.
The $<|$ relation is required to obey certain ``saturation'' conditions which capture implied relations in the presented algebra $\ang{A \mid {<|}}$; see \cref{sec:dpoly}.
We prove in \cref{thm:frm-dpoly-inj} that these ``saturation'' conditions indeed capture \emph{all} such implied relations, and not just in the presented $(\lambda, \kappa)$-frame but in the presented complete Boolean algebra:

\begin{theorem}[saturation of $(\lambda, \kappa)$-distributive polyposets]
\label{thm:intro-dpoly}
Let $(A, {<|})$ be a $(\lambda, \kappa)$-distributive polyposet.
For any $\lambda$-ary $B \subseteq A$ and $\kappa$-ary $C \subseteq A$, if $\bigwedge B \le \bigvee C$ holds in the presented complete Boolean algebra $\ang{A \mid {<|}}$, then already $B <| C$.
\end{theorem}

This easily implies, for instance, that every $(\lambda, \kappa)$-frame $A$ embeds into the free $\max(\lambda, \kappa)$-Boolean algebra it generates (by taking the canonical $<|$ consisting of all relations that hold in $A$; see \cref{thm:frm-dpoly-bifrm-inj}), or dually, that the forgetful functor $\!{\kappa Bor^+Loc} -> \!{\kappa BorLoc}$ from positive $\kappa$-Borel locales to $\kappa$-Borel locales is faithful (\cref{thm:loc-forget}).

In \cref{thm:bifrm-minterp}, we use $(\kappa, \kappa)$-distributive polyposets to prove the most general algebraic form of the separation theorems described in the preceding subsection.
This result can also be seen as an algebraic form of the Craig--Lyndon interpolation theorem for $\kappa$-ary infinitary propositional logic.
We state here a somewhat simplified special case of the full result:

\begin{theorem}[interpolation]
\label{thm:intro-interp}
Let $A, B_i, C_j$ be $<\kappa$-many $(\kappa, \kappa)$-frames, with homomorphisms $f_i : A -> B_i$ and $g_j : A -> C_j$, and let $D$ be their \defn{bilax pushout}, i.e., $D$ is presented by the disjoint union of all the elements of $A, B_i, C_j$, subject to the additional relations
\begin{align*}
f_i(a) \le a \le g_j(a) \in D \quad\text{for $a \in A$}.
\end{align*}
Thus $D$ is the universal filler in the following diagram, equipped with homomorphisms from $A, B_i, C_j$ satisfying the indicated order relations:
\begin{equation*}
\begin{tikzcd}
B_i \rar[dashed] \drar[phantom, "\le"{pos=.2,sloped}, "\le"{pos=.8,sloped}] &
D \\
A \uar["f_i"] \rar["g_j"'] \urar[dashed] &
C_j \uar[dashed]
\end{tikzcd}
\end{equation*}
If $b_i \in B_i$ and $c_j \in C_j$ satisfy
\begin{align*}
\bigwedge_i b_i \le \bigvee_j c_j \in D,
\end{align*}
then there are $a^L_i, a^R_j \in A$ such that
\begin{align*}
b_i &\le f_i(a^L_i), &
\bigwedge_i a^L_i &\le \bigvee_j a^R_j, &
g_j(a^R_j) &\le c_j.
\end{align*}
\end{theorem}

The Novikov separation theorem (\cref{thm:intro-sigma11-sep}) follows by taking each of $A, B_i, C_j$ to be $\kappa$-Boolean and then dualizing, while the more general positive $\kappa$-Borel version (\cref{thm:bor+-sep}) requires the full strength of \cref{thm:intro-interp}.
As we noted above \cref{thm:intro-sigma11-sep}, the Lusin separation theorem, which only requires \cref{thm:intro-interp} with three $\kappa$-Boolean algebras $A, B, C$, was essentially shown by LaGrange~\cite{Lamalg}; this is already enough to imply many structural results about $\!{\kappa Bool}$, e.g., that epimorphisms are regular and monomorphisms are regular and pushout-stable (see \cref{thm:bool-pushout-inj,thm:bool-mono-reg,thm:bool-epi-surj}), which dualize to \cref{thm:intro-sub}(a) and the first two parts of (e).

The proof of \cref{thm:intro-dpoly} is based on an explicit syntactic construction, in \cref{sec:frm-cbool}, of free complete Boolean algebras via the cut-free Gentzen sequent calculus for infinitary propositional logic (see \cref{thm:frm-cbool-seq}).
While this construction is new as far as we know, it has an obvious precedent in Whitman's construction of free lattices (see \cite[III]{Hcbool}).

Finally, we note that from our proof that there exist $\sigma$-analytic, non-$\infty$-Borel sets (\cref{thm:intro-sigma11-proper}), we have the following purely Boolean-algebraic consequence (see \cref{thm:bool-free-mono-bad}), which is somewhat surprising given the good behavior of monomorphisms in $\!{\kappa Bool}$ in other respects (e.g., regularity and pushout-stability, as noted above):

\begin{theorem}
There exist $\kappa < \lambda < \infty$, namely $\kappa = \omega_1$ and $\lambda = (2^{\aleph_0})^+$, such that the free functor $\!{\kappa Bool} -> \!{\lambda Bool}$ does not preserve injective homomorphisms.
\end{theorem}

\subsection{Future directions}

In this paper, we have restricted attention to the ``elementary'' parts of classical descriptive set theory.
We have not considered at all more advanced set-theoretic techniques such as determinacy, the projective hierarchy, or effective descriptive set theory (see \cite{Kcdst}, \cite{Mdst}), let alone large cardinals or forcing (although the ``cardinal collapse'' argument used in the proof of \cref{thm:intro-sigma11-proper} could be understood as an instance of sheaf-theoretic forcing; see \cite{MMshv}).
Nor have we given any serious consideration to the modern area of \emph{invariant descriptive set theory}, which studies definable (e.g., Borel) equivalence relations and group actions and their quotient spaces (see \cite{Knew}, \cite{Gidst}).
All of these areas would be ripe for future investigation of potential connections with locale theory.

Even in ``elementary'' descriptive set theory, there remain many gaps in the localic analogy which we have not addressed.
These include e.g., the perfect set theorem; universal and complete sets; a deeper study of Baire category and measure; and uniformization and selection theorems (see again \cite{Kcdst}).
We plan to address several of these topics in a sequel paper.

Conversely, there are many parts of locale theory and adjacent areas which could probably be fruitfully placed into a descriptive set-theoretic framework.
These include e.g., powerlocales (see \cite{Spowloc}, \cite{dBKpowsp}), the theory of localic groups and groupoids, and connections with toposes and categorical logic (see \cite{Jeleph}).
As we indicated near the start of this Introduction, some/all of these topics have already seen suggestive work that at least implicitly used descriptive set-theoretic ideas; what is missing is a more systematic account of the connection.
For example, there is surely a link between the Scott analysis in invariant descriptive set theory and countable model theory (see \cite[Ch.~12]{Gidst}), the $\@N$ functor on frames, and its ``first-order'' topos-theoretic analog corresponding to Morleyization of $\@L_{\infty\omega}$-theories into geometric theories (see \cite[D1.5.13]{Jeleph}, \cite[\S4]{Cscc}).

Finally, as we noted in \cref{sec:intro-loc}, in this paper we completely ignore a traditionally important aspect of locales: the use of intuitionistic logic, and the consequent possibility of internalization over an arbitrary base topos (see \cite[C1.6]{Jeleph}).
In fact, we make unrestricted use of the axiom of choice (e.g., in assuming that arbitrary subsets of frames may be well-ordered in \cref{sec:upkzfrm} in order to disjointify elements).
It would be interesting to develop a constructive version of the theory in this paper, perhaps as a bridge to effective descriptive set theory.

\subsection{Structure of paper}

This paper has three major sections.

\Cref{sec:frm} takes place entirely in the algebraic (i.e., lattice-theoretic) setting, and forms the foundations for the rest of the paper.
Its contents are described in detail in \cref{sec:intro-alg}.
As indicated there, the broad theme of this section is to study various methods for presenting our main algebraic structures of interest ($\kappa$-frames, $(\lambda, \kappa)$-frames, $\kappa$-Boolean algebras), as well as other structural results proved using presentation-based techniques.

\Cref{sec:loc} introduces the dual localic viewpoint, and develops as much of the theory as is reasonable without needing to refer to images of arbitrary $\kappa$-Borel maps (including the analysis of existing $\kappa$-Borel images); \cref{sec:sigma11} introduces the category $\!{\kappa\Sigma^1_1BorLoc}$ of analytic $\kappa$-Borel locales as the ``closure'' of the standard $\kappa$-Borel locales $\!{\kappa BorLoc}_\kappa$ under images.
The contents of these two sections are described in detail in \cref{sec:intro-loc}.
Even though we adopt the localic viewpoint whenever possible, the algebraic preliminaries in \cref{sec:frm} are of course used often in these later sections.

As we indicated above, we have tried to make this paper accessible without any prior locale theory background.
For this reason, the first few subsections in \cref{sec:frm,sec:loc} give a self-contained (though terse) introduction to all of the basic notions needed.
Likewise, we do not assume any prior descriptive set-theoretic background; all needed definitions are given in \cref{sec:loc,sec:sigma11}.

However, we have found it necessary to make fairly liberal use of category theory throughout the paper, as an organizing framework for the various types of structures introduced.
We therefore assume familiarity with basic notions of category theory, up to such notions as limits, colimits, and adjoint functors; see e.g., \cite{Bcat} for a general reference.
In \cref{sec:cat}, we give a ``logician's'' introduction to the deeper theory of categories of well-behaved (infinitary) first-order structures.

\paragraph*{Acknowledgments}

I would like to thank Matthew de~Brecht for several stimulating discussions during the early stages of this work, as well as for providing some comments and corrections to a draft of this paper.

\section{Frames and Boolean algebras}
\label{sec:frm}

This section contains the order-theoretic foundations for the paper.
In \cref{sec:frm-cat}, we define the main categories of lattices and Boolean algebras of interest to us, discuss their relation to each other, and establish some notational conventions.
In the rest of the subsections, we discuss various concrete methods (some well-known, some new) for constructing free structures, and more generally, presenting structures via generators and relations; these are then applied to study various other constructions, including certain types of limits and colimits.

Throughout this paper, $\kappa, \lambda$, and sometimes $\mu$ will denote infinite regular cardinals or the symbol $\infty$ (bigger than all cardinals).
By \defn{$\kappa$-ary}, we mean of size $<\kappa$.
We will occasionally refer to structures which are class-sized, rather than set-sized; we call the former \defn{large} and the latter \defn{small} or \defn{$\infty$-ary}.%
\footnote{To be more precise, we could work in NBG set theory instead of ZFC, or assume there to be an inaccessible cardinal called $\infty$ and take ``set'' to mean of size $<\infty$.
None of our results will depend essentially on proper classes.}
All structures except categories are assumed to be small by default.
In terminology in which ``$\kappa$'' appears, when $\kappa = \omega_1$, we write ``$\sigma$'' as a synonym, e.g., ``$\sigma$-frame'' means ``$\omega_1$-frame''.
It is occasionally helpful to regard $\kappa = 2, \{1\}$ as ``degenerate regular cardinals'', where $\{1\}$-ary means of size exactly $1$.%
\footnote{These ``possibly degenerate regular cardinals'' are called ``arity classes'' in \cite{Sksite}.}

As indicated in the Introduction, we assume familiarity with basic notions of category theory, including limits, colimits, adjoint functors; a review of some more advanced notions is given in \cref{sec:cat}.
Categories will be denoted with symbols like $\!C$.
Hom-sets will be denoted $\!C(A, B)$ for objects $A, B \in \!C$; identity morphisms will be denoted $1_A \in \!C(A, A)$.
Products will be denoted $A \times B$ or $\prod_i A_i$; fiber products/pullbacks will be denoted $A \times_C B$; the equalizer of $f, g : A \rightrightarrows B$ will be denoted $\eq(f, g)$, while their coequalizer will be denoted $\coeq(f, g)$; and coproducts and pushouts will generally be denoted with the symbol $\amalg$, but also with other notation specific to certain categories (e.g., $\otimes, \sqcup$).
The category of sets will be denoted $\!{Set}$;
the category of posets and order-preserving maps will be denoted $\!{Pos}$.

\subsection{The main categories}
\label{sec:frm-cat}

A \defn{$\kappa$-complete join-semilattice}, or \defn{$\kappa$-$\bigvee$-lattice} for short, is a poset with $\kappa$-ary joins (i.e., suprema).
An $\infty$-$\bigvee$-lattice is also called a \defn{complete join-semilattice} or a \defn{$\bigvee$-lattice}.
An $\omega$-$\bigvee$-lattice is also called a \defn{join-semilattice} or a \defn{$\vee$-lattice}, with joins usually denoted $a_1 \vee \dotsb \vee a_n$.
Note that we require $\vee$-lattices to have a least element (nullary join), which we denote by $\bot$.

We regard $\kappa$-$\bigvee$-lattices $A$ as infinitary algebraic structures with join operations $\bigvee : A^\lambda -> A$ for all $\lambda < \kappa$.
Hence, terms like \defn{homomorphism} and \defn{$\kappa$-$\bigvee$-sublattice} have their usual meanings (map preserving $\kappa$-ary joins, and subset closed under $\kappa$-ary joins).
The \defn{category of $\kappa$-$\bigvee$-lattices} will be denoted $\!{\kappa{\bigvee}Lat}$; when $\kappa = \omega, \omega_1, \omega$, we also use the names $\!{{\vee}Lat}, \!{\sigma{\bigvee}Lat}, \!{{\bigvee}Lat}$.

The category $\!{\kappa{\bigwedge}Lat}$ of \defn{$\kappa$-complete meet-semilattices} is defined similarly.
We adopt the obvious order-duals of the above conventions;
the greatest element is denoted $\top$.

A \defn{$\kappa$-frame} \cite{Mkfrm} is a poset which is a $\wedge$-lattice as well as a $\kappa$-$\bigvee$-lattice, and satisfies
\begin{align*}
a \wedge \bigvee_i b_i = \bigvee_i (a \wedge b_i)
\end{align*}
whenever the join is $\kappa$-ary.
An $\infty$-frame is also called a \defn{frame}.
An $\omega$-frame is the same thing as a \defn{distributive lattice}.
The \defn{category of $\kappa$-frames} is denoted $\!{\kappa Frm}$ (or $\!{DLat}, \!{\sigma Frm}, \!{Frm}$ when $\kappa = \omega, \omega_1, \infty$).

A \defn{$\kappa$-complete Boolean algebra}, or just \defn{$\kappa$-Boolean algebra}, is a $\kappa$-frame in which every element $a$ has a (unique) complement $\neg a$, meaning $a \wedge \neg a = \bot$ and $a \vee \neg a = \top$.
We will also use the implication operation $a -> b := \neg a \vee b$ and bi-implication $a <-> b := (a -> b) \wedge (b -> a)$.
(We will never use $\neg, ->$ to denote Heyting operations in a frame; see \cref{cvt:frm-cbool-incl}.)
An $\infty$-Boolean algebra is a \defn{complete Boolean algebra}, while an $\omega$-Boolean algebra is a \defn{Boolean algebra}.
The \defn{category of $\kappa$-Boolean algebras} is denoted $\!{\kappa Bool}$ (or $\!{Bool}, \!{\sigma Bool}, \!{CBool}$ when $\kappa = \omega, \omega_1, \infty$).

A $\kappa$-Boolean algebra also has $\kappa$-ary meets, over which finite joins distribute; it is sometimes useful to consider such meets in the absence of complements.
A \defn{$(\lambda, \kappa)$-frame}%
\footnote{When $\kappa = \lambda = \infty$, these are sometimes called \emph{biframes} (e.g., \cite[VII~1.15]{Jstone}), although that term now more commonly refers to a distinct notion (\cite[XII~5.2]{PPloc}).}
will mean a poset $A$ with $\lambda$-ary meets and $\kappa$-ary joins, over which finite joins and meets distribute, respectively; in other words, $A$ is a $\kappa$-frame, while the opposite poset $A^\op$ is a $\lambda$-frame.
The category of these will be denoted $\!{\lambda\kappa Frm}$.
Thus, $\!{\kappa Frm} = \!{\omega\kappa Frm}$, while $\!{\kappa Bool} \subseteq \!{\kappa\kappa Frm}$.
A $(\kappa, \omega)$-frame is also called a \defn{$\kappa$-coframe}; the category of these is denoted $\!{\kappa Cofrm} = \!{\kappa\omega Frm}$.
(One can also think of a $\kappa$-$\bigvee$-lattice as a ``$(\{1\}, \kappa)$-frame'', of a $\wedge$-lattice as a ``$\{1\}$-frame'', and of a poset as a ``$(\{1\}, \{1\})$-frame''.)

The following commutative diagram of categories and forgetful functors depicts the relationships between these various categories, for $\kappa \le \lambda$:
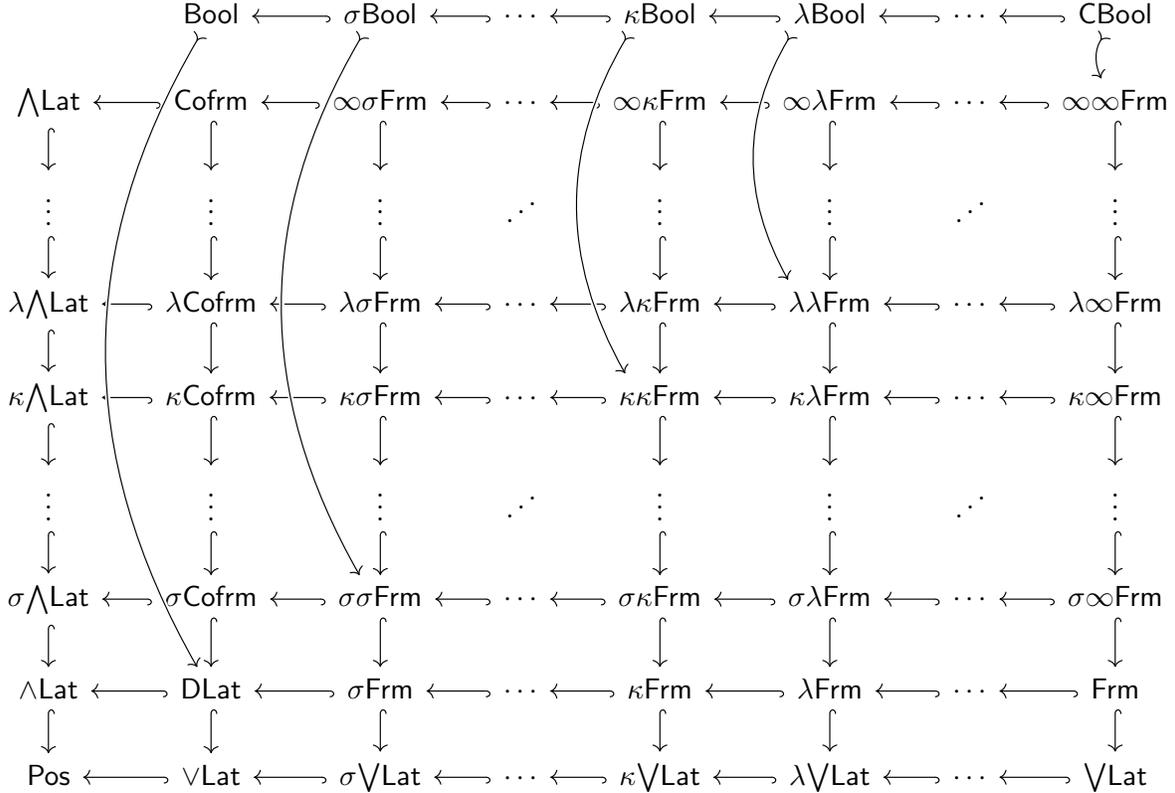
\begin{equation}
\label{diag:frm-cat}
\begin{tikzcd}[row sep=1.5em, column sep=1.75em, dashed/.style={}]
&
\!{Bool} &
\!{\sigma Bool} \lar[hook] &
\dotsb \lar[hook] &
\!{\kappa Bool} \lar[hook] &
\!{\lambda Bool} \lar[hook] &
\dotsb \lar[hook] &
\!{CBool} \lar[hook, dashed]
\\
\!{{\bigwedge}Lat} \dar[hook] &
\!{Cofrm} \lar[hook] \dar[hook] &
\!{\infty\sigma Frm} \lar[hook, dashed] \dar[hook, dashed] &
\dotsb \lar[hook, dashed] &
\!{\infty\kappa Frm} \lar[hook, dashed] \dar[hook, dashed] &
\!{\infty\lambda Frm} \lar[hook, dashed] \dar[hook, dashed] &
\dotsb \lar[hook, dashed] &
\!{\infty\infty Frm} \lar[hook, dashed] \dar[hook, dashed] \ar[u, leftarrowtail, dashed, bend left]
\\
\vdots \dar[hook] &
\vdots \dar[hook] &
\vdots \dar[hook] &
\iddots &
\vdots \dar[hook] &
\vdots \dar[hook] &
\iddots &
\vdots \dar[hook, dashed]
\\
\!{\lambda{\bigwedge}Lat} \dar[hook] &
\!{\lambda Cofrm} \lar[hook] \dar[hook] &
\!{\lambda\sigma Frm} \lar[hook] \dar[hook] &
\dotsb \lar[hook] &
\!{\lambda\kappa Frm} \lar[hook] \dar[hook] &
\!{\lambda\lambda Frm} \lar[hook] \dar[hook] \ar[uuu, leftarrowtail, shift left=4, bend left, crossing over, crossing over clearance=2pt] &
\dotsb \lar[hook] &
\!{\lambda\infty Frm} \lar[hook, dashed] \dar[hook, dashed]
\\
\!{\kappa{\bigwedge}Lat} \dar[hook] &
\!{\kappa Cofrm} \lar[hook] \dar[hook] &
\!{\kappa\sigma Frm} \lar[hook] \dar[hook] &
\dotsb \lar[hook] &
\!{\kappa\kappa Frm} \lar[hook] \dar[hook] \ar[uuuu, leftarrowtail, shift left=3, bend left, crossing over, crossing over clearance=2pt] &
\!{\kappa\lambda Frm} \lar[hook] \dar[hook] &
\dotsb \lar[hook] &
\!{\kappa\infty Frm} \lar[hook, dashed] \dar[hook, dashed]
\\
\vdots \dar[hook] &
\vdots \dar[hook] &
\vdots \dar[hook] &
\iddots &
\vdots \dar[hook] &
\vdots \dar[hook] &
\iddots &
\vdots \dar[hook, dashed]
\\
\!{\sigma{\bigwedge}Lat} \dar[hook] &
\!{\sigma Cofrm} \lar[hook] \dar[hook] &
\!{\sigma\sigma Frm} \lar[hook] \dar[hook] \ar[uuuuuu, leftarrowtail, shift left, bend left, crossing over, crossing over clearance=2pt] &
\dotsb \lar[hook] &
\!{\sigma\kappa Frm} \lar[hook] \dar[hook] &
\!{\sigma\lambda Frm} \lar[hook] \dar[hook] &
\dotsb \lar[hook] &
\!{\sigma\infty Frm} \lar[hook, dashed] \dar[hook, dashed]
\\
\!{{\wedge}Lat} \dar[hook] &
\!{DLat} \lar[hook] \dar[hook] \ar[uuuuuuu, leftarrowtail, bend left, crossing over, crossing over clearance=2pt] &
\!{\sigma Frm} \lar[hook] \dar[hook] &
\dotsb \lar[hook] &
\!{\kappa Frm} \lar[hook] \dar[hook] &
\!{\lambda Frm} \lar[hook] \dar[hook] &
\dotsb \lar[hook] &
\!{Frm} \lar[hook] \dar[hook]
\\
\!{Pos} &
\!{{\vee}Lat} \lar[hook] &
\!{\sigma{\bigvee}Lat} \lar[hook] &
\dotsb \lar[hook] &
\!{\kappa{\bigvee}Lat} \lar[hook] &
\!{\lambda{\bigvee}Lat} \lar[hook] &
\dotsb \lar[hook] &
\!{{\bigvee}Lat} \lar[hook]
\end{tikzcd}
\end{equation}
All of these forgetful functors are subcategory inclusions, in particular faithful.
Moreover, the forgetful functors $\!{\kappa Bool} -> \!{\kappa\kappa Frm}$, as well as their composites with $\!{\kappa\kappa Frm} -> \!{\kappa\lambda Frm}$ or $\!{\kappa\kappa Frm} -> \!{\lambda \kappa Frm}$ for $\omega \le \lambda \le \kappa$, are also full; this is the meaning of the $\rightarrowtail$ arrows.

The above categories with $\kappa, \lambda < \infty$, except for $\!{Pos}$, consist of algebraic structures, i.e., structures with (a small set of) infinitary operations obeying equational axioms.
Such categories are extremely well-behaved: they are \defn{monadic over $\!{Set}$}.
The axioms of partial orders are also of a special type (namely, universal Horn), so that $\!{Pos}$ belongs to the related class of \defn{locally presentable} categories.
See \cref{sec:cat} for a review of the basic theory of such categories.
In particular, they have all (small) limits and colimits, and all of the above forgetful functors between these categories have left adjoints, meaning that we have \defn{free structures}: for example, any poset $P$ generates a free distributive lattice
\begin{align*}
\ang{P}_\!{DLat} = \ang{P \qua \!{Pos}}_\!{DLat},
\end{align*}
which is the universal distributive lattice obtained by adjoining finite meets and joins to $P$ while respecting the distributive lattice axioms as well as pre-existing order relations in $P$.
We adopt this ``\defn{qua}'' notation from \cite{JVpfrm} in order to be specify, when not clear from context, what kind of pre-existing structure we want to preserve.

More generally, we may \defn{present} structures via generators and relations: for example,
\begin{align*}
\ang{P \qua \!{Pos} \mid a \wedge b \le c}_\!{DLat}
\end{align*}
denotes the universal distributive lattice obtained from $\ang{P}_\!{DLat}$ (in this case, by taking a quotient) by imposing the further relation $a \wedge b \le c$, where $a, b, c$ are some elements of $P$.
The general construction of presented structures is by a standard transfinite iteration; see \cref{sec:cat-lim}.
For a category $\!C$ of structures as above, we let $\!C_\kappa \subseteq \!C$ denote the full \defn{subcategory of $\kappa$-presented structures}, meaning those presented by $\kappa$-ary sets of generators and relations.

In the above categories of structures where $\kappa$ or $\lambda$ is $\infty$, there is a proper class of meet and join operations, which means that the structure presented by a set of generators and relations may be a proper class.
In the cases of $\!{Frm}, \!{Cofrm}, \!{{\bigvee}Lat}, \!{{\bigwedge}Lat}$, it is known via an explicit construction that small-generated free structures, and hence small-presented structures, are small; see \cref{thm:frm-small} below.
On the other hand, there is the following classical result of Gaifman~\cite{Gcbool} and Hales~\cite{Hcbool} (see \cite[I~4.10]{Jstone} and \cref{cor:gaifman-hales-2} below for different proofs):

\begin{theorem}[Gaifman--Hales]
\label{thm:gaifman-hales}
The free complete Boolean algebra on $\omega$ generators is a proper class.
Thus, the forgetful functor $\!{CBool} -> \!{Set}$ does not have a left adjoint.
\end{theorem}

\begin{corollary}
\label{cor:gaifman-hales}
None of the forgetful functors from $\!{CBool}$ or $\!{\infty\infty Frm}$ to $\!{\lambda\kappa Frm}, \!{\kappa{\bigvee}Lat}, \!{\kappa{\bigwedge}Lat}$ for $\kappa, \lambda < \infty$, or to $\!{Frm}, \!{Cofrm}, \!{{\bigvee}Lat}, \!{{\bigwedge}Lat}$, have left adjoints.
\end{corollary}
\begin{proof}
For $\!{CBool}$, this follows from \cref{thm:gaifman-hales} because these latter categories all have free functors from $\!{Set}$ (see \cref{thm:frm-small} below for the last four).
For $\!{\infty\infty Frm}$, this likewise follows from the fact that the free algebra $\ang{a_0, b_0, a_1, b_1, \dotsc}_{\!{\infty\infty Frm}}$ is a proper class; otherwise, its quotient by the relations $a_i \wedge b_i = \bot$, $a_i \vee b_i = \top$ would be the free complete Boolean algebra on $a_0, a_1, \dotsc$.
\end{proof}

In this kind of situation, we denote the category of large structures in all-caps, and denote the full subcategory of small-presented large structures with a subscript $_\infty$.
For example, $\!{CBOOL}$ is the category of large complete Boolean algebras, while $\!{CBOOL}_\infty$ is the full subcategory of small-presented algebras.
Thus, the ``left adjoints'' to the forgetful functors in \cref{thm:gaifman-hales} and \cref{cor:gaifman-hales} land in $\!{CBOOL}_\infty$ (respectively, $\!{\infty\infty FRM}_\infty$).
One can treat categories like $\!{CBOOL}$ in the same way as $\!{\kappa CBool}$, by taking $\kappa$ to be inaccessible.
However, we will only need to work with small-presented large structures, which are $\kappa$-presented for some $\kappa < \infty$; thus there are no serious set-theoretic issues involved.
(We could have completely avoided large structures by only talking about $\kappa$-ary operations for $\kappa < \infty$, passing to a larger $\kappa'$ and the free $\kappa'$-ary structure qua $\kappa$-ary structure whenever necessary; however, we feel that this would have cost some clarity.)

The above categories, other than the categories of Boolean algebras, consist of \emph{ordered} algebraic structures, i.e., structures with an underlying poset and order-preserving operations.
Such categories are \defn{locally ordered} (i.e., $\!{Pos}$-enriched), meaning each hom-set is equipped with a partial order (the pointwise order between homomorphisms), such that composition is order-preserving.
Moreover, the various categorical constructions mentioned above, e.g., limits, colimits, free/forgetful adjoint functors, all have ordered analogs.
See \cref{sec:cat-ord} for details.

\subsection{General remarks on presentations and adjunctions}
\label{sec:pres-adj}

In the following subsections, we will discuss various methods for constructing free and presented structures in the above categories, which are more structured, and therefore easier to work with, than the ``naive'' method via sets of generators and relations.

Broadly speaking, we can organize such methods as follows: given any forgetful functor $\!D -> \!C$ between two categories of structures, with a left adjoint free functor $\ang{-}_\!D : \!C -> \!D$, we can think of any $A \in \!C$ as a ``presentation'' for $\ang{A}_\!D \in \!D$.
For example:
\begin{itemize}

\item
A distributive lattice $A$ presents a Boolean algebra $\ang{A}_\!{Bool}$, where the generators are the elements of $A$ and the relations are the pre-existing lattice relations in $A$.

\item
We can also understand ``naive'' presentations in this way: define the category $\!{BoolPres}$ of Boolean algebra presentations, consisting of pairs $(G, R)$ where $G$ is any set and $R$ is a set of equations between Boolean algebra terms; a morphism $(G, R) -> (H, S)$ is a map $G -> H$ which takes every relation in $R$ to some relation in $S$.
The forgetful functor $\!{Bool} -> \!{BoolPres}$ takes a Boolean algebra $A$ to its underlying set together with all equations that hold in $A$.

\end{itemize}
Thus, broadly speaking, our goal is to study properties of adjunctions free/forgetful adjunctions $\!C \rightleftarrows \!D$ between the categories $\!D$ in \eqref{diag:frm-cat} and various other categories $\!C$ of structures (including other categories in \eqref{diag:frm-cat}), which correspond to nice properties of the presentation method for $\!D$ given by $\!C$.
Here are some of the properties we will be interested in:
\begin{itemize}

\item
If the counit $\ang{A}_\!D -> A$ is an isomorphism for each $A \in \!D$, then the presentation method is ``complete'' in that every $A \in \!D$ has a canonical presentation, namely its own underlying object in $\!C$.
For example, this is the case for $\!{Bool} -> \!{DLat}$, but not $\!{Bool} -> \!{{\vee}Lat}$.

By a standard fact about adjunctions (see e.g., \cite[3.4.1]{Bcat}), the counit is an isomorphism iff the right adjoint forgetful functor $\!D -> \!C$ is full and faithful; so we can regard $\!D$ as a reflective full subcategory of $\!C$ in this case.

\item
If the unit $A -> \ang{A}_\!D$ is an embedding for each $A \in \!C$, then presentations are ``saturated'' in that the relations are already closed under all implied relations in the presented structure.
This is the case for all forgetful functors in \eqref{diag:frm-cat}, but not for $\!{Bool} -> \!{BoolPres}$, since e.g., $\ang{a, b \mid a = \neg b}$ does not already contain the implied relation $\neg a = b$.

\item
A general category of structures may have multiple natural notions of ``embedding''.
The weakest one is a monomorphism, which usually just means an injective homomorphism.
In the case of the unit $A -> \ang{A}_\!D$, this means that if a relation in the presentation forces two generators to become identified, then they are already required to be identified in the presenting structure itself.
This is again false for $\!{Bool} -> \!{BoolPres}$, e.g., in $\ang{a, b \mid a = b}$.

Recall that an adjunction unit $A -> \ang{A}_\!D$ is a monomorphism for each $A \in \!C$ iff the left adjoint $\ang{-}_\!D : \!C -> \!D$ is faithful.
This means that every homomorphism $A -> B \in \!C$ of presentations is determined by the induced homomorphism $\ang{A}_\!D -> \ang{B}_\!D$ between the presented structures.

\item
In good categories, the strongest reasonable notion of ''embedding'' is an \defn{extremal monomorphism} (cf.\ \cref{sec:loc-intlog}), which usually means a homomorphism which is not only an embedding but whose image is also closed under all ``positively definable elements''.
For example, this is false for the unit of $\!{DLat} \rightleftarrows \!{Bool}$, since $A \in \!{DLat}$ need not be closed under complements, which are defined by the equations $a \wedge \neg a = \bot$ and $a \vee \neg a = \top$.

Usually (see \cite[end of \S3.4]{Kvcat}), the unit $A -> \ang{A}_\!D$ is an extremal monomorphism for each $A \in \!C$ iff the left adjoint $\ang{-}_\!D : \!C -> \!D$ is \defn{conservative}, meaning that if a homomorphism of presentations $f : A -> B \in \!C$ induces an isomorphism $\ang{f}_\!D : \ang{A}_\!D \cong \ang{B}_\!D$, then $f$ was already an isomorphism.
Note that this is ``orthogonal'' to the first property listed above, of the \emph{co}unit being an isomorphism: if both hold, then the adjunction $\!C \rightleftarrows \!D$ is an equivalence.%
\footnote{There is a precise sense in which these two properties are ``orthogonal complements'': any adjunction between sufficiently nice categories can be factored, essentially uniquely, into one whose left adjoint is conservative, followed by one whose right adjoint is full and faithful; see \cite{Dfactor}.}

\item
If $\ang{-}_\!D : \!C -> \!D$ is faithful and \defn{full on isomorphisms}, meaning each isomorphism $\ang{A}_\!D \cong \ang{B}_\!D$ is induced by an isomorphism $A \cong B$, this means that the presentation may be canonically recovered from the presented structure.
This implies conservativity, but is much rarer.
It tends to hold for order-theoretic ``completions'', like $\!{DLat} \rightleftarrows \!{Frm}$ (see \cref{thm:frm-idl-fulliso}), where the generators become ``atomic'' or ``compact'' in the completed structure.%
\footnote{A general theory of such ``completions'' is given by the theory of \defn{KZ-monads}; see \cite{Kkzmon}, \cref{rmk:upkzfrm-kzmon}.}

\end{itemize}

\subsection{Ideals}
\label{sec:frm-idl}

In this subsection, we recall the standard constructions of the free functors $\!{\kappa{\bigvee}Lat} -> \!{\lambda{\bigvee}Lat}$ and $\!{\kappa Frm} -> \!{\lambda Frm}$ for $\kappa \le \lambda$.

A \defn{lower set} in a poset $A$ is a downward-closed subset $D \subseteq A$ (i.e., $a \le b \in D \implies a \in D$).
Let $\@L(A)$ denote the poset of lower sets in $A$, ordered via $\subseteq$.
Each lower set $D$ can be regarded as a normal form for the $\bigvee$-lattice term $\bigvee D$, so that $\@L(A) \cong \ang{A \qua \!{Pos}}_\!{{\bigvee}Lat}$, i.e., $\@L : \!{Pos} -> \!{{\bigvee}Lat}$ is left adjoint to the forgetful functor.  The adjunction unit is given by the principal ideal embedding
\begin{align*}
\down : A &--> \@L(A) \\
a &|--> \{b \in A \mid b \le a\}
\end{align*}
for each poset $A$.
More generally, let $\@L_\kappa(A) \subseteq \@L(A)$ denote the subset of \defn{$\kappa$-generated lower sets} (i.e., those which are the downward-closure of a $\kappa$-ary subset); then $\@L_\kappa(A) \cong \ang{A \qua \!{Pos}}_\!{\kappa{\bigvee}Lat}$, so that $\@L_\kappa : \!{Pos} -> \!{\kappa{\bigvee}Lat}$ is left adjoint to the forgetful functor.
In particular, the free $\kappa$-$\bigvee$-lattice generated by a set $X$ is the \defn{$\kappa$-ary powerset} $\@P_\kappa(X) \subseteq \@P(X)$ of $\kappa$-ary subsets of $X$.

A poset $A$ is \defn{$\kappa$-directed} if every $\kappa$-ary subset has an upper bound; when $\kappa = \omega$, we say \defn{directed}, which equivalently means $A \ne \emptyset$ and every $a, b \in A$ have an upper bound.
A \defn{$\kappa$-ideal} in a poset $A$ is a $\kappa$-directed lower set; an \defn{ideal} is an $\omega$-ideal.
In a $\kappa$-$\bigvee$-lattice, a $\kappa$-ideal is equivalently a lower set closed under $\kappa$-ary joins.
Let $\kappa\@I(A) \subseteq \@L(A)$ denote the subset of $\kappa$-ideals.
Then $\kappa\@I : \!{\kappa{\bigvee}Lat} -> \!{{\bigvee}Lat}$ is left adjoint to the forgetful functor.
More generally, for $\kappa \le \lambda$, let $\kappa\@I_\lambda(A) \subseteq \@L(A)$ denote the \defn{$\lambda$-generated $\kappa$-ideals}; then $\kappa\@I_\lambda : \!{\kappa{\bigvee}Lat} -> \!{\lambda{\bigvee}Lat}$ is left adjoint to the forgetful functor.

When $A$ is a $\wedge$-lattice, respectively a $\kappa$-frame, then $\@L_\lambda(A)$, resp., $\kappa\@I_\lambda(A)$, is a $\lambda$-frame, and is the free $\lambda$-frame generated by $A$.
(See \cite[\S2]{Mkfrm}.)
In particular, if $A$ is the free $\wedge$-lattice generated by a (po)set, then $\@L_\lambda(A)$ is the free $\lambda$-frame on that same (po)set.
Taking $\lambda = \infty$ shows that despite there being a proper class of join operations,

\begin{corollary}[Bénabou]
\label{thm:frm-small}
Small-generated free frames and free $\bigvee$-lattices are small.

Hence, so are arbitrary small-generated frames and $\bigvee$-lattices, being quotients of free ones.  \qed
\end{corollary}

Summarizing, we have the following diagram of free/forgetful adjunctions, for $\kappa \le \lambda$:
\begin{equation*}
\begin{tikzcd}
\!{{\wedge}Lat} \dar[hook] \rar[shift left=2, "\@L_\omega"] &
\!{DLat} \lar[hook, shift left=2, right adjoint] \dar[hook] \rar[shift left=2, "\@I_\sigma"] &
\!{\sigma Frm} \lar[hook, shift left=2, right adjoint] \dar[hook] \rar[shift left=2, "\sigma\@I_{\omega_2}"] &
\dotsb \lar[hook, shift left=2, right adjoint] \rar[shift left=2] &
\!{\kappa Frm} \lar[hook, shift left=2, right adjoint] \dar[hook] \rar[shift left=2, "\kappa\@I_\lambda"] &
\!{\lambda Frm} \lar[hook, shift left=2, right adjoint] \dar[hook] \rar[shift left=2, "\lambda\@I"] &
\!{Frm} \lar[hook, shift left=2, right adjoint] \dar[hook]
\\
\!{Pos} \rar[shift left=2, "\@L_\omega"] &
\!{{\vee}Lat} \lar[hook, shift left=2, right adjoint] \rar[shift left=2, "\@I_\sigma"] &
\!{\sigma{\bigvee}Lat} \lar[hook, shift left=2, right adjoint] \rar[shift left=2, "\sigma\@I_{\omega_2}"] &
\dotsb \lar[hook, shift left=2, right adjoint] \rar[shift left=2] &
\!{\kappa{\bigvee}Lat} \lar[hook, shift left=2, right adjoint] \rar[shift left=2, "\kappa\@I_\lambda"] &
\!{\lambda{\bigvee}Lat} \lar[hook, shift left=2, right adjoint] \rar[shift left=2, "\lambda\@I"] &
\!{{\bigvee}Lat} \lar[hook, shift left=2, right adjoint]
\end{tikzcd}
\end{equation*}
Note that the squares with horizontal \emph{left} adjoints (and vertical right adjoints) also commute; this fact does not extend beyond the bottom two rows of the diagram \eqref{diag:frm-cat}.

The unit for all of these adjunctions is given by $\down$, which is an (order-)embedding.  By a standard fact about adjunctions (see e.g., \cite[end of \S3.4]{Kvcat}), this equivalently means that the left adjoints are (order-)faithful, i.e., restrict to (order-)embeddings on each hom-(po)set.

An unusual feature of these adjunctions, compared to arbitrary free/forgetful adjunctions between categories of algebras,
is that the generators of free algebras are uniquely determined.
For a $\lambda$-$\bigvee$-lattice $B$ and $\kappa \le \lambda$, an element $b \in B$ is called \defn{$\kappa$-compact}%
\footnote{called a \defn{$\kappa$-element} in \cite{Mkfrm}}
if whenever $D \in \kappa\@I_\lambda(B)$ with $b \le \bigvee D$, then $b \in D$.
This is easily seen to be equivalent to: every $\lambda$-ary cover of $b$ has a $\kappa$-ary subcover, i.e., whenever $b \le \bigvee C$ for a $\lambda$-ary $C \subseteq B$, then $b \le \bigvee D$ for a $\kappa$-ary $D \subseteq C$.
Let
\begin{align*}
B_\kappa := \{b \in B \mid b \text{ is $\kappa$-compact}\}.
\end{align*}
We call $B$ \defn{$\kappa$-compactly based} if it is generated by $B_\kappa$ under $\lambda$-ary joins.
If $B$ is a $\lambda$-frame which is $\kappa$-compactly based, and $B_\kappa \subseteq B$ is also closed under finite meets, then we call $B$ \defn{$\kappa$-coherent}.
The following is routine; see \cite[1.4]{Mkfrm}.

\begin{proposition}
\label{thm:vlat-cptbasis}
\leavevmode
\begin{enumerate}
\item[(a)]  For $\kappa \le \mu \le \lambda$ and $A \in \!{\kappa{\bigvee}Lat}$, we have $\kappa\@I_\lambda(A)_\mu = \kappa\@I_\mu(A)$, thus $\kappa\@I_\lambda(A)$ is $\mu$-compactly based ($\mu$-coherent if $A \in \!{\kappa Frm}$).

In particular, $\kappa\@I(A)_\lambda = \kappa\@I_\lambda(A)$, and $\kappa\@I_\lambda(A)_\kappa = \kappa\@I_\kappa(A) = \down(A) \cong A$ consists of the principal ideals, thus $\kappa\@I_\lambda(A)$ is $\kappa$-compactly based ($\kappa$-coherent if $A \in \!{\kappa Frm}$).

\item[(b)]  If $B \in \!{\lambda{\bigvee}Lat}$ is $\kappa$-compactly based, then we have an isomorphism $\bigvee : \kappa\@I_\lambda(B_\kappa) \cong B$.
\end{enumerate}
Thus, the image of $\kappa\@I_\lambda : \!{\kappa{\bigvee}Lat} -> \!{\lambda{\bigvee}Lat}$ (respectively, $\kappa\@I_\lambda : \!{\kappa Frm} -> \!{\lambda Frm}$) consists, up to isomorphism, of precisely the $\kappa$-compactly based $\lambda$-$\bigvee$-lattices (resp., the $\kappa$-coherent $\lambda$-frames).
\end{proposition}

\begin{remark}
\label{rmk:frm-idl-coh-free}
For $\kappa \le \mu \le \lambda$, whether a $\lambda$-frame $B$ is $\kappa$-coherent can depend on whether $B$ is regarded as a $\lambda$-frame or as a $\mu$-frame.  For example, if $\mu = \kappa$, clearly every $\kappa$-frame is $\kappa$-coherent, whereas not every $\lambda$-frame is $\kappa$-coherent.

What is true is that the notion of $\kappa$-coherent frame is preserved by the \emph{free} functor $\mu\@I_\lambda : \mu\!{Frm} -> \lambda\!{Frm}$ (since $\mu\@I_\lambda \circ \kappa\@I_\mu \cong \kappa\@I_\lambda$), rather than the forgetful functor.
\end{remark}

The categorical meaning of \cref{thm:vlat-cptbasis} is

\begin{corollary}
\label{thm:frm-idl-fulliso}
For $\kappa \le \lambda$, the free functors $\kappa\@I_\lambda : \!{\kappa{\bigvee}Lat} -> \!{\lambda{\bigvee}Lat}$ and $\kappa\@I_\lambda : \!{\kappa Frm} -> \!{\lambda Frm}$ are full on isomorphisms.
\end{corollary}
\begin{proof}
Every isomorphism $\kappa\@I_\lambda(A) -> \kappa\@I_\lambda(B)$ is induced by its restriction to $A \cong \kappa\@I_\lambda(A)_\kappa -> \kappa\@I_\lambda(A)_\kappa \cong B$.
\end{proof}

Recall from \cref{sec:frm-cat} that $\!{\kappa{\bigvee}Lat}_\lambda \subseteq \!{\kappa{\bigvee}Lat}$ and $\!{\kappa Frm}_\lambda \subseteq \!{\kappa Frm}$ denote the full subcategories of $\lambda$-presented structures.
Clearly, free functors preserve $\lambda$-presentability.
Furthermore, a $\kappa$-frame is $\lambda$-presented (resp., $\lambda$-generated) as a $\kappa$-frame iff it is so as a $\kappa$-$\bigvee$-lattice.
For $\lambda$-generated, this is because if $A$ is $\lambda$-generated as a $\kappa$-frame by $B \subseteq A$, then it is $\lambda$-generated as a $\kappa$-$\bigvee$-lattice by the closure of $B$ under finite meets.
For $\lambda$-presented, see \cref{thm:frm-pres-vlat} below.

The categories $\!{\kappa Frm}_\lambda$, as $\kappa, \lambda$ vary, are related as in the following diagram in which $\kappa \le \lambda$ (similarly for $\!{\kappa{\bigvee}Lat}_\lambda$):
\begin{equation*}
\begin{tikzcd}
\dotsb
    \rar[shift left=2] &
\kappa\!{Frm}
    \lar[shift left=2, right adjoint]
    \rar[shift left=2, "{\kappa\@I_\lambda}"] &
\lambda\!{Frm}
    \lar[shift left=2, right adjoint]
    \rar[shift left=2] &
\dotsb
    \lar[shift left=2, right adjoint]
    \rar[shift left=2] &
\!{Frm}
    \lar[shift left=2, right adjoint]
\\
&
\vdots
    \uar[phantom, "\subseteq"{sloped}] &
\vdots
    \uar[phantom, "\subseteq"{sloped}] &
|[xscale=-1]|\ddots &
\vdots
    \uar[phantom, "\subseteq"{sloped}]
\\
\dotsb
    \rar[shift left=2] &
\kappa\!{Frm}_\lambda
    \uar[phantom, "\subseteq"{sloped}]
    \rar[shift left=2, "{\kappa\@I_\lambda}"] &
\lambda\!{Frm}_\lambda
    \uar[phantom, "\subseteq"{sloped}]
    \rar[shift left=2, "\simeq"'] &
\dotsb
    \rar[shift left=2, "\simeq"'] &
\!{Frm}_\lambda
    \uar[phantom, "\subseteq"{sloped}]
\\
\dotsb
    \rar[shift left=2] &
\kappa\!{Frm}_\kappa
    \uar[phantom, "\subseteq"{sloped}]
    \rar[shift left=2, "{\kappa\@I_\lambda}", "\simeq"'] &
\lambda\!{Frm}_\kappa
    \uar[phantom, "\subseteq"{sloped}]
    \rar[shift left=2, "\simeq"'] &
\dotsb
    \rar[shift left=2, "\simeq"'] &
\!{Frm}_\kappa
    \uar[phantom, "\subseteq"{sloped}]
\\
|[xscale=-1]|
\ddots
&
\vdots
    \uar[phantom, "\subseteq"{sloped}] &
\vdots
    \uar[phantom, "\subseteq"{sloped}] &
&
\vdots
    \uar[phantom, "\subseteq"{sloped}]
\end{tikzcd}
\end{equation*}
Note that the free functors in each row stabilize past the diagonal.
This is due to the following simple fact, which says that in $\kappa$-generated $\kappa$-$\bigvee$-lattices, $\kappa$-ary joins already determine arbitrary joins.
Thus, ``$\kappa$-presented $\kappa$-frame'' and ``$\kappa$-presented frame'' mean exactly the same thing.

\begin{proposition}
\label{thm:frm-idl-kgen}
Let $A$ be a $\kappa$-generated $\kappa$-$\bigvee$-lattice.
Then every $\kappa$-ideal in $A$ is principal, i.e., $\down : A \cong \kappa\@I(A)$, whence $A$ is a $\bigvee$-lattice, and $A = A_\kappa$.

Thus, for $\kappa \le \lambda$, the free/forgetful adjunctions restrict to equivalences of categories
\begin{align*}
\!{\kappa{\bigvee}Lat}_\kappa \simeq \!{\lambda{\bigvee}Lat}_\kappa, &&
\!{\kappa Frm}_\kappa \simeq \!{\lambda Frm}_\kappa
\end{align*}
(more generally, between the full subcategories of $\kappa$-generated objects).
\end{proposition}
\begin{proof}
For the first statement: letting $B \subseteq A$ be a $\kappa$-ary generating set, every $\kappa$-ideal $D \subseteq A$ is equal to the principal ideal $\down \bigvee (D \cap B)$.
For the second statement: the first statement implies that the unit $\down$ of the adjunction is an isomorphism on $\kappa$-presented (resp., $\kappa$-generated) objects, thus the left adjoint $\kappa\@I_\lambda$ is full and faithful on the full subcategory of such objects; it is also essentially surjective onto the full subcategory of $\kappa$-presented (resp., $\kappa$-generated) objects by the first statement.
\end{proof}

We have the obvious order-duals of the above notions for the categories $\!{\kappa{\bigwedge}Lat}$ and $\!{\kappa Cofrm}$.
An \defn{upper set} is an upward-closed set; a \defn{$\kappa$-filter} is an upper set which is also \defn{$\kappa$-codirected}, or equivalently in a $\kappa$-$\bigwedge$-lattice, closed under $\kappa$-ary meets.
For $\kappa \le \lambda$, the set of $\lambda$-generated $\kappa$-filters, ordered by \emph{reverse} inclusion, yields the free functor $\!{\kappa{\bigwedge}Lat} -> \!{\lambda{\bigwedge}Lat}$, with unit given by the principal filter embedding
$\up a := \{b \mid b \ge a\}$.

\subsection{Quotients}
\label{sec:frm-quot}

In this subsection, we recall some basic facts about quotients of frames, Boolean algebras, and related structures.
This material is standard; in particular, see \cite[II~\S2]{Jstone} for the description of frame quotients via nuclei.
However, our account takes a somewhat different viewpoint than what is usually found in the literature, focusing more on the ordered aspects of quotients, as well as on the general context of joins before also considering finite meets.

Recall that quotients of an algebraic structure $A$ are determined by \defn{congruences} $\sim$ on $A$, i.e., equivalence relations which are also subalgebras of $A^2$.
Given a homomorphism $f : A -> B$, its \defn{kernel} $\ker(f) := \{(a, a') \in A^2 \mid f(a) = f(a')\}$ is a congruence, whose quotient is isomorphic to the image of $f$.
Conversely, any congruence $\sim$ on $A$ is the kernel of its quotient map $A ->> A/{\sim}$.
This gives a bijection between congruences on $A$ and equivalence classes of surjective homomorphisms from $A$.
See \cref{sec:cat-alg} for details.

For ordered algebraic structures, in the sense of \cref{sec:cat-ord}, we have the following ordered analogs of these notions.
The \defn{order-kernel} of a homomorphism $f : A -> B$ is
\begin{align*}
\oker(f) := \{(a, a') \in A^2 \mid f(a) \le f(a')\},
\end{align*}
and is an \defn{order-congruence} on $A$, i.e., a preorder $\lesim$ which contains the partial order $\le_A$ on $A$ and is also a subalgebra of $A^2$.
Given an order-congruence $\lesim$, we always let ${\sim} := {\lesim} \cap {\lesim}^{-1}$ denote its symmetric part, which is a congruence such that $\lesim$ descends to a partial order on $A/{\sim}$ making $\lesim$ the order-kernel of the quotient map $A ->> A/{\sim}$.
This again gives a bijection between order-congruences on $A$ and equivalence classes of surjective homomorphisms from $A$.
See \cref{sec:cat-ord} for details.

For a $\kappa$-$\bigvee$-lattice $A$, it is easily seen that a binary relation ${\lesim} \subseteq A^2$ is an order-congruence iff
\begin{itemize}
\item  $\lesim$ is a preorder containing $\le_A$; and
\item  $\kappa$-ary joins $\bigvee_i a_i$ in $A$ are also joins with respect to $\lesim$, i.e., if $a_i \lesim b$ for every $i$, then $\bigvee_i a_i \lesim b$.
\end{itemize}
In terms of the second-coordinate fibers of ${\lesim} \subseteq A^2$, i.e., the \defn{principal $\lesim$-ideals}
\begin{align*}
\simdown a := \{b \in A \mid b \lesim a\},
\end{align*}
the above conditions are equivalent to
\begin{itemize}
\item  each $\simdown a$ is a $\kappa$-ideal, and $\simdown : A -> \kappa\@I(A)$ is monotone;
\item  $\down \le \simdown$ (i.e., $\down a \subseteq \simdown a$ for each $a$);
\item  $\simdown_* \circ \simdown \le \simdown$ (i.e., $\bigcup_{b \in \simdown a} \simdown b \subseteq \simdown a$ for each $a$), where
\begin{align*}
\simdown_* : \kappa\@I(A) &--> \kappa\@I(A) \\
D &|--> \bigcup_{a \in D} \simdown a
\end{align*}
is the $\kappa$-directed-join-preserving extension of $\simdown : A -> \kappa\@I(A)$ along $\down$.
\end{itemize}
Call a map $\simdown : A -> \kappa\@I(A)$ obeying these conditions a \defn{$\kappa$-ideal closure operator} on $A$.
These are in bijection with order-congruences $\lesim$ on $A$, with $\lesim$ recovered as $\oker(\simdown)$.

When $\kappa = \infty$, for a $\bigvee$-lattice $A$, we have $\infty\@I(A) \cong A$, so we may equivalently regard $\simdown$ as a monotone map $A -> A$ which is a \defn{closure operator}, i.e., satisfying
\begin{align*}
1_A \le \simdown, &&
\simdown \circ \simdown \le \simdown \quad(\text{thus } \simdown \circ \simdown = \simdown).
\end{align*}
Each congruence class $[a]$ (for the order-congruence corresponding to $\simdown$) has a greatest element, namely $\simdown a$, so that the quotient $\bigvee$-lattice $A/{\sim}$ may be isomorphically realized as the image $\simdown(A)$ of $\simdown : A -> A$, equivalently its set of fixed points, with join in $\simdown(A)$ given by join in $A$ followed by applying $\simdown$.
The quotient map $A ->> A/{\sim}$ is then identified with $\simdown : A ->> \simdown(A)$.

For $\kappa < \infty$, we may instead consider the $\kappa$-directed-join-preserving extension $\simdown_* : \kappa\@I(A) -> \kappa\@I(A)$ of $\simdown$, which is a closure operator on $\kappa\@I(A)$, whose fixed points are those $\kappa$-ideals $D \in \kappa\@I(A)$ which are also $\lesim$-downward-closed.
The quotient $A/{\sim}$ may be identified with the subset $\simdown(A) \subseteq \simdown_*(\kappa\@I(A))$ consisting of the principal $\lesim$-ideals, i.e., the sets $\simdown a$ for $a \in A$; these are the $\kappa$-compact elements in $\simdown_*(\kappa\@I(A))$.
The quotient map $A ->> A/{\sim}$ is again identified with $\simdown$.

For a $\kappa$-frame $A$, a $\kappa$-$\bigvee$-lattice order-congruence ${\lesim} \subseteq A^2$ is a $\kappa$-frame order-congruence iff ${\lesim}$ is closed under binary meets, which is easily seen to be equivalent to either of:
\begin{itemize}
\item  binary meets in $A$ are meets with respect to $\lesim$; or equivalently,
\item  $a \lesim b \implies c \wedge a \lesim c \wedge b$ (we say that $\lesim$ is \defn{$\wedge$-stable}).
\end{itemize}
In terms of the corresponding $\kappa$-ideal closure operator $\simdown : A -> \kappa\@I(A)$, this means
\begin{itemize}
\item  $\simdown : A -> \kappa\@I(A)$ preserves binary meets; or equivalently,
\item  $a \wedge \simdown b \subseteq \simdown(a \wedge b)$.
\end{itemize}
We call such $\simdown$ a \defn{$\kappa$-ideal nucleus}.
When $\kappa = \infty$, regarding $\simdown : A -> A$ as before, $\simdown$ is called a \defn{nucleus}.

An arbitrary binary relation ${\prec} \subseteq A^2$ generates an order-congruence, which may be obtained by closing $\prec$ under the above conditions via the usual transfinite recursion.
The following results say that it suffices to close under a subset of the conditions:

\begin{proposition}
\label{thm:frm-cong-ltrans}
Let $A$ be a $\kappa$-$\bigvee$-lattice, $\prec$ be a binary relation on $A$.
Let ${\lesim} \subseteq A^2$ be the smallest binary relation which is
(i)~reflexive, and satisfies
(ii)~for $<\kappa$-many $a_i \lesim b$, we have $\bigvee_i a_i \lesim b$,
(iii)~$a \le b \lesim c \implies a \lesim c$, and
(iv)~$a \prec b \lesim c \implies a \lesim c$.
Then $\lesim$ is transitive, hence is the order-congruence generated by $\prec$.
\end{proposition}
\begin{proof}
Given $a \lesim b \lesim c$, one may induct on the derivation of $a \lesim b$ to show $a \lesim c$.

Alternatively, let $\lesim'$ be the binary relation defined by $a \lesim' b$ iff for every $c \in A$, if $b \lesim c$, then $a \lesim c$.
Then one easily checks that $\lesim'$ obeys (i--iv), hence contains $\lesim$.
\end{proof}

The significance of conditions (i--iv) in this result is that they only involve $\lesim$ with a fixed right-hand side, hence translate to the following closure conditions on the fibers $\simdown a$:

\begin{corollary}
\label{thm:frm-nuc-ltrans}
Let $A$ be a $\kappa$-$\bigvee$-lattice, $\prec$ be a binary relation on $A$, $\lesim$ be the order-congruence generated by $\prec$, and $\simdown : A -> \kappa\@I(A)$ be the corresponding $\kappa$-ideal closure operator with $\kappa$-directed-join-preserving extension $\simdown_* : \kappa\@I(A) -> \kappa\@I(A)$.
Then for each $a \in A$, the principal $\lesim$-ideal $\simdown a \subseteq A$ is the smallest $\kappa$-ideal containing $a$ which is also $\prec$-downward-closed.
Thus more generally, a $\kappa$-ideal $D \in \kappa\@I(A)$ is a fixed point of $\simdown_*$, i.e., $\lesim$-downward-closed, iff it is $\prec$-downward-closed.
\qed
\end{corollary}

Concerning $\kappa$-frame order-congruences, we have

\begin{proposition}
\label{thm:frm-cong-meet}
Let $A$ be a $\kappa$-frame, $\prec$ be a $\wedge$-stable binary relation on $A$.
Then the $\kappa$-$\bigvee$-lattice order-congruence $\lesim$ generated by $\prec$ is still $\wedge$-stable, hence is also the $\kappa$-frame order-congruence generated by $\prec$.
\end{proposition}
\begin{proof}
By transfinite induction, or by checking that $a \lesim' b \coloniff \forall c \in A\, (c \wedge a \lesim c \wedge b)$ is a $\kappa$-$\bigvee$-lattice order-congruence containing $\prec$.
\end{proof}

For a $(\lambda, \kappa)$-frame $A$ where $\lambda > \omega$, the first characterization of $\kappa$-frame order-congruences above generalizes straightforwardly, i.e., $\lambda$-ary meets in $A$ should also be $\lesim$-meets.
Similarly, the corresponding $\kappa$-ideal closure operator $\simdown : A -> \kappa\@I(A)$ should preserve $\lambda$-ary meets.
However, these notions are not as useful, due to the lack of an analog of the $\wedge$-stability characterization.

By an \defn{upper (order-)congruence} on a $(\lambda, \kappa)$-frame $A$, we mean one generated by declaring $\top \sim a$ (equivalently, $\top \lesim a$) for all $a$ in some $U \subseteq A$; we will denote such a (order-)congruence by $\sim^U$ ($\lesim^U$).
Clearly, ${\sim^U} = {\sim^{U'}}$ where $U'$ is the $\lambda$-filter generated by $U$.
If $U$ is a $\lambda$-filter, it is \emph{not} necessarily the case that there is no bigger $U' \supseteq U$ with ${\sim^U} = {\sim^{U'}}$ (see \cref{rmk:frm-dpoly-bifrm-compl} below), unless $U$ is also a $\kappa$-filter (e.g., if $\lambda \ge \kappa$), in which case it is easily seen that
\begin{align*}
b \sim^U c  \iff  \exists a \in U\, (a \wedge b = a \wedge c), &&
b \lesim^U c  \iff  \exists a \in U\, (a \wedge b \le c).
\end{align*}
It follows that $U$ may be recovered from $\sim^U$ as the congruence class $[\top]$.
We call such $\sim^U$ a \defn{$\kappa$-filter congruence}, and call $A/{\sim^U}$ a \defn{$\kappa$-filterquotient} of $A$.
In particular, for a principal filter $\up a$, we have
$b \sim^{\up a} c \iff a \wedge b = a \wedge c$,
hence $A/{\sim^{\up a}}$ may be isomorphically realized as the principal ideal $\down a \subseteq A$ with quotient map $a \wedge (-) : A ->> \down a$.
We write ${\sim^a} := {\sim^{\up a}}$, and call $\down a$ the \defn{principal filterquotient} at $a$ when we are regarding it as the quotient by $\sim^a$.

For a $\kappa$-Boolean algebra $A$, owing to $a = b \iff a <-> b = \top$, every congruence is a $\kappa$-filter congruence, corresponding to the $\kappa$-filter $[\top] \subseteq A$.
Moreover, if $\sim$ is a $\kappa$-generated congruence, then $[\top]$ is a $\kappa$-generated $\kappa$-filter, hence a principal filter, namely $\up a$ for $a = \bigwedge_i (b_i <-> c_i)$ for some $\kappa$-ary set of generators $(b_i, c_i)$ for $\sim$.

\subsection{Products and covers}
\label{sec:frm-prod}

In this subsection, we recall a standard characterization of products of frames and Boolean algebras, via pairwise disjoint covers.
We then prove some preservation results for products as well as for general covers.

A \defn{$\kappa$-ary partition of $\top$} in a $\kappa$-frame $A$ is a $\kappa$-ary family of \defn{pairwise disjoint} elements $a_i \in A$, i.e., $a_i \wedge a_j = \bot$ for $i \ne j$, which \defn{cover} $\top$, i.e., $\bigvee_i a_i = \top$.

Given a $\kappa$-ary product $\prod_i A_i$ of $\kappa$-frames, the elements $\delta_i \in \prod_i A_i$ which are $\top$ in the $i$th coordinate and $\bot$ elsewhere form a $\kappa$-ary partition of $\top$.
Moreover, each projection $\pi_i : \prod_j A_j -> A_i$ descends to an isomorphism between the principal filterquotient $\down \delta_i$ and $A_i$.
Conversely, given any $\kappa$-frame $A$ with a $\kappa$-ary partition $(\delta_i)_i$ of $\top$, we have an isomorphism
\begin{align*}
A &\cong \prod_i \down \delta_i \\
a &|-> (\delta_i \wedge a)_i \\
\bigvee_i a_i &<-| (a_i)_i.
\end{align*}
This description of $\kappa$-ary products also applies to $\kappa$-Boolean algebras and $(\lambda, \kappa)$-frames (for $\lambda \ge \omega$).
Using this, we have

\begin{proposition}
\label{thm:frm-prod-pushout}
$\kappa$-ary products in $\!{\kappa Frm}, \!{\lambda\kappa Frm}, \!{\kappa Bool}$ are pushout-stable: given $<\kappa$-many objects $A_i$ in one of these categories, and a homomorphism $f : \prod_i A_i -> B$, letting
\begin{equation*}
\begin{tikzcd}
\prod_i A_i \dar["f"'] \rar["\pi_i"] &
A_i \dar["f_i"] \\
B \rar["g_i"'] &
B_i
\end{tikzcd}
\end{equation*}
be pushout squares, we have that the $g_i$ exhibit $B$ as an isomorphic copy of the product $\prod_i B_i$.
\end{proposition}
\begin{proof}
Let $\delta_i \in \prod_i A_i$ be the partition of $\top$ defined above, so that each $A_i \cong \down \delta_i$.
Then the $f(\delta_i) \in B$ form a partition of $\top$, hence $B \cong \prod_i \down f(\delta_i)$.
Each $\down f(\delta_i)$ is the quotient of $B$ identifying $f(\delta_i)$ with $\top = f(\top)$, hence the quotient of $B$ by the image under $f$ of the congruence on $\prod_i A_i$ generated by $\delta_i \sim \top$ whose quotient is $A_i \cong \down \delta_i$; this exactly describes the pushout of $A_i$ across $f$.
\end{proof}

We next derive a presentation of a product from presentations of the factors.
More generally, given a $\kappa$-frame $A$ with a $\kappa$-ary cover $(a_i)_i$ of $\top$, we derive a presentation of $A$ from presentations of the principal filterquotients $\down a_i$.

For this purpose, it is convenient to introduce the category $\!{\lambda\kappa LFrm}$ of \defn{local $(\lambda, \kappa)$-frames}, meaning $(\lambda, \kappa)$-frames but possibly without a top element (thus, the operations are $\kappa$-ary joins and nonempty $\lambda$-ary meets).
For example, every $\kappa$-ideal in a $(\lambda, \kappa)$-frame $A$ is a local $(\lambda, \kappa)$-subframe; thus for a principal filterquotient $\down a$, the inclusion $\down a `-> A$ is a local $(\lambda, \kappa)$-subframe homomorphism which is a section of the quotient map $a \wedge (-) : A ->> \down a$.
Note that the facts about quotients from the preceding subsection all generalize easily from $(\lambda, \kappa)$-frames to local $(\lambda, \kappa)$-frames.

\begin{proposition}
\label{thm:frm-lfrm-idl-pres}
Let $A$ be a local $(\lambda, \kappa)$-frame, $A_i \subseteq A$ be $<\kappa$-many $\kappa$-ideals which together generate $A$ under $\kappa$-ary joins.
Then
\begin{align*}
A = \ang{A_i \qua \!{\lambda\kappa LFrm} \text{ for each $i$} \mid (A_i \ni a_i) \wedge (a_j \in A_j) = a_i \wedge a_j \in A_i}_\!{\lambda\kappa Frm}.
\end{align*}
(More precisely, $A$ is the quotient of the coproduct $\coprod_i A_i$ by the congruence $\sim$ generated by
\begin{align*}
\tag{$*$}
\iota_i(a_i) \wedge \iota_j(a_j) \sim \iota_i(a_i \wedge a_j)
\quad\text{for $a_i \in A_i$ and $a_j \in A_j$},
\end{align*}
where the $\iota_i : A_i -> \coprod_j A_j$ are the cocone maps.)
\end{proposition}
\begin{proof}
For each $a_i \in A_i$, by considering the relations ($*$) in the principal filterquotient $\down \iota_i(a_i) \subseteq \coprod_i A_i$, we get that every generator $\iota_i(a_i) \wedge \iota_j(a_j)$, hence every element, of $\down \iota_i(a_i)$ is $\sim \iota_i(b)$ for some $b \le a_i$.
Every $b \in \coprod_i A_i$ is $\le \bigvee_i \iota_i(a_i)$ for some $a_i \in A_i$, since the set of such elements clearly contains the image of each $\iota_i$ and is a local $(\lambda, \kappa)$-subframe; thus $b = \bigvee_i (\iota_i(a_i) \wedge b) \sim \bigvee_i \iota_i(b_i)$ for some $b_i \le a_i \in A_i$.
For two elements of this latter form, $b = \bigvee_i \iota_i(b_i)$ and $c = \bigvee_i \iota_i(c_i)$, such that $\bigvee_i b_i = \bigvee_i c_i \in A$, we have $\iota_i(b_i) = \bigvee_j \iota_i(b_i \wedge c_j) \lesim \bigvee_j \iota_j(c_j) = c$ by ($*$) for each $i$, whence $b \lesim c$; similarly $c \lesim b$, whence $b \sim c$.
This shows that $\sim$ is the kernel of the canonical map $\coprod_i A_i ->> A$.
\end{proof}

Note that both sides of ($*$) above are local $(\lambda, \kappa)$-frame homomorphisms as functions of either $a_i$ or $a_j$.
Thus, it is enough to consider ($*$) where $a_i, a_j$ are generators of $A_i, A_j$.
This yields

\begin{corollary}
\label{thm:frm-lfrm-idl-kpres}
In \cref{thm:frm-lfrm-idl-pres}, if the $A_i$ are $\mu$-presented for some $\mu \ge \kappa$, then so is $A$.
\qed
\end{corollary}

\begin{remark}
\label{rmk:frm-lfrm-disj-pres}
In \cref{thm:frm-lfrm-idl-pres}, if the $A_i$ are principal ideals $\down a_i$ such that the $a_i$ are pairwise disjoint, then $\sim$ is generated simply by
$\iota_i(a_i) \wedge \iota_j(a_j) \sim \bot$ for $i \ne j$.
\end{remark}

We now connect local $(\lambda, \kappa)$-frame presentations with $(\lambda, \kappa)$-frame presentations.
Let
\begin{align*}
A = \ang{G \mid R}_\!{\lambda\kappa Frm}
\end{align*}
be a presented $(\lambda, \kappa)$-frame.
We may turn this into a presentation of $A$ as a local $(\lambda, \kappa)$-frame, as follows.
Let $G' := G \sqcup \{\top_A\}$ where $\top_A$ is a new symbol.
For a $(\lambda, \kappa)$-frame term $t$ over $G$, we may turn it into a local $(\lambda, \kappa)$-frame term $t'$ over $G'$ by replacing every occurrence of the nullary meet operation $\top$ in $t$ with the new symbol $\top_A$.
Now let $R'$ be $R$ with every term $t$ in it replaced with $t'$ just defined, together with the new relations $x \le \top_A$ for every $x \in G$.
These new relations ensure that $\top_A$ is the top element of $\ang{G' \mid R'}_\!{\lambda\kappa LFrm}$, whence the latter is a $(\lambda, \kappa)$-frame; and the operation $t |-> t'$ on terms yields a $(\lambda, \kappa)$-frame homomorphism $f : A -> \ang{G' \mid R'}_\!{\lambda\kappa LFrm}$.
Conversely, the relations in $R'$ are clearly satisfied in $A$, whence we get a local $(\lambda, \kappa)$-frame homomorphism $g : \ang{G' \mid R'}_\!{\lambda\kappa LFrm} -> A$.
It is easily seen that $g$ is surjective and $f \circ g = 1$, whence $f = g^{-1}$, i.e.,
\begin{align*}
A = \ang{G' \mid R'}_\!{\lambda\kappa LFrm}.
\end{align*}
Using this transformation $(G, R) |-> (G', R')$, we have

\begin{proposition}
\label{thm:frm-lfrm-pres->}
If a $(\lambda, \kappa)$-frame is $\mu$-presented, then it is $\mu$-presented as a local $(\lambda, \kappa)$-frame.
\qed
\end{proposition}

Conversely, given a local $(\lambda, \kappa)$-frame $A$, the free $(\lambda, \kappa)$-frame generated by $A$ is its \defn{scone}
\begin{align*}
A_\top := A \sqcup \{\top\}
\end{align*}
where $\top = \top_{A_\top}$ is a new top element, strictly greater than all elements of $A$.
If $A$ was already a $(\lambda, \kappa)$-frame, then $A \subseteq A_\top$ is the principal ideal $\down \top_A$, i.e., the principal filterquotient by the congruence $\sim^{\top_A}$ identifying $\top_A$ with $\top_{A_\top}$.
Along with \cref{thm:frm-lfrm-pres->}, this yields

\begin{corollary}
\label{thm:frm-lfrm-pres}
A $(\lambda, \kappa)$-frame is $\mu$-presented iff it is $\mu$-presented as a local $(\lambda, \kappa)$-frame.
\qed
\end{corollary}

By \cref{thm:frm-lfrm-idl-kpres}, we get

\begin{corollary}
\label{thm:frm-cover-pres}
If a $(\lambda, \kappa)$-frame $A$ has a $\kappa$-ary cover $(a_i)_i$ of $\top$ such that each $\down a_i$ is $\mu$-presented for some $\mu \ge \kappa$, then $A$ is $\mu$-presented.
\qed
\end{corollary}

\begin{corollary}
\label{thm:frm-prod-pres}
For $\mu \ge \kappa$, a $\kappa$-ary product of $\mu$-presented $(\lambda, \kappa)$-frames is $\mu$-presented.
\qed
\end{corollary}

The above constructions of presentations also yield

\begin{proposition}
\label{thm:frm-lfrm-compl}
For a $(\lambda, \kappa)$-frame $A$ and $\lambda' \ge \lambda$, $\kappa' \ge \kappa$, we have $\ang{A \qua \!{\lambda\kappa Frm}}_\!{\lambda'\kappa' Frm} = \ang{A \qua \!{\lambda\kappa LFrm}}_\!{\lambda'\kappa' LFrm}$, i.e., the square
\begin{equation*}
\begin{tikzcd}
\!{\lambda\kappa Frm}
    \dar[shift left=2, right adjoint']
    \rar[shift left=2, "\ang{-}_\!{\lambda'\kappa' Frm}"] &
\!{\lambda'\kappa' Frm}
    \lar[shift left=2, right adjoint]
    \dar[shift left=2, right adjoint'] \\
\!{\lambda\kappa LFrm}
    \uar[shift left=2, "{(-)_\top}"]
    \rar[shift left=2, "\ang{-}_\!{\lambda'\kappa' LFrm}"] &
\!{\lambda'\kappa' LFrm}
    \lar[shift left=2, right adjoint]
    \uar[shift left=2, "{(-)_\top}"]
\end{tikzcd}
\end{equation*}
consisting of the vertical \emph{right} and horizontal \emph{left} adjoints commutes.
\end{proposition}
\begin{proof}
Let $A
= \ang{G \mid R}_\!{\lambda\kappa Frm}
= \ang{G' \mid R'}_\!{\lambda\kappa LFrm}$
where $(G', R')$ is obtained from $(G, R)$ via the procedure described above.
Note that the above procedure works equally well for $\lambda', \kappa'$ as for $\lambda, \kappa$.
Thus
$\ang{A \qua \!{\lambda\kappa Frm}}_\!{\lambda'\kappa' Frm}
= \ang{G \mid R}_\!{\lambda'\kappa' Frm}
= \ang{G' \mid R'}_\!{\lambda'\kappa' LFrm}
= \ang{A \qua \!{\lambda\kappa LFrm}}_\!{\lambda'\kappa' LFrm}$.
\end{proof}

\begin{proposition}
\label{thm:frm-prod-compl}
For $\lambda' \ge \lambda$, $\kappa' \ge \kappa$, the free functor $\ang{-}_\!{\lambda'\kappa' Frm} : \!{\lambda\kappa Frm} -> \!{\lambda'\kappa' Frm}$ preserves $\kappa$-ary products.
\end{proposition}
\begin{proof}
Let $A_i \in \!{\lambda\kappa Frm}$ be $<\kappa$-many $(\lambda, \kappa)$-frames.
Let $\prod_i A_i = \ang{G \mid R}_\!{\lambda\kappa LFrm}$ be the presentation given by \cref{thm:frm-lfrm-idl-pres}, with $R$ from \cref{rmk:frm-lfrm-disj-pres}; it is clear that the same presentation yields $\ang{G \mid R}_\!{\lambda'\kappa' LFrm} = \prod_i \ang{A_i \qua \!{\lambda\kappa LFrm}}_\!{\lambda'\kappa' LFrm}$.
Thus using \cref{thm:frm-lfrm-compl} twice, we have
\begin{align*}
\ang{\prod_i A_i \qua \!{\lambda\kappa Frm}}_\!{\lambda'\kappa' Frm}
&= \ang{\prod_i A_i \qua \!{\lambda\kappa LFrm}}_\!{\lambda'\kappa' LFrm} \\
&= \ang{G \mid R}_\!{\lambda'\kappa' LFrm} \\
&= \prod_i \ang{A_i \qua \!{\lambda\kappa LFrm}}_\!{\lambda'\kappa' LFrm} \\
&= \prod_i \ang{A_i \qua \!{\lambda\kappa Frm}}_\!{\lambda'\kappa' Frm}.
\qedhere
\end{align*}
\end{proof}

For the categories $\!{\kappa Bool}$, one may develop analogs of the above by replacing $\!{\lambda\kappa LFrm}$ with the category $\!{\kappa LBool}$ of \defn{local $\kappa$-Boolean algebras},%
\footnote{also known as Boolean r(i)ngs, Boolean (pseudo)lattices}
meaning local $(\kappa, \kappa)$-frames such that each $\down a$ is Boolean.
However, we will be content with deducing the Boolean analogs of the main results above in an \emph{ad hoc} manner.
Note that a $\kappa$-Boolean algebra presentation $A = \ang{G \mid R}_\!{\kappa Bool}$ may be turned into a $(\kappa, \kappa)$-frame presentation by closing $G$ under complements and adding relations
\begin{align*}
a \wedge \neg a = \bot, &&
a \vee \neg a = \top
\end{align*}
for each $a \in G$.
Thus \cref{thm:frm-cover-pres} yields

\begin{corollary}
\label{thm:bool-cover-pres}
If a $\kappa$-Boolean algebra $A$ has a $\kappa$-ary cover $(a_i)_i$ of $\top$ such that each $\down a_i$ is $\mu$-presented for some $\mu \ge \kappa$, then $A$ is $\mu$-presented.
\qed
\end{corollary}

\begin{corollary}
\label{thm:bool-prod-pres}
For $\mu \ge \kappa$, a $\kappa$-ary product of $\mu$-presented $\kappa$-Boolean algebras is $\mu$-presented.
\qed
\end{corollary}

\begin{proposition}
\label{thm:frm-prod-bool}
The free functor $\ang{-}_\!{\kappa Bool} : \!{\kappa\kappa Frm} -> \!{\kappa Bool}$ preserves $\kappa$-ary products.
\end{proposition}
\begin{proof}
Let $A_i \in \!{\lambda\kappa Frm}$ be $<\kappa$-many $(\kappa, \kappa)$-frames.
For each $i$, let $G_i := A_i \sqcup \neg A_i$ where $\neg A_i$ consists of the symbols $\neg a$ for each $a \in A_i$, and let $R_i$ consist of all relations between \emph{local} $(\kappa, \kappa)$-frame terms over $A_i$ which hold in $A_i$, together with the relations
\begin{align*}
\tag{$\dagger$}
a \wedge \neg a = \bot, &&
a \vee \neg a = \top_{A_i}
\end{align*}
for each $a \in A_i$; here $\bot$ denotes the nullary join operation (as a local $(\kappa, \kappa)$-frame term), while $\top_{A_i}$ denotes the top element of $A_i$, \emph{not} the nullary meet operation.
It is easily seen that $\ang{G_i \mid R_i}_\!{\kappa\kappa LFrm} = \ang{A_i}_\!{\kappa Bool}$.
Now as in the proof of \cref{thm:frm-prod-compl}, let $\prod_i A_i = \ang{G \mid R}_\!{\lambda\kappa LFrm}$ be the presentation given by \cref{thm:frm-lfrm-idl-pres}, with $R$ from \cref{rmk:frm-lfrm-disj-pres}; so $G = \bigsqcup_i A_i$, and $R$ consists of all relations which hold in each $A_i$ together with the relations
\begin{align*}
\tag{$\ddagger$}
\top_{A_i} \wedge \top_{A_j} = \bot \quad\text{for $i \ne j$}.
\end{align*}
Let $G' := \bigsqcup_i G_i$ and $R' := \bigsqcup_i R_i \cup (\ddagger)$.
Then $\ang{G' \mid R'}_\!{\kappa\kappa LFrm} = \prod_i \ang{A_i}_\!{\kappa Bool}$, since $(G', R')$ is obtained via \cref{thm:frm-lfrm-idl-pres} and \cref{rmk:frm-lfrm-disj-pres} from the presentations $(G_i, R_i)$ for $\ang{A_i}_\!{\kappa Bool}$.
On the other hand, by ($\dagger$), $\ang{G' \mid R'}_\!{\kappa\kappa LFrm}$ is obtained from $\ang{G \mid R}_\!{\kappa\kappa LFrm} = \prod_i A_i$ by freely adjoining complements in each $\down \top_{A_i} \subseteq \ang{G \mid R}_\!{\kappa\kappa LFrm}$, which is identified with $\down \delta_i \subseteq \prod_i A_i$ (where $\delta_i$ are as defined at the beginning of this subsection); since the $\delta_i$ form a partition of $\top$, adjoining complements in each $\down \top_{A_i}$ amounts to adjoining complements globally, so $\ang{G' \mid R'}_\!{\kappa\kappa LFrm} = \ang{\prod_i A_i}_\!{\kappa Bool}$.
\end{proof}

\begin{corollary}
For $\kappa \le \kappa'$, the free functor $\!{\kappa Bool} -> \!{\kappa' Bool}$ preserves $\kappa$-ary products.
\end{corollary}
\begin{proof}
The free functor is the composite $\!{\kappa Bool} -> \!{\kappa\kappa Frm} -> \!{\kappa'\kappa' Frm} -> \!{\kappa' Bool}$.
\end{proof}

\subsection{Posites}
\label{sec:frm-post}

In this subsection, we discuss a method for presenting $\kappa$-$\bigvee$-lattices and $\kappa$-frames via generators and $\bigvee$-relations which are ``saturated'' (in the sense described in \cref{sec:pres-adj}).
This method is well-known in the case $\kappa = \infty$ (see \cite[II~2.11]{Jstone}), where it is derived from the more general theory of sites in sheaf theory (see \cite[C2.1]{Jeleph}).
The case $\kappa < \infty$ does not appear to have been spelled out before, although it is likewise a special case of the more general theory of $\kappa$-ary sites of Shulman~\cite{Sksite}.
As in \cref{sec:frm-quot}, we also make a point of working in the $\bigvee$-only context for as long as possible before introducing meets.
Our presentation is partly chosen in order to highlight the analogy with some variants and generalizations of posites we will consider in \cref{sec:upkzfrm,sec:dpoly} below.

Let $A$ be a poset, and consider an arbitrary $\kappa$-$\bigvee$-lattice presented by $A$ (qua poset) together with additional relations imposed.
In other words, we are considering an arbitrary quotient of the free $\kappa$-$\bigvee$-lattice $\@L_\kappa(A)$ generated by $A$.
As described in \cref{sec:frm-quot}, such a quotient is determined by an order-congruence $\lesim$ on $\@L_\kappa(A)$, which is completely determined by the relation
\begin{align*}
a \lhd C  \coloniff  \down a \lesim C
\end{align*}
between $A$ and $\@L_\kappa(A)$, read ``\defn{$a$ is covered by $C$}'', via compatibility of $\lesim$ with $\kappa$-ary joins:
\begin{align*}
B \lesim C  \iff  \forall a \in B\, (a \lhd C).
\end{align*}
The other requirements on $\lesim$ translate to
\begin{itemize}
\item  (reflexivity) $a \in C \implies a <| C$;
\item  (left-transitivity) $a \le b <| C \implies a <| C$;
\item  (right-transitivity) if $a <| C$, and $c <| D$ for every $c \in C$, then $a <| D$.
\end{itemize}
Note that reflexivity and right-transitivity imply
\begin{itemize}
\item  (monotonicity) $a <| C \subseteq D \implies a <| D$.
\end{itemize}
We call a relation ${<|} \subseteq A \times \@L_\kappa(A)$ obeying these axioms a \defn{$\kappa$-$\bigvee$-coverage} on $A$, or just a \defn{$\bigvee$-coverage} when $\kappa = \infty$.
We call the poset $A$ equipped with a $\kappa$-$\bigvee$-coverage $<|$ a \defn{$\kappa$-$\bigvee$-posite}.
We think of a $\kappa$-$\bigvee$-posite $(A, <|)$ as a ``saturated'' presentation of a $\kappa$-$\bigvee$-lattice.
Indeed, the order-congruence $\lesim$ corresponding to $<|$ is clearly generated by declaring $\down a \lesim C = \bigvee_{c \in C} \down c$ for $a <| C$, so that the quotient is the presented $\kappa$-$\bigvee$-lattice
\begin{align*}
\ang{A \mid {<|}}_\!{\kappa{\bigvee}Lat} := \@L_\kappa(A)/{\sim} = \ang{A \qua \!{Pos} \mid a \le \bigvee C \text{ for } a <| C}_\!{\kappa{\bigvee}Lat};
\end{align*}
and $<|$ is saturated in that it contains \emph{all} relations of the form $a \le \bigvee C$ that hold in $\ang{A \mid {<|}}_\!{\kappa{\bigvee}Lat}$.

Given a $\kappa$-$\bigvee$-coverage $<|$, we write $a <| C$ for an arbitrary subset $C \subseteq A$ to mean that there is a $\kappa$-ary $B \subseteq C$ which generates a lower set $\down B \in \@L_\kappa(A)$ such that $a <| \down B$.
It is easily seen that this extended $<|$ relation continues to obey the three conditions above.
In particular, taking $C \in \@L_\lambda(A)$ for $\lambda \ge \kappa$ means that a $\kappa$-$\bigvee$-coverage may be regarded as a $\lambda$-$\bigvee$-coverage, hence as a $\bigvee$-coverage.

The $\kappa$-ideal closure operator $\simdown : \@L_\kappa(A) -> \kappa\@I(\@L_\kappa(A))$ corresponding to the order-congruence $\lesim$ corresponding to $<|$ takes $C \in \@L_\kappa(A)$ to the set of all $B \in \@L_\kappa(A)$ such that $b <| C$ for all $b \in B$.
Under the isomorphism $\kappa\@I(\@L_\kappa(A)) \cong \@L(A)$, $\simdown$ becomes
\begin{align*}
\downtri : \@L_\kappa(A) &--> \@L(A) \\
C &|--> \{a \in A \mid a <| C\}.
\end{align*}
The $\kappa$-directed-join-preserving extension $\downtri_* : \@L(A) \cong \kappa\@I(\@L_\kappa(A)) -> \@L(A)$ is defined the same way as $\downtri$ but regarding $<|$ as a $\bigvee$-coverage as above.
By \cref{sec:frm-quot}, the presented $\bigvee$-lattice $\ang{A \mid {<|}}_\!{{\bigvee}Lat}$ may be realized as the set
\begin{align*}
{<|}\@I(A) := \downtri_*(\@L(A)) \subseteq \@L(A)
\end{align*}
of \defn{$<|$-ideals} in $A$, i.e., lower sets $D \subseteq A$ which are closed under $<|$ in that $a <| C \subseteq D \implies a \in D$, under the \defn{principal $<|$-ideals} map
\begin{align*}
\downtri := \downtri \circ \down : A &--> {<|}\@I(A) \\
a &|--> \downtri \down a = \{b \in A \mid b <| \{a\}\}.
\end{align*}
The presented $\kappa$-$\bigvee$-lattice $\ang{A \mid {<|}}_\!{\kappa{\bigvee}Lat}$ may be realized as the image of $\downtri : \@L_\kappa(A) -> {<|}\@I(A)$, i.e., as the $\kappa$-generated, or equivalently $\kappa$-compact, $<|$-ideals ${<|}\@I_\kappa(A) := \downtri(\@L_\kappa(A)) = {<|}\@I(A)_\kappa \subseteq {<|}\@I(A)$.

The ``injection of generators'' $\downtri : A -> {<|}\@I(A)$ is an order-embedding iff for all $a, b \in A$, whenever $\downtri a \subseteq \downtri b$, i.e., $a <| \{b\}$, then already $a \le b$ in $A$; in other words, this means $\downtri b = \down b$, i.e., principal ideals are already $<|$-ideals.
In this case, we call $<|$ \defn{separated} (or \defn{subcanonical}).
There is a largest separated $\kappa$-$\bigvee$-coverage, called the \defn{canonical $\kappa$-$\bigvee$-coverage}, given by $a <| C$ (for $C \in \@L_\kappa(A)$ or $C \in \@P_\kappa(A)$) iff every upper bound of $C$ is $\ge a$; if $A$ is already a $\kappa$-$\bigvee$-lattice, this means
\begin{align*}
a <| C \iff a \le \bigvee C,
\end{align*}
i.e., $<|$ corresponds to the canonical presentation $A \cong \ang{A \mid \text{all relations which hold in $A$}}_\!{\kappa{\bigvee}Lat}$.
(For a general poset $A$, the canonical $\bigvee$-coverage presents the MacNeille completion of $A$.)

We regard $\kappa$-$\bigvee$-posites $(A, <|)$ as infinitary first-order structures with a partial order together with $(1+\lambda)$-ary relations $<|$ for all $\lambda < \kappa$.
Thus, by a \defn{homomorphism of $\kappa$-$\bigvee$-posites} $f : (A, {<|_A}) -> (B, {<|_B})$, we mean a monotone map $f : A -> B$ which preserves covers, i.e., $a <|_A C \implies f(a) <|_B f(C)$.
Let $\!{\kappa{\bigvee}Post}$ denote the \defn{category of $\kappa$-$\bigvee$-posites}.
Then $\!{\kappa{\bigvee}Post}$ is locally $\kappa$-presentable (see \cref{sec:cat-lim}).
For $\kappa \le \lambda$, we have a forgetful functor $\!{\lambda{\bigvee}Post} -> \!{\kappa{\bigvee}Post}$, whose left adjoint is given by regarding a $\kappa$-$\bigvee$-coverage as a $\lambda$-$\bigvee$-coverage as described above.

Regarding $\kappa$-$\bigvee$-lattices as $\kappa$-$\bigvee$-posites with the canonical $\kappa$-$\bigvee$-coverage, we have a full and faithful forgetful functor $\!{\kappa{\bigvee}Lat} -> \!{\kappa{\bigvee}Post}$, whose left adjoint is given by taking presented $\kappa$-$\bigvee$-lattices $(A, {<|}) |-> \ang{A \mid {<|}}_\!{\kappa{\bigvee}Lat} \cong {<|}\@I_\kappa(A)$; thus $\!{\kappa{\bigvee}Lat}$ is a reflective subcategory of $\!{\kappa{\bigvee}Post}$.
The adjunction unit $\downtri : (A, {<|}) -> \ang{A \mid {<|}}_\!{\kappa{\bigvee}Lat}$ ``reflects covers'', i.e., $a <| C \iff \downtri a <| \downtri(C)$, which captures the ``saturation'' of $<|$ as described above; $<|$ is separated iff $\downtri$ is also order-reflecting, hence an embedding of $\kappa$-$\bigvee$-posites.

Summarizing, we have the following commutative diagram of free/forgetful adjunctions:
\begin{equation}
\label{diag:frm-vpost}
\begin{tikzcd}
\!{Pos}
    \dar[equal]
    \rar[shift left=2, "\@L_\omega"] &
\!{{\vee}Lat}
    \lar[hook, shift left=2, right adjoint]
    \dar[rightarrowtail, shift left=2, right adjoint']
    \rar[shift left=2, "\@I_\sigma"] &
\!{\sigma{\bigvee}Lat}
    \lar[hook, shift left=2, right adjoint]
    \dar[rightarrowtail, shift left=2, right adjoint']
    \rar[shift left=2, "\sigma\@I_{\omega_2}"] &
\dotsb
    \lar[hook, shift left=2, right adjoint]
    \rar[shift left=2] &
\!{\kappa{\bigvee}Lat}
    \lar[hook, shift left=2, right adjoint]
    \dar[rightarrowtail, shift left=2, right adjoint']
    \rar[shift left=2, "\kappa\@I_\lambda"] &
\!{\lambda{\bigvee}Lat}
    \lar[hook, shift left=2, right adjoint]
    \dar[rightarrowtail, shift left=2, right adjoint']
    \rar[shift left=2, "\lambda\@I"] &
\!{{\bigvee}Lat}
    \lar[hook, shift left=2, right adjoint]
    \dar[rightarrowtail, shift left=2, right adjoint']
\\
\!{Pos}
    \rar[rightarrowtail, shift left=2] &
\!{{\vee}Post}
    \lar[shift left=2, right adjoint]
    \uar[shift left=2, "\ang{-}_\!{{\vee}Lat}"]
    \rar[rightarrowtail, shift left=2] &
\!{\sigma{\bigvee}Post}
    \lar[shift left=2, right adjoint]
    \uar[shift left=2, "\ang{-}_\!{\sigma{\bigvee}Lat}"]
    \rar[rightarrowtail, shift left=2] &
\dotsb
    \lar[shift left=2, right adjoint]
    \rar[rightarrowtail, shift left=2] &
\!{\kappa{\bigvee}Post}
    \lar[shift left=2, right adjoint]
    \uar[shift left=2, "\ang{-}_\!{\kappa{\bigvee}Lat}"]
    \rar[rightarrowtail, shift left=2] &
\!{\lambda{\bigvee}Post}
    \lar[shift left=2, right adjoint]
    \uar[shift left=2, "\ang{-}_\!{\lambda{\bigvee}Lat}"]
    \rar[rightarrowtail, shift left=2] &
\!{{\bigvee}Post}
    \lar[shift left=2, right adjoint]
    \uar[shift left=2, "\ang{-}_\!{{\bigvee}Lat}"]
\end{tikzcd}
\end{equation}
As before, the $\rightarrowtail$ arrows denote full embeddings.

We now consider posites which present frames.
Let $<|$ be a $\kappa$-$\bigvee$-coverage on a $\wedge$-lattice%
\footnote{For simplicity, we only consider posites based on a $\wedge$-lattice.
There is a more general theory of ``flat posites'' based on arbitrary posets (but still presenting frames); see \cite[C1.1.16(e)]{Jeleph}, \cite{Sksite}.}
$A$, corresponding to a $\kappa$-$\bigvee$-lattice order-congruence $\lesim$ on $\@L_\kappa(A)$ as described above.
From \cref{sec:frm-quot}, $\lesim$ is a $\kappa$-frame order-congruence iff it is $\wedge$-stable, which is easily seen to be equivalent to
\begin{itemize}
\item  ($\wedge$-stability) $a \le b <| C \implies a <| a \wedge C := \{a \wedge c \mid c \in C\}$.
\end{itemize}
Note that this (together with reflexivity and right-transitivity) implies left-transitivity.
We call a $\wedge$-stable $\kappa$-$\bigvee$-coverage $<|$ a \defn{$\kappa$-coverage}, and call a $\kappa$-$\bigvee$-posite $(A, <|)$ with $A \in \!{{\wedge}Lat}$ and $\wedge$-stable $<|$ a \defn{$\kappa$-posite}.
Thus a $\kappa$-posite $(A, <|)$ can be seen as a $\kappa$-frame presentation which is also a $\kappa$-$\bigvee$-lattice presentation (a point of view due to Abramsky--Vickers~\cite{AVquant}):
\begin{align*}
\ang{A \mid {<|}}_\!{\kappa Frm} := \@L_\kappa(A)/{\sim}
&= \ang{A \qua \!{{\wedge}Lat} \mid a \le \bigvee C \text{ for } a <| C}_\!{\kappa Frm} \\
&= \ang{A \qua \!{Pos} \mid a \le \bigvee C \text{ for } a <| C}_\!{\kappa{\bigvee}Lat}
= \ang{A \mid {<|}}_\!{\kappa{\bigvee}Lat}.
\end{align*}

By $\wedge$-stability and right-transitivity, a $\kappa$-coverage satisfies
\begin{align*}
a <| C \iff a <| a \wedge C,
\end{align*}
hence is completely determined by its restriction to pairs $(a, C)$ such that $C \subseteq \down a$ (this is how coverages are usually presented in the literature).
This restricted $<|$ continues to satisfy reflexivity and $\wedge$-stability as stated (assuming all occurrences of $<|$ obey the restriction), while right-transitivity is replaced with
\begin{itemize}
\item  (right-transitivity$'$) if $a <| C$, and $c <| c \wedge D$ for every $c \in C$, then $a <| D$.
\end{itemize}
As in the $\bigvee$-case, we also extend $<|$ to $(a, C)$ where $C \subseteq A$ is an arbitrary subset.

The largest $\kappa$-coverage on a $\wedge$-lattice $A$, called the \defn{canonical $\kappa$-coverage}, is given by $a <| C$ (for $\kappa$-ary $C$) iff for every $b \le a$, every upper bound of $b \wedge C$ is $\ge b$, i.e., $b = \bigvee (b \wedge C)$; if $A$ is already a $\kappa$-frame, this just means $a \le \bigvee C$ (by distributivity of $b \wedge \bigvee C$).
This gives a full inclusion $\!{\kappa Frm} -> \!{\kappa Post}$ into the \defn{category of $\kappa$-posites}, where by a homomorphism of $\kappa$-posites we mean a cover-preserving $\wedge$-lattice homomorphism.
The left adjoint is given by taking presented $\kappa$-frames.
We have the following commutative diagram of free/forgetful adjunctions, analogous to \eqref{diag:frm-vpost}:
\begin{equation*}
\label{diag:frm-post}
\begin{tikzcd}
\!{{\wedge}Lat}
    \dar[equal]
    \rar[shift left=2, "\@L_\omega"] &
\!{DLat}
    \lar[hook, shift left=2, right adjoint]
    \dar[rightarrowtail, shift left=2, right adjoint']
    \rar[shift left=2, "\@I_\sigma"] &
\!{\sigma Frm}
    \lar[hook, shift left=2, right adjoint]
    \dar[rightarrowtail, shift left=2, right adjoint']
    \rar[shift left=2, "\sigma\@I_{\omega_2}"] &
\dotsb
    \lar[hook, shift left=2, right adjoint]
    \rar[shift left=2] &
\!{\kappa Frm}
    \lar[hook, shift left=2, right adjoint]
    \dar[rightarrowtail, shift left=2, right adjoint']
    \rar[shift left=2, "\kappa\@I_\lambda"] &
\!{\lambda Frm}
    \lar[hook, shift left=2, right adjoint]
    \dar[rightarrowtail, shift left=2, right adjoint']
    \rar[shift left=2, "\lambda\@I"] &
\!{Frm}
    \lar[hook, shift left=2, right adjoint]
    \dar[rightarrowtail, shift left=2, right adjoint']
\\
\!{{\wedge}Lat}
    \rar[rightarrowtail, shift left=2] &
\!{\omega Post}
    \lar[shift left=2, right adjoint]
    \uar[shift left=2, "\ang{-}_\!{DLat}"]
    \rar[rightarrowtail, shift left=2] &
\!{\sigma Post}
    \lar[shift left=2, right adjoint]
    \uar[shift left=2, "\ang{-}_\!{\sigma Frm}"]
    \rar[rightarrowtail, shift left=2] &
\dotsb
    \lar[shift left=2, right adjoint]
    \rar[rightarrowtail, shift left=2] &
\!{\kappa Post}
    \lar[shift left=2, right adjoint]
    \uar[shift left=2, "\ang{-}_\!{\kappa Frm}"]
    \rar[rightarrowtail, shift left=2] &
\!{\lambda Post}
    \lar[shift left=2, right adjoint]
    \uar[shift left=2, "\ang{-}_\!{\lambda Frm}"]
    \rar[rightarrowtail, shift left=2] &
\!{Post}
    \lar[shift left=2, right adjoint]
    \uar[shift left=2, "\ang{-}_\!{Frm}"]
\end{tikzcd}
\end{equation*}
Moreover, this diagram ``sits above'' \eqref{diag:frm-vpost}, in that for each $\kappa$, we have a square
\begin{equation}
\label{diag:frm-post-vpost}
\begin{tikzcd}
\!{\kappa{\bigvee}Lat}
    \dar[rightarrowtail, shift left=2, right adjoint']
    \rar[dashed, shift left=2] &
\!{\kappa Frm}
    \lar[hook, shift left=2, right adjoint]
    \dar[rightarrowtail, shift left=2, right adjoint'] \\
\!{\kappa{\bigvee}Post}
    \uar[shift left=2, "\ang{-}_\!{\kappa{\bigvee}Lat}"] 
    \rar[dashed, shift left=2] &
\!{\kappa Post}
    \lar[hook, shift left=2, right adjoint]
    \uar[shift left=2, "\ang{-}_\!{\kappa Frm}"]
\end{tikzcd}
\end{equation}
in which not only do the forgetful functors commute, but also the two composites $\!{\kappa Post} -> \!{\kappa{\bigvee}Lat}$ agree.

We conclude this subsection with some simple consequences about presentations.
First, translating \cref{thm:frm-cong-ltrans}, \cref{thm:frm-nuc-ltrans}, and \cref{thm:frm-cong-meet} about $\lesim$ to $<|$ yields

\begin{proposition}
\label{thm:frm-cov-ltrans}
Let $A$ be a poset, $<|_0$ be a binary relation between $A$ and $\@L_\kappa(A)$ (or $\@P_\kappa(A)$).
Let $<|$ be the smallest relation which satisfies reflexivity, left-transitivity, and ``right-transitivity with respect to $<|_0$'', i.e., if $a <|_0 C$, and $c <| D$ for every $c \in C$, then $a <| D$.
Then $<|$ satisfies right-transitivity, hence is the $\kappa$-$\bigvee$-coverage generated by $<|_0$.
\qed
\end{proposition}

\begin{corollary}
\label{thm:frm-cidl-ltrans}
Let $A$ be a poset, $<|_0$ be a binary relation between $A$ and $\@L_\kappa(A)$ (or $\@P_\kappa(A)$), $<|$ be the $\kappa$-$\bigvee$-coverage generated by $<|_0$.
Then a lower set $D \in \@L(A)$ is a $<|$-ideal iff it is a $<|_0$-ideal.
\qed
\end{corollary}

\begin{proposition}
\label{thm:frm-cov-meet}
Let $A$ be a $\wedge$-lattice, $<|_0$ be a $\wedge$-stable binary relation between $A$ and $\@L_\kappa(A)$ (or $\@P_\kappa(A)$).
Then the $\kappa$-$\bigvee$-coverage $<|$ generated by $<|_0$ is still $\wedge$-stable, hence is also the $\kappa$-coverage generated by $<|_0$.
\qed
\end{proposition}

Next, any presentation of a $\kappa$-$\bigvee$-lattice or $\kappa$-frame can be canonically turned into a $\kappa$-($\bigvee$-)posite.
For suppose $A = \ang{G \mid R}_\!{\kappa{\bigvee}Lat}$ is a $\kappa$-$\bigvee$-lattice presentation, where $G$ is a set or more generally a poset, and $R$ is a set of inequalities between $\kappa$-$\bigvee$-lattice terms.
This means that we have a quotient map $\ang{G} = \@L_\kappa(G) ->> A$ whose order-kernel $\lesim$ is generated by the pairs $(B, C) \in \@L_\kappa(G)^2$ which are the ``normal forms'' of the terms $(s, t)$ in each inequality $s \le t$ in $R$.
We may replace $s \le t$ with the inequalities $s_i \le t$ whenever $s$ is a $\kappa$-ary join term $s = \bigvee_i s_i$, until we are left with inequalities $b \le t$ where $b \in G$; thus $\lesim$ is also generated by pairs $(\down b, C)$.
This means that the $\kappa$-$\bigvee$-coverage $<|$ corresponding to $\lesim$, so that $A \cong \ang{G \mid {<|}}_\!{\kappa{\bigvee}Lat}$, is generated by the pairs $(b, C)$.

If the original presentation $(G, R)$ was $\lambda$-ary for some $\lambda \ge \kappa$, then so will be the new one.
Moreover, it is a general algebraic fact (see \cref{thm:cat-alg-pres}) that a $\lambda$-presented algebraic structure is $\lambda$-presented using any given $\lambda$-ary set of generators.
We thus have

\begin{proposition}
For $\lambda \ge \kappa$, a $\kappa$-$\bigvee$-lattice $A$ is $\lambda$-presented iff it is presented by some $\lambda$-presented $\kappa$-$\bigvee$-posite, i.e., $A \cong \ang{G \mid {<|}}_\!{\kappa{\bigvee}Lat}$ for some $\lambda$-ary poset $G$ and $\lambda$-generated $\kappa$-$\bigvee$-coverage $<|$.

Moreover, in that case, we can find such $<|$ for any $\lambda$-ary poset $G$ and monotone $G -> A$ whose image generates $A$.
\qed
\end{proposition}

If $A = \ang{G \mid R}_\!{\kappa Frm}$ is a $\kappa$-frame presentation, by replacing $G$ with $\ang{G}_\!{{\wedge}Lat}$, we can suppose $A = \ang{G \qua \!{{\wedge}Lat} \mid R}_\!{\kappa Frm}$, so that the quotient $\@L_\kappa(G) ->> A$ is a $\kappa$-frame quotient, whence its order-kernel $\lesim$ is the $\kappa$-frame order-congruence generated by $R$.
By applying the distributive law, we may turn each inequality in $R$ into one between joins of finite meets, and then evaluate the finite meets in $G$, so that we are left with $\kappa$-$\bigvee$-lattice inequalities $s \le t$, corresponding to generators $(B, C)$ for $\lesim$.
As above, we may reduce to inequalities $b \le t$ with $b \in G$, corresponding to generators $(b, C)$ for the $\kappa$-coverage $<|$.
We may furthermore close $R$ under $\wedge$-stability, so that the generators $(b, C)$ for $<|$ become $\wedge$-stable as well, and so by \cref{thm:frm-cov-meet} they also generate $<|$ as a $\kappa$-$\bigvee$-coverage.

Clearly these modifications also preserve $\lambda$-presentability for $\lambda \ge \kappa$, yielding

\begin{proposition}
For $\lambda \ge \kappa$, a $\kappa$-frame $A$ is $\lambda$-presented iff it is presented by some $\lambda$-presented $\kappa$-posite, i.e., $A \cong \ang{G \mid {<|}}_\!{\kappa Frm}$ for some $\lambda$-ary $\wedge$-lattice $G$ and $\lambda$-generated $\kappa$-coverage $<|$.

Moreover, in that case, $<|$ is $\lambda$-generated also as a $\kappa$-$\bigvee$-coverage, and $G$ can be any $\lambda$-ary $\wedge$-lattice with a $\wedge$-homomorphism $G -> A$ whose image generates $A$.
\qed
\end{proposition}

Since $\ang{G \mid {<|}}_\!{\kappa Frm} = \ang{G \mid {<|}}_\!{\kappa{\bigvee}Lat}$, this yields the aforementioned (in \cref{sec:frm-idl})

\begin{corollary}
\label{thm:frm-pres-vlat}
A $\kappa$-frame is $\lambda$-presented iff it is $\lambda$-presented as a $\kappa$-$\bigvee$-lattice.
\qed
\end{corollary}

\subsection{Colimits of frames}
\label{sec:frm-colim}

In this short subsection, we record some structural facts about colimits of $\kappa$-frames.
These facts are well-known for $\kappa = \infty$ (see \cite{JTgpd}); the same proofs can be easily adapted to work when $\kappa < \infty$, although we find it more natural to use a ``presentational'' approach based on posites.%
\footnote{Such a ``presentational'' approach is usually used to prove the analogous results for Grothendieck toposes; see e.g., \cite[C2.5.14]{Jeleph}.}

Consider the computation of the colimit $\injlim F = \injlim_\!{\kappa Post} F$ of an arbitrary diagram $F : \!I -> \!{\kappa Post}$, where $\!I$ is a small indexing category.
By the general abstract method for presenting colimits of structures axiomatized by universal Horn theories (see \cref{sec:cat-lim}), $\injlim F$ may be computed by taking the colimit of the underlying $\wedge$-lattices, then taking the image of the $\kappa$-coverage $<|_{F(I)}$ on each $F(I)$ under the cocone map $\iota_I : F(I) -> \injlim F$, and finally taking the $\kappa$-coverage $<|_{\injlim F}$ on $\injlim F$ generated by all of these images.
An analogous description holds for the colimit $\injlim_\!{\kappa{\bigvee}Post} F$ of $F$ (composed with the forgetful functor) computed in $\!{\kappa{\bigvee}Post}$.

Suppose the diagram $F$ is such that the colimit of the underlying $\wedge$-lattices is preserved by the forgetful functor $\!{{\wedge}Lat} -> \!{Pos}$.
For example, this holds if $\injlim F$ is a directed colimit (i.e., $\!I$ is a directed preorder).
Then the first part of the computation of $\injlim_\!{\kappa Post} F$ agrees with that of $\injlim_\!{\kappa{\bigvee}Post} F$.
Clearly, $\wedge$-stability of each $<|_{F(I)}$ implies $\wedge$-stability of the union of their images in $\injlim F$, which by \cref{thm:frm-cov-meet} therefore also generates $<|_{\injlim F}$ as a $\kappa$-$\bigvee$-coverage; thus the last part of the computation of $\injlim_\!{\kappa Post} F$ is also preserved in $\!{\kappa{\bigvee}Post}$.
We thus have

\begin{proposition}
\label{thm:frm-colim-dir-post-vpost}
The forgetful functor $\!{\kappa Post} -> \!{\kappa{\bigvee}Post}$ preserves directed colimits.
\qed
\end{proposition}

Now for a diagram $F : \!I -> \!{\kappa Frm}$, its colimit $\injlim_\!{\kappa Frm} F$ is given by the reflection $\ang{\injlim_\!{\kappa Post} F}_\!{\kappa Frm}$ of the colimit in $\!{\kappa Post}$, which is the same as $\ang{\injlim_\!{\kappa Post} F}_\!{\kappa{\bigvee}Lat}$ (see the square \eqref{diag:frm-post-vpost});
while the colimit $\injlim_\!{\kappa{\bigvee}Lat} F$ in $\!{\kappa{\bigvee}Lat}$ is given by the reflection $\ang{\injlim_\!{\kappa{\bigvee}Post} F}_\!{\kappa{\bigvee}Lat}$.
Thus

\begin{corollary}
\label{thm:frm-colim-dir-vlat}
The forgetful functor $\!{\kappa Frm} -> \!{\kappa{\bigvee}Lat}$ preserves directed colimits.
\qed
\end{corollary}

For a directed diagram $F : \!I -> \!{\kappa{\bigvee}Lat}$, the following gives us some control over the cocone maps $\iota_I : F(I) -> \injlim F$.
For $I \le J \in \!I$, we write $F(I, J)$ for $F$ applied to the unique $I -> J \in \!I$.

\begin{proposition}
\label{thm:frm-colim-dir-ker}
Let $F : \!I -> \!{\kappa{\bigvee}Lat}$ be a directed diagram, $I \in \!I$, $a \in F(I)$.
Suppose that for each $I \le J \le K \in \!I$, $F(I, J)(a) \in F(J)$ is the greatest element of its $\oker(F(J, K))$-principal ideal (i.e., for all $b \in F(J)$, $b \le F(I, J)(a) \iff F(J, K)(b) \le F(I, K)(a)$).
Then $a$ is the greatest element of its $\oker(\iota_I)$-principal ideal.

In particular, if each $F(I, J) : F(I) -> F(J)$ is injective, then so is each $\iota_I : F(I) -> \injlim F$.
\end{proposition}
\begin{proof}
By replacing $\!I$ with $\up I$ (which is final in $\!I$), we may assume $I$ is the least element of $\!I$.
For each $J \in \!I$, let $f_J : F(J) -> 2$ be the indicator function of $F(J) \setminus \down F(I, J)(a)$.
Then the assumptions ensure $(f_J)_J$ is a cocone over $F$, hence induce a homomorphism $f : \injlim F -> 2$ such that $f \circ \iota_I = f_I$.
Then $\iota_I(b) \le \iota_I(a) \implies f_I(b) \le f_I(a) = 0 \implies b \in \down a$.
\end{proof}

We also briefly recall the construction of coproducts of frames via posites.
First, for $\wedge$-lattices $A, B$, the product $A \times B$ is a biproduct (``direct sum'') in $\!{\wedge Lat}$, hence also the coproduct, with injections $\iota_1 : A -> A \times B$ given by $\iota_1(a) := (a, \top)$ and $\iota_2$ given similarly.

Now for $\kappa$-frames $A, B$, their $\wedge$-lattice coproduct $A \times B$ generates their $\kappa$-frame coproduct, whence the latter is presented by a $\kappa$-coverage $<|$ on $A \times B$.
Namely, $<|$ imposes the relations that $\kappa$-ary joins in $A$ must be preserved, as must be $\kappa$-ary joins in $B$, hence is generated by the pairs
\begin{align}
\label{eq:frm-coprod-pres}
\begin{aligned}
(\bigvee C, b) &<| \{(c, b) \mid c \in C\} &&\text{for $\kappa$-ary $C \subseteq A$}, \\
(a, \bigvee C) &<| \{(a, c) \mid c \in C\} &&\text{for $\kappa$-ary $C \subseteq B$},
\end{aligned}
\end{align}
which are clearly $\wedge$-stable.
Thus by \cref{thm:frm-cov-meet} and \cref{thm:frm-cidl-ltrans}, the $\kappa$-frame coproduct of $A, B$ may be realized as the set of $D \in \@L(A \times B)$ which are $\kappa$-generated ideals with respect to these generators for $<|$, i.e., $D$ must be ``closed under $\kappa$-ary joins separately in each coordinate'', and $D$ must also be $\kappa$-generated as such.

As is standard, we denote the $\kappa$-frame coproduct of $A, B$ by $A \otimes B$ (in addition to the category-agnostic notation $A \amalg B$).
For $a \in A$ and $b \in B$, we write
\begin{align*}
a \times b := \iota_1(a) \wedge \iota_2(b) \in A \otimes B
\end{align*}
(this is more commonly denoted $a \otimes b$); an arbitrary element of $A \otimes B$ is thus a $\kappa$-ary join of such ``rectangles''.
Under the above realization of $A \otimes B$ as all $D \subseteq \@L(A \times B)$ closed under \eqref{eq:frm-coprod-pres}, it is easily seen that
\begin{align*}
a \times b = \down (a, b) \cup (A \times \{\bot\}) \cup (\{\bot\} \cup B) \subseteq A \times B.
\end{align*}

\begin{corollary}
\label{thm:frm-coprod}
Let $A, B$ be $\kappa$-frames, $a, a' \in A$, and $b, b' \in B$.
If $a \times b \le a' \times b' \in A \otimes B$, then either $a \le a'$ and $b \le b'$, or $a = \bot$, or $b = \bot$.

In particular, if $a \times b = \bot$, then either $a = \bot$ or $b = \bot$.
\end{corollary}
\begin{proof}
If $a \times b \le a' \times b'$, then $(a, b) \in \down (a', b') \cup (A \times \{\bot\}) \cup (\{\bot\} \cup B) \subseteq A \times B$.
\end{proof}

See \cref{thm:bifrm-coprod} below for a generalization of this to $(\lambda, \kappa)$-frames.

\subsection{Adjoining complements}
\label{sec:frm-neg}

In this subsection, we discuss the process of completing a $\kappa$-frame to a $\kappa$-Boolean algebra by repeatedly adjoining complements.
This is one of the central constructions of classical locale theory (known variously as the ``frame of nuclei'', ``dissolution locale'', or ``assembly tower''; see e.g., \cite[II~2.5--10]{Jstone}, \cite{Wasm}, \cite{Pdissolv}), and corresponds to the Borel hierarchy in descriptive set theory (see \cref{sec:loc-bor} below).
The case $\kappa < \infty$ was studied by Madden~\cite{Mkfrm}.
One feature of our approach, which is new as far as we know, is an explicit posite presentation of the frame of nuclei (\cref{thm:frm-neg-pres}).
At the end of the subsection, we also discuss the interaction between adjoining complements and $\kappa$-presentability.

For a $\kappa$-frame $A$, we let $\@N_\kappa(A)$ denote the $\kappa$-frame resulting from freely adjoining a complement for each element of $A$:
\begin{align*}
\@N_\kappa(A) := \ang{A \qua \!{\kappa Frm},\, \neg a \text{ for $a \in A$} \mid a \wedge \neg a = \bot,\, a \vee \neg a = \top}_\!{\kappa Frm}.
\end{align*}
Thus $\@N_\kappa(A)$ is the universal $\kappa$-frame equipped with a homomorphism $\eta : A -> \@N_\kappa(A)$ mapping each $a \in A$ to a complemented element.
There is an obvious extension of $\@N_\kappa$ to a functor $\!{\kappa Frm} -> \!{\kappa Frm}$, such that $\eta$ becomes a natural transformation $1_\!{\kappa Frm} -> \@N_\kappa$.
For a frame $A$, we write $\@N(A) := \@N_\infty(A)$.

By distributivity, every element of $\@N_\kappa(A)$ can be written as a $\kappa$-ary join
\begin{align}
\label{eq:frm-neg-normform}
\bigvee_i (a_i \wedge \neg b_i) \quad\text{for $a_i, b_i \in A$}
\end{align}
(where we are writing $a_i$ instead of $\eta(a_i)$; see \cref{cvt:frm-neg-incl} below).
Thus $\@N_\kappa(A)$ is generated as a $\kappa$-$\bigvee$-lattice by the image of the $\wedge$-lattice homomorphism
\begin{align*}
A \times \neg A &--> \@N_\kappa(A) \\
(a, \neg b) &|--> a \wedge \neg b
\end{align*}
where $\neg A  := A^\op$ with elements written $\neg a$ for $a \in A$.
The order-kernel of this map corresponds to a $\kappa$-coverage $<|$ on $A \times \neg A$ such that $\@N_\kappa(A) = \ang{A \times \neg A \mid <|}_\!{\kappa Frm}$.

\begin{proposition}
\label{thm:frm-neg-pres}
The $\kappa$-coverage $<|$ such that $\@N_\kappa(A) = \ang{A \times \neg A \mid <|}_\!{\kappa Frm}$ is generated by the following pairs, which are $\wedge$-stable and hence also generate $<|$ as a $\kappa$-$\bigvee$-coverage (by \cref{thm:frm-cov-meet}):
\begin{align*}
\tag{i}&&(\bigvee B, \neg c) &<| \{(b, \neg c) \mid b \in B\} &&\text{for $\kappa$-ary $B \subseteq A$ and $c \in A$}, \\
\tag{ii}&&(a, \neg b) &<| \emptyset &&\text{for $a \le b \in A$}, \\
\tag{iii}&&(c, \neg b) &<| \{(a, \neg b), (c, \neg d)\} &&\text{for $a \vee b \ge c \wedge d \in A$}.
\end{align*}
\end{proposition}

\begin{proof}
First, we check that these pairs are indeed part of $<|$, i.e., that the corresponding relations hold in $\@N_\kappa(A)$.
Pair (i) says $(\bigvee B) \wedge \neg c \le \bigvee_{b \in B} (b \wedge \neg c)$,
while (ii) says $a \wedge \neg b \le \bot$ if $a \le b$; these are clearly true.
Pair (iii) says $c \wedge \neg b \le (a \wedge \neg b) \vee (c \wedge \neg d)$ if $a \vee b \ge c \wedge d$; indeed, we have
\begin{align*}
b \vee (a \wedge \neg b) \vee (c \wedge \neg d)
&= b \vee a \vee (c \wedge \neg d) \\
&\ge (c \wedge d) \vee (c \wedge \neg d)
= c,
\end{align*}
whence $(a \wedge \neg b) \vee (c \wedge \neg d) \ge c \wedge \neg b$.

From the presentation defining $\@N_\kappa(A)$, it is clear that $<|$ is generated by the pairs
\begin{align*}
\tag{i$'$} (\bigvee B, \neg \bot) &<| \{(b, \neg \bot) \mid b \in B\}, \\
\tag{ii$'$} (a, \neg a) &<| \emptyset, \\
\tag{iii$'$} (\top, \neg \bot) &<| \{(a, \neg \bot), (\top, \neg a)\},
\end{align*}
corresponding to the relations which say that
(i$'$) $\kappa$-ary joins in $A$ are preserved,
(ii$'$) $a \wedge \neg a \le \bot$, and
(iii$'$) $\top \le a \vee \neg a$.

It is easily seen that the closure of (i$'$) and (ii$'$) under $\wedge$-stability yield (i) and (ii) respectively.
For (iii$'$), taking the meet with an arbitrary $(c, \neg b) \in A \times \neg A$ yields $(c, \neg b) <| \{(a \wedge c, \neg b), (c, \neg (a \vee b))\}$, which is of the form (iii).
Thus (i--iii) are included in $<|$ and include the generators (i$'$--iii$'$) of $<|$, and so generate $<|$.
It remains to check that (iii) is $\wedge$-stable: taking meet with $(e, \neg f) \in A \times \neg A$ yields
\begin{align*}
(c \wedge e, \neg (b \vee f)) <| \{(a \wedge e, \neg (b \vee f)), (c \wedge e, \neg (d \vee f))\}
\end{align*}
which is of the form (iii) since
\begin{align*}
(a \wedge e) \vee b \vee f
&= ((a \vee b) \wedge (e \vee b)) \vee f \\
&\ge (c \wedge d \wedge (e \vee b)) \vee f \\
&\ge (c \wedge e \wedge d) \vee f \\
&\ge c \wedge e \wedge (d \vee f).
\qedhere
\end{align*}
\end{proof}

It follows by \cref{thm:frm-cidl-ltrans} that $\@N_\kappa(A)$ may be realized as the subset of $D \in \@L(A \times \neg A)$ which are ideals with respect to the generators (i--iii) for $<|$ from \cref{thm:frm-neg-pres}, and which are also $\kappa$-generated as such.
In other words, (i--iii) describe precisely the saturation conditions characterizing uniqueness of the expressions \eqref{eq:frm-neg-normform}: two such expressions denote the same element of $\@N_\kappa(A)$ iff their sets of $(a_i, \neg b_i)$ agree when closed downward and under (i--iii).
Using this, we have

\begin{proposition}
\label{thm:frm-neg-inj}
$\eta : A -> \@N_\kappa(A)$ is injective.
\end{proposition}
\begin{proof}
$\eta(a) = a \wedge \neg \bot$, regarded as a $<|$-ideal in $A \times \neg A$ for $<|$ given by \cref{thm:frm-neg-pres}, is the principal $<|$-ideal $\downtri (a, \neg \bot)$.
We have $\downtri (a, \neg \bot) = \{(b, \neg c) \in A \times \neg A \mid b \le a \vee c\}$, since this lower set contains $(a, \neg \bot)$ and is easily seen to be the smallest such closed under (i--iii) in \cref{thm:frm-neg-pres}.
Thus
$\eta(a) \le \eta(b)
\iff \downtri (a, \neg \bot) \subseteq \downtri (b, \neg \bot)
\iff (a, \neg \bot) \in \downtri (b, \neg \bot)
\iff a \le b \vee \bot$.
\end{proof}

Note that by definition of $\@N_\kappa(A)$, $\eta$ is surjective, hence an isomorphism, iff every $a \in A$ already has a complement, i.e., $A$ is a $\kappa$-Boolean algebra.
We henceforth adopt the following

\begin{convention}
\label{cvt:frm-neg-incl}
For any $\kappa$-frame $A$, we regard $\eta : A `-> \@N_\kappa(A)$ as an inclusion.
\end{convention}

\Cref{thm:frm-neg-pres} also shows that the same $\kappa$-coverage $<|$, regarded as an $\infty$-coverage, presents the free frame generated by $A$ qua $\kappa$-frame together with complements for elements of $A$, i.e., the frame $\kappa\@I(\@N_\kappa(A))$.
This frame can thus also be realized as the set of all $<|$-ideals $D \subseteq A \times \neg A$.
As an application, we give a simple direct proof of the following result of Madden~\cite[5.1]{Mkfrm} (due to Isbell~\cite[1.3]{Iloc} when $\kappa = \infty$), which was originally proved via a more abstract method (see also \cref{rmk:loc-sub-im}):

\begin{proposition}[Madden]
\label{thm:frm-neg-cong}
For any $\kappa$-frame $A$, we have order-isomorphisms
\begin{align*}
\kappa\@I(\@N_\kappa(A)) &\cong \{\text{$\kappa$-frame order-congruences } {\lesim} \subseteq A^2\}, \\
\@N_\kappa(A) &\cong \{\text{$\kappa$-generated $\kappa$-frame order-congruences } {\lesim} \subseteq A^2\} \\
c &|-> \{(a, b) \mid a \wedge \neg b \le c\} \\
\bigvee_{a \lesim b} (a \wedge \neg b) &<-| {\lesim}
\end{align*}
where the join on the last line can be taken over any generating set for $\lesim$.
\end{proposition}
\begin{proof}
The first isomorphism is obtained by regarding a $<|$-ideal $D \subseteq A \times \neg A$ as a subset of $A \times A^\op \cong A \times \neg A$, hence as a binary relation $\lesim$ on $A$.
For $D$ to be lower means that $a \le b \lesim c \le d \implies a \le d$.
For $D$ to be closed under (i--iii) in \cref{thm:frm-neg-pres} mean that
(i)~joins in $A$ are joins with respect to $\lesim$,
(ii)~$\lesim$ contains $\le$, and
(iii)~$a \lesim b$, $c \lesim d$, and $c \wedge d \le a \vee b$ imply $c \lesim b$;
(iii) implies $\lesim$ is transitive by taking $a = d$ and
$\wedge$-stable by taking $a = b = c' \wedge b'$, $c = c' \wedge a'$, and $d = b'$,
while (iii) is also implied by these two conditions on $\lesim$ (along with the others), since from $c \lesim d$, $\wedge$-stability, and (i), we have $c \lesim c \wedge d \le a \vee b \lesim b$ whence $c \lesim b$ by transitivity.
Thus $D$ is a $<|$-ideal iff $\lesim$ is an order-congruence.

It is clear from this correspondence that a join $\bigvee_i (a_i \wedge \neg b_i)$ in $\@N_\kappa(A)$ or in $\kappa\@I(\@N_\kappa(A))$, with $a_i, b_i \in A$, corresponds to the $<|$-ideal generated by the pairs $(a_i, \neg b_i)$, which corresponds to the order-congruence $\lesim$ generated by $a_i \lesim b_i$.
In particular, a $\kappa$-ary such join corresponds to a $\kappa$-generated order-congruence, yielding the second isomorphism.
\end{proof}

Now define inductively for each ordinal $\alpha$
\begin{align*}
\@N_\kappa^0(A) &:= A, \\
\@N_\kappa^{\alpha+1}(A) &:= \@N_\kappa(\@N_\kappa^\alpha(A)), \\
\@N_\kappa^\alpha(A) &:= \injlim_{\beta < \alpha} \@N_\kappa^\beta(A) \quad\text{for $\alpha$ limit},
\end{align*}
where the colimit is taken in $\!{\kappa Frm}$ with respect to the inclusions $\@N_\kappa^\beta(A) `-> \@N_\kappa^{\beta+1}(A) `-> \dotsb$.
By induction, using \cref{thm:frm-neg-inj} for the successor stages and \cref{thm:frm-colim-dir-ker} for the limit stages, for $\alpha \le \beta$, we have an injective homomorphism $\eta_{\alpha\beta} : \@N_\kappa^\alpha(A) -> \@N_\kappa^\beta(A)$.
Since the operations of $\kappa$-frames are $\kappa$-ary, the forgetful functor $\!{\kappa Frm} -> \!{Set}$ preserves $\kappa$-directed limits; thus every element of $\@N_\kappa^\kappa(A)$ belongs to some earlier stage, and so already has a complement in $\@N_\kappa^\kappa(A)$.
Thus
\begin{align*}
\@N_\kappa^\kappa(A) = \@N_\kappa^\infty(A) = \ang{A \qua \!{\kappa Frm}}_\!{\kappa Bool}.
\end{align*}

\begin{corollary}
\label{thm:frm-neginf-inj}
The unit $\eta = \eta_{0\kappa} = \eta_{0\infty} : A -> \@N_\kappa^\infty(A) = \ang{A \qua \!{\kappa Frm}}_\!{\kappa Bool}$ of the free/forgetful adjunction $\!{\kappa Frm} \rightleftarrows \!{\kappa Bool}$ is injective.
Equivalently, the left adjoint $\@N_\kappa^\infty = \ang{-}_\!{\kappa Bool}$ is faithful.
\qed
\end{corollary}

This allows us to transfer results from \cref{sec:frm-colim} about colimits in $\!{\kappa Frm}$ to $\!{\kappa Bool}$:

\begin{corollary}
\label{thm:bool-colim-dir-ker}
Let $F : \!I -> \!{\kappa Bool}$ be a directed diagram, $I \in \!I$, $a \in F(I)$.
Suppose that for each $I \le J \le K \in \!I$, $F(I, J)(a) \in F(J)$ is the greatest element of its $\oker(F(J, K))$-principal ideal.
Then $a$ is the greatest element of its $\oker(\iota_I)$-principal ideal.

In particular, if each $F(I, J) : F(I) -> F(J)$ is injective, then so is each $\iota_I : F(I) -> \injlim F$.
\end{corollary}
\begin{proof}
The colimit $\injlim_\!{\kappa Bool} F$ in $\!{\kappa Bool}$ is the reflection $\@N_\kappa^\infty(\injlim_\!{\kappa Frm} F)$ of the colimit in $\!{\kappa Frm}$; the result thus follows from \cref{thm:frm-colim-dir-ker}, \cref{thm:frm-colim-dir-vlat}, and \cref{thm:frm-neginf-inj}.
\end{proof}

\begin{corollary}
\label{thm:bool-coprod}
Let $A, B$ be $\kappa$-Boolean algebras, $a, a' \in A$, and $b, b' \in b$.
If $a \times b \le a' \times b'$ in the coproduct $A \amalg B$ of $A, B$ in $\!{\kappa Bool}$, then either $a \le a'$ and $b \le b'$, or $a = \bot$, or $b = \bot$.

In particular, if $a \times b = \bot$, then either $a = \bot$ or $b = \bot$.
\end{corollary}
\begin{proof}
This similarly follows from \cref{thm:frm-coprod} and \cref{thm:frm-neginf-inj}.
\end{proof}

We can also use $\@N_\kappa$ to construct the free functor $\ang{-}_\!{\lambda Bool} : \!{\kappa Bool} -> \!{\lambda Bool}$ for $\kappa \le \lambda$.
Indeed, since $\!{\kappa Bool} \subseteq \!{\kappa Frm}$ is a reflective subcategory,
\begin{align*}
\ang{A \qua \!{\kappa Bool}}_\!{\lambda Bool}
&\cong \ang{\ang{A \qua \!{\kappa Frm}}_\!{\kappa Bool}}_\!{\lambda Bool} \\
&\cong \ang{\ang{A \qua \!{\kappa Frm}}_\!{\lambda Frm}}_\!{\lambda Bool} \\
&\cong \@N_\lambda^\infty(\kappa\@I_\lambda(A))
\end{align*}
(see the diagram \eqref{diag:frm-cat}).
Since the units for $\kappa\@I_\lambda$ and $\@N_\lambda^\infty$ are both injective,

\begin{corollary}
\label{thm:frm-neginf-bool-inj}
For $\kappa \le \lambda$, the unit of the free/forgetful adjunction $\!{\kappa Bool} \rightleftarrows \!{\lambda Bool}$ is injective.
Equivalently, the left adjoint $\ang{-}_\!{\lambda Bool}$ is faithful.
\qed
\end{corollary}

\begin{remark}
The preceding result is also easily seen ``semantically'', via MacNeille completion.
On the other hand, \cref{thm:frm-dpoly-bifrm-inj} below yields yet another, purely ``syntactic'', proof.
\end{remark}

\begin{remark}
For $\kappa \le \lambda$, unlike the free/forgetful adjunctions $\!{\kappa Frm} \rightleftarrows \!{\lambda Frm}$ and $\!{\kappa{\bigvee}Lat} \rightleftarrows \!{\lambda{\bigvee}Lat}$ (see \cref{thm:frm-idl-fulliso}), the free functor $\!{\kappa Bool} -> \!{\lambda Bool}$ is \emph{not} full on isomorphisms.
However, we will see below (\cref{thm:bool-free-cons}) that it is at least conservative.
\end{remark}

By \cref{thm:frm-neginf-inj,thm:frm-neginf-bool-inj}, every $\kappa$-frame or $\kappa$-Boolean algebra is embedded in the free (large) complete Boolean algebra it generates.
We now extend \cref{cvt:frm-neg-incl} to

\begin{convention}
\label{cvt:frm-cbool-incl}
For any $\kappa$-frame, $\kappa$-coframe, or $\kappa$-Boolean algebra $A$, \textbf{we regard $A$ as a subalgebra of the free complete Boolean algebra it generates}.
That is, we regard the unit $A `-> \ang{A}_\!{CBOOL}$ as an inclusion.
Hence we also regard all of the intermediate units $A `-> \ang{A}_\!{\lambda Frm} = \kappa\@I_\lambda(A)$ for $\lambda \ge \kappa$, $A `-> \@N_\kappa^\alpha(A)$, etc., as inclusions.

In particular, we adopt the (uncommon in locale theory) convention that \textbf{from now on, unless otherwise specified, all infinitary Boolean operations take place in $\ang{A}_\!{CBOOL}$}.
Thus for example, for elements $a_i$ in a frame $A$, $\bigwedge_i a_i$ will by default no longer refer to their meet in the complete lattice $A$, but will instead refer to their meet in $\ang{A}_\!{CBOOL}$.
Similarly, $->, \neg$ will never be used to denote the Heyting implication or pseudocomplement in a frame $A$, but will instead denote Boolean implication or complement in $\ang{A}_\!{CBOOL}$.
Note that for $a, b, a_i \in A$, we have $a -> b, \neg a \in \@N(A) \subseteq \@N^\infty(A) = \ang{A}_\!{CBOOL}$; thus $\bigwedge_i a_i = \neg \bigvee_i \neg a_i \in \@N^2(A) \subseteq \@N^\infty(A)$.

We also extend this convention to $(\lambda, \kappa)$-frames, once we know they also embed into their free complete Boolean algebras; see \cref{thm:frm-dpoly-bifrm-inj}.
\end{convention}

We conclude this subsection by showing that $\kappa$-Boolean algebras can be presented by ``enough'' $\kappa$-frames.
First, from general categorical principles, we have

\begin{proposition}
\label{thm:frm-bool-pres}
For every $\kappa$-presented $\kappa$-Boolean algebra $B$, there is a $\kappa$-presented $\kappa$-frame $A$ such that $\@N_\kappa^\infty(A) = \ang{A}_\!{\kappa Bool} \cong B$.
\end{proposition}
\begin{proof}
If $\kappa = \omega$, then $B$ is finite, so we may take $A = B$.
For uncountable $\kappa$, this follows from reflectivity and local $\kappa$-presentability of $\!{\kappa Bool} \subseteq \!{\kappa Frm}$; see \cref{thm:cat-lim-refl}.
\end{proof}

\begin{remark}
\label{rmk:frm-bool-gen-free}
By \cref{thm:frm-neginf-inj}, the conclusion above implies that (up to isomorphism) $A \subseteq B$ is a $\kappa$-subframe generating $B$ as a $\kappa$-Boolean algebra.
However, this condition is not sufficient for the conclusion to hold: for example, the topology $\@O(2^\#N)$ of Cantor space $2^\#N$ is a subframe of $\@P(2^\#N)$ generating it as a complete Boolean algebra, but not freely (e.g., by the Gaifman--Hales \cref{thm:gaifman-hales}, since $\@O(2^\#N)$ is the ideal completion of the algebra of clopen sets in $2^\#N$ which is the free Boolean algebra on $\#N$, whence $\ang{\@O(2^\#N) \qua \!{Frm}}_\!{CBOOL} \cong \ang{\#N \qua \!{Set}}_\!{CBOOL}$).
\end{remark}

Unravelling the proof of \cref{thm:frm-bool-pres} yields the following refined statement:

\begin{proposition}
\label{thm:frm-bool-pres-cof}
Let $B$ be a $\kappa$-presented $\kappa$-Boolean algebra which is freely generated by a $\kappa$-subframe $A \subseteq B$.
For any $\kappa$-ary subset $C \subseteq A$, there is a $\kappa$-presented $\kappa$-subframe $A' \subseteq A$ which contains $C$ and still freely generates $B$.
\end{proposition}
\begin{proof}
If $\kappa = \omega$, this is again trivial (take $A' := A$), so suppose $\kappa$ is uncountable.
As in the proof of \cref{thm:cat-lim-refl}, write $A$ as a $\kappa$-directed colimit $\injlim_{i \in I} A_i$ of $\kappa$-presented frames; we may assume that each $A_i -> A$ has image containing $C$ (by taking the $G_i$ in that proof to contain $C$).
Then $\injlim_i \@N_\kappa^\infty(A_i) \cong \@N_\kappa^\infty(\injlim_i A_i) = B$, so as in that proof, we have $B \cong \injlim_j \@N_\kappa^\infty(A_{i_j}) \cong \@N_\kappa^\infty(A')$ where $A' := \injlim_j A_{i_j}$ for some $i_0 \le i_1 \le \dotsb$.
From that proof, the unit $A' -> \@N_\kappa^\infty(A') \cong B$ is the colimit comparison $A' -> A \subseteq B$, and is injective (\cref{thm:frm-neginf-inj}), whence up to isomorphism, $A'$ is a $\kappa$-presented $\kappa$-subframe of $A \subseteq B$, which contains $C$ since the image of each $A_{i_j} -> A$ does.
\end{proof}

\begin{corollary}
\label{thm:frm-neginf-pres}
Let $A$ be a $\kappa$-presented $\kappa$-frame.
For any $\alpha$ and $\kappa$-ary subset $C \subseteq \@N_\kappa^\alpha(A)$, there is a $\kappa$-presented $\kappa$-subframe $A' \subseteq \@N_\kappa^\alpha(A)$ such that $A \cup C \subseteq A'$ and $\@N_\kappa^\infty(A) \cong \@N_\kappa^\infty(A')$ (canonically).
\end{corollary}
\begin{proof}
Take $A$ in \cref{thm:frm-bool-pres-cof} to be $\@N_\kappa^\alpha(A)$, and $C$ to be the present $C$ together with a $\kappa$-ary generating set for $A$.
\end{proof}

\begin{remark}
If $C \subseteq \@N_\kappa^\alpha(A)$ in \cref{thm:frm-neginf-pres} consists of complemented elements, then we may ensure they become complemented in $A'$, by adding the complements to $C$.
\end{remark}

\subsection{Zero-dimensionality and ultraparacompactness}
\label{sec:upkzfrm}

In this subsection, which is somewhat technical and will not be used in an essential way until \cref{sec:sigma11-invlim}, we study $\kappa$-frames which have ``enough'' complemented elements, as well as a more subtle infinitary refinement of such a condition.
The main results are towards the end of the subsection, concerning the connection of such frames with $\kappa$-Boolean algebras (\cref{thm:upkzfrm-cabool-dircolim,thm:upkzfrm-bool-pres}) as well as with the $\@N_\kappa$ functor from the preceding subsection (\cref{thm:frm-neg-upkz}).
Many of these results are known for $\kappa = \infty$, due to work of Paseka~\cite{Pupk} and Plewe~\cite{Psubloc}.
However, in generalizing to $\kappa < \infty$, we once again follow a ``presentational'' approach (which was essentially introduced by van~Name~\cite{vNultra} when $\kappa = \infty$).
Here, our approach yields proofs which are entirely different in spirit from the known ones for $\kappa = \infty$, and which (we believe) make the underlying algebraic ideas quite transparent.

For a $\kappa$-frame $A$, let $A_\neg \subseteq A$ denote the Boolean subalgebra of complemented elements.
We call $A$ \defn{zero-dimensional} if it is generated by $A_\neg$, and let $\!{\kappa ZFrm} \subseteq \!{\kappa Frm}$ denote the full subcategory of zero-dimensional $\kappa$-frames.

It is easily seen that $\!{\kappa ZFrm} \subseteq \!{\kappa Frm}$ is a coreflective subcategory, with the coreflection of $A \in \!{\kappa Frm}$ being the $\kappa$-subframe generated by $A_\neg$.
We may view this coreflective adjunction $\!{\kappa ZFrm} \rightleftarrows \!{\kappa Frm}$ as a free/forgetful adjunction, provided we change the ``underlying set'' of a zero-dimensional $\kappa$-frame $A$ to mean $A_\neg$, as we now explain (see \cref{sec:cat-lim} for a general explanation of this method of regarding arbitrary adjunctions as free/forgetful).
The following notion was essentially introduced by van~Name~\cite{vNultra} when $\kappa = \infty$, under different terminology.

Let $A$ be a zero-dimensional $\kappa$-frame.
Since $A_\neg \subseteq A$ is a $\wedge$-sublattice generating $A$, the canonical $\kappa$-coverage $<|$ on $A$ restricted to $B := A_\neg$ yields a $\kappa$-posite $(B, {<|})$ presenting $A$.
Since $B$ is Boolean, $<|$ is completely determined by its restriction to left-hand side $\top$, via
\begin{align}
\label{eq:frm-bpost}
b <| C \iff \top <| \{\neg b\} \cup C.
\end{align}
Since $B = A_\neg \subseteq A$ is closed under finite joins, we have
\begin{itemize}
\item  ($\omega$-supercanonicity) $<|$ extends the canonical $\omega$-coverage on $B$, i.e., for finite $C \subseteq B$, we have $\bigvee C = \top \implies \top <| C$.
\end{itemize}
The other axioms on a $\kappa$-coverage (see \cref{sec:frm-post}) imply
\begin{itemize}
\item  (monotonicity) if $\top <| C \subseteq D$, then $\top <| D$;
\item  (transitivity) if $\top <| C \cup D$, and $\top <| \{\neg c\} \cup D$ for every $c \in C$, then $\top <| D$.
\end{itemize}

\begin{lemma}
These three axioms on the unary predicate $\top <|$ on $\@P_\kappa(B)$ imply that $<|$ given by \eqref{eq:frm-bpost} is a $\kappa$-coverage on $B$.
\end{lemma}
\begin{proof}
Reflexivity follows from $\omega$-supercanonicity (whence $\top <| \{a, \neg a\}$) and monotonicity.

Right-transitivity of $<|$ follows from transitivity of $<|$.

$\wedge$-stability: suppose $a \le b <| C$, i.e., $\top <| \{\neg b\} \cup C$; we must show $a <| a \wedge C$, i.e., $\top <| \{\neg a\} \cup (a \wedge C)$.
Since $a \le b$, we have $\neg a \vee b = \top$, whence by $\omega$-supercanonicity, $\top <| \{\neg a, b\}$.
By monotonicity, we have $\top <| \{\neg a, b\} \cup C$ and $\top <| \{\neg a, \neg b\} \cup C$, whence by transitivity, $\top <| \{\neg a\} \cup C$, whence by monotonicity, $\top <| \{\neg a\} \cup C \cup (a \wedge C)$.
For each $c \in C$, we have $\top <| \{\neg a, \neg c, a \wedge c\}$ by $\omega$-supercanonicity whence $\top <| \{\neg a, \neg c\} \cup (a \wedge C)$ by monotonicity.
Thus $\top <| \{\neg a\} \cup (a \wedge C)$ by transitivity.
\end{proof}

We call a Boolean algebra $B$ equipped with a unary predicate $\top <|$ on $\@P_\kappa(B)$ obeying these three axioms a \defn{Boolean $\kappa$-posite}, and denote the category of all such by $\!{\kappa BPost}$.
Via \eqref{eq:frm-bpost}, this is the same thing as a $\kappa$-posite $(B, {<|})$ such that $<|$ remembers finite joins in $B$, hence yields a $\kappa$-frame $A = \ang{B \mid {<|}}_\!{\kappa Frm}$ together with a Boolean homomorphism $\eta : B -> A$ whose image generates $A$.
Let $\!{\kappa SBPost} \subseteq \!{\kappa BPost}$ denote the full subcategory of separated Boolean $\kappa$-posites $(B, {<|})$; this means
\begin{align*}
\top <| C \implies \bigvee C = \top \in B.
\end{align*}
We have a ``forgetful'' functor $\!{\kappa Frm} -> \!{\kappa SBPost} \subseteq \!{\kappa BPost}$ taking $A$ to $A_\neg$ equipped with the restriction of the canonical $<|$ on $A$, with left adjoint $(B, {<|}) |-> \ang{B \mid {<|}}_\!{\kappa Frm}$, where the counit $\ang{A_\neg \mid {<|}}_\!{\kappa Frm} -> A$ at $A \in \!{\kappa Frm}$ is an isomorphism iff $A$ was zero-dimensional.
Thus, we have represented $\!{\kappa ZFrm}$ as equivalent to a full reflective subcategory of $\!{\kappa SBPost}$, consisting of all those separated Boolean $\kappa$-posites $(B, {<|})$ for which the unit $\eta : B -> (\ang{B \mid {<|}}_\!{\kappa Frm})_\neg$ is an isomorphism.
Since $<|$ is assumed to be separated, $\eta$ is always an embedding; thus $\eta$ is an isomorphism iff it is surjective, which, by expressing a complemented element of $\ang{B \mid {<|}}_\!{\kappa Frm}$ as a join of elements of the image of $\eta$, is easily seen to mean
\begin{itemize}
\item  (complement-completeness)
for $C, D \in \@P_\kappa(B)$, if $\top <| C \cup D$, and $c \wedge d = \bot \in B$ for every $c \in C$ and $d \in D$, then there is $b \in B$ such that $\top <| \{\neg b\} \cup C$ and $\top <| \{b, \neg c\}$ for every $c \in C$
(i.e., $b <| C$ and $c <| \{b\}$ for every $c \in C$; by separatedness, this implies $b = \bigvee C$ in $B$).
\end{itemize}
We let $\!{\kappa CCBPost} \subseteq \!{\kappa SBPost}$ denote the full subcategory of complement-complete separated Boolean $\kappa$-posites.
Summarizing, we have represented the coreflective adjunction $\!{\kappa ZFrm} \rightleftarrows \!{\kappa Frm}$ as equivalent to the free/forgetful adjunction $\!{\kappa CCBPost} \rightleftarrows \!{\kappa Frm}$:
\begin{equation}
\label{diag:zfrm}
\begin{tikzcd}
\!{\kappa BPost} \rar[shift left=2] &
\!{\kappa CCBPost} \lar[shift left=2, rightarrowtail, right adjoint] \rar[phantom, "\simeq"] \ar[rr, shift left=2, bend left=15, "\ang{-}_\!{\kappa Frm}"] &[-2em]
\!{\kappa ZFrm} \rar[shift left=2, rightarrowtail] &
\!{\kappa Frm} \lar[shift left=2, right adjoint] \ar[ll, shift left=2, bend left=15, "{(-)_\neg}"]
\end{tikzcd}
\end{equation}

Note that the above categories $\!{\kappa BPost}, \!{\kappa SBPost}, \!{\kappa CCBPost}$ are all locally $\kappa$-presentable.
For $\!{\kappa CCBPost}$, this requires the use of unique existential quantifiers to axiomatize the existence of joins; see \cref{ex:cat-lim-essalg}.

We have the following weakenings of zero-dimensionality.
A $\kappa$-frame $A$ is \defn{regular} \cite[\S3]{Mkfrm} if every $a \in A$ is a $\kappa$-ary join of $b \in A$ for which there exist $c \in \neg A$ ($= \{\neg d \mid d \in A\}$) with $b \le c \le a$ in $\ang{A}_\!{CBOOL}$ (or in $\@N_\kappa(A)$); \defn{fit} if every $a \in A$ is a $\kappa$-ary join, in $\ang{A}_\!{CBOOL}$ (or in $\@N_\kappa(A)$), of $\neg A \ni b \le a$; and \defn{subfit} if every $a \in A$ is the least, in $A$, upper bound (in $\ang{A}_\!{CBOOL}$) of a $\kappa$-ary set of $\neg A \ni b \le a$.
(The latter notions are due to Isbell~\cite[\S2]{Iloc} when $\kappa = \infty$; we have given an equivalent definition in terms of $\ang{A}_\!{CBOOL}$.)
Clearly,
\begin{align*}
\text{zero-dimensional}
\implies
\text{regular}
\implies
\text{fit}
\implies
\text{subfit}.
\end{align*}
(Sometimes, zero-dimensional is called \emph{ultraregular}; see e.g., \cite{Eultra}.)
We also have the following strengthening.
A $\kappa$-frame $A$ is \defn{ultranormal} if every disjoint $a, b \in \neg A$ are separated by some $c \in A_\neg$, i.e., $a \le c \le \neg b$; equivalently, every $\neg a, \neg b \in A$ which cover $\top$ can be refined to a partition $\neg c, c$.
If $A$ is also subfit, then this implies that each $\neg b \in A$ is a $\kappa$-ary join of $c \in A_\neg$, i.e.,
\begin{align*}
\text{subfit} + \text{ultranormal} \implies \text{zero-dimensional}.
\end{align*}
See \cite[Ch.~V]{PPloc} for an extended discussion of these and other localic separation axioms.

We next consider the $\kappa$-ary generalization of ultranormality.
We say that $A$ is \defn{ultraparacompact} if every $\kappa$-ary cover $(b_i)_i$ of $\top \in A$ refines to a partition $(c_i)_i$, i.e., whenever $\top \le \bigvee_i b_i$, then there are $c_i \le b_i$ with $\top \le \bigvee_i c_i$ and $c_i \wedge c_j = \bot$ for $i \ne j$ (whence in fact each $c_i \in A_\neg$).
Clearly,
\begin{align*}
\text{ultraparacompact} \implies \text{ultranormal};
\end{align*}
hence
\begin{align*}
\text{subfit} + \text{ultraparacompact} \implies \text{zero-dimensional}.
\end{align*}
More generally, we say that $a \in A$ is \defn{ultraparacompact} if the same condition defining ultraparacompactness of $A$ holds replacing $\top$ with $a$ (without requiring $c_i \in A_\neg$), i.e., every $\kappa$-ary cover $(b_i)_i$ of $a$ refines to a pairwise disjoint $\kappa$-ary cover $(c_i)_i$ (whence by replacing $c_i$ with $a \wedge c_i$, we may ensure $a = \bigvee_i c_i$).
We let $A_\upk \subseteq A$ denote the ultraparacompact elements.

\begin{lemma}
For any $\kappa$-frame $A$, $A_\upk \subseteq A$ is closed under pairwise disjoint $\kappa$-ary joins.
\end{lemma}
\begin{proof}
Let $a_i \in A_\upk$ be $<\kappa$-many pairwise disjoint ultraparacompact elements.
If $\bigvee_i a_i \le \bigvee_j b_j$ for $<\kappa$-many $b_j$, then for each $i$, we have $a_i \le \bigvee_j c_{ij}$ for some pairwise disjoint $c_{ij} \le b_j$, whence $\bigvee_i a_i = \bigvee_{i,j} (a_i \wedge c_{ij})$ is a pairwise disjoint cover of $\bigvee_i a_i$.
\end{proof}

As with zero-dimensionality, it is convenient to introduce an auxiliary category to facilitate the study of ultraparacompactness.
By a \defn{disjunctive $\kappa$-frame} (see \cite{Cborin}), we will mean a poset $A$ which is a $\wedge$-lattice, has a least element $\bot$ (whence we may speak of disjointness), and has pairwise disjoint $\kappa$-ary joins, which we also denote using $\bigsqcup$ instead of $\bigvee$, over which finite meets distribute.
Denote the category of all such by $\!{\kappa DjFrm}$.
Then $\!{\kappa DjFrm}$ is locally $\kappa$-presentable, where the partial join operations are again axiomatized using unique existentials (see \cref{ex:cat-ord-djfrm}).

We have a forgetful functor $\!{\kappa Frm} -> \!{\kappa DjFrm}$, preserving the underlying $\wedge$-lattice but forgetting about non-pairwise-disjoint joins, whose left adjoint may be constructed as follows.
Given a disjunctive $\kappa$-frame $A$, the \defn{canonical disjunctive $\kappa$-coverage} on $A$ is given by
\begin{align*}
a <| C
&\coloniff  \exists \text{ $\kappa$-ary pairwise disjoint } \{d_i\}_i \subseteq \down C\, (a \le \bigsqcup_i d_i)  \quad\text{for $C \in \@P_\kappa(A)$} \\
&\iff  \exists \text{ $\kappa$-ary pairwise disjoint } \{d_i\}_i \subseteq \down C\, (a = \bigsqcup_i d_i).
\end{align*}
It is readily seen that this defines a separated $\kappa$-coverage on $A$, such that ${<|}\@I_\kappa(A) = \ang{A \mid {<|}}_\!{\kappa Frm} \cong \ang{A \qua \!{\kappa DjFrm}}_\!{\kappa Frm}$.
Call the $<|$-ideals \defn{disjunctive $\kappa$-ideals}, denoted $\dj\kappa\@I(A) := {<|}\@I(A)$, with $\kappa$-generated ideals $\dj\kappa\@I_\kappa(A) \subseteq \dj\kappa\@I(A)$.
Note that the disjunctive $\kappa$-ideal generated by $C \subseteq A$ is
\begin{align*}
\downtri C = \{a \mid a <| C\} = \{a \mid \exists \text{ $\kappa$-ary pairwise disjoint } \{d_i\}_i \subseteq \down C\, (a = \bigsqcup_i d_i)\}.
\end{align*}
Thus, for $B \in \!{\kappa Frm}$, $b \in B$ is ultraparacompact iff whenever $C \in \@P_\kappa(A)$ with $b \le \bigvee C$, then $b \in \downtri C$.
We now have the following, which is analogous to \cref{thm:vlat-cptbasis}:

\begin{proposition}
\label{thm:frm-upkbasis}
\leavevmode
\begin{enumerate}
\item[(a)]  For $A \in \!{\kappa DjFrm}$, we have $\dj\kappa\@I_\kappa(A)_\upk = \down(A) \cong A$.
In particular, $\dj\kappa\@I_\kappa(A)_\upk \subseteq \dj\kappa\@I_\kappa(A)$ forms a generating disjunctive $\kappa$-subframe.
\item[(b)]  If $B \in \!{\kappa Frm}$ and $A \subseteq B_\upk$ is a generating disjunctive $\kappa$-subframe of $B$, then we have an isomorphism $\bigvee : \dj\kappa\@I_\kappa(A) \cong B$, hence in fact $A = B_\upk$.
\end{enumerate}
\end{proposition}
\begin{proof}
We will need

\begin{lemma}
\label{lm:frm-upkbasis}
Let $B \in \!{\kappa Frm}$ and $A \subseteq B$ be closed under pairwise disjoint $\kappa$-ary joins and generate $A$ under $\kappa$-ary joins.
Let $c_i \in B_\upk$ be $<\kappa$-many ultraparacompact elements, and let $\bigvee_i c_i = \bigsqcup_i d_i$ with $c_i \ge d_i \in B$ be a pairwise disjoint refinement.
Then each $d_i \in A$.
\end{lemma}
\begin{proof}
Since $A$ generates $B$ under $\kappa$-ary joins, each $d_i = \bigvee D_i$ for some $D_i \in \@P_\kappa(A)$.
Since each $c_i$ is ultraparacompact and $c_i \le \bigvee_i d_i = \bigvee \bigcup_i D_i$, we have $c_i = \bigsqcup_j d_{ij}$ for some $d_{ij}$ each of which is a pairwise disjoint join of elements from $D_j$.
For each $d \in D_i$, since $\bigvee D_i = d_i \le c_i$, we have $d = d \wedge c_i = \bigsqcup_j (d \wedge d_{ij}) = d \wedge d_{ii} \le d_{ii}$, since for $j \ne i$ we have $d \wedge d_{ij} \le \bigvee D_i \wedge \bigvee D_j = d_i \wedge d_j = \bot$.
Thus each $d_i = \bigvee D_i = d_{ii} \in A$.
\end{proof}

(a)  Clearly principal ideals are ultraparacompact in $\dj\kappa\@I_\kappa(A)$.
Conversely, for $C \in \dj\kappa\@I_\kappa(A)_\upk$ generated by $\kappa$-ary $\{c_i\}_i \subseteq C$, we have $C = \bigvee_i \down c_i$ in $\dj\kappa\@I_\kappa(A)$, so using ultraparacompactness of $C$ and \cref{lm:frm-upkbasis} (with $A := \down(A) \subseteq \dj\kappa\@I_\kappa(A) =: B$), we have pairwise disjoint $d_i \le c_i$ with $C = \bigvee_i \down d_i = \down \bigsqcup_i d_i$.

(b)  Since $A \subseteq B$ forms a generating $\wedge$-sublattice, we have $B \cong \ang{A \mid {<|}}_\!{\kappa Frm}$ for the restriction $<|$ of the canonical $\kappa$-coverage on $B$; thus it is enough to show that $<|$ agrees with the canonical disjunctive $\kappa$-coverage $<|'$ on $A$.
Clearly the former contains the latter.
Conversely, for $b <| \{c_i\}_i \in \@P_\kappa(A)$, i.e., $b \le \bigvee_i c_i$ in $B$, we have $b = \bigvee_i (b \wedge c_i)$ with $b \wedge c_i \in A \subseteq B_\upk$ by assumption, so using ultraparacompactness of $b$ and \cref{lm:frm-upkbasis}, we have pairwise disjoint $b \wedge c_i \ge d_i \in A$ with $b = \bigsqcup_i d_i$, which witnesses $b <|' \{c_i\}_i$.
The last assertion follows from (a).
\end{proof}

\begin{remark}
\label{rmk:upkzfrm-kzmon}
As with \cref{thm:frm-idl-fulliso}, it follows that the free functor $\dj\kappa\@I_\kappa : \!{\kappa DjFrm} -> \!{\kappa Frm}$ is full on isomorphisms.
The $\kappa$-frames in its essential image, i.e., satisfying \cref{thm:frm-upkbasis}(b), might be called ``ultraparacoherent'', by analogy with coherent frames.

The preceding discussion, along with \cref{thm:vlat-cptbasis}, both fit into the categorical framework of KZ-monads; see \cite[B1.1.13--15]{Jeleph}, \cite{Kkzmon}, \cite{Mmoncomon}.
Namely, we have a KZ-monad $\dj\kappa\@I_\kappa$ on $\!{\kappa DjFrm}$, with category of algebras $\!{\kappa Frm}$; \cref{thm:frm-upkbasis}(a) shows that this monad has regular monic unit, whence $\!{\kappa DjFrm}$ embeds fully faithfully into the category of coalgebras for the induced comonad.
These coalgebras are the retracts of ``ultraparacoherent'' frames, or equivalently frames with a $\wedge$-stable ``relatively ultraparacompact'' or ``way-below'' relation $\ll$, hence might be called ``stably ultraparacompact'' frames, in analogy with stably compact frames (see \cite[VII~4.6]{Jstone}), localic $\vee$-lattices (see \cite[\S6.4]{Spowloc}), etc.
We do not study these in detail here, but make one preliminary observation: spaces of points of ``stably ultraparacompact'' frames have connected meets, since these commute with pairwise disjoint joins in $2 = \{0 < 1\}$; thus such frames can be thought of as dual to certain locales with connected meets.
\end{remark}

Finally, we consider $\kappa$-frames which are both ultraparacompact and zero-dimensional (or equivalently subfit, as noted above).
Let $\!{\kappa UPKZFrm} \subseteq \!{\kappa ZFrm}$ denote the full subcategory of these.

\begin{proposition}
\label{thm:upkzfrm}
For a zero-dimensional $\kappa$-frame $A$, the following are equivalent:
\begin{enumerate}
\item[(i)]  $A$ is ultraparacompact;
\item[(ii)]  $A_\upk \subseteq A$ forms a generating disjunctive $\kappa$-subframe (i.e., $A$ is ``ultraparacoherent'');
\item[(iii)]  every $a \in A_\neg$ is ultraparacompact;
\item[(iv)]  every $\kappa$-ary cover of $a \in A_\neg$ by complemented elements refines to a pairwise disjoint cover;
\item[(v)]  every $\kappa$-ary cover of $a \in A_\neg$ refines to a pairwise disjoint cover by complemented elements.
\end{enumerate}
In this case, $A_\upk \subseteq A$ is the closure of $A_\neg$ under pairwise disjoint $\kappa$-ary joins.
\end{proposition}
\begin{proof}
Clearly (ii)$\implies$(i) and (v)$\implies$(iii)$\implies$(iv).

(iv)$\implies$(iii): Given a $\kappa$-ary cover $a \le \bigvee_i c_i$ of $a \in A_\neg$ by $c_i \in A$, write each $c_i$ as a $\kappa$-ary join of $c_{ij} \in A_\neg$, and refine these to $a = \bigsqcup_{i,j} d_{ij}$ with $d_{ij} \le c_{ij}$; then $d_i := \bigsqcup_j d_{ij} \le c_i$ with $a = \bigsqcup_i d_i$.

(iii)$\implies$(ii) follows from \cref{thm:frm-upkbasis}(b) (with $A :=$ closure of $A_\neg$ under pairwise disjoint $\kappa$-ary joins).

(i)$\implies$(v): Given a $\kappa$-ary cover $a \le \bigvee_i c_i$ of $a \in A_\neg$, letting $a = \bigsqcup_i d_i$ with $d_i \le a \wedge c_i$, since $a$ is complemented and the $d_i$ are pairwise disjoint, each $d_i$ has complement $\neg a \vee \bigvee_{j \ne i} d_j$.

The last statement follows from (iii) and \cref{thm:frm-upkbasis}(b) (with $A :=$ closure of $A_\neg$ under pairwise disjoint $\kappa$-ary joins).
\end{proof}

Given $A \in \!{\kappa UPKZFrm}$, as described above \eqref{diag:zfrm}, we may equivalently represent $A$ as the canonical Boolean $\kappa$-posite $(A_\neg, {<|})$.
For $<\kappa$-many $a, c_i \in A_\neg$ with $a \le \bigvee_i c_i$, we may find a pairwise disjoint refinement $a = \bigsqcup_i d_i$ with $d_i \in A_\neg$ (by \cref{thm:upkzfrm}(v)); in other words, $<|$ is generated by its restriction to pairwise disjoint right-hand side.
Call such a $\kappa$-posite \defn{disjunctive}, and let $\!{\kappa DjBPost} \subseteq \!{\kappa BPost}$ be the full subcategory of disjunctive Boolean $\kappa$-posites, with the further full subcategories $\!{\kappa CCDjBPost} \subseteq \!{\kappa SDjBPost} \subseteq \!{\kappa DjBPost}$ of complement-complete and separated posites, respectively.
Conversely, given any disjunctive Boolean $\kappa$-posite $(B, {<|}) \in \!{\kappa DjBPost}$, the presented zero-dimensional $\kappa$-frame $\ang{B \mid {<|}}_\!{\kappa Frm}$ is ultraparacompact: indeed, since $<|$ is generated by its restriction to pairwise disjoint right-hand side, $\ang{B \mid {<|}}_\!{\kappa Frm}$ may be obtained by first taking the presented \emph{disjunctive} $\kappa$-frame $\ang{B \mid {<|}}_\!{\kappa DjFrm}$, and then completing to a $\kappa$-frame as in \cref{thm:frm-upkbasis}.
Thus, the equivalence of categories $\!{\kappa CCBPost} \simeq \!{\kappa ZFrm}$ from above \eqref{diag:zfrm} restricts to an equivalence $\!{\kappa CCDjBPost} \simeq \!{\kappa UPKZFrm}$ between complement-complete disjunctive Boolean $\kappa$-posites and ultraparacompact zero-dimensional $\kappa$-frames.

The following diagram, extending \eqref{diag:zfrm}, depicts the relationship between all the categories we have considered in this subsection (as before, $\rightarrowtail$ denotes full subcategory):
\begin{equation}
\label{diag:upkzfrm}
\begin{tikzcd}
\!{\kappa BPost}
    \dar[shift left=2, right adjoint']
    \rar[shift left=2] &
\!{\kappa CCBPost}
    \lar[shift left=2, rightarrowtail, right adjoint]
    \dar[shift left=2, right adjoint']
    \rar[phantom, "\scriptstyle\simeq"] \rar[shift left=2, "\ang{-}_\!{\kappa Frm}"] &
\!{\kappa ZFrm}
    \lar[shift left=2, "{(-)_\neg}"]
    \dar[shift left=2, right adjoint']
    \rar[shift left=2, rightarrowtail] &
\!{\kappa Frm}
    \lar[shift left=2, right adjoint]
    \dar[shift left=2, right adjoint']
    \rar[shift left=2] &
\!{\kappa Bool}
    \lar[shift left=2, rightarrowtail, right adjoint]
    \ar[dll, dashed, rightarrowtail, start anchor=south]
\\
\!{\kappa DjBPost}
    \uar[shift left=2, rightarrowtail]
    \rar[shift left=2] &
\!{\kappa CCDjBPost}
    \lar[shift left=2, rightarrowtail, right adjoint]
    \uar[shift left=2, rightarrowtail]
    \rar[start anchor=north east, end anchor=south west, shift left=2, phantom, "\scriptstyle\simeq"{sloped}]
    \rar[start anchor=north east, end anchor=south west, shift left=4, "\ang{-}_\!{\kappa Frm}"{inner sep=1pt}]
    \ar[rr, shift left=2, rightarrowtail, "\ang{-}_\!{\kappa DjFrm}"{pos=.7}] &
|[yshift=3em]|
\!{\kappa UPKZFrm}
    \lar[start anchor=south west, end anchor=north east, shift left=0, "{(-)_\neg}"{pos=.3,inner sep=1pt}]
    \uar[shift left=2, rightarrowtail]
&
\!{\kappa DjFrm}
    \uar[shift left=2, "{\ang{-}_\!{\kappa Frm}}"{pos=.3}]
    \ar[ll, shift left=2, right adjoint, "{(-)_\neg}"]
\end{tikzcd}
\end{equation}
Most of these functors are the obvious ones, or else were described above.
The right adjoint $\!{\kappa BPost} -> \!{\kappa DjBPost}$ takes a Boolean $\kappa$-posite $(B, {<|})$ and replaces $<|$ with the $\kappa$-coverage generated by its restriction to pairwise disjoint right-hand side
(this can be regarded as ``forgetful'', if we view a disjunctive $\kappa$-coverage $<|$ as already restricted to pairwise disjoint right-hand side).
The right adjoint $\!{\kappa DjFrm} -> \!{\kappa CCDjBPost} \subseteq \!{\kappa DjBPost}$ takes a disjunctive $\kappa$-frame $A$ to $A_\neg \subseteq A$, same as $\!{\kappa Frm} -> \!{\kappa CCBPost} \subseteq \!{\kappa BPost}$; the only thing to note is that pairwise disjoint joins are already enough to define complements and prove their basic properties (see e.g., \cite{Cborin}).

\begin{remark}
\label{rmk:upkzfrm-free}
As $\kappa$ varies, the above diagrams \eqref{diag:upkzfrm} fit together in the obvious manner.
It follows that ultraparacompactness and zero-dimensionality are preserved by the \emph{free} functors $\!{\kappa Frm} -> \!{\lambda Frm}$ for $\kappa \le \lambda$ (but not necessarily by the forgetful functors, in analogy with \cref{rmk:frm-idl-coh-free}).
\end{remark}

The above diagram \eqref{diag:upkzfrm} also indicates that the composite forgetful functor $\!{\kappa Bool} -> \!{\kappa UPKZFrm}$ (which takes a $\kappa$-Boolean algebra to itself) is full, which is just to say that $\!{\kappa Bool} \subseteq \!{\kappa UPKZFrm}$, i.e., every $\kappa$-Boolean algebra $A$ is ultraparacompact and (clearly) zero-dimensional.
Ultraparacompactness is because covers may be disjointified via the usual transfinite iteration: given $a = \bigvee_{i < \alpha} c_i$ for some ordinal $\alpha < \kappa$, we have $a = \bigsqcup_{i < \alpha} (c_i \wedge \neg \bigvee_{j < i} c_j)$.

Since $\!{\kappa UPKZFrm} \subseteq \!{\kappa Frm}$ is coreflective, it follows that colimits of $\kappa$-Boolean algebras in $\!{\kappa Frm}$ remain ultraparacompact zero-dimensional.
We now show that $\!{\kappa UPKZFrm}$ is precisely the closure of $\!{\kappa Bool} \subseteq \!{\kappa Frm}$ under colimits.
This result is essentially due to Paseka~\cite{Pupk} when $\kappa = \infty$.

Recall that a \defn{complete atomic Boolean algebra} is an isomorphic copy of a full powerset $\@P(X)$; let $\!{CABool}$ denote the category of these.
By a \defn{complete $\kappa$-atomic Boolean algebra}, we will mean a copy of $\@P(X)$ for $\kappa$-ary $X$; let $\!{\kappa CABool}_\kappa$ denote the category of these.
For $\kappa$-ary $X$, it is easily seen that $\@P(X)$ is presented, as a $\kappa$-frame (and hence as a $\kappa$-Boolean algebra), by
\begin{align}
\label{eq:frm-caba-pres}
\@P(X) = \ang{\{x\} \text{ for } x \in X \mid \{x\} \wedge \{y\} = \bot \text{ for $x \ne y$, } \top \le \bigvee_{x \in X} \{x\}};
\end{align}
thus $\!{\kappa CABool}_\kappa \subseteq \!{\kappa Frm}_\kappa, \!{\kappa Bool}_\kappa$.
Moreover, the same presentation yields a $\kappa$-presented disjunctive Boolean $\kappa$-posite $(A, {<|})$ (namely $A :=$ all finite and cofinite subsets of $X$, and $<|$ generated by $\top <| \{\{x\} \mid x \in X\}$) such that $\@P(X) = \ang{A \mid {<|}}_\!{\kappa Frm}$; thus $\!{\kappa CABool}_\kappa \subseteq \!{\kappa UPKZFrm}_\kappa :\simeq \!{\kappa CCDjBPost}_\kappa$.

Now given any ultraparacompact zero-dimensional $\kappa$-frame $A$, for $\kappa$-ary $X$, a homomorphism $f : \@P(X) -> A$ is, by the above presentation for $\@P(X)$, the same thing as a pairwise disjoint cover $\top \le \bigsqcup_{x \in X} f(x)$ in $A$.
Using this, we may convert a presentation of $A$ into an expression of $A$ as a colimit of complete $\kappa$-atomic Boolean algebras.
Namely, suppose $A = \ang{B \mid {<|}}_\!{\kappa Frm}$ for a disjunctive Boolean $\kappa$-posite $(B, {<|})$ (e.g., the canonical posite $(A_\neg, {<|})$), and suppose that $<|$, or rather its one-sided restriction $\top <|$, is in turn generated by some unary relation $\top <|_0$ on pairwise disjoint $\kappa$-ary subsets of $B$.
Then for each $C$ such that $\top <|_0 C$, so that $\top = \bigsqcup C$ in $A$, we get a homomorphism $f_C : \@P(C) -> A$ as above.
Suppose, furthermore, that
\begin{itemize}
\item  $\top <|_0 C  \implies  C \subseteq B \setminus \{\bot\}$;
\item  ($\omega$-supercanonicity) for every finite pairwise disjoint $C \subseteq B \setminus \{\bot\}$, we have $\bigvee C = \top \implies \top <|_0 C$;
\item  (closure under common refinement) for every $C, D$ such that $\top <|_0 C$ and $\top <|_0 D$, we have $\top <|_0 (C \wedge D) \setminus \{\bot\}$, where $C \wedge D := \{c \wedge d \mid c \in C \AND d \in D\}$.
\end{itemize}
Note that we can always modify $<|_0$ to satisfy these, by removing $\bot$ from all $C$ such that $\top <|_0 C$ and then closing under the last two conditions, which does not change the generated $<|$.

\begin{theorem}
\label{thm:upkzfrm-cabool-pres}
Under the above assumptions, the homomorphisms $f_C : \@P(C) -> A$ for $\top <|_0 C$ form cocone maps exhibiting $A$ as the directed colimit, in $\!{\kappa Frm}$, of the $\@P(C)$ together with the preimage maps $h_{CD}^{-1} : \@P(C) -> \@P(D)$ for $\top <|_0 C$ and $\top <|_0 D$ such that $D \subseteq \down C$, where $h_{CD} : D -> C$ takes $d \in D$ to the unique element above it in $C$.
\end{theorem}
\begin{proof}
By the general construction of colimits of presented structures (see \cref{sec:cat-lim}), the colimit $\injlim_{\top <|_0 C} \@P(C)$ is presented by all the $\@P(C)$ qua $\kappa$-frame, together with new relations which identify $c \in \@P(C)$ with $h_{CD}^{-1}(c) = \down c \cap D \in \@P(D)$.
Hence, incorporating the above presentations \eqref{eq:frm-caba-pres} of $\@P(C)$, we get
\begin{align*}
\injlim_{\top <|_0 C} \@P(C)
= \ang*{f_C(c) \text{ for } \top <|_0 C \ni c \relmiddle| \begin{aligned}
f_C(c) \wedge f_C(c') &= \bot &&\text{for $c \ne c' \in C$} \\
\bigvee_{c \in C} f_C(c) &= \top &&\text{for $\top <|_0 C$} \\
f_C(c) &= \bigvee_{c \ge d \in D} f_D(d) &&\text{for $D \subseteq \down C$}
\end{aligned}}_\!{\kappa Frm}.
\end{align*}
On the other hand, since $A = \ang{B \mid {<|}}_\!{\kappa Frm}$ and $<|_0$ generates $<|$,
\begin{align*}
A = \ang{B \qua \!{Bool} \mid \top \le \bigvee C \text{ for $\top <|_0 C$}}_\!{\kappa Frm}.
\end{align*}
Our goal is to show that these two presentations are equivalent.

First, note that every generator $b \in B$ in the second presentation, other than $\bot, \top$ which are redundant since they are $\kappa$-frame operations, corresponds to some generator $f_C(b)$ in the first presentation, namely for $C := \{b, \neg b\}$ by $\omega$-supercanonicity of $<|_0$.
Also, for each $b$, the $f_C(b)$ for different $C$ are identified by the first presentation: indeed, for $\top <|_0 C, D \ni b$, we have $f_C(b) = \bigvee_{b \ge c \wedge d \in (C \wedge D) \setminus \{\bot\}} f_{(C \wedge D) \setminus \{\bot\}}(c \wedge d) = f_D(b)$ by the third relation.
So the generators (other than $\bot, \top$) in the two presentations correspond to each other.

Since the relations in the first presentation clearly hold in $A$, it remains only to show that they imply the relations in the second presentation.
For any $\top <|_0 C$, the first presentation says that $C$ forms a partition of $\top$; in particular, the relations $\top \le \bigvee C$ for $\top <|_0 C$ are implied by the first presentation, as are finite partitions of $\top$ in $B$ by $\omega$-supercanonicity of $<|_0$.
For each $b \in B$, by considering the partition $\top = b \sqcup \neg b$, we get that complements are preserved by the first presentation.
For $b \le b' \in B$, by considering the partition $\top = b \sqcup \neg b' \sqcup (b' \wedge \neg b)$, we get that the partial order on $B$ is preserved.
Finally, for a binary join $b \vee b' \in B$, by considering the partition $\top = b \sqcup (b' \wedge \neg b) \sqcup \neg (b \vee b')$ where $b' \wedge \neg b \le b'$, we get that the join $b \vee b'$ is preserved.
\end{proof}

Taking $B := A_\neg$ and $<|_0$ to be the canonical $<|$ with all occurrences of $\bot$ removed yields

\begin{corollary}
\label{thm:upkzfrm-cabool-dircolim}
Every ultraparacompact zero-dimensional $\kappa$-frame is a directed colimit, in $\!{\kappa Frm}$, of complete $\kappa$-atomic Boolean algebras.
\qed
\end{corollary}

\begin{remark}
For a zero-dimensional $\kappa$-frame $A$, of course the generating set $A_\neg \subseteq A$ is the directed union of all finite (hence complete atomic) Boolean subalgebras $B \subseteq A_\neg$.
However, this does not generally yield an expression of $A$ as the directed colimit $\injlim_{B \subseteq A_\neg} B$ in $\!{\kappa Frm}$.
Indeed, since each $B = \ang{B \qua \!{Bool}}_\!{\kappa Frm}$, this colimit $\injlim_{B \subseteq A_\neg} B$ in $\!{\kappa Frm}$ is the same as the free functor $\ang{-}_\!{\kappa Frm} : \!{Bool} -> \!{\kappa Frm}$ applied to the colimit in $\!{Bool}$, i.e., to $A_\neg$; and of course $A$ is usually a proper quotient of $\ang{A_\neg \qua \!{Bool}}_\!{\kappa Frm} = \@I(A_\neg)$.
(The topological analog of $A |-> \@I(A_\neg)$ is the Banaschewski, i.e., universal zero-dimensional, compactification of a zero-dimensional space.)
\end{remark}

\begin{remark}
It is well-known, and easily shown, that $\!{CABool} \subseteq \!{Frm}$ is closed under finite colimits (namely, $\@P(X) \otimes \@P(Y) \cong \@P(X \times Y)$, and the coequalizer of $f^{-1}, g^{-1} : \@P(X) \rightrightarrows \@P(Y)$ for $f, g : Y \rightrightarrows X$ is $\@P(\eq(f, g))$).
The same holds for $\!{\kappa CABool}_\kappa \subseteq \!{\kappa Frm}$.

It follows that for any $A \in \!{\kappa Frm}$, the \defn{canonical diagram} of all complete $\kappa$-atomic Boolean $B$ equipped with a homomorphism $B -> A$ is filtered (see e.g., \cite[0.4, 1.4]{ARlpac}).
The proof of \cref{thm:upkzfrm-cabool-pres} easily adapts to show that for $A \in \!{\kappa UPKZFrm}$, the colimit of all such $B$ over the canonical diagram is $A$, i.e., $\!{\kappa CABool}_\kappa \subseteq \!{\kappa UPKZFrm}$ is a dense subcategory (see e.g., \cite[1.23]{ARlpac}).
\end{remark}

Now suppose $A$ is a $\kappa$-presented ultraparacompact zero-dimensional $\kappa$-frame.
Since $\!{\kappa CCDjBPost} \simeq \!{\kappa UPKZFrm} \subseteq \!{\kappa Frm}$ is coreflective, this means that $A = \ang{B \mid {<|}}_\!{\kappa Frm}$ for a $\kappa$-presented complement-complete disjunctive Boolean $\kappa$-posite $(B, {<|})$ (see \cref{sec:cat-lim}),
which means the same without requiring complement-completeness, since $\!{\kappa CCDjBPost} \subseteq \!{\kappa DjBPost}$ is reflective (see \cref{thm:cat-lim-refl}; for $\kappa = \omega$, use instead that finitely generated Boolean algebras are finite).
So $B$ is a $\kappa$-presented Boolean algebra, hence $\kappa$-ary; and $<|$ is a $\kappa$-generated Boolean $\kappa$-coverage.
Modifying the generators $<|_0$ to fulfill the assumptions of \cref{thm:upkzfrm-cabool-pres} clearly preserves the cardinality bound, yielding

\begin{corollary}
\label{thm:upkzfrm-kcabool-dircolim}
Every $\kappa$-presented ultraparacompact zero-dimensional $\kappa$-frame is a $\kappa$-ary directed colimit, in $\!{\kappa Frm}$, of complete $\kappa$-atomic Boolean algebras.
\qed
\end{corollary}

We may combine this with \cref{thm:frm-bool-pres} to get the following strengthening of the latter:

\begin{proposition}
\label{thm:upkzfrm-bool-pres}
Every $\kappa$-presented $\kappa$-Boolean algebra is freely generated by a $\kappa$-presented ultraparacompact zero-dimensional $\kappa$-frame
(hence is a $\kappa$-ary directed colimit, in $\!{\kappa Bool}$, of complete $\kappa$-atomic Boolean algebras).
\end{proposition}
\begin{proof}
As in the proof of \cref{thm:frm-bool-pres}, applied instead to the reflective inclusion of locally $\kappa$-presentable categories $\!{\kappa Bool} \subseteq \!{\kappa UPKZFrm} \simeq \!{\kappa CCDjBPost}$.
\end{proof}

As with \cref{thm:frm-bool-pres-cof}, we have the following refinement:

\begin{proposition}
\label{thm:upkzfrm-bool-pres-cof}
Let $B$ be a $\kappa$-presented $\kappa$-Boolean algebra which is freely generated by an ultraparacompact zero-dimensional $\kappa$-subframe $A \subseteq B$.
For any $\kappa$-ary subset $C \subseteq A$, there is a $\kappa$-presented ultraparacompact zero-dimensional $\kappa$-subframe $A' \subseteq A$ which contains $C$ and still freely generates $B$.
\end{proposition}
\begin{proof}
As in the proof of \cref{thm:frm-bool-pres-cof}, with the role of the reflective adjunction $\!{\kappa Frm} \rightleftarrows \!{\kappa Bool}$ in that proof replaced by $\!{\kappa CCDjBPost} \rightleftarrows \!{\kappa Bool}$, with $A \subseteq B$ replaced by the subposite $A_\neg \subseteq B$, and with $C \subseteq A$ replaced by some $C' \subseteq A_\neg$ which generates $C$ under $\kappa$-ary joins.
\end{proof}

The following result connects ultraparacompactness with the frames $\@N_\kappa^\alpha(A)$ from the preceding subsection.
The case $\kappa = \infty$ was proved by Plewe~\cite[Th.~17]{Psubloc}, using very different methods.

\begin{theorem}
\label{thm:frm-neg-upkz}
For any $\kappa$-frame $A$, $\@N_\kappa(A)$ is ultraparacompact (and zero-dimensional).
\end{theorem}
\begin{proof}
Let $B \subseteq \@N_\kappa(A)_\neg$ be the Boolean subalgebra generated by $A \subseteq \@N_\kappa(A)_\neg$, and let $<|$ be the restriction to $B$ of the canonical $\kappa$-coverage on $\@N_\kappa(A)$.
Since $B$ is a generating $\wedge$-sublattice of $\@N_\kappa(A)$, we have $\@N_\kappa(A) = \ang{B \mid {<|}}_\!{\kappa Frm}$.
Thus, it suffices to show that $<|$ is disjunctive.

We claim that $<|$ is generated by the canonical $\kappa$-coverage on $A$, together with the canonical $\omega$-coverage on $B$.
Let $<|'$ be the $\kappa$-coverage on $B$ generated by these.
Clearly ${<|'} \subseteq {<|}$, whence we have a quotient map $f : \ang{B \mid {<|'}}_\!{\kappa Frm} ->> \ang{B \mid {<|}}_\!{\kappa Frm} = \@N_\kappa(A)$.
From the definition of $<|'$, $\kappa$-ary joins in $A$ are preserved in $\ang{B \mid {<|'}}_\!{\kappa Frm}$, as are finite meets and joins in $B$ whence so are complements of $a \in A$; thus the inclusion $A `-> B `-> \ang{B \mid {<|'}}_\!{\kappa Frm}$ extends to $g : \@N(A) -> \ang{B \mid {<|'}}_\!{\kappa Frm}$.
We have $g \circ f = 1$, since both restrict to the inclusion $A `-> B `-> \ang{B \mid {<|'}}_\!{\kappa Frm}$, which is a $\kappa$-frame epimorphism since $\ang{B \mid {<|'}}_\!{\kappa Frm}$ is generated by the sublattice $B$ which is in turn generated by elements of $A$ and their complements.
It follows that $f$ is an isomorphism, whence ${<|'} = {<|}$ (since both present the same $\kappa$-frame).

So it suffices to show that every $\kappa$-ary covering relation $a <| C$ which holds in $A$, as well as every finite covering relation which holds in $B$, is implied by the pairwise disjoint $\kappa$-ary covering relations which hold in $B$.
Enumerate $C = \{c_i\}_{i < \alpha}$ for some ordinal $\alpha < \kappa$.
In both cases, we have the pairwise disjoint cover $a <| \{c_i \wedge \neg \bigvee_{j < i} c_j\}_{i < \alpha}$ in $B$, which implies $a <| C$ by transitivity.
\end{proof}

\begin{corollary}
\label{thm:frm-neginf-upkz}
For any $\kappa$-frame $A$ and $\alpha \ge 1$, $\@N_\kappa^\alpha(A)$ is ultraparacompact zero-dimensional.
\end{corollary}
\begin{proof}
Follows by induction from the preceding result and coreflectivity of $\!{\kappa UPKZFrm} \subseteq \!{\kappa Frm}$.
\end{proof}

\begin{corollary}
\label{thm:upkzfrm-neginf-pres}
Let $A$ be a $\kappa$-presented $\kappa$-frame.
For any $\alpha \ge 1$ and $\kappa$-ary subset $C \subseteq \@N_\kappa^\alpha(A)$, there is a $\kappa$-presented ultraparacompact zero-dimensional $\kappa$-subframe $A' \subseteq \@N_\kappa^\alpha(A)$ such that $A \cup C \subseteq A'$ and $\@N_\kappa^\infty(A) \cong \@N_\kappa^\infty(A')$ (canonically).
\end{corollary}
\begin{proof}
As in \cref{thm:frm-neginf-pres}.
\end{proof}

\subsection{Free complete Boolean algebras}
\label{sec:frm-cbool}

In this subsection, we give a ``brute-force'' syntactic construction of free complete Boolean algebras, using cut admissibility for a Gentzen sequent calculus for infinitary propositional logic.
This is analogous to Whitman's construction of free lattices (see \cite[III]{Hcbool}, \cite[I~4.6--7]{Jstone}, \cite{CSsigmapi}).
We are not aware of a prior application of the technique to infinitary Boolean algebras, although it is possibly folklore.
For background on sequent calculi in the finitary context, see \cite{TSprf}.

Let $X$ be a set (of generators), and let $\neg X$ be the set of symbols $\neg x$ for $x \in X$.
Let $\Term(X)$ be the class of \defn{prenex complete Boolean terms} over $X$, constructed inductively as follows:
\begin{itemize}
\item  $X \sqcup \neg X \subseteq \Term(X)$;
\item  for any subset $A \subseteq \Term(X)$, the symbols $\bigwedge A, \bigvee A$ are in $\Term(X)$.
\end{itemize}
Define the involution $\neg : \Term(X) -> \Term(X)$ in the obvious way, interchanging $x, \neg x$ and $\bigwedge, \bigvee$.

Let $|-$ be the unary predicate on \emph{finite} subsets $A \subseteq \Term(X)$ defined inductively as follows:
\begin{itemize}
\item  (reflexivity/identity) if $x, \neg x \in A$ for some $x \in X$, then $|- A$;
\item  ($\bigwedge$R) if $\bigwedge B \in A$, and $|- A \cup \{b\}$ for every $b \in B$, then $|- A$;
\item  ($\bigvee$R) if $\bigvee B \in A$, and $|- A \cup \{b\}$ for some $b \in B$, then $|- A$.
\end{itemize}
For two finite sets $A, B \subseteq \Term(X)$, we use the abbreviation%
\footnote{For simplicity, we use a one-sided sequent calculus, with formulas in prenex form and $\neg$ treated as an operation on terms.
It is also possible to directly formulate a two-sided calculus, treating $\neg$ as a primitive symbol, and prove all of the following results for two-sided sequents.
See \cite[Ch.~3]{TSprf}.}
\begin{align*}
A |- B \coloniff |- \neg A \cup B
\end{align*}
(where $\neg A := \{\neg a \mid a \in A\}$).
Thus $|- A \iff \emptyset |- A$.
The rules above imply their obvious generalizations when $|-$ has nonempty left-hand side, and also imply
\begin{itemize}
\item  (reflexivity/identity) if $x \in A \cap B$ or $x, \neg x \in A$ for some $x \in X$, then $A |- B$;
\item  ($\bigvee$L) if $\bigvee B \in A$, and $A \cup \{b\} |- C$ for every $b \in B$, then $A |- C$;
\item  ($\bigwedge$L) if $\bigwedge B \in A$, and $A \cup \{b\} |- C$ for some $b \in B$, then $A |- C$.
\end{itemize}
From now on, we will freely switch between one-sided and two-sided $|-$ without further comment; in particular, all of the following results stated for one-sided $|-$ have two-sided generalizations.

Every term $a \in \Term(X)$ can be represented as a well-founded tree, hence has an ordinal rank.
Similarly, whenever $|- A$ holds, its derivation using the rules above has an ordinal rank.

\begin{lemma}[generalized reflexivity/identity]
If $a, \neg a \in A$ for some $a \in \Term(X)$, then $|- A$.
\end{lemma}
\begin{proof}
By induction on $a$.
The case $a \in X$ is immediate.
Otherwise, by $\neg$-duality, it suffices to consider $a = \bigwedge B$.
For every $b \in B$, we have $|- A \cup \{b, \neg b\}$ by the induction hypothesis, whence $|- A \cup \{b\}$ by ($\bigvee$R) since $\bigvee \neg B = \neg a \in A$.
Thus $|- A$ by ($\bigwedge$R) since $\bigvee B = a \in A$.
\end{proof}

\begin{lemma}[monotonicity/weakening]
\label{thm:frm-cbool-weak}
If $|- A \subseteq B$, then $|- B$.
Moreover, the derivations have the same rank.
\end{lemma}
\begin{proof}
By straightforward induction on the (rank of the) derivation of $|- A$.
\end{proof}

\begin{lemma}[transitivity/cut]
\label{thm:frm-cbool-cut}
If $|- A \cup \{a\}$ and $|- A \cup \{\neg a\}$ for some $a \in \Term(X)$, then $|- A$.
\end{lemma}
\begin{proof}
By lexicographical induction on $a$ followed by the ranks of the derivations of $|- A \cup \{a\}$ and $|- A \cup \{\neg a\}$.
Up to $\neg$-duality, it suffices to consider the following cases:
\begin{itemize}
\item
Suppose $|- A \cup \{a\}$ by reflexivity because $\neg a \in A$.
Then $|- A = A \cup \{\neg a\}$ by hypothesis.
\item
Suppose $|- A \cup \{a\}$ by reflexivity because some $x, \neg x \in A$.
Then $|- A$ for the same reason.
\item
Suppose $|- A \cup \{a\}$ by ($\bigwedge$R) because some $\bigwedge B \in A$ with $|- A \cup \{a, b\}$ for each $b \in B$.
Then by weakening, $|- A \cup \{\neg a, b\}$ for each $b \in B$, derived with the same rank as $|- A \cup \{\neg a\}$.
Thus by the induction hypothesis applied to $|- A \cup \{a, b\}$ and $|- A \cup \{\neg a, b\}$ where the rank of the former is less than that of $|- A \cup \{a\}$, we have $|- A \cup \{b\}$ for each $b \in B$.
So by ($\bigwedge$R), $|- A$.
\item
Similarly if $|- A \cup \{a\}$ by ($\bigvee$R) because some $\bigvee B \in A$ with $|- A \cup \{a, b\}$ for some $b \in B$.
\item
Finally, suppose $a = \bigwedge B$,
\begin{enumerate}
\item[(i)] $|- A \cup \{a\}$ by ($\bigwedge$R) because
\item[(ii)] $|- A \cup \{a, b\}$ for each $b \in B$, and
\item[(iii)] $|- A \cup \{\neg a\}$ by ($\bigvee$R) because
\item[(iv)] $|- A \cup \{\neg a, \neg b\}$ for some $b \in B$.
\end{enumerate}
By the induction hypothesis applied to (iv) and a weakening of (i) (where the former has lesser rank than $|- A \cup \{a, b\}$ and the latter has the same rank as (i)), we have $|- A \cup \{\neg b\}$ for some $b \in B$.
Again by the induction hypothesis applied to (ii) and a weakening of (iii) (where again the former has lesser rank), we have $|- A \cup \{b\}$ for that $b$.
Thus by the induction hypothesis (using that $b$ is a subterm of $a$), $|- A$.
\qedhere
\end{itemize}
\end{proof}

Note that by weakening, all of the rules above have equivalent (``multiplicative'') forms where the hypotheses are allowed to have different sets which are combined in the conclusion, e.g.,
\begin{itemize}
\item  ($\bigwedge$R) if $|- A_b \cup \{b\}$ for every $b \in B$, and $\bigcup_b A_b$ is finite, then $|- \bigcup_{b \in B} A_b \cup \{\bigwedge B\}$;
\item  (cut) if $|- A \cup \{a\}$ and $|- B \cup \{\neg a\}$ for some $a \in \Term(X)$, then $|- A \cup B$;
\end{itemize}
etc.  We will freely use these forms from now on (without mentioning weakening).

Now let $\le$ be defined on $\Term(X)$ by
\begin{align*}
a \le b  \coloniff  \{a\} |- \{b\}.
\end{align*}
By (generalized) reflexivity and (multiplicative) cut, this is a preorder; let ${\equiv} := {\le} \cap {\ge}$ be its symmetric part.
By ($\bigvee$R), we have $\{a\} |- \{\bigvee A\}$, i.e., $a \le \bigvee A$, for $a \in A$, while ($\bigvee$L) gives that $\bigvee A$ is the join of $A$ with respect to $\le$; similarly, $\bigwedge$ is meet.
For every $a \in \Term(X)$ and $B \subseteq \Term(X)$, we have the derivation
\begin{align*}
\inference[$\bigwedge$L $\times 2$]{
    \inference[$\bigvee$L]{
        \inference[$\bigvee$R]{
            \inference[$\bigwedge$R]{
                \overline{\{a, b\} |- \{a\}} &
                \overline{\{a, b\} |- \{b\}}
            }{\{a, b\} |- \bigwedge \{a, b\}}
        }{\{a, b\} |- \{\bigvee \{\bigwedge \{a, b\} \mid b \in B\}\}}
        \quad\text{for each $b \in B$}
    }{\{a, \bigvee B\} |- \{\bigvee \{\bigwedge \{a, b\} \mid b \in B\}\}}
}{\{\bigwedge \{a, \bigvee B\}\} |- \{\bigvee \{\bigwedge \{a, b\} \mid b \in B\}\}}
\end{align*}
which shows $a \wedge \bigvee B = \bigvee_{b \in B} (a \wedge b)$ in $\Term(X)/{\equiv}$, whence $\Term(X)/{\equiv}$ is a (large) frame.
And from $|- \{a, \neg a\}$, applying ($\bigvee$R) twice and ($\bigwedge$L) once yields $\{\bigwedge \emptyset\} |- \{\bigvee \{a, \neg a\}\}$, i.e., $\top \le a \vee \neg a$; thus since $\neg$ is an order-reversing involution, $\neg$ descends to complement in $\Term(X)/{\equiv}$.
Thus

\begin{lemma}
$\Term(X)/{\equiv}$ is a (large) complete Boolean algebra.
\qed
\end{lemma}

\begin{lemma}[soundness]
\label{thm:frm-cbool-sound}
Whenever $|- A$, then $\bigvee A = \top$ holds in $\ang{X}_\!{CBOOL}$.
\end{lemma}
\begin{proof}
By straightforward induction on the derivation of $|- A$.
\end{proof}

\begin{theorem}
\label{thm:frm-cbool-ltalg}
$\ang{X}_\!{CBOOL} = \Term(X)/{\equiv}$.
\end{theorem}
\begin{proof}
Clearly $\Term(X)/{\equiv}$ is generated as a complete Boolean algebra by (the $\equiv$-classes of) $x \in X$, so it suffices to check that if two terms $a, b \in \Term(X)$ obey $a \le b$ in $\Term(X)/{\equiv}$, i.e., $\{a\} |- \{b\}$, then they also obey $a \le b$ in $\ang{X}_\!{CBOOL}$; this follows from the preceding lemma.
\end{proof}

The restriction of $|-$ to \emph{finite} sets of terms was needed for the inductive proof of \cref{thm:frm-cbool-cut} to go through.
It is often more convenient to use an infinitary version of $|-$, which we now give.

We inductively define the relation $|-$ on \emph{arbitrary} subsets $A \subseteq \Term(X)$ via the following rules:
\begin{itemize}
\item  (reflexivity/identity) if $x, \neg x \in A$ for some $x \in X$, then $|- A$;
\item  ($\bigwedge$R) if $\bigwedge B \in A$, and $|- A \cup \{b\}$ for every $b \in B$, then $|- A$;
\item  ($\bigvee$R) if $\bigvee B \in A$, and $|- A \cup B$, then $|- A$.
\end{itemize}
As before, we also use the obvious two-sided generalization without further comment.
It will turn out (by \cref{thm:frm-cbool-seq}) that this relation restricted to finite sets is the same as defined before.
Whenever we need to disambiguate, we will refer to these rules as the \defn{infinitary sequent calculus} (and the previous ones as the \defn{finitary calculus}; note that this does not refer to the arity of $\bigwedge, \bigvee$).

\begin{lemma}
Weakening (\cref{thm:frm-cbool-weak}) and soundness (\cref{thm:frm-cbool-sound}) also hold for the infinitary sequent calculus.
\end{lemma}
\begin{proof}
By induction on derivations in the infinitary calculus.
\end{proof}

\begin{lemma}
\label{thm:frm-cbool-seq-fin-inf}
If $|- A$ in the finitary calculus, then $|- A$ in the infinitary calculus.
\end{lemma}
\begin{proof}
By induction.
In the only non-obvious case of ($\bigvee$R), use weakening in the infinitary calculus to go from $|- A \cup \{b\}$ to $|- A \cup B$.
\end{proof}

\begin{lemma}
In the infinitary calculus, if $|- A \cup \{\bigvee B\}$, then $|- A \cup B$.
\end{lemma}
\begin{proof}
By induction.
\begin{itemize}
\item  If $|- A \cup \{\bigvee B\}$ by reflexivity, then there is some $x, \neg x \in A$, so $|- A \cup B$ by reflexivity.
\item  If $|- A \cup \{\bigvee B\}$ by ($\bigwedge$R), then there is some $\bigwedge C \in A$ such that $|- A \cup \{\bigvee B\} \cup \{c\}$ for every $c \in C$; by the induction hypothesis, $|- A \cup B \cup \{c\}$ for every $c \in C$, whence by ($\bigwedge$R), $|- A \cup B$.
\item  Suppose $|- A \cup \{\bigvee B\}$ by ($\bigvee$R) because there is $\bigvee C \in A \cup \{\bigvee B\}$ such that $|- A \cup \{\bigvee B\} \cup C$.
By the induction hypothesis, we have $|- A \cup B \cup C$.
If $B = C$, this gives $|- A \cup B$; otherwise, we must have $\bigvee C \in A$, whence by ($\bigvee$R), $|- A \cup B$.
\qedhere
\end{itemize}
\end{proof}

\begin{theorem}
\label{thm:frm-cbool-seq}
For any subset $A \subseteq \Term(X)$, the following are equivalent:
\begin{enumerate}
\item[(i)]  $\bigvee A = \top$ in $\ang{X}_\!{CBOOL}$;
\item[(ii)]  $|- \{\bigvee A\}$ in the infinitary sequent calculus;
\item[(iii)]  $|- A$ in the infinitary sequent calculus;
\item[(iv)]  $|- \{\bigvee A\}$ in the finitary sequent calculus;
\item[(v)]  $|- A$ in the finitary sequent calculus, when $A$ is finite.
\end{enumerate}
\end{theorem}
\begin{proof}
(i) is implied by all the other conditions by soundness.

(i)$\iff \{\bigwedge \emptyset\} |- \{\bigvee A\}$ (in the finitary calculus) by \cref{thm:frm-cbool-ltalg}; the latter implies (iv) by cut against $|- \{\bigwedge \emptyset\}$ which is immediate from ($\bigwedge$R).

(iv)$\implies$(ii) by \cref{thm:frm-cbool-seq-fin-inf}.

(ii)$\iff$(iii) by the preceding lemma and ($\bigvee$R).

Similarly, when $A$ is finite, (v)$\implies$(iv) by ($\bigvee$R), and (iv)$\implies$(v) by cut against $\{\bigvee A\} |- A$ which follows from ($\bigvee$L).
\end{proof}

\begin{corollary}[infinitary transitivity/cut]
\label{thm:frm-cbool-inf-cut}
In the infinitary calculus, if $|- A \cup B$, and $|- A \cup \{\neg b\}$ for every $b \in B$, then $|- A$.
\end{corollary}
\begin{proof}
In $\ang{X}_\!{CBOOL}$, we have $\bigvee A \vee \bigvee B = \top$, and $\bigvee A \vee \neg b = \top$, i.e., $b \le \bigvee A$, for every $b \in B$; this clearly implies $\bigvee A = \top$.
\end{proof}

\subsection{Distributive polyposets}
\label{sec:dpoly}

In this subsection, we introduce a method for presenting $(\lambda, \kappa)$-frames for arbitrary $\lambda, \kappa$.
This can be seen as an abstraction of the properties of the infinitary sequent calculus from the preceding subsection (see \cref{ex:frm-dpoly-seq}), or alternatively, as a ``two-sided'' generalization of posite presentations from \cref{sec:frm-post} (see \cref{ex:frm-dpoly-vpost}).

A \defn{$(\lambda, \kappa)$-distributive prepolyposet}%
\footnote{``Poly'' derives from the theory of polycategories, and refers to the fact that both sides of $<|$ may consist of multiple elements; see e.g., \cite{Schu}.
``Distributive'' refers to the fact that our axioms force existing finite meets/joins to distribute over existing $\kappa$-joins/$\lambda$-meets; this follows from \cref{thm:frm-dpoly-meetjoin} and \cref{thm:frm-dpoly-inj}.}
consists of an underlying set $A$ equipped with a binary relation ${<|} \subseteq \@P_\lambda(A) \times \@P_\kappa(A)$ between $\lambda$-ary subsets and $\kappa$-ary subsets of $A$, obeying:
\begin{itemize}
\item  (monotonicity) $B \supseteq C <| D \subseteq E \implies B <| E$;
\item  (reflexivity) $\{a\} <| \{a\}$ for all $a \in A$;
\item  (transitivity) if $B \cup C <| D \cup E$, $B <| \{c\} \cup E$ for every $c \in C$, and $B \cup \{d\} <| E$ for every $d \in D$, then $B <| E$.
\end{itemize}
Using monotonicity, transitivity is easily seen to be equivalent to the combination of its special cases with $D = \emptyset$ or $C = \emptyset$, which we call \defn{left-transitivity} and \defn{right-transitivity} respectively, or to the following ``multiplicative'' form; we will use all of these freely:
\begin{itemize}
\item  if $B \cup C <| D \cup E$, $B_c <| \{c\} \cup E_c$ for every $c \in C$, and $B_d \cup \{d\} <| E_d$ for every $d \in D$, then $B \cup \bigcup_{d \in D} B_d <| \bigcup_{c \in C} E_c \cup E$, assuming $\bigcup_{d \in D} B_d$ is $\lambda$-ary and $\bigcup_{c \in C} E_c$ is $\kappa$-ary.
\end{itemize}
We refer to $<|$ as a \defn{$(\lambda, \kappa)$-distributive prepolyorder} on $A$.
We define the preorder $\le$ on $A$ by
\begin{align*}
a \le b  \coloniff  \{a\} <| \{b\},
\end{align*}
and let ${\equiv} := {\le} \cap {\ge}$.
By transitivity, $<|$ is invariant under $\equiv$ on both sides.
If $\le$ is a partial order, i.e, $\equiv$ is equality, then we call $<|$ \defn{separated} or a \defn{$(\lambda, \kappa)$-distributive polyorder}, and $A$ a \defn{$(\lambda, \kappa)$-distributive polyposet}.
We denote the categories of $(\lambda, \kappa)$-distributive (pre)polyposets and $<|$-preserving maps by $\!{\lambda\kappa DP(r)oly}$; for $\lambda, \kappa < \infty$, both are locally $\max(\lambda, \kappa)$-presentable, and $\!{\lambda\kappa DPoly} \subseteq \!{\lambda\kappa DProly}$ is reflective, with reflection given by quotienting by $\equiv$.

\begin{example}
\label{ex:frm-dpoly-seq}
Let $X$ be any set, and let $A \subseteq \Term(X)$ be a set of prenex complete Boolean terms over $X$, as in the preceding subsection.
Declare $B <| C$ iff $B |- C$ in the infinitary sequent calculus from the preceding subsection.
Then $(A, {<|})$ is a $(\lambda, \kappa)$-distributive prepolyposet, where the transitivity axiom is by (the two-sided version of) \cref{thm:frm-cbool-inf-cut}.
\end{example}

\begin{example}
\label{ex:frm-dpoly-bifrm}
Let $A$ be a $(\lambda, \kappa)$-frame, and define the \defn{canonical polyorder}
\begin{align*}
B <| C \coloniff \bigwedge B \le \bigvee C.
\end{align*}
Then $(A, <|)$ is a $(\lambda, \kappa)$-distributive polyposet.
To check the left-transitivity axiom: if
\begin{enumerate}
\item[(i)] $\bigwedge B \wedge \bigwedge C \le \bigvee D$ and
\item[(ii)] $\bigwedge B \le c \vee \bigvee D$ for every $c \in C$, i.e., $\bigwedge B \le \bigwedge C \vee \bigvee D$ (by distributivity),
\end{enumerate}
then
$
\bigwedge B
= \bigwedge B \wedge (\bigwedge C \vee \bigvee D)
\le (\bigwedge B \wedge \bigwedge C) \vee \bigvee D
= \bigwedge D
$;
right-transitivity is dual.
This defines a forgetful functor $\!{\lambda\kappa Frm} -> \!{\lambda\kappa DPoly}$.
\end{example}

\begin{proposition}
\label{thm:frm-dpoly-meetjoin}
Let $A$ be a $(\lambda, \kappa)$-distributive prepolyposet.
\begin{enumerate}
\item[(a)]
For $B \in \@P_\kappa(A)$, a $\le$-join $\bigvee B$, if it exists, is characterized by the relations $b \le \bigvee B$ for each $b \in B$ together with $\{\bigvee B\} <| B$.
\item[(b)]
Dually, a meet of $B \in \@P_\lambda(A)$ is characterized by $\bigwedge B \le b$ for $b \in B$ together with $B <| \{\bigwedge B\}$.
\item[(c)]
Thus, any $(\lambda, \kappa)$-distributive prepolyposet homomorphism preserves existing $\lambda$-ary meets and $\kappa$-ary joins in the domain.
\item[(d)]
Thus, the forgetful functor $\!{\lambda\kappa Frm} -> \!{\lambda\kappa DPoly}$ is full (and faithful).
\end{enumerate}
\end{proposition}
\begin{proof}
(a)  If $b \le c$ for every $b \in B$, then by transitivity with $\{\bigvee B\} <| B$, $\bigvee B \le c$.

The rest immediately follows.
\end{proof}

Given a $(\lambda, \kappa)$-distributive prepolyposet $(A, {<|})$, we let
\begin{align*}
\ang{A \mid {<|}}_\!{\lambda\kappa Frm} := \ang{(A, {<|}) \qua \!{\lambda\kappa DProly}}_\!{\lambda\kappa Frm} = \ang{A \qua \!{Set} \mid \bigwedge B \le \bigvee C \text{ for } B <| C}_\!{\lambda\kappa Frm}
\end{align*}
be the \defn{presented $(\lambda, \kappa)$-frame}.
This is the left adjoint to the forgetful functor $\!{\lambda\kappa Frm} -> \!{\lambda\kappa DProly}$.
We use similar notation for any of our main categories ``above'' $\!{\lambda\kappa Frm}$ in the diagram \eqref{diag:frm-cat}: for $\lambda' \ge \lambda$, $\kappa' \ge \kappa$, and $\mu \ge \lambda, \kappa$, we let $\ang{A \mid {<|}}_\!{\lambda'\kappa' Frm}$ and $\ang{A \mid {<|}}_\!{\mu Bool}$ denote the corresponding presented structure in the respective category.
We now have the main result of this subsection, which says that our three axioms on $<|$ imply that it is already ``saturated'' under all relations that hold in the presented algebras, analogously to posites (\cref{sec:frm-post}):

\begin{theorem}
\label{thm:frm-dpoly-inj}
For any $(\lambda, \kappa)$-distributive prepolyposet $(A, {<|})$, the unit $\eta : A -> \ang{A \mid {<|}}_\!{CBOOL}$ reflects $<|$, i.e., for any $B \in \@P_\lambda(A)$ and $C \in \@P_\kappa(A)$, we have $B <| C \iff \bigwedge \eta(B) \le \bigvee \eta(C)$.
Thus if $<|$ is separated, then $\eta$ is an embedding of $(\lambda, \kappa)$-distributive polyposets.

Thus, the same holds for all of the intermediate units $A -> \ang{A \mid {<|}}_\!{\lambda'\kappa' Frm}$ and $A -> \ang{A \mid {<|}}_\!{\mu Bool}$.
\end{theorem}
\begin{proof}
By the general construction of presented algebras, $\ang{A \mid {<|}}_\!{CBOOL}$ is the free complete Boolean algebra $\ang{A}_\!{CBOOL} = \ang{A \qua \!{Set}}_\!{CBOOL}$ quotiented by the relations $\bigwedge B \le \bigvee C$, i.e., ${\bigwedge B -> \bigvee C} = \top$, for $B <| C$.
By the description of Boolean algebra quotients at the end of \cref{sec:frm-quot}, an element $t \in \ang{A}_\!{CBOOL}$ becomes $\top$ in $\ang{A \mid {<|}}_\!{CBOOL}$ iff it belongs to the principal $\infty$-filter $\up \bigwedge_{B <| C} (\bigwedge B -> \bigvee C)$.
Thus for $D \in \@P_\kappa(A)$ and $E \in \@P_\lambda(A)$, we have
\begin{align*}
\bigwedge \eta(D) \le \bigvee \eta(E)
&\iff \bigwedge_{B <| C} (\bigwedge B -> \bigvee C) \le \bigwedge D -> \bigvee E \in \ang{A}_\!{CBOOL}, \\
\intertext{which, using the infinitary sequent calculus from the preceding subsection, is by \cref{thm:frm-cbool-seq}}
&\iff \{\bigvee (\neg B \cup C) \mid B <| C\} |- \neg D \cup E.
\end{align*}
We induct on derivations to show that for every such $D, E$, we have $D <| E$.
\begin{itemize}
\item
If the last rule used is reflexivity, then there must be some $a \in A$ with $a \in D \cap E$, whence $D <| E$ by reflexivity of $<|$.
\item
If the last rule used is ($\bigwedge$R), i.e., ($\bigvee$L), then there is some $B_0 <| C_0$ such that
\begin{itemize}
\item  $\{\bigvee (\neg B \cup C) \mid B <| C\} |- \{b\} \cup \neg (D) \cup E$ for every $b \in B_0$ and
\item  $\{\bigvee (\neg B \cup C) \mid B <| C\} |- \{\neg c\} \cup \neg (D) \cup E$ for every $c \in C_0$.
\end{itemize}
By the induction hypothesis, this means $D <| \{b\} \cup E$ for every $b \in B_0$ and $D \cup \{c\} <| E$ for every $c \in C_0$.
By transitivity with $B_0 <| C_0$, we get $D <| E$.
\item
The last rule cannot be ($\bigvee$R), i.e., ($\bigwedge$L), due to the form of the sequent.
\qedhere
\end{itemize}
\end{proof}

Applying \cref{thm:frm-dpoly-inj} to the canonical polyorder on a $(\lambda, \kappa)$-frame yields

\begin{corollary}
\label{thm:frm-dpoly-bifrm-inj}
For $\lambda' \ge \lambda$, $\kappa' \ge \kappa$, and $\mu \ge \lambda, \kappa$, the units of the free/forgetful adjunctions $\!{\lambda\kappa Frm} \rightleftarrows \!{\lambda'\kappa' Frm}$ and $\!{\lambda\kappa Frm} \rightleftarrows \!{\mu Bool}$ are injective.
Equivalently, the left adjoints are faithful.
\qed
\end{corollary}

Note that this gives alternative proofs of \cref{thm:frm-neginf-inj,thm:frm-neginf-bool-inj}, which are more ``syntactic'', being ultimately based on the infinitary sequent calculus.
A ``semantic'' proof of \cref{thm:frm-dpoly-bifrm-inj} in the case $\kappa = \lambda = \infty$ was essentially given by Picado~\cite[2.2]{Pbifrm}.

We now know that all of our main categories defined in \cref{sec:frm-cat} admit faithful free functors to all of the ones ``above'' them, i.e., each of these types of structures embeds into the free complete Boolean algebra it generates.
As promised in \cref{cvt:frm-cbool-incl}, for a $(\lambda, \kappa)$-frame $A$, we henceforth also regard the unit $\eta : A `-> \ang{A}_\!{CBOOL}$ (hence also all intermediate units) as an inclusion.

Next, we note that distributive polyposets are indeed a ``two-sided'' generalization of posites:

\begin{example}
\label{ex:frm-dpoly-vpost}
Let $(A, <|)$ be a $\kappa$-$\bigvee$-posite (\cref{sec:frm-post}).
We may extend $<|$ to a $(\lambda, \kappa)$-distributive prepolyorder, for any $\lambda$, by defining $B <| C$ to mean that there is $b \in B$ with $b <| C$.
\begin{proof}[Proof of transitivity]
Left-transitivity: from $B \cup C <| D$, we either have some $b \in B$ with $b <| D$ in which case $B <| D$, or some $c \in C$ with $c <| D$, in which case from $B <| \{c\} \cup D$, we have some $b \in B$ with $b <| \{c\} \cup D$, whence by right-transitivity of the original $<|$, we have $b <| D$ and so $B <| D$.

Right-transitivity: from $B <| C \cup D$, we have some $b \in B$ with $b <| C \cup D$, while for every $c \in C$, from $B \cup \{c\} <| D$, which means either $c <| D$ or there is $b_c \in B$ with $b_c <| D$; if the latter holds for some $c$ then we immediately get $B <| D$, otherwise $B <| D$ by right-transitivity of the original $<|$.
\end{proof}

The extended $<|$ is clearly the $(\lambda, \kappa)$-distributive prepolyorder generated by declaring $\{a\} <| C$ whenever $a <| C$.
Conversely, any $(\lambda, \kappa)$-distributive prepolyorder $<|$ on $A$ defines a preorder $\le$ on $A$, as well as a $\kappa$-$\bigvee$-coverage by restricting the left-hand side of $<|$ to be singletons, which is separated (as defined in \cref{sec:frm-post}, despite $\le$ being possibly not antisymmetric).
Letting $\!{\kappa{\bigvee}Prost}$ be the category of \defn{$\kappa$-$\bigvee$-preposites}, meaning $\kappa$-$\bigvee$-posites where the underlying poset may instead be a preordered set, we thus have a free/forgetful adjunction $\!{\kappa{\bigvee}Prost} \rightleftarrows \!{\lambda\kappa DProly}$, such that the unit is $<|$-reflecting (but the preorder on the image is the restriction of $<|$ to singletons).
This restricts to a reflective adjunction $\!{\kappa{\bigvee}SPost} \rightleftarrows \!{\lambda\kappa DPoly}$ between the full subcategories of separated $\kappa$-$\bigvee$-posites and $(\lambda, \kappa)$-distributive polyposets, respectively.

(It would perhaps have been more natural for us to only consider separated preposites and separated posites in the first place, since given a non-separated preposite $(A, {<|})$, we may always replace the preorder $\le_A$ with $<|$ on singletons, without affecting the presented $\kappa$-$\bigvee$-lattice/$\kappa$-frame.
This would mirror our definition of prepolyorders via a single relation $<|$, from which $\le$ is then derived.
However, in the literature, separation is usually regarded as an extra condition.)

In fact, we can view a $\kappa$-$\bigvee$-(pre)posite as a ``$(\{1\}, \kappa)$-distributive (pre)polyposet'', just as a $\kappa$-$\bigvee$-lattice is a ``$(\{1\}, \kappa)$-frame'' (see \cref{sec:frm-cat}).
Regarding general $(\lambda, \kappa)$-distributive (pre)polyposets as presentations for $(\lambda, \kappa)$-frames then mirrors the role of (separated) $\kappa$-$\bigvee$-(pre)posites as presentations for $\kappa$-$\bigvee$-lattices, so that we have a commutative square of adjunctions, analogous to \eqref{diag:frm-vpost}:
\begin{equation*}
\begin{tikzcd}
\!{\kappa{\bigvee}Lat}
    \dar[rightarrowtail, shift left=2, right adjoint']
    \rar[shift left=2] &
\!{\lambda\kappa Frm}
    \lar[hook, shift left=2, right adjoint]
    \dar[rightarrowtail, shift left=2, right adjoint'] \\
\!{\kappa {\bigvee}(S)P(r)ost}
    \uar[shift left=2, "\ang{-}_\!{\kappa{\bigvee}Lat}"]
    \rar[shift left=2] &
\!{\lambda\kappa DP(r)oly}
    \lar[shift left=2, right adjoint]
    \uar[shift left=2, "\ang{-}_\!{\lambda\kappa Frm}"]
\end{tikzcd}
\end{equation*}
The bottom left adjoint is full and faithful if we take $\!{\kappa {\bigvee}SP(r)ost}$, and an isomorphism if $\lambda = \{1\}$.
\end{example}

\begin{example}
\label{ex:frm-dpoly-post}
Now let $(A, {<|})$ be a $\kappa$-posite.
If we regard $(A, {<|})$ as a $(\lambda, \kappa)$-distributive prepolyposet as in \cref{ex:frm-dpoly-vpost}, we lose the information that finite meets in $A$ are to be preserved.
Instead, we regard $(A, {<|})$ as a $(\omega, \kappa)$-distributive prepolyposet via
\begin{align*}
B <| C  \coloniff  \bigwedge B <| C \quad\text{for $B \in \@P_\omega(A)$}.
\end{align*}
\begin{proof}[Proof of transitivity]
Left-transitivity: from $B <| \{c\} \cup D$ for every $c \in C$, we get
$\bigwedge B <| \{c\} \cup D$ for every $c \in C$.
Letting $c = \{c_1, \dotsc, c_n\}$, using $\wedge$-stability and right-transitivity, we have
$\bigwedge B <| \{\bigwedge B \wedge c_1\} \cup D$, and
$\bigwedge B \wedge c_1 <| \{\bigwedge B \wedge c_1 \wedge c_2\} \cup E$, whence
$\bigwedge B <| \{\bigwedge B \wedge c_1 \wedge c_2\} \cup D$;
continuing inductively, we get
$\bigwedge B <| \{\bigwedge B \wedge c_1 \wedge \dotsb \wedge c_n\} \cup D$.
Now from $B \cup C <| D$, we get
$\bigwedge B \wedge \bigwedge C <| D$,
whence by right-transitivity,
$\bigwedge B <| D$.
Right-transitivity is similar to \cref{ex:frm-dpoly-vpost}.
\end{proof}

This defines a functor $\!{\kappa Prost} -> \!{\omega\kappa DProly}$, which, when restricted to $\!{\kappa SPost}$, is a full and faithful \emph{right} adjoint.
Its left adjoint takes an $(\omega, \kappa)$-distributive prepolyposet $(A, {<|})$, takes the separated quotient, and then freely adjoins finite meets while preserving separatedness (hence does nothing for finite meets that already exist).
Moreover, this adjunction is compatible (i.e., forms the obvious commutative triangle) with the previously defined forgetful functors $\!{\kappa Frm} -> \!{\kappa SPost}$ (from \cref{sec:frm-post}) and $\!{\kappa Frm} -> \!{\omega\kappa DPoly}$ (\cref{ex:frm-dpoly-bifrm}).
\end{example}

\begin{remark}
\label{rmk:frm-dpoly-bifrm-compl}
For $\lambda > \omega$, combining the preceding two examples, one might try to define the free $(\lambda, \kappa)$-distributive prepolyposet generated by a $\kappa$-posite $(A, {<|})$ by taking $B <| C$ iff there is a finite $B' \subseteq B$ with $\bigwedge B' <| C$.
However, the resulting $<|$ \emph{need not} be transitive.
When $A$ is a $\kappa$-frame, this corresponds to the fact that given infinite $B, C \subseteq A$, $\bigwedge B \le \bigvee C$ in $\ang{A}_\!{CBOOL}$ does not necessarily mean that there is a finite $B' \subseteq B$ with $\bigwedge B' \le \bigvee C$.
(For example, take $\kappa = \lambda = \omega_1$ and $A = \@O(\#R)$; by \cref{rmk:loc-ctbpres-bor}, countable meets are intersections, thus $\bigwedge_{n \in \#N} (n, \infty) = \emptyset$.)

For similar reasons, for an arbitrary $(\lambda, \kappa)$-distributive polyposet $(A, {<|})$, attempting to extend $<|$ to arbitrary subsets by $B <| C$ iff there are $\lambda$-ary $B' \subseteq B$ and $\kappa$-ary $C' \subseteq C$ with $B' <| C'$ need not yield a transitive relation.
\end{remark}

\subsection{Interpolation}
\label{sec:frm-interp}

In this subsection, we apply distributive polyposets to study colimits in the categories $\!{\kappa\kappa Frm}, \!{\kappa Bool}$.
Our main goal is to prove \cref{thm:bifrm-minterp}, a structural result about pushouts and their ordered analogs, which can be seen as an algebraic version of the Craig--Lyndon interpolation theorem for infinitary propositional logic.

First, we consider coproducts:

\begin{proposition}
\label{thm:frm-dpoly-coprod}
Let $A_i$ for $i \in I$ be $(\lambda, \kappa)$-distributive prepolyposets.
Then the coproduct of the $A_i$ in $\!{\lambda\kappa DProly}$ is given by the disjoint union $\bigsqcup_i A_i$ with $<|$ given by the closure under monotonicity of the union of the $<|_{A_i}$ on each $A_i$, i.e.,
\begin{align*}
B <| C  \coloniff  \exists i\, (A_i \cap B <|_{A_i} A_i \cap C) \quad\text{for $B \in \@P_\lambda(\bigsqcup_i A_i)$ and $C \in \@P_\kappa(\bigsqcup_i A_i)$}.
\end{align*}
The coproduct in $\!{\kappa BProly}$ is given similarly, with $\neg$ defined as in each $A_i$.
\end{proposition}
\begin{proof}
By the general construction of colimits (see \cref{sec:cat-lim}), $<|$ is the closure of the union of the $<|_{A_i}$ under monotonicity, reflexivity, and transitivity; since the union is clearly reflexive, it suffices to check that the closure under monotonicity is already transitive.
So suppose
\begin{itemize}
\item[(i)]  $B \cup C <| D$, i.e., there is some $i$ such that $A_i \cap (B \cup C) <|_{A_i} A_i \cap D$;
\item[(ii)]  for each $c \in C$, $B <| \{c\} \cup D$, i.e., there is some $i_c$ such that $A_{i_c} \cap B <|_{A_{i_c}} A_{i_c} \cap (\{c\} \cup D)$.
\end{itemize}
If there is some $c \in A_i \cap C$ such that $i_c \ne i$, then $c \not\in A_{i_c}$, so (ii) gives $A_{i_c} \cap B <|_{A_{i_c}} A_{i_c} \cap D$, whence $B <| D$.
Otherwise, left-transitivity of $<|_{A_i}$ gives $A_i \cap B <|_{A_i} A_i \cap D$, whence $B <| D$.
This proves left-transitivity; right-transitivity is dual.
\end{proof}

We may apply this to deduce that only ``trivial'' relations hold in coproducts of $(\lambda, \kappa)$-frames:

\begin{corollary}
\label{thm:bifrm-coprod}
Let $A_i$ for $i \in I$ be $(\lambda, \kappa)$-frames, let $b_i \in A_i$ for each $i \in I$ with $b_i = \top$ for all but $<\lambda$-many $i$, and let $c_i \in A_i$ for each $i \in I$ with $c_i = \bot$ for all but $<\kappa$-many $i$.
If $\bigwedge_i b_i \le \bigvee_i c_i$ in the coproduct $(\lambda, \kappa)$-frame $\coprod_i A_i$, then some $b_i \le c_i$.
\end{corollary}
\begin{proof}
The coproduct $\coprod_i A_i$ may be computed as the reflection $\ang{\bigsqcup_i A_i \mid {<|}}_\!{\lambda\kappa Frm}$ into $\!{\lambda\kappa Frm}$ of the coproduct in $\!{\lambda\kappa DProly}$ given by \cref{thm:frm-dpoly-coprod}.
By \cref{thm:frm-dpoly-inj}, we have $\bigwedge_i b_i \le \bigvee_i c_i$ in $\coprod_i A_i$ iff $B := \{b_i \mid b_i \ne \top\} <| \{c_i \mid c_i \ne \bot\} =: C$ in $\bigsqcup_i A_i$.
By \cref{thm:frm-dpoly-coprod}, the latter means there is some $i$ such that $A_i \cap B <| A_i \cap C$ in $A_i$, which means $b_i \le c_i$.
\end{proof}

When $\kappa = \lambda$, we have the following generalization for coproducts with certain relations imposed, which says that the only relations which hold are those implied in one step by transitivity:

\begin{theorem}[interpolation]
\label{thm:bifrm-minterp}
Let $A, B_i, C_j$ be $<\kappa$-many $(\kappa, \kappa)$-frames, with homomorphisms $f_i : A -> B_i$ and $g_j : A -> C_j$, and let $D$ be their \defn{bilax pushout}, i.e., the universal $(\kappa, \kappa)$-frame equipped with morphisms as in the diagram
\begin{equation*}
\begin{tikzcd}
B_i \rar["\iota_i"] \drar[phantom, "\le"{pos=.2,sloped}, "\le"{pos=.8,sloped}] &
D \\
A \uar["f_i"] \rar["g_j"'] \urar["\iota"{pos=.2,inner sep=1pt}] &
C_j \uar["\iota_j"']
\end{tikzcd}
\end{equation*}
obeying $\iota_i \circ f_i \le \iota$ for each $i$ and $\iota \le \iota_j \circ g_j$ for each $j$.
Let $a^L, a^R \in A$, $b^L_i, b^R_i \in B_i$ for each $i$, and $c^L_j, c^R_j \in C_j$ for each $j$.
If
\begin{align*}
\tag{$*$}
a^L \wedge \bigwedge_i b^L_i \wedge \bigwedge_j c^L_j \le a^R \vee \bigvee_i b^R_i \vee \bigvee_j c^R_j \quad \in D
\end{align*}
(where we omitted mentioning $\iota, \iota_i, \iota_j$),
then there are $a^L_i, a^R_j \in A$ such that
\begin{align*}
\tag{$\dagger$}
\begin{aligned}
\bigwedge_i a^L_i \wedge a^L &\le a^R \vee \bigvee_j a^R_j &&\in A, \\
b^L_i &\le f_i(a^L_i) \vee b^R_i &&\in B_i \;\forall i, \\
c^L_j \wedge g_j(a^R_j) &\le c^R_j &&\in C_j \;\forall j.
\end{aligned}
\end{align*}
\end{theorem}

Note that ($\dagger$) does indeed imply ($*$) by transitivity; here our omission of the $\iota$'s in ($*$) somewhat obscures the significance of the directions of the $\le$ in the bilax pushout, which ensure that e.g., the $a^L_i$ in the second line of ($\dagger$) is $\le$ the $a^L_i$ in the first line.
The bilax pushout is an instance of a \defn{weighted colimit} in a locally ordered category; see \cref{sec:cat-ord}.

\begin{proof}
The bilax pushout $D$ is presented by $A, B_i, C_j$ qua $\!{\kappa\kappa Frm}$ subject to the additional relations $f_i(a) \le a$ and $a \le g_j(a)$ for each $a \in A$.
For $L, R \in \@P_\kappa(A \sqcup \bigsqcup_i B_i \sqcup \bigsqcup_j C_j)$, define
\begin{align*}
L <| R \coloniff \exists a^L_i, a^R_j \in A \text{ s.t.\ } \left\{
\begin{aligned}
\bigwedge_i a^L_i \wedge \bigwedge (A \cap L) &\le \bigvee (A \cap R) \vee \bigvee_j a^R_j &&\in A, \\
\bigwedge (B_i \cap L) &\le f_i(a^L_i) \vee \bigvee (B_i \cap R) &&\in B_i \;\forall i, \\
\bigwedge (C_i \cap L) \wedge g_j(a^R_j) &\le \bigvee (C_j \cap R) &&\in C_j \;\forall j.
\end{aligned}
\right.
\end{align*}
We claim that $<|$ is a $(\kappa, \kappa)$-distributive prepolyorder on $A \sqcup \bigsqcup_i B_i \sqcup \bigsqcup_j C_j$ containing the canonical $<|$ on $A, B_i, C_j$ and also obeying $\{f_i(a)\} <| \{a\}$ and $\{a\} <| \{g_j(a)\}$ for each $a \in A$.
This will imply that $<|$ contains the relations needed to present $D$, thus whenever ($*$) holds in $D$, then it also holds in $\ang{A \sqcup \bigsqcup_i B_i \sqcup \bigsqcup_j C_j \mid {<|}}_\!{\kappa\kappa Frm}$, whence by \cref{thm:frm-dpoly-inj}, we have $\{a^L, b^L_i, c^L_j\}_{i,j} <| \{a^R, b^R_i, c^R_j\}_{i,j}$, which exactly gives ($\dagger$).

To check that $L <|_A R \implies L <| R$: take $a^L_i := \top$ and $a^R_j := \bot$.

To check that $L <|_{B_i} R \implies L <| R$: take $a^L_i := \bot$, $a^L_{i'} := \top$ for all $i' \ne i$, and $a^R_j := \bot$.

Dually, $L <|_{C_j} R \implies L <| R$.
Reflexivity follows, since $<|_A, <|_{B_i}, <|_{C_j}$ are reflexive.
Monotonicity is obvious.

Finally, we check left-transitivity; right-transitivity is dual.
Suppose $L <| R \cup S$, i.e., there are $\-a^L_i, \-a^R_j \in A$ such that
\begin{align*}
\tag{a} \bigwedge_i \-a^L_i \wedge \bigwedge (A \cap L) &\le \bigvee (A \cap (R \cup S)) \vee \bigvee_j \-a^R_j, \\
\tag{b} \bigwedge (B_i \cap L) &\le f_i(\-a^L_i) \vee \bigvee (B_i \cap (R \cup S)), \\
\tag{c} \bigwedge (C_j \cap L) \wedge g_j(\-a^R_j) &\le \bigvee (C_j \cap (R \cup S)); \\
\intertext{and for each $s \in S$, we have $L \cup \{s\} <| R$, i.e., there are $\-a^L_{s,i}, \-a^R_{s,j}$ such that}
\tag{a$_s$} \bigwedge_i \-a^L_{s,i} \wedge \bigwedge (A \cap (L \cup \{s\})) &\le \bigvee (A \cap R) \vee \bigvee_j \-a^R_{s,j}, \\
\tag{b$_s$} \bigwedge (B_i \cap (L \cup \{s\})) &\le f_i(\-a^L_{s,i}) \vee \bigvee (B_i \cap R), \\
\tag{c$_s$} \bigwedge (C_j \cap (L \cup \{s\})) \wedge g_j(\-a^R_{s,j}) &\le \bigvee (C_j \cap R).
\end{align*}
We must show $L <| R$.
Put
\begin{align*}
a^L_i &:= (\-a^L_i \vee \bigvee_{s \in B_i \cap S} \-a^L_{s,i}) \wedge \bigwedge_{s \in S \setminus B_i} \-a^L_{s,i}, \\
a^R_j &:= (\-a^R_j \wedge \bigwedge_{s \in C_j \cap S} \-a^R_{s,j}) \vee \bigvee_{s \in S \setminus C_j} \-a^R_{s,j}.
\end{align*}
Then
\begin{align*}
f_i(a^L_i) \vee \bigvee (B_i \cap R)
&= ((f_i(\-a^L_i) \vee \bigvee_{s \in B_i \cap S} f_i(\-a^L_{s,i})) \wedge \bigwedge_{s \in S \setminus B_i} f_i(\-a^L_{s,i})) \vee \bigvee (B_i \cap R) \\
\text{(by (b$_s$))} \qquad
&\ge ((f_i(\-a^L_i) \vee \bigvee_{s \in B_i \cap S} \bigwedge (B_i \cap (L \cup \{s\}))) \wedge \bigwedge (B_i \cap L)) \vee \bigvee (B_i \cap R) \\
&= ((f_i(\-a^L_i) \vee (\bigwedge (B_i \cap L) \wedge \bigvee (B_i \cap S))) \wedge \bigwedge (B_i \cap L)) \vee \bigvee (B_i \cap R) \\
&\ge (f_i(\-a^L_i) \vee \bigvee (B_i \cap R) \vee \bigvee (B_i \cap S)) \wedge \bigwedge (B_i \cap L) \\
\text{(by (b))} \qquad
&= \bigwedge (B_i \cap L), \\[1ex]
\bigwedge (C_j \cap L) \wedge g_j(a^R_j)
&= (\bigwedge (C_j \cap L) \wedge g_j(\-a^R_j) \wedge \bigwedge_{s \in C_j \cap S} g_j(\-a^R_{s,j})) \vee \bigvee_{s \in S \setminus C_j} (\bigwedge (C_j \cap L) \wedge g_j(\-a^R_{s,j})) \\
\text{(by (c), (c$_s$))} \qquad
&\le (\bigvee (C_j \cap (R \cup S)) \wedge \bigwedge (C_j \cap L) \wedge \bigwedge_{s \in C_j \cap S} g_j(\-a^R_{s,j})) \vee \bigvee (C_j \cap R) \\
&\le \bigvee (C_j \cap R) \vee \bigvee_{s \in C_j \cap S} (s \wedge \bigwedge (C_j \cap L) \wedge g_j(\-a^R_{s,j})) \\
\text{(by (c$_s$))} \qquad
&= \bigvee (C_j \cap R),
\end{align*}
and letting
$t := \bigwedge_i \bigwedge_{s \in S \setminus B_i} \-a^L_{s,i} \wedge \bigwedge (A \cap L)$,
\begin{align*}
\bigwedge_i a^L_i \wedge \bigwedge (A \cap L)
&= t \wedge \bigwedge_i (\-a^L_i \vee \bigvee_{s \in B_i \cap S} \-a^L_{s,i}) \\
\text{(using $t$)} \qquad
&= t \wedge \bigwedge_i (\-a^L_i \vee \bigvee_{s \in B_i \cap S} (\-a^L_{s,i} \wedge \bigwedge_{i' \ne i} \-a^L_{s,i'} \wedge \bigwedge (A \cap L))) \\
&= t \wedge \bigwedge_i (\-a^L_i \vee \bigvee_{s \in B_i \cap S} (\bigwedge_{i'} \-a^L_{s,i'} \wedge \bigwedge (A \cap L))) \\
\text{(by (a$_s$))} \qquad
&\le t \wedge \bigwedge_i (\-a^L_i \vee \bigvee_{s \in B_i \cap S} (\bigvee (A \cap R) \vee \bigvee_j \-a^R_{s,j})) \\
\text{(using $t$)} \qquad
&\le t \wedge ((\bigwedge_i \-a^L_i \wedge \bigwedge (A \cap L)) \vee \bigvee (A \cap R) \vee \bigvee_i \bigvee_{s \in B_i \cap S} \bigvee_j \-a^R_{s,j}) \\
\text{(by (a))} \qquad
&\le t \wedge (\bigvee (A \cap R) \vee \bigvee (A \cap S) \vee \bigvee_j \-a^R_j \vee \bigvee_i \bigvee_{s \in B_i \cap S} \bigvee_j \-a^R_{s,j}) \\
\text{(using $t$)} \qquad
&= t \wedge (\bigvee (A \cap R) \vee \bigvee_{s \in A \cap S} (s \wedge \-a^L_{s,i} \wedge \bigwedge (A \cap L)) \vee \bigvee_j \-a^R_j \vee \bigvee_i \bigvee_{s \in B_i \cap S} \bigvee_j \-a^R_{s,j}) \\
\text{(by (a$_s$))} \qquad
&= t \wedge (\bigvee (A \cap R) \vee \bigvee_{s \in A \cap S} \bigvee_j \-a^R_{s,j} \vee \bigvee_j \-a^R_j \vee \bigvee_i \bigvee_{s \in B_i \cap S} \bigvee_j \-a^R_{s,j}) \\
&= t \wedge (\bigvee (A \cap R) \vee \bigvee_j \-a^R_j \vee \bigvee_{s \in S \setminus \bigsqcup_j C_j} \bigvee_j \-a^R_{s,j}) \\
\text{(using $t$)} \qquad
&\le \bigvee (A \cap R) \vee \bigvee_j (\-a^R_j \wedge \bigwedge_{s \in C_j \cap S} \bigwedge_i \-a^L_{s,i} \wedge \bigwedge (A \cap L)) \vee \bigvee_{s \in S \setminus \bigsqcup_j C_j} \bigvee_j \-a^R_{s,j} \\
\text{(by (a$_s$))} \qquad
&\le \bigvee (A \cap R) \vee \bigvee_j (\-a^R_j \wedge \bigwedge_{s \in C_j \cap S} (\bigvee (A \cap R) \vee \bigvee_{j'} \-a^R_{s,j'})) \vee \bigvee_{s \in S \setminus \bigsqcup_j C_j} \bigvee_j \-a^R_{s,j} \\
&= \bigvee (A \cap R) \vee \bigvee_j (\-a^R_j \wedge \bigwedge_{s \in C_j \cap S} (\-a^R_{s,j} \vee \bigvee_{j' \ne j} \-a^R_{s,j'})) \vee \bigvee_{s \in S \setminus \bigsqcup_j C_j} \bigvee_j \-a^R_{s,j} \\
&\le \bigvee (A \cap R) \vee \bigvee_j (\-a^R_j \wedge \bigwedge_{s \in C_j \cap S} \-a^R_{s,j}) \vee \bigvee_j \bigvee_{s \in S \setminus C_j} \-a^R_{s,j} \\
&= \bigvee (A \cap R) \vee \bigvee_j a^R_j.
\end{align*}
Thus the $a^L_i, a^R_j$ witness $L <| R$.
\end{proof}

We have the following special cases of \cref{thm:bifrm-minterp}:

\begin{remark}
If $A = 2 = \{\bot < \top\}$ is the initial $(\kappa, \kappa)$-frame, the relations $\iota_i \circ f_i \le \iota$ and $\iota \le \iota_j \circ g_j$ become trivial, so the bilax pushout in \cref{thm:bifrm-minterp} is just the coproduct $\coprod_i B_i \amalg \coprod_j C_j$.
If we take $a^L = \top$ and $a^R = \bot$, then ($\dagger$) says that either $a^L_i = \bot$ for some $i$ in which case $b^L_i \le b^R_i$, or $a^R_j = \top$ for some $j$ in which case $c^L_j \le c^R_j$.
We thus recover \cref{thm:bifrm-coprod} (for $\lambda = \kappa$).
\end{remark}

Now consider \cref{thm:bifrm-minterp} when $A, B_i, C_j$ are $\kappa$-Boolean algebras.
Then $D$, being generated by complemented elements, is also a $\kappa$-Boolean algebra.
Moreover, the relation $f(a) \le f(b) \iff f(\neg a) \ge f(\neg b)$ means that imposing inequalities between morphisms is the same as imposing equalities, so that a bilax pushout is simply a pushout.
Moving the inequalities in \cref{thm:bifrm-minterp} to one side, we get the following simplified statement:

\begin{corollary}
\label{thm:bool-minterp}
Let $A, B_i$ be $<\kappa$-many $\kappa$-Boolean algebras, with homomorphisms $f_i : A -> B_i$, and let $D$ be their pushout.
Let $a \in A$ and $b_i \in B_i$ for each $i$ with $\bigwedge_i b_i \le a \in D$.
Then there are $a_i \in A$ such that $\bigwedge_i a_i \le a \in A$ and $b_i \le f_i(a_i) \in B_i$.
\qed
\end{corollary}

The case of a single $B_i$ and $C_j$ in \cref{thm:bifrm-minterp} is particularly useful:

\begin{corollary}
\label{thm:bifrm-interp}
Let
\begin{equation*}
\begin{tikzcd}
B \rar["\iota_1"] \drar[phantom,"\le"{sloped}] & D \\
A \uar["f"] \rar["g"'] & C \uar["\iota_2"']
\end{tikzcd}
\end{equation*}
be a \defn{cocomma} square in $\!{\kappa\kappa Frm}$, i.e., a universal square such that $\iota_1 \circ f \le \iota_2 \circ g$.
Let $b, b' \in B$ and $c, c' \in C$ with $b \wedge c \le b' \vee c' \in D$.
Then there is $a \in A$ with $b \le f(a) \vee b'$ and $c \wedge g(a) \le c'$.

In particular, for $b \in B$ and $c \in C$ with $b \le c \in D$, there is $a \in A$ with $b \le f(a)$ and $g(a) \le c$.

The same holds for pushout squares in $\!{\kappa Bool}$.
\end{corollary}
\begin{proof}
Since $\iota_1 \circ f \le \iota_2 \circ g$ is implied by $\iota_1 \circ f \le \iota \le \iota_2 \circ g$ for $\iota : A -> D$, the cocomma $D$ maps into the bilax pushout, thus $b \wedge c \le b' \vee c' \in D$ implies that the same relation holds in the bilax pushout, whence \cref{thm:bifrm-minterp} applies.
As above, when $A, B, C \in \!{\kappa Bool}$, we get the Boolean version.
\end{proof}

The $\kappa$-Boolean version of \cref{thm:bifrm-interp} was essentially proved  by LaGrange~\cite{Lamalg}, via a different method, and in the form of the easily equivalent \cref{thm:bool-sap} (see \cite[3.2]{Cborin}).

\begin{corollary}[LaGrange]
\label{thm:bool-pushout-inj}
Pushout in $\!{\kappa Bool}$ preserves monomorphisms, i.e., in a pushout
\begin{equation*}
\begin{tikzcd}
B \rar["g'"] & D \\
A \uar["f"] \rar["g"'] & C \uar["f'"']
\end{tikzcd}
\end{equation*}
if $f$ is injective, then so is $f'$.
More generally, $f'$ is injective as long as $\ker(f) \subseteq \ker(g)$.
\end{corollary}
\begin{proof}
Let $c \in C$ with $f'(c) = \bot$.
Then $f'(c) \le g'(\bot)$, so by \cref{thm:bifrm-interp}, there is $a \in A$ with $c \le g(a)$ and $f(a) \le \bot$.
Since $\ker(f) \subseteq \ker(g)$, we get $c \le g(a) = \bot$.
\end{proof}

A category has the \defn{strong amalgamation property} if in the above pushout square, whenever $f, g$ are both injective, then so are $f', g'$ and the square is also a pullback.

\begin{corollary}[LaGrange]
\label{thm:bool-sap}
$\!{\kappa Bool}$ has the strong amalgamation property.
\end{corollary}
\begin{proof}
Let $(b, c) \in B \times_D C$, i.e., $g'(b) = f'(c)$.
Then by \cref{thm:bifrm-interp}, there are $a, a' \in A$ such that $f(a') \le b \le f(a)$ and $g(a) \le c \le g(a')$.
Since $f, g$ are injective, $a' \le a$ and $a \le a'$, whence $f(a) = b$ and $g(a) = c$.
\end{proof}

A morphism is a \defn{regular monomorphism} if it is the equalizer of its cokernel (i.e., pushout with itself); see \cref{sec:cat-alg}.

\begin{corollary}
\label{thm:bool-mono-reg}
All monomorphisms in $\!{\kappa Bool}$ are regular.
\end{corollary}
\begin{proof}
If $f : A -> B$ is a monomorphism, then letting $\iota_1, \iota_2 : B \rightrightarrows C$ be its cokernel, for any $b \in \eq(\iota_1, \iota_2)$, i.e., $\iota_1(b) = \iota_2(b)$, by the strong amalgamation property we have $b \in \im(f)$.
\end{proof}

\begin{corollary}[LaGrange]
\label{thm:bool-epi-surj}
All epimorphisms in $\!{\kappa Bool}$ are regular, i.e., surjective.
\end{corollary}
\begin{proof}
$f : A -> B$ is an epimorphism iff its pushout with itself is the identity $1_B$; then for any $b \in B$, by \cref{thm:bifrm-interp}, there is $a \in A$ with $b \le f(a) \le b$.
\end{proof}

\begin{corollary}
\label{thm:bool-free-cons}
For $\kappa \le \lambda$, the free functor $\ang{-}_\!{\lambda Bool} : \!{\kappa Bool} -> \!{\lambda Bool}$ is conservative.
\end{corollary}
\begin{proof}
Since $\ang{-}_\!{\lambda Bool}$ is faithful (\cref{thm:frm-neginf-bool-inj}), it reflects monomorphisms and epimorphisms;
since a morphism in $\!{\kappa Bool}$ which is both a monomorphism and an epimorphism is a bijection by \cref{thm:bool-epi-surj}, hence an isomorphism, it follows that $\ang{-}_\!{\lambda Bool}$ reflects isomorphisms.
\end{proof}

We have analogous consequences of \cref{thm:bifrm-interp} for $\!{\kappa\kappa Frm}$.
Given a morphism $f : A -> B$ in a locally ordered category, consider its cocomma with itself:
\begin{equation*}
\begin{tikzcd}
B \rar["\iota_1"] \drar[phantom, "\le"{sloped}] & C \\
A \uar["f"] \rar["f"'] & B \uar["\iota_2"']
\end{tikzcd}
\end{equation*}
We say that $f$ is an \defn{order-regular monomorphism} if $f$ exhibits $A$ as the \defn{inserter} $\ins(\iota_1, \iota_2)$, which is the universal subobject of $B$ on which $\iota_1 \le \iota_2$ holds; in categories of ordered algebraic structures, we have $\ins(\iota_1, \iota_2) = \{b \in B \mid \iota_1(b) \le \iota_2(b)\}$.
We say that $f$ is an \defn{order-epimorphism} if $g \circ f \le h \circ f \implies g \le h$ for all $g, h : B \rightrightarrows D$, or equivalently $\iota_1 \le \iota_2$, i.e., $\ins(\iota_1, \iota_2) = B$.
See \cref{sec:cat-ord}.

\begin{corollary}
In any cocomma square in $\!{\kappa\kappa Frm}$ as above, $f$ surjects onto $\ins(\iota_1, \iota_2)$.
\end{corollary}
\begin{proof}
For any $b \in \ins(\iota_1, \iota_2)$, by \cref{thm:bifrm-interp} there is $a \in A$ such that $b \le f(a) \le b$.
\end{proof}

\begin{corollary}
\label{thm:bifrm-mono-oreg}
All monomorphisms in $\!{\kappa\kappa Frm}$ are order-regular.
\qed
\end{corollary}

\begin{corollary}
\label{thm:bifrm-oepi-surj}
All order-epimorphisms in $\!{\kappa\kappa Frm}$ are surjective.
\qed
\end{corollary}

\begin{corollary}
\label{thm:bifrm-free-cons}
For $\kappa \le \lambda$, the free functor $\!{\kappa\kappa Frm} -> \!{\lambda\lambda Frm}$ is conservative.
\end{corollary}
\begin{proof}
This follows similarly to \cref{thm:bool-free-cons} from \cref{thm:bifrm-oepi-surj} and the fact that the free functor is order-faithful, i.e., an order-embedding on each hom-poset (because the units are order-embeddings, by \cref{thm:frm-dpoly-bifrm-inj}).
\end{proof}

\begin{remark}
Epimorphisms in $\!{\kappa\kappa Frm}$ are \emph{not} necessarily surjective: the unit $A -> \ang{A}_\!{\kappa Bool}$ is an epimorphism, essentially because $\neg$ is uniquely defined in terms of $\wedge, \vee$.
Indeed, the set of $b \in \ang{A}_\!{\kappa Bool}$ at which the value of every $f : \ang{A}_\!{\kappa Bool} -> C \in \!{\kappa Bool}$ is determined by $f|A$ is easily seen to be a $\kappa$-Boolean subalgebra containing $A$, hence is all of $\ang{A}_\!{\kappa Bool}$.

It follows that monomorphisms in $\!{\kappa\kappa Frm}$ are \emph{not} necessarily regular: $A -> \ang{A}_\!{\kappa Bool}$ is a monomorphism and an epimorphism; if it were a regular monomorphism, then (being an epimorphism) it would have to be an isomorphism.
\end{remark}

\begin{remark}
\label{rmk:bifrm-oepi-surj}
Order-epimorphisms in $\!{\lambda\kappa Frm}$ for $\lambda \ne \kappa$ are \emph{not} necessarily surjective.
Take $\lambda = \omega$ and $\kappa \ge \omega_1$, and consider the topology $\@O(\-{\#N})$ of $\-{\#N} = \#N \cup \{\infty\}$, the one-point compactification of the discrete space $\#N$.
The inclusion $\@O(\-{\#N}) `-> \@P(\-{\#N})$ is a non-surjective $\kappa$-frame order-epimorphism.
Indeed, every $a \in \@P(\-{\#N}) \setminus \@O(\-{\#N})$ is the join $a = (a \cap \#N) \cup \{\infty\}$ of some $a \cap \#N \in \@O(\-{\#N})$ with an element which is a meet $\{\infty\} = \bigcap_{n \in \#N} [n, \infty] = \neg \bigcup_n \neg [n, \infty]$ of complemented elements of $\@O(\-{\#N})$.
Thus for $\kappa$-frame homomorphisms $f, g : \@P(\-{\#N}) \rightrightarrows B$, which automatically preserve complements, if $f \le g$ on $\@O(\-{\#N})$, then also $f(\{\infty\}) \le g(\{\infty\})$, whence $f \le g$ on all of $\@P(\-{\#N})$.

It follows that \cref{thm:bifrm-minterp} must also fail for $\!{\lambda\kappa Frm}$ with $\lambda \ne \kappa$, or else the proof of \cref{thm:bifrm-oepi-surj} would go through.
\end{remark}

\begin{remark}
The proof of \cref{thm:bifrm-minterp} is inspired by the well-known technique of proving the Craig--Lyndon interpolation theorem by induction on derivations in the cut-free sequent calculus.
See e.g., \cite[\S4.4]{TSprf}.
However, we note that \cref{thm:bifrm-minterp} does not immediately follow from applying this technique to the sequent calculus from \cref{sec:frm-cbool}: since we do not assume the algebras are free (or even $\kappa$-presented), we would need to include extra ``non-logical axioms'', which the cut admissibility proof using \emph{finitary} sequents from \cref{thm:frm-cbool-cut} does not easily handle.
\end{remark}

\section{Locales and Borel locales}
\label{sec:loc}

In this section, we study the categories of ``formal spaces'' which are the duals of the algebraic categories $\!{\kappa Frm}, \!{\kappa\kappa Frm}, \!{\kappa Bool}$ studied in the previous section.
In \cref{sec:loc-cat}, we define these categories and discuss their relations to each other as well as to the corresponding categories of spaces; in \cref{sec:loc-sp}, we complete this discussion by defining the spatialization adjunctions between the corresponding categories of locales and spaces.
In \cref{sec:loc-bor}, we define the $\kappa$-Borel hierarchy on a $\kappa$-locale via the dual of the $\@N_\kappa$ functor from \cref{sec:frm-neg}, and show how several algebraic results we proved in \cref{sec:frm-neg,sec:upkzfrm} correspond to classical descriptive set-theoretic results on ``change of topology''.
In the crucial \cref{sec:loc-im}, we give a detailed analysis of notions of ``sublocale''/``subobject'', ``embedding'', ``injection'', ``surjection'', and ``image'' in our main categories.

The rest of the subsections cover somewhat more specialized topics.
In \cref{sec:loc-ctbpres}, we review some known results which say that the spatialization adjunctions become equivalences when $\kappa = \omega_1$ and furthermore the algebras/spaces involved are restricted to being countably (co)presented; we then use these to give several important examples.
In \cref{sec:loc-intlog}, we review some abstract categorical properties of our various categories of (Borel) locales, as well as the method from categorical logic of interpreting ``pointwise'' expressions in arbitrary categories obeying such properties.
In \cref{sec:poloc}, we show that our category of ``positive $\kappa$-Borel locales'' (dual to $\!{\kappa\kappa Frm}$) can be meaningfully regarded as $\kappa$-Borel locales equipped with an intrinsic ``specialization order''.
Finally, in \cref{sec:loc-baire} we recall Isbell's well-known localic Baire category theorem and discuss some of its consequences as well as a generalization.

\subsection{The main categories}
\label{sec:loc-cat}

A topological space $X = (X, \@O(X))$ can be defined as an underlying set $X$ together with a subframe $\@O(X) \subseteq \@P(X)$; a continuous map $f : X -> Y$ between topological spaces is a map such that $f^{-1} : \@P(Y) -> \@P(X)$ restricts to a frame homomorphism $\@O(Y) -> \@O(X)$.
A \defn{locale} $X$ is, formally, the same thing as an arbitrary frame $\@O(X)$, but thought of as the frame of ``open sets'' in a generalized ``space'' without an underlying set.
A \defn{continuous map} $f : X -> Y$ between locales is a frame homomorphism $f^* : \@O(Y) -> \@O(X)$.
The \defn{category of locales} will be denoted $\!{Loc}$; thus by definition, $\!{Loc} = \!{Frm}^\op$.

\begin{convention}
\label{cvt:loc-frm}
There are various commonly used conventions regarding the locale/frame distinction; see e.g., \cite{PPloc}.
We follow Isbell \cite{Iloc} in \defn{strictly distinguishing between locales and frames}: the former are ``geometric'' objects, while the latter are their ``algebraic'' duals.

This convention means that we may unambiguously speak of notions such as \defn{product locales} $X \times Y$ (coproduct frames $\@O(X \times Y) := \@O(X) \otimes \@O(Y)$), which only implicitly refer to the direction of morphisms (the product projection $\pi_1 : X \times Y -> X$ corresponding to the coproduct injection $\iota_1 : \@O(X) -> \@O(X) \otimes \@O(Y)$).
\end{convention}

\begin{convention}
\label{cvt:loc-sets}
In the locale theory literature, it is common to refer to the elements of the frame $\@O(X)$ corresponding to a locale $X$ as \emph{open parts}, \emph{opens}, or some other name emphasizing that they need not be subsets in the usual sense; see e.g., \cite{Idst}.
However, we will be dealing extensively with localic analogs of other types of sets, e.g., closed, Borel, and analytic sets, and we feel it would be rather awkward and confusing to similarly rename all of these.
Therefore, we will freely refer to elements $U \in \@O(X)$ as \defn{open sets of $X$}.
Note that there should be no ambiguity, since by \cref{cvt:loc-frm}, there is no other notion of ``subset'' of $X$.
When working in $\@O(X)$, we will freely use set-theoretic notation interchangeably with order-theoretic ones, e.g.,
\begin{align*}
U \subseteq V &\coloniff U \le V, &
U \cap V &:= U \wedge V, &
X &:= \top, &
\emptyset &:= \bot.
\end{align*}
\end{convention}

For each of the other categories defined in \cref{sec:frm-cat}, consisting of algebraic structures with a subset of the operations of a complete Boolean algebra, we can consider an analogous notion of ``space'' consisting of an underlying set $X$ equipped with a subalgebra (of the specified type) of the powerset $\@P(X)$.
We can then define a pointless generalization of such ``spaces'' by dropping the underlying set and the requirement that the algebra be a subalgebra of a powerset.

Perhaps the most familiar example of such a ``space'', other than topological, is a \defn{Borel space} (or \emph{measurable space}) $X = (X, \@B(X))$, consisting of a set $X$ together with a $\sigma$-algebra $\@B(X)$, i.e., a $\sigma$-Boolean subalgebra $\@B(X) \subseteq \@P(X)$, whose elements are called \defn{Borel sets} (or \emph{measurable sets}).
We analogously define a \defn{($\sigma$-)Borel locale} $X$ to mean an arbitrary $\sigma$-Boolean algebra $\@B(X) = \@B_\sigma(X)$, whose elements are called \defn{($\sigma$-)Borel sets} of $X$.
A \defn{($\sigma$-)Borel map} $f : X -> Y$ between $\sigma$-Borel locales is a $\sigma$-Boolean homomorphism $f^* : \@B_\sigma(Y) -> \@B_\sigma(X)$.
The \defn{category of ($\sigma$-)Borel locales} will be denoted $\!{\sigma BorLoc} := \!{\sigma Bool}^\op$.

Analogously, by a \defn{$\kappa$-Borel space} we mean a set $X$ equipped with a $\kappa$-Boolean subalgebra $\@B_\kappa(X) \subseteq \@P(X)$.
We then define a \defn{$\kappa$-Borel locale} $X$ to be an arbitrary $\kappa$-Boolean algebra $\@B_\kappa(X)$, whose elements are called \defn{$\kappa$-Borel sets} of $X$.
A \defn{$\kappa$-Borel map} $f : X -> Y$ is a $\kappa$-Boolean homomorphism $\@B _\kappa(Y) -> \@B_\kappa(X)$.
The \defn{category of $\kappa$-Borel locales} is $\!{\kappa BorLoc} := \!{\kappa Bool}^\op$.
We include the case $\kappa = \infty$ by taking an \defn{$\infty$-Borel locale} $X$ to be a \emph{small-presented large} complete Boolean algebra $\@B_\infty(X)$; thus $\!{\infty BorLoc} := \!{CBOOL}_\infty^\op$.

\begin{example}
An $\omega$-Borel space is a set $X$ equipped with a Boolean algebra of subsets $\@B_\omega(X)$.
This can be seen as a description of a zero-dimensional compactification of $X$ (namely the Stone space of $\@B_\omega(X)$), hence might be called an ``ultraproximity'' on $X$.
On the other hand, an $\omega$-Borel locale is, equivalently by Stone duality, just a Stone space.
\end{example}

\begin{example}
\label{ex:loc-infbor}
An $\infty$-Borel space is a set $X$ equipped with the complete (atomic) Boolean algebra $\@B_\infty(X)$ of invariant sets for some equivalence relation on $X$; if $\@B_\infty(X)$ separates points of $X$, then it must be the full powerset $\@P(X)$.
On the other hand, an $\infty$-Borel locale can, in particular, be given by an arbitrary small complete Boolean algebra, possibly atomless.
For example, the regular open or Lebesgue measure algebra of $[0, 1]$ correspond to $\infty$-Borel locales, which, intuitively, are the ``intersections'' of all dense open, respectively conull, sets in $[0, 1]$ (see \cref{ex:loc-baire-perfect}).
\end{example}

Analogously in the topological setting, by a \defn{$\kappa$-topological space} we will mean a set $X$ equipped with a $\kappa$-subframe $\@O_\kappa(X) \subseteq \@P(X)$.
We then define a \defn{$\kappa$-locale} $X$ to be an arbitrary $\kappa$-frame $\@O_\kappa(X)$, whose elements are called \defn{$\kappa$-open sets} of $X$; a $\kappa$-frame homomorphism $\@O_\kappa(Y) -> \@O_\kappa(X)$ is a \defn{$\kappa$-continuous map} $X -> Y$, and the \defn{category of $\kappa$-locales} is denoted $\!{\kappa Loc} := \!{\kappa Frm}^\op$.
(Note that when we omit $\kappa$, we mean $\kappa = \omega_1$ in the Borel context but $\kappa = \infty$ in the topological context, in accord with standard terminology.)

The main example of a $\kappa$-topological space is a topological space which is \defn{$\kappa$-based}, i.e., has a $\kappa$-ary (sub)basis; in such a space, arbitrary unions of open sets are determined by $\kappa$-ary ones, i.e., a $\kappa$-based topological space is the same thing as a $\kappa$-based $\kappa$-topological space (see \cref{thm:frm-idl-kgen}).
For example, a second-countable (= $\sigma$-based) topological space can be faithfully regarded as a $\sigma$-topological space.
The analogous class of $\kappa$-locales correspond to the $\kappa$-generated ($\kappa$-)frames.

A (second-countable $\sigma$-)topological space contains more structure than a Borel space, which may be forgotten by replacing the ($\sigma$-)topology with the Borel $\sigma$-algebra it generates.
Motivated by this, we adopt the following

\begin{convention}
\label{cvt:loc-forget}
We regard all of the \emph{free} functors from \cref{sec:frm-cat} as nameless \emph{forgetful} functors between the dual categories of locales (and variations thereof).
Thus, for example, a $\sigma$-locale $X$ has an \defn{underlying $\sigma$-Borel locale}, given by $\@B_\sigma(X) := \ang{\@O_\sigma(X) \qua \!{\sigma Frm}}_\!{\sigma Bool}$.

Note that this convention is consistent with our \cref{cvt:frm-cbool-incl}, according to which $\@O_\sigma(X) \subseteq \@B_\sigma(X)$ is regarded as a $\sigma$-subframe, namely the $\sigma$-subframe of $\sigma$-open sets.
We also have $\@O_\sigma(X) \subseteq \@O(X) := \ang{\@O_\sigma(X) \qua \!{\sigma Frm}}_\!{Frm} \cong \sigma\@I(\@O_\sigma(X))$; note that $\@O_\sigma(X) = \@O(X)_\sigma$ (by \cref{thm:vlat-cptbasis}), i.e., ``$\sigma$-open'' = ``$\sigma$-compact open''.
All of these lattices are contained in the (large) complete Boolean algebra $\@B_\infty(X) := \ang{\@O_\sigma(X) \qua \!{\sigma Frm}}_\!{CBOOL}$.
It is often helpful to regard $\@B_\infty(X)$ as playing role of the ``powerset $\@P(X)$'' in the localic context; see \cref{ex:loc-infbor}.
For example, a locale $X$ consists of an underlying $\infty$-Borel locale together with a ``topology'' $\@O(X) \subseteq \@B_\infty(X)$.
(However, $\@O(X)$ is not allowed to be an arbitrary subframe of $\@B_\infty(X)$: it needs to be small and to freely generate $\@B_\infty(X)$ qua frame.
See \cref{rmk:frm-bool-gen-free}.)

Via this analogy, we can define the analogs of various common topological set notions in a locale $X$:
a \defn{closed set} will mean an element of $\neg \@O(X) = \{\neg U \mid U \in \@O(X)\} \subseteq \@B_\infty(X)$;
a \defn{clopen set} is an element of $\@O(X)_\neg$;
the \defn{interior} $B^\circ$ and \defn{closure} $\-B$ of an arbitrary $\infty$-Borel $B \in \@B_\infty(X)$ are the greatest open set $\le B$ and least closed set $\ge B$, respectively;
$B$ is \defn{dense} if $\-B = X$ (= $\top \in \@B_\infty(X)$).
We also call $X$ \defn{$\kappa$-based} if $\@O(X)$ is $\kappa$-generated (as a frame or $\bigvee$-lattice).
\end{convention}

\begin{remark}
\label{rmk:loc-reglind-baire}
If $X$ is a compact Hausdorff space, or more generally a regular Lindelöf space, then $\@O(X)$ is a $\sigma$-coherent frame, whose $\sigma$-compact elements $\@O_\sigma(X) \subseteq \@O(X)$ are the cozero open sets in $X$.
Thus, $\sigma$-locales can also be regarded as a generalization of regular Lindelöf spaces; see \cite{MVlind} for this point of view.

The \defn{Baire $\sigma$-algebra} in a topological space is that generated by the (co)zero sets (see e.g., \cite[4A3K]{Fmeas}; this definition agrees with other commonly used definitions in compact Hausdorff spaces).
Thus, it would perhaps have been more consistent for us to use ``Baire set'' instead of ``Borel set'' in our generalized terminology.
We have chosen the latter, due to its far more ubiquitous use (as well as the various other unrelated uses of ``Baire'') in classical descriptive set theory.
\end{remark}

We also introduce analogous terminology corresponding to the categories $\!{\kappa\kappa Frm}$.
A \defn{positive $\kappa$-Borel space} is a set $X$ equipped with a $(\kappa, \kappa)$-subframe $\@B^+_\kappa(X) \subseteq \@P(X)$;
a \defn{positive $\kappa$-Borel locale} $X$ is an arbitrary $(\kappa, \kappa)$-frame $\@B^+_\kappa(X)$, whose elements are called \defn{positive $\kappa$-Borel sets} of $X$; and the category of these and \defn{positive $\kappa$-Borel maps} (i.e., $(\kappa, \kappa)$-frame homomorphisms in the opposite direction) is denoted $\!{\kappa Bor^+Loc} := \!{\kappa\kappa Frm}^\op$.
As with $\kappa$-Borel locales, we take a \defn{positive $\infty$-Borel locale} to mean a small-presented large $(\infty, \infty)$-frame.
These notions sit in between locales and Borel locales: every $\kappa$-locale $X$ has an underlying positive $\kappa$-Borel locale (given by $\@B^+_\kappa(X) := \ang{\@O_\kappa(X) \qua \!{\kappa Frm}}_\!{\kappa\kappa Frm}$), while every positive $\kappa$-Borel locale $X$ has an underlying $\kappa$-Borel locale (given by $\@B_\kappa(X) := \ang{\@B^+_\kappa(X) \qua \!{\kappa\kappa Frm}}_\!{\kappa Bool}$).

The main purpose of introducing positive Borel locales is to provide the localic analog of the specialization preorder on a topological space:

\begin{example}
\label{ex:loc-infbor+}
A positive $\infty$-Borel space is a set $X$ equipped with a collection of subsets $\@B^+_\infty(X)$ closed under arbitrary intersections and unions; such a collection is also known as an \defn{Alexandrov topology} on $X$, and consists of all the upper sets with respect to some preorder $\le$ on $X$.
The \defn{specialization preorder} on an arbitrary topological space $X$ is defined by
\begin{align*}
x \le y \coloniff x \in \-{\{y\}} \iff \forall x \in U \in \@O(X)\, (y \in U),
\end{align*}
and corresponds to the Alexandrov topology generated by the topology on $X$.
Thus, the specialization preorder can be regarded as the ``underlying positive $\infty$-Borel space'' of a topological space, where positive $\infty$-Borel sets are precisely the upper sets;
the underlying positive $\infty$-Borel locale of a locale is analogous.
\end{example}

\begin{example}
\label{ex:loc-sbor+}
It is a classical theorem from descriptive set theory (see \cite[28.11]{Kcdst}) that a Borel subset of Cantor space $2^\#N$ is upward-closed in the product of the usual ordering $2 = \{0 < 1\}$ iff it belongs to the closure under countable intersections and unions of the (subbasic) open sets for the Scott topology on $2^\#N$, i.e., the product $\#S^\#N$ of the \defn{Sierpinski space} $\#S := 2$ with the topology where $\{1\}$ is open but not closed.
In other words, the underlying positive ($\sigma$-)Borel space of $\#S^\#N$ remembers precisely the Borel $\sigma$-algebra together with the specialization preorder.
We will prove the $\kappa$-localic analog of the theorem just cited in \cref{thm:loc-pos-upper} below.
\end{example}

We locally order the categories $\!{\kappa Bor^+Loc}$, as well as $\!{\kappa Loc}$, via the \emph{same} pointwise partial order on the hom-sets of the dual categories $\!{\kappa\kappa Frm}$ and $\!{\kappa Frm}$.
That is, for $f, g : X -> Y$, say in $\!{\kappa Loc}$,
\begin{align*}
f \le g  \coloniff  \forall U \in \@O_\kappa(Y)\, (f^*(U) \le g^*(U)).
\end{align*}
(Note that we do not flip the partial order on the hom-sets; in 2-categorical terminology, we flip the $1$-cells but not the $2$-cells, i.e., $\!{\kappa Loc} = \!{\kappa Frm}^\op \ne \!{\kappa Frm}^\mathrm{coop}$.)
To motivate this, note that for functions $f, g : X -> Y$ to a preorder $Y$, we have $f \le g$ pointwise iff for every upper $U \subseteq Y$, we have $f^{-1}(U) \subseteq g^{-1}(U)$.

The following diagram, dual to part of the diagram \eqref{diag:frm-cat}, depicts the relationships between the categories introduced above, for $\kappa \le \lambda$:
\begin{equation}
\label{diag:loc-cat}
\begin{tikzcd}[column sep=2em]
\mathllap{\scriptstyle \!{Stone} \;\simeq\;}
\!{\omega BorLoc}
    \dar[shift left=2]
    \rar[shift left=2, right adjoint'] &
\!{\sigma BorLoc}
    \lar[shift left=2]
    \dar[shift left=2]
    \rar[shift left=2, right adjoint'] &
\dotsb
    \lar[shift left=2]
    \rar[shift left=2, right adjoint'] &
\!{\kappa BorLoc}
    \lar[shift left=2]
    \dar[shift left=2]
    \rar[shift left=2, right adjoint'] &
\!{\lambda BorLoc}
    \lar[shift left=2]
    \dar[shift left=2]
    \rar[shift left=2] &
\!{\infty BorLoc}
    \dar[shift left=2]
\\
\!{\omega Bor^+Loc}
    \uar[shift left=2, right adjoint']
    \rar[shift left=2, right adjoint'] &
\!{\sigma Bor^+Loc}
    \lar[shift left=2]
    \uar[shift left=2, right adjoint']
    \dar[shift left=2]
    \rar[shift left=2, right adjoint'] &
\dotsb
    \lar[shift left=2]
    \rar[shift left=2, right adjoint'] &
\!{\kappa Bor^+Loc}
    \lar[shift left=2]
    \uar[shift left=2, right adjoint']
    \dar[shift left=2]
    \rar[shift left=2, right adjoint'] &
\!{\lambda Bor^+Loc}
    \lar[shift left=2]
    \uar[shift left=2, right adjoint']
    \dar[shift left=2]
    \rar[shift left=2] &
\!{\infty Bor^+Loc}
    \uar[shift left=2, right adjoint']
\\
\mathllap{\scriptstyle \!{Spec} \;\simeq\;}
\!{\omega Loc}
    \uar[equal]
    \rar[shift left=2, right adjoint'] &
\!{\sigma Loc}
    \lar[shift left=2]
    \uar[shift left=2, right adjoint']
    \rar[shift left=2, right adjoint'] &
\dotsb
    \lar[shift left=2]
    \rar[shift left=2, right adjoint'] &
\!{\kappa Loc}
    \lar[shift left=2]
    \uar[shift left=2, right adjoint']
    \rar[shift left=2, right adjoint'] &
\!{\lambda Loc}
    \lar[shift left=2]
    \uar[shift left=2, right adjoint']
    \rar[shift left=2, right adjoint'] &
\!{Loc}
    \lar[shift left=2]
    \uar[shift left=2]
\end{tikzcd}
\end{equation}
The functors pointing upward or to the right are the forgetful functors from \cref{cvt:loc-forget}, corresponding to the \emph{free} functors in \eqref{diag:frm-cat}; the left adjoints, when they exist, correspond to the \emph{forgetful} functors in \eqref{diag:frm-cat}.
(Most free functors from $\!{\infty BorLoc}, \!{\infty Bor^+Loc}$ do not exist, since they are dual to categories of small-presented large structures; but the free functor from the former category to the latter, i.e., the forgetful functor $\!{CBOOL}_\infty -> \!{\infty\infty FRM}_\infty$, does exist, because a small-presented complete Boolean algebra is also small-presented as an $(\infty, \infty)$-frame by adding complements of generators, as described before \cref{thm:bool-cover-pres}.)
Here $\!{Stone}$ denotes the category of Stone spaces, while $\!{Spec}$ denotes the category of \defn{spectral spaces} (or \defn{coherent spaces}), which is dual to $\!{\omega Frm} = \!{DLat}$ by Stone duality for distributive lattices; see \cite[II~3.4]{Jstone}.

By the various results in \cref{sec:frm} on preservation properties of free functors, we have

\begin{proposition}
\label{thm:loc-forget}
\leavevmode
\begin{enumerate}
\item[(a)]  All of the forgetful functors in the above diagram \eqref{diag:loc-cat} are faithful, as well as order-faithful for the functors with domain and codomain in the bottom two rows.
\item[(b)]  The horizontal forgetful functors are also conservative, as well as full-on-isomorphisms in the case of $\!{\kappa Loc} -> \!{\lambda Loc}$.
\item[(c)]  The forgetful functors whose domain is in column ``$\kappa$'' preserve $\kappa$-ary coproducts (as well as all small limits).
\end{enumerate}
\end{proposition}
\begin{proof}
By \cref{thm:frm-idl-fulliso,thm:frm-neginf-inj,thm:frm-neginf-bool-inj,thm:frm-dpoly-bifrm-inj,thm:bool-free-cons,thm:bifrm-free-cons}, the results in \cref{sec:frm-prod}, and the fact that free functors always preserve colimits.
\end{proof}

Intuitively, the fact that the forgetful functors are faithful means that all of the above categories can be seen as consisting of $\infty$-Borel locales equipped with additional structure, while the morphisms are $\infty$-Borel maps which preserve that additional structure.
Fullness of $\!{\kappa Loc} -> \!{\lambda Loc}$ on isomorphisms means that being a $\kappa$-locale can be regarded as a mere \emph{property} of a locale, rather than additional structure; that is, we may treat $\kappa$-locales as particular kinds of locales.%
\footnote{In the literature (e.g., \cite{Jstone}), $\kappa$-locales are usually called \defn{$\kappa$-coherent locales}.}
Conservativity of the other horizontal functors means that e.g., if a $\sigma$-Borel map $f : X -> Y$ between $\sigma$-Borel locales is an $\infty$-Borel isomorphism, then its inverse (as an $\infty$-Borel map) is automatically a $\sigma$-Borel map.
This can be seen as an analog of the result from classical descriptive set theory that a Borel bijection between standard Borel spaces automatically has Borel inverse (see \cite[14.12]{Kcdst}).

We use the following terminology for (co)limits in the above categories.
As already mentioned (\cref{cvt:loc-frm}), \defn{product locales} will mean categorical products, dual to coproducts in the dual algebraic categories.
We will also refer to $\kappa$-ary coproducts in $\!{\kappa Loc}, \!{\kappa BorLoc}, \!{\kappa Bor^+Loc}$ as \defn{disjoint unions}, denoted $\bigsqcup$, and corresponding to products of the dual algebras; the nullary coproduct is also called the \defn{empty locale} $\emptyset$ (with $\@O(\emptyset) = \@B_\infty(\emptyset) = 1$).
In addition to \cref{thm:loc-forget}(c), the dual of \cref{thm:frm-prod-pushout} tells us that such coproducts are well-behaved:

\begin{proposition}
\label{thm:loc-coprod-pullback}
$\kappa$-ary disjoint unions in $\!{\kappa Loc}, \!{\kappa BorLoc}, \!{\kappa Bor^+Loc}$ are pullback-stable.
\qed
\end{proposition}

A \defn{$\kappa$-sublocale} of a $\kappa$-locale $X$ will mean a \emph{regular} subobject of $X \in \!{\kappa Loc}$, hence an equivalence class of $\kappa$-locales $Y$ equipped with a $\kappa$-continuous map $f : Y -> X$ which is a regular monomorphism, i.e., $f^* : \@O_\kappa(X) -> \@O_\kappa(Y) \in \!{\kappa Frm}$ is a regular epimorphism, i.e., surjective (see \cref{sec:cat-alg}).
In other words, every $\kappa$-open set in $Y$ is the ``restriction'' $f^*(U)$ of a $\kappa$-open set $U$ in $X$.
We call such an $f : Y -> X$ an \defn{embedding of $\kappa$-locales}.
By the usual abuse of notation for subobjects (again see \cref{sec:cat-alg}), we also refer to $Y$ itself as a $\kappa$-sublocale of $X$, denoted $Y \subseteq X$, and refer to the embedding $f$ as the ``inclusion'' $Y `-> X$ (the notation $Y \subseteq X$ is consistent with our notation for $\kappa$-Borel sets $B \subseteq X$, by \cref{cvt:loc-sub-pi02} below).
Let
\begin{align*}
\RSub_\!{\kappa Loc}(X) := \{\text{regular subobjects of } X \in \!{\kappa Loc}\} = \{\text{$\kappa$-sublocales of } X\}.
\end{align*}
The terms \defn{$\kappa$-Borel sublocale} and \defn{positive $\kappa$-Borel sublocale} likewise mean regular subobjects in $\!{\kappa BorLoc}, \!{\kappa Bor^+Loc}$ respectively, hence quotients in $\!{\kappa Bool}, \!{\kappa\kappa Frm}$ respectively.
However, due to \cref{thm:bool-epi-surj,thm:bifrm-oepi-surj}, these are the same as ordinary subobjects in $\!{\kappa BorLoc}$, respectively order-embedded subobjects in $\!{\kappa Bor^+Loc}$ (see \cref{sec:cat-ord}, and \cref{thm:loc-mono-epi} below).

Finally, a (positive) $\kappa$-(Borel )locale will be called \defn{$\lambda$-copresented} if its dual algebra is $\lambda$-presented, and the category of all these will be denoted $\!{\kappa (Bor^{(+)})Loc}_\lambda \subseteq \!{\kappa (Bor^{(+)})Loc}$.
This implies that it is $\kappa$-based.
By the dual of \cref{thm:frm-idl-kgen},
\begin{align*}
\!{Loc}_\kappa = \!{\kappa Loc}_\kappa = \!{\lambda Loc}_\kappa \quad\text{for $\lambda \ge \kappa$}.
\end{align*}
We also call a $\kappa$-copresented (positive) $\kappa$-(Borel )locale a \defn{standard (positive) $\kappa$-(Borel ) locale}.
This is motivated by the equivalence between $\sigma$-copresented ($\sigma$-)Borel locales and \defn{standard Borel spaces} (Borel subspaces of $2^\#N$) from classical descriptive set theory (see \cref{sec:loc-ctbpres} below).
Note that by our definition, all (positive) $\infty$-(Borel )locales are standard.

\begin{proposition}
$\kappa$-ary limits and disjoint unions, as well as the forgetful functors in \eqref{diag:loc-cat} from column ``$\kappa$'', preserve $\kappa$-copresentability.
\end{proposition}
\begin{proof}
$\kappa$-ary disjoint unions are by \cref{thm:frm-prod-pres,thm:bool-prod-pres}; the rest are general facts about $\kappa$-presentable objects in the dual categories.
\end{proof}

\subsection{Spatialization}
\label{sec:loc-sp}

Given a topological space $X$, we may forget its underlying set and keep only the frame $\@O(X)$ of open sets, yielding the \defn{underlying locale} of $X$; we treat this as a nameless forgetful functor
\begin{align*}
\!{Top} --> \!{Loc},
\end{align*}
which is formally the same as the functor $\@O : \!{Top} -> \!{Frm}^\op$.

Conversely, given a locale $X$, we may define a canonical topological space from it as follows.
A \defn{point} in $X$ is a continuous map $x : 1 -> X$, or equivalently a frame homomorphism $x^* : \@O(X) -> \@O(1) = 2$.
We may identify points $x : 1 -> X$ with the sets $(x^*)^{-1}(\top) \subseteq \@O(X)$, which are \defn{completely prime filters}, i.e., complements of principal ideals $(x^*)^{-1}(\bot)$, thought of as the neighborhood filter of $x$; the complement (in $\@B_\infty(X)$) of the greatest element of $(x^*)^{-1}(\top)$ can be thought of as the closure $\-x := \-{\{x\}}$, and is an \defn{irreducible} closed set, meaning $\up \-x$ is a prime filter in $\neg \@O(X)$.
The \defn{spatialization} of $X$ is the space of points
\begin{align*}
\Sp(X) := \!{Loc}(1, X)
\end{align*}
equipped with the topology consisting of the open sets
\begin{align*}
\Sp(U) := \{x \in \Sp(X) \mid x^*(U) = \top \in 2\}
\end{align*}
for $U \in \@O(X)$.
In other words, $\Sp : \@O(X) -> \@P(X)$ is easily seen to be a frame homomorphism; $\@O(\Sp(X))$ is by definition its image.

Given a topological space $X$ and a locale $Y$, a continuous map $f : X -> Y$ induces a map $X -> \Sp(Y)$ taking each point $x \in X$ to the composite map of locales $f(x) : 1 --->{x} X --->{f} Y$, given by $f(x)^*(V) = x^*(f^*(V)) = (\top \text{ iff $x \in V$})$ for $V \in \@O(Y)$.
In other words, $f$ is equivalently a frame homomorphism $f^* : \@O(Y) -> \@O(X) \subseteq \@P(X) \cong 2^X$, hence transposes to a map $X -> 2^{\@O(Y)}$ whose image lands in the set of frame homomorphisms, i.e., in $\Sp(Y)$; this is the map induced by $f$.
Conversely, a map $g : X -> \Sp(Y)$ corresponds in this way to a frame homomorphism $(V |-> \{x \in X \mid g(x)^*(V) = \top\} = g^{-1}(\Sp(V))) : \@O(Y) -> \@P(X)$, which lands in $\@O(X)$ iff $g$ is continuous.
This yields a natural bijection
\begin{align*}
\!{Loc}(X, Y) \cong \!{Top}(X, \Sp(Y)),
\end{align*}
i.e., $\Sp : \!{Loc} -> \!{Top}$ is right adjoint to the forgetful functor.
(This is a Stone-type adjunction induced by the ``commuting'' structures of a topology and a frame on $2$; see \cite[VI~\S4]{Jstone}.)

The adjunction unit is given by, for each topological space $X$, the continuous map
\begin{alignat*}{3}
\eta_X : X &--> \Sp(X) &{}\cong{}& \{\text{completely prime filters in } \@O(X)\} &{}\cong{}& \{\text{irreducible closed sets}\} \subseteq \neg \@O(X) \\
x &|--> && \{U \in \@O(X) \mid x \in U\} &{}|->{}& \-x,
\end{alignat*}
and is an isomorphism iff $X$ is \defn{sober}, meaning $T_0$ (whence $\eta_X$ is injective) and every irreducible closed set has a dense point.
Examples include any Hausdorff space (irreducible closed sets being singletons) as well as most naturally occurring non-Hausdorff spaces in ``ordinary'' mathematics, such as spectra of rings or continuous dcpos (see \cite{Jstone}).
The adjunction counit is given by, for each locale $Y$, the map
\begin{align*}
\epsilon_Y : \Sp(Y) &--> Y
\end{align*}
defined by
\begin{align*}
\epsilon_Y^* := \Sp : \@O(Y) --> \@O(\Sp(Y)).
\end{align*}
By definition, $\epsilon_Y^*$ is always surjective, i.e., $\epsilon_Y$ is a sublocale embedding; while $\epsilon_Y^*$ is injective iff $Y$ is \defn{spatial}, meaning it admits enough points to separate open sets.
In particular, the underlying locale of a topological space $X$ is always spatial; hence the adjunction $\!{Top} \rightleftarrows \!{Loc}$ is \defn{idempotent} (see e.g., \cite[VI~4.5]{Jstone}), meaning it factors into the reflection into sober spaces followed by the coreflective inclusion of spatial locales:
\begin{equation*}
\begin{tikzcd}
\!{Top} \rar[shift left=2, two heads, "\Sp"] &
\!{SobTop} \lar[shift left=2, rightarrowtail, right adjoint] \rar[phantom, "\simeq"] &[-2em]
\!{SpLoc} \rar[shift left=2, rightarrowtail] &
\!{Loc} \lar[shift left=2, two heads, right adjoint, "\Sp"]
\end{tikzcd}
\end{equation*}

For each of the other categories of locale-like structures defined in the preceding subsection, dual to a category of lattice-theoretic algebras, we may set up an idempotent adjunction with the corresponding category of spaces equipped with a subalgebra of $\@P(X)$ in exactly the same way; we refer to all of these as right adjoints as \defn{spatialization}.
Thus we have idempotent adjunctions
\begin{equation*}
\begin{tikzcd}
\!{\kappa Top} \rar[shift left=2] &
\!{\kappa Loc}, \lar[shift left=2, right adjoint, "\Sp"] &
\!{\kappa Bor} \rar[shift left=2] &
\!{\kappa BorLoc}, \lar[shift left=2, right adjoint, "\Sp"] &
\!{\kappa Bor^+} \rar[shift left=2] &
\!{\kappa Bor^+Loc} \lar[shift left=2, right adjoint, "\Sp"]
\end{tikzcd}
\end{equation*}
between the respective categories of $\kappa$-topological spaces and $\kappa$-locales, of $\kappa$-Borel spaces and $\kappa$-Borel locales, and of positive $\kappa$-Borel spaces and positive $\kappa$-Borel locales.
These adjunctions fit into the obvious three-dimensional commutative diagram together with \eqref{diag:loc-cat}, one ``slice'' of which is
\begin{equation}
\label{diag:loc-sp}
\begin{tikzcd}
\!{\kappa BorLoc}
    \dar[shift left=2]
    \rar[shift left=2, right adjoint', "\Sp"] &
\!{\kappa Bor}
    \lar[shift left=2]
    \dar[shift left=2]
\\
\!{\kappa Bor^+Loc}
    \uar[shift left=2, right adjoint']
    \dar[shift left=2]
    \rar[shift left=2, right adjoint', "\Sp"] &
\!{\kappa Bor^+}
    \lar[shift left=2]
    \uar[shift left=2, right adjoint']
    \dar[shift left=2]
\\
\!{\kappa Loc}
    \uar[shift left=2, right adjoint']
    \rar[shift left=2, right adjoint', "\Sp"] &
\!{\kappa Top}
    \lar[shift left=2]
    \uar[shift left=2, right adjoint']
\end{tikzcd}
\end{equation}
Commutativity of the right adjoints means that e.g., for a $\kappa$-locale $X$, its $\kappa$-topological spatialization $\Sp(X)$, which is the set of points $\Sp(X) \cong \!{\kappa Frm}(\@O_\kappa(X), 2)$ equipped with the image $\@O_\kappa(\Sp(X))$ of the $\kappa$-frame homomorphism $\Sp : \@O_\kappa(X) -> \@P(\Sp(X))$, has underlying $\kappa$-Borel space given by $\Sp(X)$ equipped with the $\kappa$-Boolean subalgebra of $\@P(\Sp(X))$ generated by $\@O_\kappa(\Sp(X))$, which is the same as $\Sp(X) \cong \!{\kappa Frm}(\@O_\kappa(X), 2) \cong \!{\kappa Bool}(\ang{\@O_\kappa(X) \qua \!{\kappa Frm}}_\!{\kappa Bool}, 2) = \!{\kappa Bool}(\@B_\kappa(X), 2)$ equipped with the image of $\Sp : \@B_\kappa(X) -> \@P(\Sp(X))$, which is the $\kappa$-Borel spatialization of $X$.
Note that this justifies our using the same notation $\Sp$ to denote all of these spatializations.

\begin{remark}
\label{rmk:loc-sp-forget}
However, the horizontal \emph{left} adjoints (the forgetful functors from spaces to locales) in the above diagram do not always commute with the vertical \emph{right} adjoints (the forgetful functors from topological to positive Borel to Borel).
Dually, this means that e.g., given a subframe $\@O(X) \subseteq \@P(X)$, the free complete Boolean algebra $\ang{\@O(X)}_\!{CBOOL}$ it generates need not be the complete Boolean subalgebra of $\@P(X)$ generated by $\@O(X)$.
A counterexample is given by taking $\@O(X) := \ang{\#N}_\!{Frm}$, which is a subframe $\@L(\ang{\#N}_\!{\wedge Lat}) \subseteq \@P(\ang{\#N}_\!{\wedge Lat})$ of a powerset (see \cref{sec:frm-idl}); but $\ang{\@O(X)}_\!{CBOOL} = \ang{\#N}_\!{CBOOL}$ is not a complete Boolean subalgebra of a powerset.

A related issue is that many familiar (say, topological) properties and constructions need not be preserved by the horizontal forgetful functors above (even when starting with, say, sober spaces).
For example, the localic product $\#Q \times \#Q$ is not spatial, hence is not the underlying locale of the topological product (see \cite[II~2.14]{Jstone}).

These issues mean that when we are regarding a space as a locale, we must be careful about \emph{when} we apply the nameless horizontal forgetful functor above.
In general, our convention will be to \textbf{pass to locales as soon as possible}.
For example, ``the $\infty$-Borel locale $\#R$'' will mean the underlying $\infty$-Borel locale of the underlying locale, rather than the (discrete) underlying $\infty$-Borel locale of the underlying $\infty$-Borel space; while $\#Q \times \#Q$ will refer to the localic product.

We now list some well-behaved classes of topological spaces, for which many spatial notions agree with their localic analogs.
Usually, we will only be regarding these kinds of spaces as locales, so that the above issues never arise.
\begin{itemize}

\item  Compact Hausdorff spaces form a full subcategory of $\!{SobTop} \simeq \!{SpLoc} \subseteq \!{Loc}$ which is closed under arbitrary limits in $\!{Loc}$ (see \cite[III~1.6--7]{Jstone}).

\item  Locally compact Hausdorff spaces are closed under finite limits in $\!{Loc}$ (see \cite[II~2.13]{Jstone}).

\item  In particular, identifying sets with discrete spaces, we may regard $\!{Set} \subseteq \!{SobTop}$ as a full subcategory of $\!{Loc}$ closed under finite limits (and arbitrary colimits), the \defn{discrete locales}, dual to complete atomic Boolean algebras $\!{CABool} \subseteq \!{Frm}$ (see \cref{sec:upkzfrm}).

Similarly, we call $\kappa$-ary sets \defn{$\kappa$-discrete locales}, regarded as a full subcategory $\!{Set}_\kappa \subseteq \!{Loc}_\kappa = \!{\kappa Loc}_\kappa$ of standard $\kappa$-locales, dual to $\!{\kappa CABool}_\kappa \subseteq \!{\kappa Frm}_\kappa$ from \cref{sec:upkzfrm}.
Note that $\!{Set}_\kappa$ is also a full subcategory of standard $\kappa$-Borel locales $\!{\kappa BorLoc}_\kappa$.

\item  Completely metrizable spaces are closed under countable limits in $\!{Loc}$ (see \cite[4.1]{Ifun}).

\item  The \defn{Sierpinski space} $\#S = \{0 < 1\}$, with $\{1\}$ open but not closed, is sober; and all topological and localic powers of it agree (with $\@O(\#S^X) = \ang{X}_\!{Frm}$; see \cite[VII~4.9]{Jstone}, \cref{thm:loc-sierpinski}).

\item  More generally, countably copresented subspaces of powers of $\#S$ are well-behaved, and also in the (positive) Borel contexts; see \cref{sec:loc-ctbpres}.

\end{itemize}
\end{remark}

\subsection{The Borel hierarchy}
\label{sec:loc-bor}

Classically, the Borel hierarchy in a metrizable topological space $X$ is defined by letting the $\*\Sigma^0_1$ sets be the open sets,
the $\*\Pi^0_\alpha$ sets be the complements of $\*\Sigma^0_\alpha$ sets for each ordinal $\alpha$,
the $\*\Delta^0_\alpha$ sets be those which are both $\*\Sigma^0_\alpha$ and $\*\Pi^0_\alpha$,
and the $\*\Sigma^0_\alpha$ sets, for $\alpha \ge 2$, be countable unions of sets which are $\*\Pi^0_\beta$ for some $\beta < \alpha$; see \cite[11.B]{Kcdst}.
Thus, for example, a $\*\Sigma^0_2$ set is an $F_\sigma$ set (countable union of closed), while a $\*\Pi^0_2$ set is a $G_\delta$ set (countable intersection of open).
In the non-metrizable setting, a better-behaved definition, due to Selivanov~\cite{Sdom}, takes instead $\*\Sigma^0_\alpha$ sets to be countable unions
\begin{align*}
\bigcup_{i \in \#N} (A_i \setminus B_i) \quad\text{for $A_i, B_i \in \Sigma^0_{\beta_i}$, $\beta_i < \alpha$}.
\end{align*}
It is easily seen by induction that this only makes a difference for $\*\Sigma^0_2$, where the effect is to ensure that $\*\Sigma^0_2$ is the $\sigma$-subframe of $\@P(X)$ generated by the open and closed sets.

Motivated by this definition and the formula \eqref{eq:frm-neg-normform} for elements of $\@N_\kappa$, we define the \defn{$\kappa$-Borel hierarchy} of a $\kappa$-locale $X$ as follows:
\begin{align*}
\kappa\Sigma^0_{1+\alpha}(X) &:= \@N_\kappa^\alpha(\@O_\kappa(X)) \subseteq \@B_\kappa(X), \\
\kappa\Pi^0_{1+\alpha}(X) &:= \neg \@N_\kappa^\alpha(\@O_\kappa(X)) \subseteq \@B_\kappa(X), \\
\kappa\Delta^0_{1+\alpha}(X) &:= \kappa\Sigma^0_{1+\alpha}(X) \cap \kappa\Pi^0_{1+\alpha}(X) = \@N_\kappa^\alpha(\@O_\kappa(X))_\neg \subseteq \@B_\kappa(X).
\end{align*}
We call the elements $B \in \kappa\Sigma^0_\alpha(X)$ the \defn{$\kappa\Sigma^0_\alpha$ sets of $X$}; similarly for the other classes.

\begin{convention}
Classically, boldface $\*\Sigma, \*\Pi, \*\Delta$ is used in order to distinguish from the \emph{effective} Borel hierarchy, for which lightface is reserved (see \cite{Mdst}).
Since we will never consider effective notions in this paper, we have chosen to use lightface, in order to avoid cluttering the notation any more than it already is.
To avoid any possible confusion, we will always write the prefix $\kappa$ when denoting the localic Borel hierarchy (even when $\kappa = \omega_1$ or $\kappa = \infty$).
However, we will continue to use boldface for the classical ($\sigma$-)Borel hierarchy in topological spaces.
\end{convention}

We have the usual picture of the $\kappa$-Borel hierarchy (see \cite[11.B]{Kcdst}):
\begin{equation*}
\begin{tikzcd}[row sep=0.5em,column sep=1em,every arrow/.style={phantom,"\subseteq" sloped}]
& \mathllap{\@O_\kappa ={}} \kappa\Sigma^0_1 \drar & \subseteq & \kappa\Sigma^0_2 \drar & \subseteq & \dotsb & \subseteq & \kappa\Sigma^0_\alpha \drar & \subseteq & \dotsb \\
\mathllap{(\@O_\kappa)_\neg ={}} \kappa\Delta^0_1 \urar \drar && \kappa\Delta^0_2 \urar \drar && \dotsb && \kappa\Delta^0_\alpha \urar \drar && \dotsb && \subseteq \@B_\kappa(X) \\
& \mathllap{\neg \@O_\kappa ={}} \kappa\Pi^0_1 \urar & \subseteq & \kappa\Pi^0_2 \urar & \subseteq & \dotsb & \subseteq & \kappa\Pi^0_\alpha \urar & \subseteq & \dotsb
\end{tikzcd}
\hspace{-3em}
\end{equation*}

The $\sigma$-localic Borel hierarchy is connected to the classical one as follows: for a second-countable ($\sigma$-)topological space $X$, the Borel $\sigma$-algebra is the image in $\@P(X)$ of the Borel $\sigma$-algebra of the underlying $\sigma$-locale of $X$ (see \cref{rmk:loc-sp-forget}); this quotient map restricts to give
\begin{align}
\label{eq:loc-sp-bor}
\*\Sigma^0_\alpha(X) = \im(\sigma\Sigma^0_\alpha(X) -> \@P(X)).
\end{align}
Similarly for $\*\Pi^0_\alpha$.

\begin{remark}
It is possible for the spatial $\*\Sigma^0_\alpha(X)$ to be a nontrivial quotient of the localic $\sigma\Sigma^0_\alpha(X)$ (even for a second-countable $\sigma$-topological $X$).
Take $X = \#Q$; by the localic Baire category theorem (see \cref{thm:loc-baire} below), the join of the closed singletons $\{x\}$ is $< \top$ in $\sigma\Sigma^0_2(X)$, but of course not in $\*\Sigma^0_2(X)$.
In \cref{sec:loc-ctbpres}, we will see a large class of spaces for which $\*\Sigma^0_\alpha(X) \cong \sigma\Sigma^0_\alpha(X)$.
\end{remark}

For a $\kappa$-locale $X$ and $\lambda \ge \kappa$, we may regard $X$ as a $\lambda$-locale, hence obtain a $\lambda$-Borel hierarchy of $X$.
As $\lambda$ varies, these are related as follows:
\begin{equation*}
\begin{tikzcd}[row sep=1em, column sep=0.8em, every arrow/.style={phantom,"\subseteq" sloped}]
\infty\Sigma^0_1(X) \rar &
\infty\Sigma^0_2(X) \rar &
\dotsb \rar &
\infty\Sigma^0_\alpha(X) \rar &
\dotsb \rar &
\infty\Sigma^0_\kappa(X) \rar &
\dotsb \rar &
\infty\Sigma^0_\lambda(X) \rar &
\dotsb \rar &
\@B_\infty(X)
\\
\vdots \uar &
\vdots \uar &
&
\vdots \uar &
&
\vdots \uar &
&
\vdots \uar
\\
\lambda\Sigma^0_1(X) \uar \rar &
\lambda\Sigma^0_2(X) \uar \rar &
\dotsb \rar &
\lambda\Sigma^0_\alpha(X) \uar \rar &
\dotsb \rar &
\lambda\Sigma^0_\kappa(X) \uar \rar &
\dotsb \rar &
\lambda\Sigma^0_\lambda(X) \mathrlap{{}= \@B_\lambda(X)} \uar
\\
\kappa\Sigma^0_1(X) \uar \rar &
\kappa\Sigma^0_2(X) \uar \rar &
\dotsb \rar &
\kappa\Sigma^0_\alpha(X) \uar \rar &
\dotsb \rar &
\kappa\Sigma^0_\kappa(X) \mathrlap{{}= \@B_\kappa(X)} \uar
\end{tikzcd}
\end{equation*}
Note the differing lengths of the rows.
Indeed, if $X$ is the $\sigma$-locale with $\@O(X) = \ang{\#N}_\!{Frm}$ (namely $X = \#S^\#N$; see \cref{thm:loc-sierpinski}), then by the Gaifman--Hales \cref{thm:gaifman-hales}, $\@B_\infty(X) = \ang{\#N}_\!{CBOOL}$ is a proper class, while each $\infty\Sigma^0_{1+\alpha}(X) = \@N^\alpha(\ang{\#N}_\!{Frm})$ is a set; thus

\begin{corollary}[of Gaifman--Hales]
There is a $\sigma$-locale $X$ whose $\infty$-Borel hierarchy is strictly increasing.
\qed
\end{corollary}

For a locale $X$, each $\infty\Sigma^0_\alpha(X) \subseteq \@B_\infty(X)$ forms a small $\bigvee$-sublattice, while each $\infty\Pi^0_\alpha(X) \subseteq \@B_\infty(X)$ forms a small $\bigwedge$-sublattice; thus these inclusions have right/left adjoints, i.e., for any $\infty$-Borel set $B \in \@B_\infty(X)$, we may define its \defn{$\infty\Sigma^0_\alpha$-interior} and \defn{$\infty\Pi^0_\alpha$-closure}
\begin{align*}
B^{\circ\alpha} &:= \max \{A \in \infty\Sigma^0_\alpha(X) \mid A \le B\}, \\
\-B^\alpha &:= \min \{C \in \infty\Pi^0_\alpha(X) \mid B \le C\}.
\end{align*}
Thus $\-B^1 = \-B$, while $\-B^\alpha = B \iff B \in \infty\Pi^0_\alpha(X)$.
We say that $B$ is \defn{$\infty\Pi^0_\alpha$-dense} if $\-B^\alpha = X$.
Note that even if $X$ is a $\kappa$-locale and $B \in \@B_\kappa(X)$, $\-B^\alpha$ need not be in $\kappa\Pi^0_\alpha(X)$.
Similarly, there is no useful spatial analog of these notions for $\alpha \ge 2$.

For a $\kappa$-Borel locale $X$, since each $\kappa\Sigma^0_{1+\alpha}(X)$ is a $\kappa$-frame, we may define a $\kappa$-locale by
\begin{align*}
\@O_\kappa(\@D_\kappa^\alpha(X)) := \kappa\Sigma^0_{1+\alpha}(X) = \@N_\kappa^\alpha(\@O_\kappa(X)).
\end{align*}
This $\kappa$-locale $\@D_\kappa^\alpha(X)$ is called the \defn{$\alpha$th dissolution} of $X$ (a term due to Johnstone~\cite{Jsep}), and is equipped with a $\kappa$-continuous map
\begin{align*}
\epsilon_\kappa : \@D_\kappa^\alpha(X) &--> X,
\end{align*}
corresponding to the unit $\epsilon_\kappa^* : \@O_\kappa(X) `-> \@N_\kappa^\alpha(\@O_\kappa(X))$ (from \cref{sec:frm-neg}),
such that $\epsilon_\kappa$ is also a $\kappa$-Borel isomorphism and such that $\epsilon_\kappa^*(\kappa\Sigma^0_{1+\alpha}(X)) = \@O_\kappa(\@D_\kappa^\alpha(X))$.
In other words, we can think of $\@D_\kappa^\alpha(X)$ as ``$X$ with a finer $\kappa$-topology, inducing the same $\kappa$-Borel structure, in which $\kappa\Sigma^0_{1+\alpha}$ sets in $X$ become open''.

Note that for $\kappa < \infty$, $\@D_\kappa^\kappa(X) = \@D_\kappa^\infty(X)$, which corresponds to $\@N_\kappa^\infty(\@O_\kappa(X)) = \ang{\@O_\kappa(X)}_\!{\kappa Bool}$ but regarded as a $\kappa$-frame via the forgetful functor $\!{\kappa Bool} -> \!{\kappa Frm}$, is the ``free $\kappa$-locale'' on the underlying $\kappa$-Borel locale of $X$ (see \eqref{diag:loc-cat}).
It can thus be thought of as a $\kappa$-localic analog of the discrete topology on a topological space.
(However, the dependence on $\kappa < \infty$ is needed, since the forgetful functor $\!{Loc} -> \!{\infty BorLoc}$ in \eqref{diag:loc-cat} does not have a left adjoint.)

\begin{remark}
\label{rmk:loc-bor-pfrm}
In \cite{BWZpfrm}, the authors construct a variant dissolution which only adds complements for cozero sets of completely regular locales; the Lindelöf case is an instance of the dissolution on $\sigma$-locales (see \cref{rmk:loc-reglind-baire}).
\end{remark}

A basic result in classical descriptive set theory says that Borel sets may be made (cl)open while preserving $\sigma$-copresentability of the topology (see \cite[13.A]{Kcdst} and \cref{sec:loc-ctbpres} below).
The localic analog of this is given by the duals of \cref{thm:frm-bool-pres,thm:frm-neginf-pres}:

\begin{proposition}
\label{thm:loc-bor-loc}
Every standard $\kappa$-Borel locale is the underlying $\kappa$-Borel locale of a standard $\kappa$-locale.
\qed
\end{proposition}

\begin{proposition}
\label{thm:loc-bor-dissolv}
Let $X$ be a standard $\kappa$-locale.
For any $\alpha$ and $\kappa$-ary $\@C \subseteq \kappa\Sigma^0_\alpha(X)$, there is a standard $\kappa$-locale $X'$ equipped with a $\kappa$-continuous $\kappa$-Borel isomorphism $f : X' -> X$ such that, identifying $\@B_\kappa(X) = \@B_\kappa(X')$ via $f^*$, we have $\@C \subseteq \@O_\kappa(X') \subseteq \kappa\Sigma^0_\alpha(X)$.
(Thus we can also ensure $\@C \cap \kappa\Delta^0_\alpha(A) \subseteq \@O_\kappa(X')_\neg$.)
\qed
\end{proposition}

We will henceforth refer to the conclusion of this result by saying that $X'$ is a \defn{partial dissolution} of $X$ making each $C \in \@C$ $\kappa$-open.

We call a $\kappa$-locale $X$ \defn{ultraparacompact}, respectively \defn{zero-dimensional}, if the $\kappa$-frame $\@O_\kappa(X)$ is, as defined in \cref{sec:upkzfrm}.
Recall also \defn{$\kappa$-discrete} locales $\!{Set}_\kappa \subseteq \!{Loc}_\kappa$ from \cref{rmk:loc-sp-forget}.
By \cref{thm:upkzfrm-cabool-dircolim,thm:upkzfrm-kcabool-dircolim},

\begin{proposition}
\label{thm:loc-upkz-invlim}
A (standard) $\kappa$-locale is ultraparacompact zero-dimensional iff it is a ($\kappa$-ary) codirected limit, in $\!{\kappa Loc}$, of $\kappa$-discrete locales.
\qed
\end{proposition}

We now have the following strengthened ``change of topology'' results, dual to \cref{thm:upkzfrm-bool-pres,thm:upkzfrm-neginf-pres}:

\begin{proposition}
\label{thm:loc-bor-upkz}
Every standard $\kappa$-Borel locale is the underlying $\kappa$-Borel locale of an ultraparacompact zero-dimensional standard $\kappa$-locale (hence a $\kappa$-ary codirected limit, in $\!{\kappa BorLoc}$, of $\kappa$-ary sets).
\qed
\end{proposition}

\begin{proposition}
\label{thm:loc-bor-dissolv-upkz}
Let $X$ be a standard $\kappa$-locale.
For any $\alpha \ge 2$ and $\kappa$-ary $\@C \subseteq \kappa\Sigma^0_\alpha(X)$, there is an ultraparacompact zero-dimensional partial dissolution $X' -> X$ such that $\@C \subseteq \@O_\kappa(X') \subseteq \kappa\Sigma^0_\alpha(X)$
(and $\@C \cap \kappa\Delta^0_\alpha(A) \subseteq \@O_\kappa(X')_\neg$).
\qed
\end{proposition}

We also record here the dual of \cref{thm:frm-neginf-upkz}:

\begin{proposition}
For any $\kappa$-locale $X$ and $\alpha \ge 2$, if $C_i \in \kappa\Sigma^0_\alpha(X)$ are $<\kappa$-many $\kappa\Sigma^0_\alpha$ sets which cover $X$, then there are pairwise disjoint $D_i \le C_i$ (whence $D_i \in \kappa\Delta^0_\alpha(X)$) that still cover $X$.
\qed
\end{proposition}

\begin{remark}
This is a weak form of the $\kappa$-ary analog of the \defn{generalized reduction property} which holds classically for $\*\Sigma^0_\alpha$, $\alpha \ge 2$ in metrizable spaces (see \cite[22.16]{Kcdst}): any countable family of $\*\Sigma^0_\alpha$ sets has a pairwise disjoint refinement with the same union.
In other words, every $\*\Sigma^0_\alpha$ set is ultraparacompact in the $\sigma$-frame $\*\Sigma^0_\alpha$.
We do not know if analogously, every $\kappa\Sigma^0_\alpha$ set in every (standard) $\kappa$-locale is ultraparacompact.
\end{remark}

For a $\kappa$-continuous map $f : X -> Y$ between $\kappa$-locales, the $\kappa$-Boolean homomorphism $f^* : \@B_\kappa(Y) -> \@B_\kappa(X)$ (extending the $\kappa$-frame homomorphism $f^* : \@O_\kappa(Y) -> \@O_\kappa(X)$) restricts to $\kappa$-frame homomorphisms $f^* : \kappa\Sigma^0_\alpha(Y) -> \kappa\Sigma^0_\alpha(X)$ for each $\alpha$, i.e., ``preimages of $\kappa\Sigma^0_\alpha$ sets are $\kappa\Sigma^0_\alpha$''.

If $f$ is merely a $\kappa$-Borel map, then for each $V \in \@O_\kappa(Y)$, we have $f^*(V) \in \@B_\kappa(X)$, thus $f^*(V) \in \kappa\Sigma^0_\alpha(X)$ for some $\alpha < \kappa$.
By taking a partial dissolution of $X'$ making $f^*(V)$ $\kappa$-open for all (or only generating) $V \in \@O_\kappa(Y)$ using \cref{thm:loc-bor-dissolv} (or \cref{thm:loc-bor-dissolv-upkz}), we get the following counterpart ``change of topology'' results for maps:

\begin{proposition}
\label{thm:loc-mor-dissolv}
Let $f : X -> Y$ be a $\kappa$-Borel map between $\kappa$-locales.
Then there is a partial dissolution $X' -> X$ such that $f$ lifts to a $\kappa$-continuous map $f' : X' -> Y$.
Moreover:
\begin{enumerate}
\item[(a)]  if $X, Y$ are standard, we can ensure $X'$ is standard;
\item[(b)]  if $f^*(\@O_\kappa(Y)) \subseteq \kappa\Sigma^0_\alpha(X)$, we can ensure $\@O_\kappa(X') \subseteq \kappa\Sigma^0_\alpha(X)$;
\item[(c)]  for $\alpha \ge 2$, we can ensure $X'$ is ultraparacompact zero-dimensional.
\qed
\end{enumerate}
\end{proposition}

\begin{remark}
Classically, a map $f : X -> Y$ between topological spaces is \defn{Baire class $\alpha$} if it is in the $\alpha$th iterated closure of the continuous maps under pointwise sequential limits, while a \defn{Baire map} is a map which is Baire class $\alpha$ for some $\alpha < \omega_1$.
For sufficiently nice spaces $X, Y$, $f$ is Baire class $\alpha$ iff $f^*(\@O(Y)) \subseteq \*\Sigma^0_{\alpha+1}(X)$, hence Baire iff it is Borel; see \cite[24.3]{Kcdst}.

Our notion of $\kappa$-Borel map, as well as the above result, is thus related to Ball's~\cite{Bbaire} notion of Baire maps from a locale to $\#R$, in the same way as in \cref{rmk:loc-bor-pfrm}: Ball's notion in the regular Lindelöf case is an instance of our notion when $\kappa = \omega_1$.
\end{remark}

\begin{remark}
\label{rmk:loc-bor-localization}
We note the following abstract categorical consequence of the above (which is somewhat implicit in \cite{MMfrmepi}): the forgetful functor $\!{\kappa Loc} -> \!{\kappa BorLoc}$ exhibits $\!{\kappa BorLoc}$ as a \defn{localization} of $\!{\kappa Loc}$, i.e., $\!{\kappa BorLoc}$ is obtained from $\!{\kappa Loc}$ by freely adjoining inverses for some morphisms.
See e.g., \cite[Ch.~5]{Bcat}.
This follows (see e.g., \cite[p.~486]{MPlfp}) from the fact that $\!{\kappa Loc} -> \!{\kappa BorLoc}$ preserves finite limits, and that every $\kappa$-Borel map $f : X -> Y$ with $Y \in \!{\kappa Loc}$ factors as $X \cong X' -> Y$ for some $\kappa$-locale $X'$ which is $\kappa$-Borel isomorphic to $X$ and some $\kappa$-continuous map $X' -> Y$, by \cref{thm:loc-bor-loc,thm:loc-mor-dissolv}.
Moreover, we can take $X'$ standard if $X, Y$ both are; thus $\!{\kappa BorLoc}_\kappa$ is also a localization of $\!{\kappa Loc}_\kappa$.

Once we know we have a localization, we can of course describe $\!{\kappa BorLoc}_{(\kappa)}$ as the localization of $\!{\kappa Loc}_{(\kappa)}$ which adjoins an inverse for all those $\kappa$-continuous $f : X -> Y \in \!{\kappa Loc}_{(\kappa)}$ which do in fact become invertible in $\!{\kappa BorLoc}_{(\kappa)}$; these are exactly the partial dissolutions.
In other words,
\begin{quote}
``(Standard) $\kappa$-Borel locales are what's left of (standard) $\kappa$-locales after $\kappa$-continuous bijections are declared to be isomorphisms.''
\end{quote}
(Note that ``bijection'' can be interpreted as either $\kappa$-Borel or $\infty$-Borel isomorphism, by conservativity of $\!{\kappa BorLoc} -> \!{\infty BorLoc}$ (\cref{thm:loc-forget}).)
\end{remark}

\subsection{Borel images and sublocales}
\label{sec:loc-im}

We now give a detailed analysis of various notions of ``image'' for maps between locales.

For an $\infty$-Borel map $f : X -> Y$ between $\infty$-Borel locales, and any $\infty$-Borel set $B \in \@B_\infty(X)$, we define the \defn{$\infty$-Borel image} of $B$ under $f$ to be
\begin{align*}
f(B) := f^{\@B_\infty}(B) := \min \{C \in \@B_\infty(Y) \mid B \le f^*(C)\}, \text{ if it exists}.
\end{align*}
Thus, $f : \@B_\infty(X) \rightharpoonup \@B_\infty(Y)$ is the partial left adjoint to $f^* : \@B_\infty(Y) -> \@B_\infty(X)$.
More generally, for any class of $\infty$-Borel sets $\Gamma(Y) \subseteq \@B_\infty(Y)$, we define the \defn{$\Gamma$-image} of $B$ under $f$ to be
\begin{align*}
f^\Gamma(B) := \min \{C \in \Gamma(Y) \mid B \le f^*(C)\}, \text{ if it exists}.
\end{align*}
We also define the \defn{$\Gamma$-image}, and in particular the \defn{$\infty$-Borel image}, of $f$ to be
\begin{align*}
\im^\Gamma(f) &:= f^\Gamma(X), \\
\im(f) &:= f(X) = \im^{\@B_\infty}(f).
\end{align*}
(In \cref{thm:sigma11-proper} we will show that the $\infty$-Borel image need not exist.)

If $\Gamma(Y) \subseteq \@B_\infty(Y)$ is a small $\bigwedge$-sublattice, then the $\Gamma$-image always exists (by the adjoint functor theorem).
For example, if $Y$ is a locale, the \defn{$\infty\Pi^0_\alpha$-image} exists for every $\alpha$; we denote it by
\begin{align*}
\-f^\alpha(B) &:= f^{\infty\Pi^0_\alpha}(B) = \bigwedge \{C \in \infty\Pi^0_\alpha(Y) \mid B \le f^*(C)\}, \\
\-\im^\alpha(f) &:= \-f^\alpha(X) = \bigwedge \{C \in \infty\Pi^0_\alpha(y) \mid f^*(C) = \top\}.
\end{align*}
There is a close relationship between $\infty\Pi^0_\alpha$-image and $\infty\Pi^0_\alpha$-closure (\cref{sec:loc-bor}): the latter is the former under the identity map.
Conversely, if the $\infty$-Borel image exists, then its $\infty\Pi^0_\alpha$-closure gives the $\infty\Pi^0_\alpha$-image, hence our notation.
(This will generalize to all $\infty\Pi^0_\alpha$-images, once we define the notion of ``$\infty\Sigma^1_1$-image'' by adjoining nonexistent $\infty$-Borel images in \cref{sec:sigma11-cat} below.)

\begin{remark}
\label{rmk:loc-im-pi02}
In the locale theory literature (see e.g., \cite[C1.2]{Jeleph}), the ``image'' of a continuous locale map usually refers to the image sublocale, i.e., what we are calling the ``$\infty\Pi^0_2$-image'' (see \cref{cvt:loc-sub-pi02} below).
It is well-known that this notion of ``image'' behaves poorly in many ways, e.g., it is not pullback-stable (see \cite[C1.2.12]{Jeleph} and \cref{rmk:loc-epi-pullback}).
This poor behavior is explained by the decomposition of ``$\infty\Pi^0_2$-image'' as ``$\infty\Sigma^1_1$-image'' (which is well-behaved; see \cref{thm:sigma11-cat}) followed by $\infty\Pi^0_2$-closure (which one does not expect to behave like an ``image'').

To avoid confusion, we will henceforth avoid saying ``image'' without a prefix.
\end{remark}

Clearly, if $\Gamma \subseteq \Gamma'$, and the $\Gamma'$-image exists and is in $\Gamma$, then it is also the $\Gamma$-image.
Conversely, if $\Gamma$ generates $\Gamma'$ under arbitrary meets, then any $\Gamma$-image is also the $\Gamma'$-image; for example, if $Y$ is a $\kappa$-locale, then a $\kappa$-closed image exists iff a closed image exists and is a $\kappa$-closed set.
If the $\Gamma$-image and $\Gamma'$-image both exist, then the former is always bigger; e.g.,
\begin{align*}
\-f(B) := \-f^1(B) \ge \-f^2(B) \ge \dotsb \ge \-f^\alpha(B) \ge \dotsb \quad (\ge f(B), \text{ if it exists}).
\end{align*}
Also, if $\Gamma$ is a directed union of subclasses $\Gamma_i$, then the $\Gamma$-image is the eventual value of the $\Gamma_i$-images, if the latter eventually all exist and stabilize; e.g.,
\begin{align*}
f(B)
&= \min_{\alpha < \infty} \-f^\alpha(B)
= \text{eventual value of } \-f^\alpha(B) \text{, if it exists}.
\end{align*}
If $f : X -> Y$ with $Y$ a $\kappa$-Borel locale, we denote the \defn{$\kappa$-Borel image} by
\begin{align*}
f^\kappa(B) &:= f^{\@B_\kappa}(B) = \min \{C \in \@B_\kappa(Y) \mid B \le f^*(C)\};
\end{align*}
thus for $\kappa < \infty$,
\begin{align*}
f(B)
&= \text{eventual value of } f^\lambda(B) \text{ for $\kappa \le \lambda < \infty$, if it exists}.
\end{align*}
(The above formulas are implicit in the work of Wilson~\cite{Wasm} and Madden--Molitor~\cite{MMfrmepi}.)

The $\kappa$-Borel image has a simple dual algebraic meaning:

\begin{proposition}
\label{thm:loc-im-alg}
Let $f : X -> Y$ be an $\infty$-Borel map from an $\infty$-Borel locale to a $\kappa$-Borel locale, where $\kappa \le \infty$.
Then the $\kappa$-Borel image $f^\kappa(B)$ of some $B \in \@B_\infty(X)$ exists iff the composite of $f^* : \@B_\kappa(Y) -> \@B_\infty(X)$ followed by the principal filterquotient $B \wedge (-) : \@B_\infty(X) ->> \down B$ has kernel which is a $\kappa$-generated $\kappa$-Boolean congruence on $\@B_\kappa(Y)$, hence a principal filter congruence $\sim^C$ for some $C \in \@B_\kappa(Y)$, in which case $f^\kappa(B) = C$.

Thus, the $\kappa$-Borel image $\im^\kappa(f) := f^\kappa(X)$ of $f$ exists iff $f^* : \@B_\kappa(Y) -> \@B_\infty(X)$ has kernel which is a $\kappa$-generated congruence, hence a principal filter congruence $\sim^{\im^\kappa(f)}$.
\end{proposition}
\begin{proof}
This is a direct unravelling of the definitions: $f^\kappa(B) = C$ iff $C \in \@B_\kappa(Y)$ is least such that $B \le f^*(C)$, i.e., $B \wedge f^*(C) = B = \top_{\down B}$, which means precisely that $\ker(B \wedge f^*(-) : \@B_\kappa(Y) -> \down B) = {\sim^C}$ by the general description of $\kappa$-Boolean algebra congruences (end of \cref{sec:frm-quot}).
\end{proof}

Recall from the end of \cref{sec:loc-cat} that by an \defn{embedding} in one of our categories of locales, we mean a regular monomorphism, i.e., a map $f$ such that $f^*$ is a surjective homomorphism.

\begin{corollary}
\label{thm:loc-emb-im}
Let $f : X -> Y$ be an embedding in (i) $\!{\kappa Loc}$, (ii) $\!{\kappa Bor^+Loc}$, or (iii) $\!{\kappa BorLoc}$.
Then the $\infty$-Borel image $\im(f)$ exists, and is given by the respective formulas
\begin{align*}
\tag{i}
\im(f) &= \bigwedge \{B -> C \mid B, C \in \@O_\kappa(Y) \AND f^*(B) \le f^*(C)\}, \\
\tag{ii}
\im(f) &= \bigwedge \{B -> C \mid B, C \in \@B^+_\kappa(Y) \AND f^*(B) \le f^*(C)\}, \\
\tag{iii}
\im(f) &= \bigwedge \{B \in \@B_\kappa(Y) \mid f^*(B) = \top\},
\end{align*}
where the meets may be restricted to any generating set for the (order-)kernel of $f^*$.

If $X, Y$ are both standard, then the meets may be restricted to be $\kappa$-ary; thus $\im(f) \in \@B_\kappa(Y)$, and so $\im(f)$ is also the $\kappa$-Borel image $\im^\kappa(f)$.
\end{corollary}
(Note that when $\kappa = \infty$, these formulas still make sense, because we always assume (positive) $\infty$-Borel locales are standard.)
\begin{proof}
Since $f^*$ is surjective, so is its complete Boolean extension $\@B_\infty(Y) -> \@B_\infty(X)$, with kernel generated by any generators of the kernel of the original $f^*$.
In case (i), for example, $f^* : \@O_\kappa(Y) ->> \@O_\kappa(X)$ is the quotient of $\@O_\kappa(Y)$ by the relations $B \le C$ for all (generating) pairs $(B, C)$ in the order-kernel of $f^*$, hence $\@B_\infty(Y) ->> \@B_\infty(X)$ is the quotient by these same relations, which (again by the general description of Boolean congruences from \cref{sec:frm-quot}) equivalently means $\@B_\infty(X)$ is, up to isomorphism, the principal filterquotient $\@B_\infty(Y) ->> \down \bigwedge_{f^*(B) \le f^*(C)} (B -> C)$, which by \cref{thm:loc-im-alg} means $\im(f) = \bigwedge_{f^*(B) \le f^*(C)} (B -> C)$, which yields the formula (i).
The cases (ii) and (iii) are similar.

If $X, Y$ are both standard, then $f^*$ is a quotient map between $\kappa$-presented algebras, hence has $\kappa$-generated kernel (see \cref{thm:cat-alg-cong-pres}).
\end{proof}

\begin{remark}
\label{rmk:loc-mono-epi-std}
The notions of monomorphism, regular monomorphism, and epimorphism in the categories of standard objects $\!{\kappa (Bor^{(+)})Loc}_\kappa$ are the restrictions of those in their respective parent categories $\!{\kappa (Bor^{(+)})Loc}$; this is a general algebraic fact about their dual categories (see \cref{thm:cat-alg-mono-epi-pres}).
Thus, there is no ambiguity when we say ``standard $\kappa$-sublocale'', for example.
\end{remark}

Let us define the following \emph{ad hoc} classes of sets: for a (i) $\kappa$-locale, (ii) positive $\kappa$-Borel locale, or (iii) $\kappa$-Borel locale $Y$, respectively,
\begin{align*}
\tag{i}
\infty(\kappa\Pi^0_2)_\delta(Y)
&:= \text{closure of $\kappa\Pi^0_2(Y) \subseteq \@B_\infty(Y)$ under arbitrary $\bigwedge$} \\
&= \{\bigwedge_{i \in I} (B_i -> C_i) \mid I \in \!{Set} \AND B_i, C_i \in \@O_\kappa(Y)\}, \\
\tag{ii}
\lambda I(\@B^+_\kappa)_\delta(Y)
&:= \{\bigwedge_{i \in I} (B_i -> C_i) \mid I \in \!{Set}_\lambda \AND B_i, C_i \in \@B^+_\kappa(Y)\}, \\
\tag{iii}
\infty(\@B_\kappa)_\delta(Y)
&:= \{\bigwedge_{i \in I} B_i \mid I \in \!{Set} \AND B_i \in \@B_\kappa(Y)\}.
\end{align*}
(The ``$I$'' in (ii) stands for ``implication''.)
The formulas in \cref{thm:loc-emb-im} clearly imply

\begin{corollary}
\label{thm:loc-emb-im-cls}
In the situations of \cref{thm:loc-emb-im}, $\im(f)$ is the
(i)~$\infty(\kappa\Pi^0_2)_\delta$-image (hence also $\infty\Pi^0_2$-image),
(ii)~$\infty I(\@B^+_\kappa)_\delta$-image, or
(iii)~$\infty(\@B_\kappa)_\delta$-image of $f$, respectively.

If $X, Y$ are standard, then $\im(f)$ is also the
(i)~$\kappa\Pi^0_2$-image,
(ii)~$\kappa I(\@B^+_\kappa)_\delta$-image, or
(iii)~$\kappa$-Borel image of $f$, respectively.
\qed
\end{corollary}

The images in these different categories are related as follows.
Since the forgetful functors $\!{\kappa Loc} -> \!{\kappa Bor^+Loc} -> \!{\kappa BorLoc}$ and $\!{\kappa (Bor^{(+)})Loc} -> \!{\lambda (Bor^{(+)})Loc}$ ($\kappa \le \lambda$) from diagram \eqref{diag:loc-cat} preserve limits, they preserve regular monomorphisms as well as pullbacks, and so restrict to $\bigwedge$-lattice homomorphisms on the regular subobject lattices of each object in the domain.
Since the functors are also faithful (\cref{thm:loc-forget}), these $\bigwedge$-lattice homomorphisms are injective (see \cref{thm:cat-funct-sub}).
Thus for $X \in \!{\kappa Loc}$, say, we have embeddings
\begin{equation}
\label{diag:loc-sub}
\begin{tikzcd}[every arrow/.append style={hook}]
\RSub_\!{\kappa BorLoc}(X) \rar &
\RSub_\!{\lambda BorLoc}(X) \rar &
\dotsb \rar &
\RSub_\!{\infty BorLoc}(X)
\\
\RSub_\!{\kappa Bor^+Loc}(X) \uar \rar &
\RSub_\!{\lambda Bor^+Loc}(X) \uar \rar &
\dotsb \rar &
\RSub_\!{\infty Bor^+Loc}(X) \uar
\\
\RSub_\!{\kappa Loc}(X) \uar \rar &
\RSub_\!{\lambda Loc}(X) \uar \rar &
\dotsb \rar &
\RSub_\!{Loc}(X) \uar
\end{tikzcd}
\end{equation}
The dual of $\RSub_\!{\kappa Loc}(X) `-> \RSub_\!{\infty BorLoc}(X)$, say, takes a $\kappa$-frame quotient $f^* : \@O_\kappa(X) ->> \@O_\kappa(Y)$ to the induced complete Boolean quotient $f^* : \@B_\infty(X) ->> \@B_\infty(Y)$.
Thus, all of these embeddings commute with taking the $\infty$-Borel image (by our definition of $\infty$-Borel image).

The following is also immediate from \cref{thm:loc-im-alg}:

\begin{corollary}
\label{thm:loc-emb-im-inj}
Each $\im : \RSub_\!{\infty BorLoc}(X) -> \@B_\infty(X)$ is a $\bigwedge$-lattice embedding.

Thus, each $\im : \RSub_\!{\kappa (Bor^{(+)})Loc}(X) -> \@B_\infty(X)$ is a $\bigwedge$-lattice embedding.

Thus, each $\im : \RSub_{\!{\kappa (Bor^{(+)})Loc}_\kappa}(X) -> \@B_\kappa(X)$ is a $\kappa$-$\bigwedge$-lattice embedding.
\end{corollary}
\begin{proof}
Clearly, a principal filterquotient of $\@B_\infty(X)$ determines a unique element of $\@B_\infty(X)$.
Meets in $\RSub_\!{\infty BorLoc}(X)$ are given by pullback, i.e., pushout of the corresponding principal filterquotients of $\@B_\infty(X)$, which corresponds to meet of the corresponding elements of $\@B_\infty(X)$.
\end{proof}

We now consider the inverses of $\im$.
Given, say, a $\kappa$-locale $X$, and an $\infty$-Borel set $B \in \@B_\infty(X)$, we define the \defn{induced $\kappa$-sublocale} $Y \subseteq X$ on $B$ by taking $\@O_\kappa(Y)$ to be the image of the inclusion $\@O_\kappa(X) `-> \@B_\infty(X)$ composed with the principal filterquotient $\@B_\infty(X) ->> \down B$.
In other words, $\@O_\kappa(Y)$ is the quotient of $\@O_\kappa(X)$ by the $\kappa$-frame congruence ``identifying sets which are the same in $B$'':
\begin{align*}
U \sim^B V  \coloniff  U \cap B = V \cap B \in \@B_\infty(X).
\end{align*}
Similarly, an $\infty$-Borel set in a (positive) $\kappa$-Borel locale induces a (positive) $\kappa$-Borel sublocale.

\begin{remark}
When $\kappa = \infty$, clearly small-presentability $\@B_\infty(X)$ implies small-presentability of $\down B$; thus the induced $\infty$-Borel sublocale on $B$ makes sense.

For the induced positive $\infty$-Borel sublocale, we likewise need to require small-presentability of the image of the composite $\@B^+_\infty(X) `-> \@B_\infty(X) ->> \down B$.
We do not know if this is automatic.
\end{remark}

The connection with $\infty$-Borel images is given by

\begin{proposition}
\label{thm:loc-sub-im}
Let $X \in $
(i)~$\!{\kappa Loc}$,
(ii)~$\!{\kappa Bor^+Loc}$, or
(iii)~$\!{\kappa BorLoc}$.
\begin{enumerate}

\item[(a)]  We have order-isomorphisms
\begin{equation*}
\begin{tikzcd}[row sep=1em]
\text{(i)} &
\RSub_\!{\kappa Loc}(X)
    \rar[shift left, "\im"] &[2in]
\infty(\kappa\Pi^0_2)_\delta(X),
    \lar[shift left,"{\textup{induced $\kappa$-sublocale}}"] &
~ \\
\text{(ii)} &
\RSub_\!{\kappa Bor^+Loc}(X)
    \rar[shift left, "\im"] &[2in]
\infty I(\@B^+_\kappa)_\delta(X),
    \lar[shift left,"{\textup{induced positive $\kappa$-Borel sublocale}}"] &
~ \\
\text{(iii)} &
\RSub_\!{\kappa BorLoc}(X)
    \rar[shift left, "\im"] &[2in]
\infty(\@B_\kappa)_\delta(X).
    \lar[shift left,"{\textup{induced $\kappa$-Borel sublocale}}"] &
~
\end{tikzcd}
\end{equation*}

\item[(b)]  For every other $B \in \@B_\infty(X)$, the induced (positive) $\kappa$-(Borel )sublocale on $B$ corresponds, via (a), to the closure of $B$ in the respective class $\Gamma \subseteq \@B_\infty$ on the right-hand side of (a).

\item[(c)]  If $X$ is standard, then the order-isomorphisms in (a) restrict to, respectively,
\begin{align*}
\tag{i}
\RSub_{\!{\kappa Loc}_\kappa}(X) &\cong \kappa\Pi^0_2(X), \\
\tag{ii}
\RSub_{\!{\kappa Bor^+Loc}_\kappa}(X) &\cong \kappa I(\@B^+_\kappa)_\delta(X), \\
\tag{iii}
\RSub_{\!{\kappa BorLoc}_\kappa}(X) &\cong \@B_\kappa(X).
\end{align*}

\end{enumerate}
\end{proposition}
\begin{proof}
We only do case (i); cases (ii) and (iii) are completely analogous.

(a)  By \cref{thm:loc-emb-im-cls,thm:loc-emb-im-inj} it remains only to check that for every $B = \bigwedge_i (U_i -> V_i) \in \infty(\kappa\Pi^0_2)_\delta(X)$ where $U_i, V_i \in \@O_\kappa(X)$, letting $f : Y `-> X$ be the induced $\kappa$-sublocale on $B$, we have $B = \im(f)$.
By definition, $\@O_\kappa(Y)$ is the quotient of $\@O_\kappa(X)$ by the kernel of $\@O_\kappa(X) `-> \@B_\infty(X) ->> \down B$.
By the general description of Boolean congruences (\cref{sec:frm-quot}), $\down B$ is the quotient of $\@B_\infty(X)$ by the relations $U_i \le V_i$, hence is also the free complete Boolean algebra generated by the quotient $Q$ of $\@O_\kappa(X)$ by these same relations, i.e., $\@O_\kappa(X) `-> \@B_\infty(X) ->> \down B$ also factors as $\@O_\kappa(X) ->> Q `-> \ang{Q}_\!{CBOOL} \cong \down B$, hence has kernel generated by the relations $U_i \le V_i$.
By \cref{thm:loc-emb-im}, this means $\im(f) = \bigwedge_i (U_i -> V_i) = B$.

(b)  By definition of induced $\kappa$-sublocale and \cref{thm:loc-emb-im}, the induced $\kappa$-sublocale $f : Y `-> X$ on $B$ has
\begin{align*}
\im(f)
&= \bigwedge \{U -> V \mid U, V \in \@O_\kappa(X) \AND U \wedge B \le V \le B\} \\
&= \bigwedge \{U -> V \mid U, V \in \@O_\kappa(X) \AND B \le U -> V\},
\end{align*}
which is clearly the $\infty(\kappa\Pi^0_2)_\delta$-closure of $B$.

(c) is clear from \cref{thm:loc-emb-im} and the above proof of (a).
\end{proof}

\begin{remark}
\label{rmk:loc-sub-im}
\Cref{thm:loc-sub-im}(a)(i) is a restatement of \cref{thm:frm-neg-cong}, in which the two maps are composed with the order-reversing bijections $\neg : \kappa\@I(\@N_\kappa(\@O_\kappa(X))) = \ang{\@N_\kappa(\@O_\kappa(X))}_\!{Frm} \cong \infty(\kappa\Pi^0_2)_\delta(X)$ and $\RSub_\!{\kappa Loc}(X) \cong \{\text{order-congruences on } \@O_\kappa(X)\}$.
Thus, the above yields another proof of that result, which is somewhat closer in spirit to the original proof of Madden \cite[5.1]{Mkfrm} rather than the concrete computation via posites in \cref{thm:frm-neg-cong}.
\end{remark}

\begin{convention}
\label{cvt:loc-sub-pi02}
From now on, \textbf{we identify regular subobjects in $\!{\kappa Loc}, \!{\kappa Bor^+Loc}, \!{\kappa BorLoc}$ with their $\infty$-Borel images} via \cref{thm:loc-sub-im}(a).
Thus, we also regard the embeddings in \eqref{diag:loc-sub} as inclusions.
For a $\kappa$-locale $X$, say, by \cref{thm:loc-sub-im}(b), the induced $\kappa$-sublocale on $B \in \@B_\infty(X)$ becomes simply its $\infty(\kappa\Pi^0_2)_\delta$-closure (hence its $\infty\Pi^0_2$-closure if $X$ is $\kappa$-based).

For an $\infty$-Borel sublocale $Y \subseteq X$, i.e., $Y \in \@B_\infty(X)$, \textbf{we also identify $\infty$-Borel sets in $Y$ with $\infty$-Borel subsets of $Y$ in $X$}, i.e., we treat $\@B_\infty(Y) \cong \down Y \subseteq \@B_\infty(X)$ as an inclusion.
Thus, the inclusion $f : Y `-> X$ is given by $f^* = Y \wedge (-) : \@B_\infty(X) ->> \down Y$.
Note that if $Y \subseteq X$ is a $\kappa$-sublocale, say, then $\@O_\kappa(Y)$ becomes identified with a subset of $\infty(\kappa\Pi^0_2)_\delta(X)$: namely, $V \in \@O_\kappa(Y)$ is identified with $Y \wedge U$ for any $U \in \@O_\kappa(X)$ such that $V = f^*(U)$.

(This is incompatible with another commonly used convention in locale theory, where for a sublocale $f : Y `-> X$, one instead identifies $V \in \@O(Y)$ with the \emph{greatest} $U \in \@O(X)$ such that $V = f^*(U)$; see e.g., \cite[II~2.3]{Jstone}, \cite[III~2.1]{PPloc}.)
\end{convention}

We next note that the above correspondence between regular subobjects and their image sets interacts well with other operations on either side.
We have already noted in \cref{thm:loc-emb-im-inj} that meet (pullback) of regular subobjects corresponds to meet of their images.

\begin{proposition}
\label{thm:loc-sub-im-union}
\leavevmode
\begin{enumerate}

\item[(a)]  For
\begin{enumerate}
\item[(i)]  $X \in \!{\kappa Loc}$ and $Y_i \in \@O_\kappa(X)$,
\item[(ii)]  $X \in \!{\kappa Bor^+Loc}$ and $Y_i \in \@B^+_\kappa(X)$, or
\item[(iii)]  $X \in \!{\kappa BorLoc}$ and $Y_i \in \@B_\kappa(X)$
\end{enumerate}
such that the $Y_i$ are $<\kappa$-many and pairwise disjoint, regarding them as regular subobjects $Y_i \subseteq X$ via \cref{thm:loc-sub-im}(a),
the disjoint union (i.e., coproduct) $\bigsqcup_i Y_i$ with the induced map $\bigsqcup_i Y_i -> X$ is a regular subobject, and corresponds to the join $\bigvee_i Y_i \in \@B_\infty(X)$.

\item[(b)]  For any $<\kappa$-many objects $Y_i$ in
(i)~$\!{\kappa Loc}$,
(ii)~$\!{\kappa Bor^+Loc}$, or
(iii)~$\!{\kappa BorLoc}$,
their disjoint union $\bigsqcup_i Y_i$ is such that the cocone maps $\iota_i : Y_i -> \bigsqcup_j Y_j$ exhibit the $Y_i \subseteq \bigsqcup_j Y_j$ as pairwise disjoint regular subobjects, belonging to the respective class $\Gamma \subseteq \@B_\infty$ in (a), with union $\bigsqcup_j Y_j$.

\end{enumerate}
\end{proposition}
\begin{proof}
This follows from the description of products in the dual algebraic categories from \cref{sec:frm-prod}.
For example, in case (i), the assumption in (a) says the $Y_i \in \@O_\kappa(X)$ are pairwise disjoint, hence descend to a partition of $\top$ in the principal filterquotient $\down \bigvee_i Y_i$, which is thus isomorphic to $\prod_i \down Y_i = \@O_\kappa(\bigsqcup_i Y_i)$; the induced map $f : \bigsqcup_i Y_i -> X$ corresponds to the quotient map $f^* : \@O_\kappa(X) ->> \down \bigvee_i Y_i$, hence $f$ is a $\kappa$-locale embedding with $\infty$-Borel image $\bigvee_i Y_i$ by \cref{thm:loc-im-alg}.
For (b), the $\iota_i : Y_i -> \bigsqcup_j Y_j$ are dual to the product projections $\iota_i^* : \@O_\kappa(\bigsqcup_j Y_j) = \prod_j \@O_\kappa(Y_j) -> \@O_\kappa(Y_i)$, which are isomorphic to the principal filterquotients $\prod_j \@O_\kappa(Y_j) ->> \down \delta_i$ where $\delta_i \in \prod_j \@O_\kappa(Y_j)$ are as in \cref{sec:frm-prod}; this easily yields the conclusions in (b).
Cases (ii) and (iii) are similar.
\end{proof}

\begin{remark}
Since $<\kappa$-many $\kappa$-Borel sets may be disjointified, \cref{thm:loc-sub-im-union}(a) allows the expression of arbitrary $\kappa$-ary joins in $\@B_\kappa$ in terms of coproducts of $\kappa$-Borel sublocales.
\end{remark}

\begin{proposition}
\label{thm:loc-im-factor}
Let $f : X -> Y$ be a morphism in
(i)~$\!{\kappa Loc}$,
(ii)~$\!{\kappa Bor^+Loc}$, or
(iii)~$\!{\kappa BorLoc}$.
\begin{enumerate}
\item[(a)]  For a regular subobject $Z \subseteq Y$, identified with $Z \in \@B_\infty(Y)$ via \cref{thm:loc-sub-im}(a), $f^*(Z) \in \@B_\infty(X)$ corresponds to the pullback of $Z `-> Y$ along $f$.
\item[(b)]  Thus, $f^*(Z) = X$ iff $f$ factors through $Z `-> Y$.
\item[(c)]  The factorization $X ->> Z `-> Y$ of $f$ into an epimorphism followed by a regular monomorphism is given by the
(i)~$\infty(\kappa\Pi^0_2)_\delta$-image,
(ii)~$\infty I(\@B^+_\kappa)_\delta$-image, or
(iii)~$\infty(\@B_\kappa)_\delta$-image of $f$, with either existing if the other does (which is always if $\kappa < \infty$ or in case (i)).
\end{enumerate}
\end{proposition}
\begin{proof}
(a)  If $Z = \bigwedge_i (U_i -> V_i)$, i.e., $\@O_\kappa(Z)$ (say) is the quotient of $\@O_\kappa(Y)$ by the relations $U_i \le V_i$, then $f^*(Z) = \bigwedge_i (f^*(U_i) -> f^*(V_i))$, i.e., $\@O_\kappa(f^*(Z))$ is the quotient of $\@O_\kappa(X)$ by the relations $f^*(U_i) \le f^*(V_i)$, which exactly describes the pushout $\@O_\kappa(X) \otimes_{\@O_\kappa(Y)} \@O_\kappa(Z)$.

(b) follows since $f$ factors through a subobject $Z \subseteq Y$ iff its pullback along $f$ is all of $X$.

(c) follows since $Z$ is by definition the least regular subobject through which $f$ factors.
\end{proof}

\begin{remark}
\label{rmk:loc-im-bad}
In \cref{thm:sigma11-proper} we will show that there is a single $\sigma$-Borel map $f : X -> Y$ between standard $\sigma$-Borel locales which does \emph{not} have a $\kappa$-Borel image for any $\kappa$.

It follows from this that the $\infty(\@B_\sigma)_\delta$-image $Z$ of $f$ (which always exists) is \emph{not} in $\@B_\sigma(Y)$ (or else it would be the $\sigma$-Borel image).
In other words, there is a morphism $f : X -> Y$ in the category $\!{\sigma Bor}_\sigma$ of standard $\sigma$-Borel locales, whose epi--regular mono factorization $X ->> Z `-> Y$ is \emph{not} standard.

Moreover, since $Z \subseteq Y$ is itself $\infty$-Borel, hence $\lambda$-Borel for some $\omega_1 < \lambda < \infty$ (namely $\lambda = (2^{\aleph_0})^+$, since $\abs{\@B_\sigma(Y)} \le 2^{\aleph_0}$ and so every meet of $\sigma$-Borel sets is $(2^{\aleph_0})^+$-ary), $Z$ cannot also be the $\infty(\@B_\lambda)_\delta$-image of $f$, or else it would be the $\lambda$-Borel image; thus the $\infty(\@B_\lambda)_\delta$-image must be strictly smaller.
In other words, we have $\kappa = \omega_1 < \lambda < \infty$ such that the forgetful functor $\!{\kappa BorLoc} -> \!{\lambda BorLoc}$ does \emph{not} preserve epi--regular mono factorizations, or equivalently does not preserve epimorphisms (since it does preserve regular monos).
See \cref{thm:loc-im-bad}.
\end{remark}

For an $\infty$-Borel map $f : X -> Y$ and a class of sets $\Gamma$, we call $f$ \defn{$\Gamma$-surjective} if $Y = \top \in \@B_\infty(Y)$ is the $\Gamma$-image of $f$.
(Traditionally, a ``surjective'' continuous map between locales usually means a ``$\infty\Pi^0_2$-surjective'' map in our terminology; see \cref{rmk:loc-im-pi02}.)
By \cref{thm:loc-im-factor}(c), and the fact that the middle classes below are the closures of the left classes under arbitrary meets,
\begin{align}
\label{eq:loc-epi-surj}
\begin{alignedat}{3}
\text{$\kappa\Pi^0_2$-surjective}
&={}& \text{$\infty(\kappa\Pi^0_2)_\delta$-surjective}
&= \text{epimorphism in } \!{\kappa Loc}, \\
\text{$\kappa I(\@B^+_\kappa)_\delta$-surjective}
&={}& \text{$\infty I(\@B^+_\kappa)_\delta$-surjective}
&= \text{epimorphism in } \!{\kappa Bor^+Loc}, \\
\text{$\kappa$-Borel surjective}
&={}& \text{$\infty(\@B_\kappa)_\delta$-surjective}
&= \text{epimorphism in } \!{\kappa BorLoc}.
\end{alignedat}
\end{align}
In this terminology, the epi--regular mono factorization in $\!{\kappa Loc}$, say, consists of a $\kappa\Pi^0_2$-surjection followed by a $\kappa$-locale embedding.

By \cref{thm:bool-pushout-inj}, we have

\begin{proposition}
\label{thm:loc-bor-epi-pullback}
Epimorphisms, hence epi--regular mono factorizations that exist, in $\!{\kappa BorLoc}$ are pullback-stable.
\qed
\end{proposition}

\begin{remark}
\label{rmk:loc-epi-pullback}
The analogous statement in $\!{\kappa Loc}$ is well-known to fail; likewise, the analogous statement in $\!{\kappa Bor^+Loc}$ fails.
Let $X = [0, 1]$ with the Scott topology, where the open sets are $(r, 1]$ for each $r \in [0, 1]$; this is a quasi-Polish space (see \cref{sec:loc-ctbpres}).
Let $Y = [0, 1]$ with the discrete topology (which is $\kappa$-based for $\kappa = (2^{\aleph_0})^+$).
The continuous identity map $f : Y -> X$ is given by the inclusion $f^* : \@O(X) = \@O([0, 1]) `-> \@P([0, 1]) = \@O(Y)$, whence $f$ is an epimorphism of ($\kappa$-)locales.
Moreover, by \cref{ex:loc-ctbpres-dirint} below, $\@B^+_\infty(X) = \@B^+_\sigma(X)$ consists precisely of the open and closed intervals $(r, 1]$ and $[r, 1]$; thus $f^* : \@B^+_\infty(X) -> \@P([0, 1]) = \@B^+_\infty(X)$ is still injective, whence $f$ is also an epimorphism of positive $\kappa$-Borel locales.
Also, the image of $f^*$ generates $\@P([0, 1])$ as a complete Boolean algebra, i.e., $f^* : \@B_\infty(X) -> \@P([0, 1]) = \@B_\infty(Y)$ is surjective, whence $f$ is an $\infty$-Borel locale embedding.

The $\infty$-Borel image $f(Y) \subseteq X$ is the join of singletons $\bigvee_{x \in [0, 1]} \{x\} \in \kappa\Sigma^0_2(X)$, which is $< \top \in \@B_\infty(X)$ by localic Baire category (\cref{thm:loc-baire}, applied to $[0, 1]$ with the \emph{Euclidean} topology which is $\sigma$-Borel isomorphic to $X$) or Gaifman--Hales (\cref{thm:gaifman-hales}, which together with \cref{rmk:loc-ctbpres-perfect} implies that $f^* : \@B_\infty(X) \cong \ang{\#N}_\!{CBOOL} ->> \@P([0, 1]) = \@B_\infty(Y)$ cannot be injective).
Let $Z = \neg f(Y) \in \kappa\Pi^0_2(X)$; thus $Z \subseteq X$ is a $\kappa$-sublocale, hence also a positive $\kappa$-Borel sublocale, which is nonempty.
But the pullback of $f$ along $Z `-> X$ is empty (since it corresponds to intersection of $f(Y)$ and $Z$ by \cref{thm:loc-emb-im-inj}), hence not an epimorphism in either $\!{\kappa Loc}$ or $\!{\kappa Bor^+Loc}$.
\end{remark}

Indeed, we have the following characterization of pullback-stable epimorphisms in $\!{\kappa Loc}, \!{\kappa Bor^+Loc}$, due to Wilson~\cite[28.5]{Wasm} in the case of $\!{Loc}$:

\begin{proposition}
\label{thm:loc-epi-pullback}
A morphism $f : X -> Y$ in (i)~$\!{\kappa Loc}$ or (ii)~$\!{\kappa Bor^+Loc}$, or their full subcategories of standard objects, is a pullback-stable epimorphism iff it becomes an epimorphism in $\!{\kappa BorLoc}$.
\end{proposition}
\begin{proof}
($\Longleftarrow$)
If $f$ becomes an epimorphism in $\!{\kappa BorLoc}$, then since the forgetful functors $\!{\kappa Loc} -> \!{\kappa Bor^+Loc} -> \!{\kappa BorLoc}$ preserve pullbacks, every pullback of $f$ becomes an epimorphism in $\!{\kappa BorLoc}$, hence is already an epimorphism, since the forgetful functors are also faithful.

($\Longrightarrow$) (i)
Suppose $f$ is a pullback-stable epimorphism in $\!{\kappa Loc}$ or $\!{\kappa Loc}_\kappa$; we must show it is an epimorphism in $\!{\kappa BorLoc}$, i.e., $\kappa$-Borel surjective.
Let $B \in \@B_\kappa(Y)$ such that $f^*(B) = \top$; we must show $B = \top$.
By \cref{thm:loc-bor-dissolv}, let $Y' -> Y$ be a partial dissolution with $B \in \@O_\kappa(Y')$, and with $Y'$ standard if $Y$ is.
Let $f' : X' -> Y'$ be the pullback of $f$ along $Y' -> Y$.
\begin{equation*}
\begin{tikzcd}
X' \dar["f'"'] \rar & X \dar["f"] \\
Y' \rar & Y
\end{tikzcd}
\end{equation*}
Then $f'$ is an epimorphism in $\!{\kappa Loc}_{(\kappa)}$, i.e., $\kappa\Pi^0_2$-surjective.
Since $Y' -> Y$ is a $\kappa$-Borel isomorphism and $\!{\kappa Loc} -> \!{\kappa BorLoc}$ preserves pullbacks, $X' -> X$ is also a $\kappa$-Borel isomorphism, i.e., we may identify $\@B_\kappa(X') = \@B_\kappa(X)$.
So since $f^{\prime*}(B) = f^*(B) = \top \in \@B_\kappa(X) = \@B_\kappa(X')$ and $B \subseteq Y'$ is open, $B = \top \in \@B_\kappa(Y') = \@B_\kappa(Y)$, as desired.

The proof of ($\Longrightarrow$) (ii) is analogous, using that we may ``dissolve'' a positive $\kappa$-Borel locale $X$ to make any $\kappa$-Borel set $B \in \@B^+_\kappa(X)$ positive, dually by taking $\ang{\@B^+_\kappa(X)}_\!{\kappa Bool}$, and in the standard case, passing to a $\kappa$-presented subalgebra containing $\@B^+_\kappa(X)$ and $B$ using \cref{thm:cat-lim-refl} as in the proof of \cref{thm:frm-neginf-pres}.
\end{proof}

\begin{remark}
Thus far, we have mostly focused on images of maps $f : X -> Y$; but we can reduce images of arbitrary $\infty$-Borel sets $B \subseteq X$ under $f$ to images of maps.
Indeed, $B$ itself corresponds to an $\infty$-Borel sublocale $g : B `-> X$; hence the $\Gamma$-image of $B$ under $f$, for any $\Gamma$, is the same as the $\Gamma$-image of $f \circ g$.
Furthermore:
\begin{enumerate}
\item[(a)]  if $X$ is a $\kappa$-Borel locale and $B \in \@B_\kappa(X)$, then $B$ is a $\kappa$-Borel sublocale of $X$, which is standard if $X$ is (by \cref{thm:loc-sub-im});
\item[(b)]  if $X$ is furthermore a $\kappa$-locale, then by \cref{thm:loc-bor-dissolv} we can take a partial dissolution $X' -> X$, which is standard if $X$ is, such that $B \in \@O_\kappa(X')$, whence the composite $g : B `-> X' -> X$ is a $\kappa$-continuous monomorphism with $\infty$-Borel image $B$.
\end{enumerate}
\end{remark}

Applying this reduction to \cref{thm:loc-bor-epi-pullback} yields the following ``Beck--Chevalley equation'':

\begin{corollary}
\label{thm:loc-bor-im-pullback}
Let
\begin{equation*}
\begin{tikzcd}
X \times_Z Y \dar["g'"'] \rar["f'"] & Y \dar["g"] \\
X \rar["f"'] & Z
\end{tikzcd}
\end{equation*}
be a pullback square in $\!{\kappa BorLoc}$.
For any $\infty(\@B_\kappa)_\delta$-set $B \subseteq X$ such that $f^{\infty(\@B_\kappa)_\delta}(B)$ exists, we have
\begin{align*}
f^{\prime\infty(\@B_\kappa)\delta}(g^{\prime*}(B)) = g^*(f^{\infty(\@B_\kappa)_\delta}(B)).
\end{align*}
Thus if $B \subseteq X$ is $\kappa$-Borel and the $\kappa$-Borel image $f^\kappa(B) \subseteq Z$ exists, then so does
\begin{align*}
f^{\prime\kappa}(g^{\prime*}(B)) = g^*(f^\kappa(B)).
\end{align*}
\end{corollary}
\begin{proof}
The first equation follows from \cref{thm:loc-bor-epi-pullback} applied to the epi--regular mono factorization of $B `-> X --->{f} Z$; the second follows because if $f^\kappa(B)$ exists, it must be $f^{\infty(\@B_\kappa)_\delta}(B)$.
\end{proof}

Another important aspect of epimorphisms in $\!{\kappa BorLoc}$ is their nice behavior with respect to inverse limits:

\begin{proposition}
\label{thm:loc-bor-epi-invlim}
Let $(X_i)_{i \in I}$ be a codirected diagram in $\!{\kappa BorLoc}$, for a directed poset $I$, with morphisms $f_{ij} : X_i -> X_j$ for $i \ge j$.
If each $f_{ij}$ is an epimorphism, then so is each limit projection $\pi_i : \projlim_j X_j -> X_i$.
\end{proposition}
\begin{proof}
By \cref{thm:bool-colim-dir-ker}.
\end{proof}

\begin{corollary}
\label{thm:loc-bor-epi-prod}
Epimorphisms in $\!{\kappa BorLoc}$ are closed under arbitrary products.
\end{corollary}
\begin{proof}
Closure under finite products follows from pullback-stability (\cref{thm:loc-bor-epi-pullback}).
For an arbitrary family of epimorphisms $(f_i : X_i ->> Y_i)_{i \in I}$, we can write $\prod_i X_i$ as the limit of all products of the form $\prod_{i \in F} X_i \times \prod_{i \not\in F} Y_i$ for finite $F \subseteq I$, where for $F \subseteq G \subseteq I$ we have the map $\prod_{i \in G} X_i \times \prod_{i \not\in G} Y_i -> \prod_{i \in F} X_i \times \prod_{i \not\in F} Y_i$ which is the product of the $f_i$ for $i \in G \setminus F$ and the identity on other coordinates; each such map is a finite product of epimorphisms, hence an epimorphism.
Thus by \cref{thm:loc-bor-epi-invlim}, all of the limit projections $\prod_i X_i -> \prod_{i \in F} X_i \times \prod_{i \not\in F} Y_i$ are epimorphisms; taking $F = \emptyset$ yields $\prod_i f_i : \prod_i X_i -> \prod_i Y_i$.
\end{proof}

\begin{remark}
\label{rmk:loc-bor-epi-invlim}
Note that even though \cref{thm:loc-bor-epi-prod} holds in $\!{Set}$ by the axiom of choice, \cref{thm:loc-bor-epi-invlim} only holds for \emph{countable} inverse limits (consider the tree of bounded well-ordered sequences of rationals, whose empty set of branches is the inverse limit of its $\omega_1$ levels).
Thus, this is an example where locales are \emph{better-behaved} than topological spaces.
(The corresponding fact for epimorphisms in $\!{Loc}$, i.e., $\infty\Pi^0_2$-surjective maps, is standard; see \cite[IV~4.2]{JTgpd}.)
\end{remark}

For the sake of completeness, we also record the following direct translations of results from \cref{sec:frm-interp}.
Recall that the \defn{kernel} of a morphism $f : X -> Y$ in an arbitrary category is its pullback $\ker(f) := X \times_Y X$ with itself; $f$ is a monomorphism iff $\ker(f) \subseteq X \times X$ is the diagonal $X `-> X \times X$, and is a regular epimorphism iff it is the coequalizer of its kernel.
Similarly, the \defn{order-kernel} of $f$ in a locally ordered category is the comma object $\oker(f) := X \downarrow_Y X$ with itself; $f$ is an \defn{order-monomorphism} iff $\oker(f) \subseteq X \times X$ is the \defn{(internal) order ${\le_X}$ on $X$}, which is the universal object equipped with projections $\pi_1, \pi_2 : {\le_X} \rightrightarrows X$ obeying $\pi_1 \le \pi_2$; and $f$ is an \defn{order-regular epimorphism} iff it is the coinserter of its order-kernel.
An equivalence class of monomorphisms to $X$ is a \defn{subobject} of $X$, while an equivalence class of order-monomorphisms is a \defn{order-embedded subobject}.
See \cref{sec:cat-alg,sec:cat-ord} for details.

\begin{proposition}
\label{thm:loc-mono-epi}
In $\!{\kappa Bor^{(+)}Loc}$,
\begin{enumerate}
\item[(a)]  every (order-)monomorphism is regular, i.e., (order-embedded) subobject = regular subobject;
\item[(b)]  every epimorphism is (order-)regular.
\end{enumerate}
\end{proposition}
\begin{proof}
By \cref{thm:bool-mono-reg,thm:bool-epi-surj,thm:bifrm-mono-oreg,thm:bifrm-oepi-surj}.
\end{proof}

\begin{remark}
\label{rmk:lusin-suslin}
Classically, the \defn{Lusin--Suslin theorem} says that every injective Borel map $f : X -> Y$ between standard Borel spaces has Borel image; see \cite[15.1]{Kcdst}.
Since injectivity is trivially equivalent to being a monomorphism, \cref{thm:loc-mono-epi}(a) in the case of $\!{\kappa BorLoc}$, together with \cref{thm:loc-emb-im}(iii), can be seen as a Lusin--Suslin theorem for $\kappa$-Borel locales: the former result says that every ``injective'' map is an embedding, while the latter says that every embedding has a $\kappa$-Borel image (which is also the $\infty$-Borel image).
Moreover, the proofs of \cref{thm:bool-epi-surj} by LaGrange~\cite{Lamalg} and \cref{thm:loc-emb-im}(iii) can together be seen as an algebraic analog of a proof of the classical Lusin--Suslin theorem (see \cite[3.6]{Cborin} for an elaboration of this point).

Likewise, \cref{thm:loc-mono-epi}(a) for $\!{\kappa Bor^+Loc}$ says that every positive $\kappa$-Borel map $f : X -> Y$ which is ``internally an order-embedding'' is in fact an embedding of positive $\kappa$-Borel locales.
Recalling from \cref{ex:loc-infbor+} that we regard positive $\kappa$-Borel locales as generalizations of the specialization preorder, this result can be seen as an analog of the following classical fact: given a Borel monotone map $f : X -> Y$ between quasi-Polish spaces, every open $U \subseteq X$ is the preimage of a positive Borel $B \subseteq Y$ (i.e., a set in the $(\sigma, \sigma)$-subframe generated by $\@O(Y) \subseteq \@P(Y)$).

(The analogous statement for $\!{\kappa Loc}$ is false by \cref{rmk:bifrm-oepi-surj}: the identity $\abs{\-{\#N}} -> \-{\#N}$ to the one-point compactification $\-{\#N}$ of $\#N$ from its discrete space is an order-embedding but not an embedding.)

\Cref{thm:loc-mono-epi}(b) corresponds to the following classical fact and its ordered analog: given a Borel map $f : X -> Y$ between standard Borel spaces, if $f^{-1}(B) = \emptyset \implies B = \emptyset$ for all Borel $B \subseteq Y$ (which clearly just means $f$ is surjective), then every Borel $A \subseteq X$ which is invariant under the equivalence relation $\ker(f) \subseteq X^2$ (i.e., the preimages of $A$ under the two projections $\ker(f) \rightrightarrows A$ are the same) is $f^{-1}(B)$ for a unique Borel $B \subseteq Y$ (namely $B = f(A)$, which is Borel by the Lusin separation theorem \cite[14.7]{Kcdst} which corresponds to \cref{thm:bifrm-interp}; see also \cref{thm:bor-blackwell}).
\end{remark}

\begin{remark}
Since the forgetful functors in \eqref{diag:loc-cat} preserve limits and are faithful, they preserve and reflect (plain) monomorphisms.
Thus, for example, a $\kappa$-continuous map $f : X -> Y \in \!{\kappa Loc}$ is a monomorphism iff it is a $\kappa$-Borel monomorphism in $\!{\kappa BorLoc}$, hence a $\kappa$-Borel embedding, in which case it has an $\infty(\kappa\Pi^0_2)_\delta$-image $f(X) \subseteq Y$.
In particular, any monomorphism of locales (hence $\kappa$-locales for sufficiently large $\kappa$) has an $\infty$-Borel image.
This fact was implicitly shown by Madden--Molitor~\cite{MMfrmepi}, also ultimately using LaGrange's interpolation theorem (\cref{thm:bool-epi-surj}).

Recall that a finite-limit-preserving functor is conservative iff it restricts to an injection on each subobject $\wedge$-lattice (see \cref{thm:cat-funct-sub}).
Since the horizontal forgetful functors in \eqref{diag:loc-cat} are conservative (\cref{thm:loc-forget}), they are thus injective on subobject $\bigwedge$-lattices.
For $\!{\kappa BorLoc} -> \!{\lambda BorLoc}$, this is equivalent to injectivity on regular subobjects (\eqref{diag:loc-sub}, in turn equivalent to faithfulness) by \cref{thm:loc-mono-epi}(a); but for $\!{\kappa Bor^+Loc} -> \!{\lambda Bor^+Loc}$, the latter result only gives injectivity on order-embedded subobjects.
\end{remark}

\begin{remark}
\label{rmk:loc-sub-sp}
Finally, we record the interaction of sublocales with spatialization.
The right adjoint functor $\Sp : \!{\kappa Loc} -> \!{\kappa Top}$ (say) preserves limits, hence preserves regular monomorphisms, which are (certain) $\kappa$-subspace embeddings in $\!{\kappa Top}$;
thus each $\kappa$-sublocale $Y \subseteq X$ yields a $\kappa$-subspace $\Sp(Y) \subseteq \Sp(X)$.
Points $x \in \Sp(Y)$ are homomorphisms $x^* : \@O_\kappa(Y) -> 2$, hence homomorphisms $\@O_\kappa(X) -> 2$ which respect the kernel of $\@O_\kappa(X) ->> \@O_\kappa(Y)$.
If $Y$ corresponds via \cref{thm:loc-sub-im} to an $\infty(\kappa\Pi^0_2)_\delta$-set
\begin{align*}
Y = \bigwedge_i (U_i -> V_i) \subseteq X,
\end{align*}
where $U_i, V_i \in \@O_\kappa(X)$, then the (order-)kernel of $\@O_\kappa(X) ->> \@O_\kappa(Y)$ is generated by the relations $U_i \le V_i$; thus $x \in \Sp(X)$ is in $\Sp(Y)$ iff $x^*(U_i) \le x^*(V_i)$, i.e., $x \in \Sp(U_i) \implies x \in \Sp(V_i)$ (recall \cref{sec:loc-sp}) for each $i$, i.e., $x \in \bigcap_i (\Sp(U_i) -> \Sp(V_i))$; that is, we have
\begin{align*}
\Sp(Y) = \bigcap_i (\Sp(U_i) -> \Sp(V_i)) \subseteq \Sp(X).
\end{align*}
In other words, letting $\infty(\kappa\*\Pi^0_2)_\delta(\Sp(X))$ consist of all sets in $\Sp(X)$ of the form $\bigcap_i (U_i -> V_i)$ for $\kappa$-open $U_i, V_i \subseteq \Sp(X)$, so that $\infty(\kappa\*\Pi^0_2)_\delta(\Sp(X))$ is the image of $\infty(\kappa\Pi^0_2)_\delta(X)$ under $\@B_\infty(X) ->> \@B_\infty(\Sp(X)) = \@P(\Sp(X))$ as in \eqref{eq:loc-sp-bor}, we have that the following square commutes:
\begin{equation}
\begin{tikzcd}
\infty(\kappa\Pi^0_2)_\delta(X) \dar[two heads] \rar[phantom, "\cong"] &[-1em] \RSub_\kappa(X) = \{\text{$\kappa$-sublocales of } X\} \dar["\Sp"] \\
\infty(\kappa\*\Pi^0_2)_\delta(\Sp(X)) \rar[phantom, "\subseteq"] & \@P(\Sp(X)) = \{\text{$\kappa$-subspaces of } X\}
\end{tikzcd}
\end{equation}
We clearly have analogs of this for $\!{\kappa Bor^{(+)}Loc}$, and for the full subcategories of standard objects.
\end{remark}

\subsection{Standard $\sigma$-locales}
\label{sec:loc-ctbpres}

The standard objects in our main categories $\!{\kappa (Bor^{(+)})Loc}$, when $\kappa = \omega_1$, form subcategories which are particularly convenient when working with concrete examples: they are all spatial, corresponding to well-behaved classes of spaces; and moreover, other localic notions (such as products, sublocales, etc.), under suitable countability restrictions, also correspond to their spatial analogs.
In this subsection, we review these correspondences, and use them to give several key examples.

Classical descriptive set theory takes place in the topological context of \defn{Polish spaces}, i.e., second-countable completely metrizable spaces, and the Borel context of \defn{standard Borel spaces}, i.e., ($\sigma$-)Borel subspaces of Cantor space $2^\#N$; see \cite{Kcdst}.
Relatively recently, de~Brecht~\cite{dBqpol} introduced \defn{quasi-Polish spaces}, which are $\*\Pi^0_2$ (in the sense of Selivanov; see \cref{sec:loc-bor}) subspaces of countable powers of Sierpinski space $\#S^\#N$, and showed that they are a robust generalization of Polish spaces to the non-$T_1$ context, with suitable analogs of most well-known descriptive set-theoretic properties of Polish spaces.
Heckmann~\cite{Hqpol} proved that the spatialization adjunction $\!{Top} \rightleftarrows \!{Loc}$ restricts to an equivalence between the full subcategories of quasi-Polish spaces and countably copresented locales (or equivalently in our terminology, standard $\sigma$-locales).
As we noted in the Introduction, this result, and the consequent analogies it yields between descriptive set theory and locale theory, is one of the primary motivations behind the present paper.

As Heckmann~\cite[\S4.1]{Hqpol} points out, the key ingredient in his result, namely that every countably copresented locale is spatial, has a much longer and more disguised history.
It was shown by Fourman--Grayson~\cite[3.12]{FGformal}, in the related context of \emph{formal topology}.
If one takes for granted certain purely formal order-theoretic constructions, e.g., the injectivity of $A -> \ang{A}_\!{CBOOL}$ for a frame $A$ (shown already by Isbell~\cite[1.3, 1.5]{Iloc}), then the result can arguably be traced back even further; e.g., it follows from
the Loomis--Sikorski representation theorem for $\sigma$-Boolean algebras,
which is itself equivalent to the spatiality of standard $\sigma$-Borel locales (see \cite[4.1]{Cborin}).
We may complete the triad by adding the spatiality of standard positive $\sigma$-Borel locales (see below).
All of these results easily imply each other, and essentially boil down to a (classical) Baire category argument.
In the following, we have arbitrarily chosen to take the Loomis--Sikorski theorem as fundamental, and to view the other results as consequences.

\begin{theorem}[Loomis--Sikorski]
\label{thm:loomis-sikorski}
Every countably presented $\sigma$-Boolean algebra $A$ admits enough homomorphisms to $2$ to separate its elements.
\end{theorem}
\begin{proof}
See \cite[29.1]{Sbool} or \cite[4.1]{Cborin}.  (Note that since every countably presented algebra is a principal filterquotient of a free algebra, it suffices to assume $A$ is free.)
\end{proof}

\begin{corollary}
Every countably presented ($\sigma$-)frame or $(\sigma, \sigma)$-frame $A$ admits enough homomorphisms to $2$ to separate its elements.
\end{corollary}
\begin{proof}
Follows from injectivity of the units $A -> \ang{A}_\!{\sigma Bool}$ (\cref{thm:frm-neginf-inj,thm:frm-dpoly-bifrm-inj}).
\end{proof}

\begin{corollary}
\label{thm:loc-ctbpres-sp}
Every standard (positive) $\sigma$-(Borel )locale $X$ is spatial.
\qed
\end{corollary}

It follows that the spatialization adjunctions from \cref{sec:loc-sp}, when $\kappa = \omega_1$, restrict to equivalences between $\!{\sigma (Bor^{(+)})Loc}_\sigma$ and full subcategories of the corresponding categories of spaces.
In order to identify these full subcategories of spaces, we use the following well-known fact (see e.g., \cite[VII~4.9]{Jstone}):

\begin{proposition}[folklore]
\label{thm:loc-sierpinski}
For any set $X$, the power $\#S^X$ of Sierpinski space is sober, and corresponds to the locale with $\@O(\#S^X) = \ang{X}_\!{Frm}$ where the generators $x \in X \subseteq \ang{X}_\!{Frm}$ are identified with the subbasic open sets in the product $\#S^X$.
\end{proposition}
\begin{proof}
The free frame on one generator is clearly $\ang{1}_\!{Frm} \cong \@O(\#S) = \{\emptyset < \{1\} < \#S\}$ where $\{1\}$ is the generator; and $\#S$ is clearly sober.
It follows that the right adjoint spatialization functor $\Sp : \!{Loc} -> \!{Top}$ takes the localic power $\#S^X$, corresponding to the frame coproduct $\@O(\#S)^{{\bigotimes}X} \cong \ang{X}_\!{Frm}$, to the spatial power $\#S^X$, whence the latter is sober.
But since $\ang{X}_\!{Frm} = \@L(\ang{X}_\!{\wedge Lat}) \subseteq \@P(\ang{X}_\!{\wedge Lat}) \cong 2^{\ang{X}_\!{\wedge Lat}}$ (see \cref{sec:frm-idl}) is a subframe of a power of $2$, it clearly admits enough homomorphisms to $2$ (the projections), i.e., the localic power $\#S^X$ is spatial, hence agrees with the spatial power.
Under this correspondence, each generator $x \in X$ corresponds to the coproduct injection $\iota_x : \@O(\#S) -> \@O(\#S)^{{\bigotimes}X}$, hence to the product projection $\pi_x : \#S^X -> \#S$, hence to the $x$th subbasic open set in $\#S^X$.
\end{proof}

\begin{corollary}
\label{thm:loc-bor-sierpinski}
For any countable set $X$, the countable power $\#S^X$ in $\!{\sigma (Bor^{(+)})Loc}$ is spatial, and corresponds to the spatial power of $\#S$ in $\!{(\sigma)Top}, \!{(\sigma)Bor^+}, \!{(\sigma)Bor}$ respectively.
\end{corollary}
\begin{proof}
Spatiality is by \cref{thm:loc-ctbpres-sp}; the rest follows from chasing the product $\#S^X \in \!{\sigma Loc}$ through the various right adjoint forgetful and spatialization functors \eqref{diag:loc-sp}.
\end{proof}

For a second-countable ($\sigma$-)topological space $X$ to have a spatial underlying $\sigma$-Borel locale (of the underlying $\sigma$-locale) means that the localic $\sigma$-Borel algebra of $X$ (which is \emph{a priori} $\ang{\@O_\sigma(X)}_\!{\sigma Bool}$) agrees with the spatial Borel $\sigma$-algebra; this implies (see \eqref{eq:loc-sp-bor}) that we have levelwise agreement of the ($\sigma$-)Borel hierarchies, i.e.,
\begin{align}
\label{eq:loc-ctbpres-bor}
\sigma\Sigma^0_\alpha(X) \cong \*\Sigma^0_\alpha(X) \quad \forall \alpha < \omega_1.
\end{align}
Similarly, if a positive Borel space $X$ has a spatial underlying $\sigma$-Borel locale, then countable Boolean combinations of positive Borel sets in $X$, as subsets, are in bijection with the respective combinations of positive Borel sets in the underlying $\sigma$-Borel locale; for example, if $\sigma I(\@B_\sigma)_\delta$ is the class of countable meets of implications between positive $\sigma$-Borel sets defined above \cref{thm:loc-emb-im-cls},
\begin{align}
\label{eq:loc-ctbpres-bor+}
\sigma I(\@B_\sigma)_\delta(X) \cong \{\bigcap_{i \in \#N} (B_i -> C_i) \mid B_i, C_i \in \@B^+_\sigma(X)\} \subseteq \@P(X).
\end{align}

Now for a standard $\sigma$-locale $X$, say, $\@O(X)$ is a countably presented quotient of some countably generated free $\sigma$-frame $\@O(\#S^Y)$ with $Y$ countable, i.e., $X \subseteq \#S^Y$ is a standard $\sigma$-sublocale, hence a $\sigma\Pi^0_2$ set in $\#S^Y$ by \cref{thm:loc-sub-im}(c), hence a $\*\Pi^0_2$ set in $\#S^Y$ by \eqref{eq:loc-ctbpres-bor} and \cref{rmk:loc-sub-sp}, i.e., a quasi-Polish space.
Similarly, a standard $\sigma$-Borel locale $X$ corresponds to a Borel set in $\#S^Y$ (which is Borel isomorphic to $2^Y$), i.e., a standard Borel space;
while a standard positive $\sigma$-Borel locale $X$ corresponds to a set of the form given by the right-hand side of \eqref{eq:loc-ctbpres-bor+} in some $X = \#S^Y$.
Call positive Borel spaces of the latter form \defn{standard positive Borel spaces}.
Let
\begin{align*}
\!{QPol} \subseteq \!{(\sigma)Top}, &&
\!{SBor^+} \subseteq \!{Bor^+}, &&
\!{SBor} \subseteq \!{Bor}
\end{align*}
denote the categories of quasi-Polish spaces, standard positive Borel spaces, and standard Borel spaces, respectively.
We have thus recovered the equivalences of categories mentioned above:

\begin{theorem}
\label{thm:loc-ctbpres}
The spatialization adjunctions from \cref{sec:loc-sp} restrict to equivalences
\begin{align*}
&\mathrlap{\textup{(Heckmann)}}& \hspace{1in}
\!{QPol} &\simeq \!{(\sigma)Loc}_\sigma, \hspace{1in} && \\
&& \!{SBor^+} &\simeq \!{\sigma Bor^+Loc}_\sigma, && \\
&\mathrlap{\textup{(Loomis--Sikorski)}}&
\!{SBor} &\simeq \!{\sigma BorLoc}_\sigma. &&
\qed
\end{align*}
\end{theorem}

In the rest of this subsection, we record some other key notions whose spatial and localic versions agree via the above equivalences, as well as some important examples:

\begin{remark}
\label{rmk:loc-ctbpres-bor}
As noted above \eqref{eq:loc-ctbpres-bor}, for a quasi-Polish space $X$, the localic and spatial $\sigma$-Borel hierarchy agree levelwise.

For example, this tells us that in $\#R$, a countable intersection of open sets $\bigcap_{n \in \#N} U_n$ is also a countable meet in $\@B_\infty(\#R)$ (which is not true for uncountable intersections; see \cref{ex:loc-baire-perfect}).

Also, standard $\sigma$-sublocales of $X$ correspond to localic $\sigma\Pi^0_2$ sets (\cref{thm:loc-sub-im}), hence to spatial $\*\Pi^0_2$ sets, hence to quasi-Polish subspaces (see \cite[Th.~23]{dBqpol}, or \cite[4.2]{Cqpol} for a proof via a spatial analog \cite[4.1]{Cqpol} of \cref{thm:cat-alg-cong-pres} used in the proof of \cref{thm:loc-emb-im}).
\end{remark}

\begin{remark}
Of course, all existing categorical limits and colimits transfer across the equivalences in \cref{thm:loc-ctbpres}.
In particular, this applies to $\kappa$-ary limits and $\kappa$-ary coproducts, which on the left-hand side are the usual limit and disjoint union of spaces, while on the right-hand side are the localic notions in $\!{\sigma (Bor^{(+)})Loc}$ (since $\!{\sigma (Bor^{(+)})Loc}_\sigma \subseteq \!{\sigma (Bor^{(+)})Loc}$ is closed under these operations).
\end{remark}

\begin{remark}
Also, (regular) monomorphisms and epimorphisms transfer across \cref{thm:loc-ctbpres}.
\begin{itemize}
\item
Monomorphisms are clearly injective maps in all three spatial categories, which hence correspond to monomorphisms in the localic categories.

\item
Regular monomorphisms $f : X -> Y$ are embeddings in all three spatial categories; this can be seen \emph{via} \cref{thm:loc-ctbpres}, since $f^*$ being surjective says precisely that each open/(positive) Borel set in $X$ is the restriction of such a set in $Y$.

\item
Epimorphisms in $\!{SBor}$ are easily seen to be surjective Borel maps (given a non-surjective $f : X -> Y$, duplicate some $y \in Y \setminus f(X)$ to get two distinct maps from $Y$ equalized by $f$); these thus correspond by \eqref{eq:loc-epi-surj} to $\sigma$-Borel surjective maps in $\!{\sigma BorLoc}_\sigma$.

(We do \emph{not} know if all such maps are $\infty$-Borel surjective; see \cref{rmk:sigma11-absolute}.)

\item
Epimorphisms in $\!{QPol}$ need not be surjective: consider $\-{\#N} = \#N \sqcup \{\infty\}$ with the Scott topology, with open sets $\emptyset, [n, \infty]$ for all $n \in \#N$ (which is indeed quasi-Polish; see \cite[\S8]{dBqpol}), and the inclusion $\#N `-> \-{\#N}$ where $\#N$ has the discrete topology.
Every continuous map $\-{\#N} -> X$ to a sober space $X$ is determined by its restriction to $\#N$, since directed joins are topologically definable in sober spaces (see \cite[II~1.9]{Jstone}).
\end{itemize}
\end{remark}

\begin{remark}
Polish spaces are precisely the regular (i.e., $T_3$) quasi-Polish spaces (see \cite[5.3]{Cqpol}).
Also, a space is clearly regular iff its underlying locale is, as defined in \cref{sec:upkzfrm}.
Thus, we have an equivalence of categories between Polish spaces and regular standard $\sigma$-locales.
\end{remark}

\begin{remark}
\label{rmk:loc-ctbpres-perfect}
It is a classical result that the only standard Borel spaces, up to Borel isomorphism, are $0, 1, 2, \dotsc, \#N$ with the discrete Borel structure, and $\#R$ with the usual Borel structure (see \cite[15.6]{Kcdst}).
Thus, any two uncountable standard Borel spaces are Borel isomorphic.

It follows that for any uncountable standard Borel space $X$, we have
\begin{align*}
\@B_\infty(X) \cong \ang{\#N}_\!{CBOOL}
\end{align*}
(which is a proper class, by the Gaifman--Hales \cref{thm:gaifman-hales}).
Indeed, $\@B_\infty(X) = \ang{\@B_\sigma(X)}_\!{CBOOL} \cong \ang{\@B_\sigma(\#S^\#N)}_\!{CBOOL} = \ang{\ang{\#N}_\!{\sigma Bool}}_\!{CBOOL} = \ang{\#N}_\!{CBOOL}$, since $X \cong \#S^\#N \in \!{\sigma BorLoc}_\sigma$ and by \cref{thm:loc-bor-sierpinski}.
(Several instances of this were already pointed out by Isbell~\cite{Iloc}, including $X = \#S^\#N, 2^\#N, \#R$.)
\end{remark}

\begin{example}
\label{ex:loc-ctbpres-dirint}
Consider the interval $X = [0, 1]$ with the Scott topology, with open sets $(r, 1]$; this is a quasi-Polish space (see \cite[\S8]{dBqpol}), whose specialization preorder is the usual linear order on $[0, 1]$.
Since it is uncountable, we thus have $\@B_\infty(X) \cong \ang{\#N}_\!{CBOOL}$.

However, $\@B^+_\infty(X)$ is much simpler: it is the same as $\@B^+_\sigma(X)$, the closure of the open sets $\@O(X)$ under countable unions and intersections, i.e., consisting of all open and closed intervals $(r, 1]$ and $[r, 1]$.
Indeed, note that every lower set $D \subseteq \@B^+_\sigma(X)$ is countably generated: it is either a principal ideal, or else contains a strictly increasing cofinal sequence $[r_0, 1] \subsetneq [r_1, 1] \subsetneq \dotsb$ for some $r_0 > r_1 > \dotsb$.
Thus, every $(\sigma, \sigma)$-frame homomorphism $\@B^+_\sigma(X) -> A$ already preserves arbitrary joins, and dually, already preserves arbitrary meets, which implies $\@B^+_\infty(X) = \ang{\@B^+_\sigma(X)}_\!{\infty\infty FRM} = \@B^+_\sigma(X)$.
\end{example}

\begin{example}
On the other hand, if $X$ is Polish (or just fit quasi-Polish; see \cref{sec:upkzfrm}), then $\@B^+_\sigma(X) = \@B_\sigma(X)$, since every closed set is a countable intersection of opens.
Hence in this case, $\@B^+_\infty(X) = \ang{\@B^+_\sigma(X)}_\!{\infty\infty FRM} = \ang{\@B_\sigma(X)}_\!{\infty\infty FRM} = \ang{\@B_\sigma(X)}_\!{CBOOL} = \@B_\infty(X)$ ($\cong \ang{\#N}_\!{CBOOL}$, if $X$ is uncountable).
\end{example}

\subsection{The internal logic}
\label{sec:loc-intlog}

In categorical logic, there is a method for formally interpreting ``pointwise'' expressions, such as
\begin{align}
\label{eq:loc-intlog-ker}
\ker(f) := \{(x, y) \in X^2 \mid f(x) = f(y)\},
\end{align}
in arbitrary categories.
This can be used to give intuitive-looking definitions of the localic analogs of various spatial notions such as $\ker$.
In this subsection, we briefly review this method, concentrating on the categories $\!{\kappa BorLoc}$; see \cite[D1.2]{Jeleph} for a more comprehensive reference.

Let $\!C$ be a category with $\kappa$-ary limits; we have in mind $\!C = \!{\kappa BorLoc}_{(\kappa)}$.
The \defn{internal language} of $\!C$ is the $\kappa$-ary multi-sorted infinitary first-order language $\@L(\!C)$ with
\begin{itemize}
\item  a sort for each object $X \in \!C$;
\item  a $\kappa$-ary function symbol for each morphism $f : \prod_i X_i -> Y \in \!C$, where $\vec{X} = (X_i)_i$ is a $\kappa$-ary family of objects in $\!C$;
\item  a $\kappa$-ary relation symbol for each subobject $R \subseteq \prod_i X_i$ in $\!C$.
\end{itemize}
Now let $\phi(\vec{x})$ be a formula in the \defn{$\kappa$-ary infinitary first-order logic $\@L_{\kappa\kappa}(\!C)$} over $\@L(\!C)$, with $<\kappa$-many free variables $\vec{x}$ with sorts $\vec{X}$ (for a $\kappa$-ary family of objects $\vec{X}$ in $\!C$), using $\kappa$-ary conjunctions $\bigwedge$ and disjunctions $\bigvee$ and $\exists$ over $\kappa$-ary families of variables; we treat $\forall$ as an abbreviation for $\neg \exists \neg$.
For certain such formulas $\phi$, we define its \defn{interpretation in $\!C$}, which is a certain subobject
\begin{align*}
\{\vec{x} \in \prod \vec{X} \mid \phi(\vec{x})\}_\!C \subseteq \prod \vec{X} \in \!C;
\end{align*}
the formulas which have an interpretation are called \defn{interpretable}.
\begin{itemize}

\item  First, each term $t(\vec{x})$ of sort $Y$ with variables $\vec{x}$ of sort $\vec{X}$ has an interpretation as a morphism $t_\!C : \prod \vec{X} -> Y$, defined by induction on $t$ in the obvious way.

\item  Atomic formulas are interpretable: for variables $\vec{x} = (x_i)_{i \in I}$ (where $I$ is a $\kappa$-ary set) of sorts $\vec{X} = (X_i)_{i \in I}$, for terms $\vec{t}(\vec{x}) = (t_j(\vec{x}))_{j \in J}$ (where $J$ is $\kappa$-ary) of sorts $(Y_j)_{j \in J}$, and for a relation $R \subseteq \prod_{j \in J} Y_j$, the atomic formula $R(\vec{t})$ is interpreted as the pullback
\begin{equation*}
\begin{tikzcd}
\{\vec{x} \in \prod_i X_i \mid R(\vec{t})\} \dar[hook] \rar & R \dar[hook] \\
\prod_i X_i \rar["((t_j)_\!C)_j"'] & \prod_j Y_j
\end{tikzcd}
\end{equation*}
(the binary equality relation $=$ is treated as the diagonal $X `-> X^2$ for each $X \in \!C$).

\item  $\kappa$-ary conjunction $\bigwedge_i \phi_i$ is interpreted as the meet (i.e., wide pullback) of the interpretations of the $\phi_i$.

\item  $\kappa$-ary disjunction $\bigvee_i \phi_i$ is interpreted as the join, in the subobject lattice, of the interpretations of the $\phi_i$, \emph{provided the join exists and is pullback-stable}.

In particular, the nullary disjunction $\bot$ is interpreted as a pullback-stable least subobject, which is the same thing (see \cite[A1.4.1]{Jeleph}) as a \defn{strict initial object} $\emptyset \in \!C$, i.e., an initial object such that every morphism to it is an isomorphism.

\item  If $\phi(\vec{x})$ has an interpretation $\{\vec{x} \mid \phi(\vec{x})\} \subseteq \prod \vec{X}$, and this interpretation has a \emph{pullback-stable complement}, meaning another subobject $\neg \{\vec{x} \mid \phi(\vec{x})\} \subseteq \prod \vec{X}$ whose meet with $\{\vec{x} \mid \phi(\vec{x})\}$ is the pullback-stable least subobject $\emptyset$ and whose join with $\{\vec{x} \mid \phi(\vec{x})\}$ is all of $\prod \vec{X}$ and such that this join is pullback-stable, then this complement $\neg \{\vec{x} \mid \phi(\vec{x})\}$ is defined to be $\{\vec{x} \mid \neg \phi(\vec{x})\}$.

(Note that pullback-stability of a join implies that meets distribute over it; thus pullback-stable complements that exist are unique.)

\item  Finally, for $\exists$: if $\phi(\vec{x}, \vec{y})$ has an interpretation $\{(\vec{x}, \vec{y}) \mid \phi(\vec{x}, \vec{y})\} \subseteq \prod \vec{X} \times \prod \vec{Y}$, and the projection $\pi_{\vec{X}} : \prod \vec{X} \times \prod \vec{Y} -> \prod \vec{X}$ is such that there is a least subobject $\pi_{\vec{X}}(\{(\vec{x}, \vec{y}) \mid \phi(\vec{x}, \vec{y})\}) \subseteq \prod \vec{X}$ whose pullback is $\ge \{(\vec{x}, \vec{y}) \mid \phi(\vec{x}, \vec{y})\}$, and furthermore it remains least when pulled back along any $Z -> \prod \vec{X}$, then we call $\pi_{\vec{X}}(\{(\vec{x}, \vec{y}) \mid \phi(\vec{x}, \vec{y})\})$ the \defn{pullback-stable image} of $\{(\vec{x}, \vec{y}) \mid \phi(\vec{x}, \vec{y})\}$ under $\pi_{\vec{x}}$, and define it to be $\{\vec{x} \mid \exists \vec{y}\, \phi(\vec{x}, \vec{y})\}$.

Alternatively, say that an arbitrary morphism $f : A -> X$ has \defn{pullback-stable image} $\im(f) \subseteq X$ when $\im(f)$ is least such that $f^*(\im(f)) = A$, and remains so when pulled back along any $Z -> X$, i.e., $A ->> \im(f) `-> X$ is a pullback-stable factorization of $f$ into an \defn{extremal epimorphism} (a morphism not factoring through any proper subobject of its codomain) followed by a monomorphism.
Then the image of a subobject, as defined above, is the same thing as the image of the composite
\begin{align}
\label{eq:loc-intlog-proj}
\{(\vec{x}, \vec{y}) \mid \phi(\vec{x}, \vec{y})\} `-> \prod \vec{X} \times \prod \vec{Y} --->{\pi_{\vec{X}}} \prod \vec{X}.
\end{align}

\end{itemize}
If $\phi$ is an interpretable formula with no free variables, then its interpretation $\{() \mid \phi\}$, if it exists, is a subobject of the terminal object $1 \in \!C$; if it is all of $1$, then we say that $\!C$ \defn{satisfies} $\phi$, denoted
\begin{align*}
\!C |= \phi  \coloniff  \{() \mid \phi\} = 1.
\end{align*}
For interpretable $\phi(\vec{x})$ with free variables of sorts $\vec{X}$, if $\forall \vec{x}\, \phi(\vec{x})$ is also interpretable, then it is easily seen that $\{\vec{x} \mid \phi(\vec{x})\} \subseteq \prod \vec{X}$ is all of $\prod \vec{X}$ iff $\!C |= \forall \vec{x}\, \phi(\vec{x})$; thus more generally, regardless of whether $\forall \vec{x}\, \phi(\vec{x})$ is interpretable, we put
\begin{align*}
\!C |= \forall \vec{x}\, \phi(\vec{x})  \coloniff  \{\vec{x} \mid \phi(\vec{x})\} = \prod \vec{X}.
\end{align*}
More generally still, for interpretable $\phi(\vec{x}), \psi(\vec{x})$ both with free variables of sorts $\vec{X}$, we put
\begin{align*}
\!C |= \forall \vec{x}\, (\phi(\vec{x}) -> \psi(\vec{x}))  \coloniff  \{\vec{x} \mid \phi(\vec{x})\} \subseteq \{\vec{x} \mid \psi(\vec{x})\} \subseteq \prod \vec{X};
\end{align*}
again if $\phi(\vec{x}) -> \psi(\vec{x})$ is in fact interpretable, this is consistent with the above.

\begin{example}
It is straightforward to see that for any morphism $f : X -> Y \in \!C$, \eqref{eq:loc-intlog-ker} above does define the same subobject of $X^2$ as the usual kernel, i.e., the pullback $X \times_Y X$.
\end{example}

Formulas such as \eqref{eq:loc-intlog-ker}, whose interpretation only uses ($\kappa$-ary) limits in $\!C$, are called \defn{($\kappa$-)limit formulas} (cf.\ \cref{sec:cat-lim}); they are always interpretable (provided $\!C$ has $\kappa$-ary limits), and moreover one may reason about them exactly as in $\!{Set}$ (since the Yoneda embedding $\!C -> \!{Set}^{\!C^\op}$ preserves limits).
Atomic formulas are clearly limit formulas, as are conjunctions thereof.
Certain existentials are also limit formulas: namely, if the composite \eqref{eq:loc-intlog-proj} above is already a monomorphism (which is expressible via pullbacks), then it is its own pullback-stable image, hence is the interpretation of $\exists \vec{y}\, \phi(\vec{x}, \vec{y})$ using only limits (and composition).
Such existentials are called \defn{provably unique}.

In order for \emph{all} $\@L_{\kappa\kappa}(\!C)$-formulas to be interpretable, $\!C$ needs to have (i) all pullback-stable images, or equivalently pullback-stable extremal epi--mono factorizations; and (ii) all pullback-stable joins and complements of subobjects.
Categories with (i) are called \defn{regular} (cf.\ \cref{sec:cat-alg}); in such categories, all extremal epimorphisms are in fact regular epimorphisms (see \cite[A1.3.4]{Jeleph}), hence one may equivalently speak of regular epi--mono factorizations.
By \cref{thm:loc-bor-epi-pullback,thm:loc-mono-epi}(b),

\begin{proposition}
\label{thm:loc-bor-reg}
For $\kappa < \infty$, $\!{\kappa BorLoc}$ is a regular category.
\qed
\end{proposition}

\begin{remark}
$\!{\kappa BorLoc}_\kappa$ is \emph{not} regular, not being closed under images in $\!{\kappa BorLoc}$; see \cref{thm:sigma11-proper} below.
In particular, $\!{\infty BorLoc}$ (= $\!{\infty BorLoc}_\infty$) is \emph{not} regular.
Furthermore, for $\kappa \le \lambda$, images in $\!{\kappa BorLoc}$ need not be preserved in $\!{\lambda BorLoc}$, i.e., the forgetful functor from the former category to the latter is not regular (see \cref{rmk:loc-im-bad}, \cref{thm:loc-im-bad}).
\end{remark}

A category with (finite limits and) all pullback-stable $\kappa$-ary joins of \emph{pairwise disjoint} subobjects is called \defn{$\kappa$-subextensive} in \cite{Cborin}.
A \defn{Boolean} $\kappa$-subextensive category is one which furthermore has complements (which can be defined using meets and pairwise disjoint joins; cf.\ \cref{sec:upkzfrm}); this implies arbitrary $\kappa$-ary joins, by disjointifying.
There does not seem to be a standard name for a category with pullback-stable $\kappa$-ary joins in general; however, if such a category is also regular, it is called \defn{$\kappa$-coherent} (see \cite[A1.4]{Jeleph}).
In a $\kappa$-subextensive category $\!C$, $\kappa$-ary pairwise disjoint joins of subobjects are also coproducts; thus if every $\kappa$-ary family of objects $X_i$ embeds as pairwise disjoint subobjects of a common object $Y$, then $\!C$ has $\kappa$-ary coproducts, which can be defined purely in terms of (meets and) $\kappa$-ary pairwise disjoint joins of subobjects, hence are pullback-stable.
Such a category $\!C$ is also called \defn{$\kappa$-extensive}, and can be equivalently defined as a category with (finite limits and) pullback-stable $\kappa$-ary coproducts in which coproduct injections are pairwise disjoint and the initial object is strict.
(For these last facts, see \cite{CLWext}, \cite[A1.4.3]{Jeleph}, \cite[\S2]{Cborin}.)
A $\kappa$-extensive regular category is $\kappa$-coherent, i.e., again has arbitrary $\kappa$-ary joins, by taking an image of a coproduct.

\begin{proposition}[see {\cite[3.8--9]{Cborin}}]
\label{thm:loc-bor-ext}
\leavevmode
\begin{enumerate}
\item[(a)]  $\!{\kappa BorLoc}$ and $\!{\kappa BorLoc}_\kappa$ are $\kappa$-extensive.
\item[(b)]  For $\kappa < \infty$, $\!{\kappa BorLoc}$ is $\kappa$-coherent.
\item[(c)]  $\!{\kappa BorLoc}_\kappa$ is Boolean.
\end{enumerate}
\end{proposition}
\begin{proof}
(a) and (c) are by \cref{thm:loc-coprod-pullback} and the fact that subobjects in these categories are the same as $\infty(\@B_\kappa)_\delta$, respectively $\kappa$-Borel, sets (by \cref{thm:loc-sub-im,thm:loc-mono-epi}), with pullback along $f : X -> Y$ corresponding to $f^* : \@B_\infty(Y) -> \@B_\infty(X)$ (\cref{thm:loc-im-factor}) and $\kappa$-ary coproduct corresponding to pairwise disjoint union of subobjects (\cref{thm:loc-sub-im-union}).

(b) follows from \cref{thm:loc-bor-reg} and $\kappa$-extensivity.
\end{proof}

It follows that large fragments of $\@L_{\kappa\kappa}$ may be interpreted in the categories $\!{\kappa BorLoc}_{(\kappa)}$.
Namely, starting in $\!{\kappa BorLoc}_\kappa$, we may interpret any formula, as long as we ensure that whenever an $\exists$ appears, its interpretation still lands in $\!{\kappa BorLoc}_\kappa$; while arbitrary $\exists$ may be interpreted by passing to $\!{\kappa BorLoc}$, as long as we no longer take complements outside of such arbitrary $\exists$.
We will continue and extend this discussion of interpreting $\exists$ when we introduce $\kappa\Sigma^1_1$ sets in \cref{sec:sigma11} below.

\begin{remark}
In \cite{Cborin} we showed that $\!{\kappa BorLoc}_\kappa$ is the \emph{initial} Boolean $\kappa$-extensive category with $\kappa$-ary limits (in the $2$-category of all such categories).
\end{remark}

\begin{remark}
\label{rmk:loc-bor-creg}
\Cref{thm:loc-bor-epi-invlim} also gives a compatibility condition between \emph{arbitrary} limits and images in $\!{\kappa BorLoc}$.
In particular, it implies Makkai's \cite[2.2]{Mbarr} \emph{principle of dependent choices} (for all lengths $\lambda$), whence by the main result of that paper, we may also interpret \emph{$\infty$-regular formulas} (built using arbitrary $\bigwedge$ and $\exists$) in $\!{\kappa BorLoc}$, with the same implications between such formulas holding as in $\!{Set}$.
(This is stronger than Carboni--Vitale's \cite[\S4.5]{CVregex} notion of \defn{completely regular} category, i.e., a regular category in which regular epimorphisms are closed under arbitrary products, which $\!{\kappa BorLoc}$ is, by \cref{thm:loc-bor-epi-prod}; see also \cref{sec:cat-alg}.)
\end{remark}

\begin{remark}
For a $\kappa$-subextensive category, $\kappa$-extensivity imposes the existence of additional structure (disjoint coproducts) which is uniquely defined in terms of the $\kappa$-subextensive structure, if it exists.
There is a similar counterpart to regularity, called \defn{(Barr\nobreakdash-)exactness} (cf.\ \cref{sec:cat-alg}): every equivalence relation ${\sim} \subseteq X^2$ (i.e., $\!C \models{}$``$\sim$ is an equivalence relation on $X$'') has a coequalizer whose kernel is $\sim$.
In contrast to \cref{thm:loc-bor-reg}, the category $\!{\kappa BorLoc}$ is \emph{not} exact.
Its ``exact completion'' is related to the theory of \emph{Borel equivalence relations} in descriptive set theory; see e.g., \cite{Gidst}.
We leave a detailed study of the exact completion for future work.
\end{remark}

\subsection{Positive Borel locales and partial orders}
\label{sec:poloc}

Recall (see \cref{ex:loc-infbor+}) that our main motivation for positive $\kappa$-Borel locales is to provide a pointless approach to the specialization order on a locale.
We now make this precise.

By an \defn{(internal) partial order} on an object $X$ in a category $\!C$ with finite limits, we mean a subobject ${\le} \subseteq X^2$ which internally (see \cref{sec:loc-intlog}) obeys the axioms of reflexivity, transitivity, and antisymmetry (which are implications between finite limit formulas, thus interpretable in $\!C$).
Thus, an internal partial order on a $\kappa$-Borel locale $X \in \!{\kappa BorLoc}$ is an $\infty(\@B_\kappa)_\delta$ binary relation ${\le} \subseteq X^2$ obeying these axioms, which explicitly means
\begin{align*}
(X --->{\delta} X^2)^*({\le}) &= X, \\
(X^3 --->{\pi_{12}} X^2)^*({\le}) \cap (X^3 --->{\pi_{23}} X^2)^*({\le}) &\subseteq (X^3 --->{\pi_{13}} X^2)^*({\le}) \subseteq X^3, \\
({\le}) \cap (X^2 --->{\sigma} X^2)^*({\le}) &= ({=_X}) \subseteq X^2
\end{align*}
where $\delta$ is the diagonal, $\pi_{ij} : X^3 -> X^2$ is the projection onto the $i$th and $j$th coordinates, $\sigma$ is the ``swap'' map, and ${=_X}$ is the image of $\delta$; while an internal partial order on a standard $\kappa$-Borel locale $X \in \!{\kappa BorLoc}_\kappa$ is a $\kappa$-Borel such $\le$.
By a \defn{$\kappa$-Borel polocale}, we mean a $\kappa$-Borel locale $X$ equipped with a $\infty(\@B_\kappa)_\delta$ partial order ${\le} = {\le_X}$, while by a \defn{standard $\kappa$-Borel polocale}, we mean a standard $\kappa$-Borel locale equipped with a $\kappa$-Borel partial order.
Let $\!{\kappa BorPOLoc}, \!{\kappa BorPOLoc}_\kappa$ be the categories of these and $\kappa$-Borel order-preserving maps.
We have forgetful functors
\begin{equation}
\label{diag:poloc-forget}
\begin{tikzcd}
\!{\kappa BorPOLoc}_\kappa \dar \rar[rightarrowtail] & \!{\kappa BorPOLoc} \dar \\
\!{\kappa BorLoc}_\kappa \rar[rightarrowtail] & \!{\kappa BorLoc}
\end{tikzcd}
\end{equation}
Moreover, for $\kappa \le \lambda$, since the forgetful functor $\!{\kappa BorLoc} -> \!{\lambda BorLoc}$ preserves finite limits, it preserves internal partial orders; thus we have forgetful functors $\!{\kappa BorLoc}_{(\kappa)} -> \!{\lambda BorLoc}_{(\lambda)}$, compatible with \eqref{diag:poloc-forget} in the obvious sense.

Recall that for an object $X$ in a locally ordered category, such as $\!{\kappa Bor^+Loc}$, the \defn{order ${\le_X} \subseteq X^2$ on $X$} is the weighted limit defined as the universal object equipped with projections $\pi_1, \pi_2 : {\le_X} \rightrightarrows X$ such that $\pi_1 \le \pi_2$, which is to say the order-kernel ${\le_X} = \oker(1_X)$ of the identity, or equivalently the power ${\le_X} = X^\#S$ of $X$ by the poset $\#S = \{0 < 1\}$ (see paragraph before \cref{thm:loc-mono-epi}, \cref{sec:cat-ord}).
In $\!{\kappa Bor^+Loc}$, we also call $\le_X$ the \defn{specialization order} on $X$.
By a straightforward calculation (or by the enriched Yoneda lemma \cite{Kvcat}), the order on $X$ is indeed an internal partial order on $X$ in the above sense.
Moreover, every morphism $f : X -> Y$ is order-preserving.
Thus, by assigning to every (standard) positive $\kappa$-Borel locale $X$ its specialization order $\le_X$ in $\!{\kappa Bor^+Loc}_{(\kappa)}$, and then applying the forgetful functor to $\!{\kappa BorLoc}_{(\kappa)}$, we get a (faithful) forgetful functor
\begin{align*}
\!{\kappa Bor^+Loc}_{(\kappa)} --> \!{\kappa BorPOLoc}_{(\kappa)}
\end{align*}
which factors the forgetful functor $\!{\kappa Bor^+Loc}_{(\kappa)} -> \!{\kappa BorLoc}_{(\kappa)}$ through \eqref{diag:poloc-forget}.

Recall from \eqref{diag:loc-cat} that the forgetful functor $\!{\kappa Bor^+Loc} -> \!{\kappa BorLoc}$ has a left adjoint ``free functor'' $D : \!{\kappa BorLoc} -> \!{\kappa Bor^+Loc}$, dual to the forgetful functor $\!{\kappa Bool} -> \!{\kappa\kappa Frm}$ (or $\!{CBOOL}_\infty -> \!{\infty\infty FRM}_\infty$ when $\kappa = \infty$), which exhibits $\!{\kappa BorLoc}$ as equivalent to a coreflective subcategory of $\!{\kappa Bor^+Loc}$.
For $X \in \!{\kappa BorLoc}$, we can think of $D(X)$ as the ``discrete'' positive $\kappa$-Borel locale on $X$.
For $X \in \!{\kappa Bor^+Loc}$, the counit $\epsilon_X : D(X) -> X$ is dual to the inclusion $\epsilon_X^* : \@B^+_\kappa(X) `-> \@B^+_\kappa(D(X)) = \@B_\kappa(X)$; hence $\epsilon_X$ is an epimorphism.

\begin{lemma}
\label{thm:loc-pos-ord-disc}
For $X \in \!{\kappa Bor^+Loc}$, the specialization order ${\le_X} \subseteq X^2$, regarded as a subobject in $\!{\kappa BorLoc}$, is also the order-kernel in $\!{\kappa Bor^+Loc}$ of the counit $\epsilon_X : D(X) -> X$ regarded as in $\!{\kappa BorLoc}$.
\end{lemma}
\begin{proof}
For any $Y \in \!{\kappa BorLoc}$,
a $\kappa$-Borel map $Y -> {\le_X} \in \!{\kappa BorLoc}$ is equivalently a positive $\kappa$-Borel map $D(Y) -> {\le_X} \in \!{\kappa Bor^+Loc}$,
which by definition of $\le_X$ is equivalently a pair $f, g : D(Y) \rightrightarrows X \in \!{\kappa Bor^+Loc}$ such that $f \le g$,
which (by coreflectivity) is equivalently a pair $f', g' : D(Y) \rightrightarrows D(X) \in \!{\kappa Bor^+Loc}$ such that $f = \epsilon_X \circ f' \le \epsilon_X \circ g' = g$,
which is equivalently a map $D(Y) -> \oker(\epsilon_X) \in \!{\kappa Bor^+Loc}$, which is equivalently a map $Y -> \oker(\epsilon_X) \in \!{\kappa BorLoc}$.
\end{proof}

Given any $\kappa$-Borel polocale $X$, we say that a $\kappa$-Borel set $B \subseteq X$ is \defn{upper} if it is internally so (with respect to the order $\le_X$) in $\!{\kappa BorLoc}$, i.e.,
\begin{align*}
({\le_X} --->{\pi_1} X)^*(B) \subseteq ({\le_X} --->{\pi_2} X)^*(B) \subseteq {\le_X}.
\end{align*}
(Note that since $\!{\kappa BorLoc} -> \!{\infty BorLoc}$ is conservative (\cref{thm:loc-forget}), i.e., injective on subobjects (\cref{thm:cat-funct-sub}), a $\kappa$-Borel set $B$ is upper in $\!{\kappa BorLoc}$ iff it is so as an $\infty$-Borel set in $\!{\infty BorLoc}$.)

We now have the main result connecting positive $\kappa$-Borel locales with $\kappa$-Borel polocales:

\begin{theorem}
\label{thm:loc-pos-upper}
For a positive $\kappa$-Borel locale $X$, a $\kappa$-Borel set $B \in \@B_\kappa(X)$ is upper (with respect to the specialization order) iff it is a positive $\kappa$-Borel set.
\end{theorem}
\begin{proof}
Letting $D(X)$ as above be the ``discrete'' positive $\kappa$-Borel locale on the underlying $\kappa$-Borel locale of $X$, the counit $\epsilon_X : D(X) -> X$ is an epimorphism, hence an order-regular epimorphism by \cref{thm:loc-mono-epi}.
This means $\epsilon_X$ is the coinserter of its order-kernel, which by \cref{thm:loc-pos-ord-disc} is the underlying $\kappa$-Borel locale of the specialization order $\le_X$ in $\!{\kappa Bor^+Loc}$.
Dually, this means
\begin{align*}
\@B^+_\kappa(X)
&= \ins(\pi_1^*, \pi_2^* : \@B^+_\kappa(D(X)) \rightrightarrows \@B_\kappa(\le_X)) \\
&= \ins(\pi_1^*, \pi_2^* : \@B_\kappa(X) \rightrightarrows \@B_\kappa(\le_X)) \\
&= \{B \in \@B_\kappa(X) \mid \pi_1^*(B) \le \pi_2^*(B)\},
\end{align*}
which consists precisely of the upper sets by definition.
\end{proof}

Since the preimage of an upper set under an order-preserving map is clearly upper,

\begin{corollary}
\label{thm:loc-pos-poloc}
The forgetful functor $\!{\kappa Bor^+Loc} -> \!{\kappa BorPOLoc}$ is full.
\qed
\end{corollary}

\begin{remark}
Given any $\kappa$-Borel polocale $X$, the upper $\kappa$-Borel sets form a $(\kappa, \kappa)$-subframe of $\@B_\kappa(X)$, hence when $\kappa < \infty$ define a positive $\kappa$-Borel locale $X'$ equipped with a positive $\kappa$-Borel epimorphism $\eta : X -> X'$ (with $\eta^*$ given by the inclusion).
This $X'$ is the reflection of $X$ into $\!{\kappa Bor^+Loc}$: indeed, for any order-preserving $\kappa$-Borel map $f : X -> Y$ into a positive $\kappa$-Borel locale $Y$, $f^* : \@B_\kappa(Y) -> \@B_\kappa(X)$ must preserve upper sets, i.e., take positive $\kappa$-Borel maps in $Y$ to positive $\kappa$-Borel maps in $X'$, which yields a factorization of $f$ through $X'$.
So $\!{\kappa Bor^+Loc} \subseteq \!{\kappa BorPOLoc}$ is an epireflective subcategory, for $\kappa < \infty$.
\end{remark}

\begin{example}
Consider the interval $[0, 1]$ with the Scott topology (see \cref{ex:loc-ctbpres-dirint}), a quasi-Polish space, i.e., standard $\sigma$-locale.
Let $X = [0, 1] \times \#S$, regarded as a standard Borel space (standard $\sigma$-Borel locale) equipped with the \emph{lexicographical} order $\le$ (where $\#S = \{0 < 1\}$).
This is a linear order, i.e., it satisfies the axiom $\forall x, y\, ((x \le y) \vee (y \le x))$ internally in $\!{SBor}$, hence also in $\!{\infty BorLoc}$ since the forgetful functor preserves unions of subobjects (being given by join in $\@B_\sigma \subseteq \@B_\infty$).

We claim that \emph{$A := [0, 1] \times \{0\} \subseteq X$ is not in the complete Boolean subalgebra of $\@B_\infty(X)$ generated by the upper $\infty$-Borel sets}.
To see this, first recall from \cref{ex:loc-ctbpres-dirint} that $\@B^+_\infty([0, 1])$ consists only of the intervals $(r, 1]$ and $[r, 1]$.
Next, any upper $\infty$-Borel $B \subseteq X$ is of the form
\begin{align*}
\tag{$*$}
B = (C \times \{0\}) \cup (D \times \{1\}),
\end{align*}
where $C, D \in \@B_\infty([0, 1])$ are the preimages of $B$ under the two ``cross-section'' inclusions $[0, 1] `-> X$ which are order-preserving, whence $C, D$ are upper, hence intervals of the aforementioned form; by definition of the lexicographical order, it is easily seen that the only possibilities are (i) $C = D$ or (ii) $C = (r, 1]$ and $D = [r, 1]$ for some $r$.
Thus, the complete Boolean subalgebra of $\@B_\infty(X)$ generated by the upper sets is contained in the complete Boolean subalgebra consisting of all $B$ as in ($*$) such that $C \triangle D \in \@B_\infty([0, 1])$ is a union of points (and conversely, it is easy to see that every such $B$ can be generated from $\@B^+_\infty(X)$).
Since $[0, 1] \in \@B_\infty([0, 1])$ is not a union of points (by localic Baire category or Gaifman--Hales; cf.\ \cref{rmk:loc-epi-pullback}), it follows that $A$ is not in this complete Boolean subalgebra.

Thus, \emph{$X$ is a standard $\sigma$-Borel polocale which is not a positive $\kappa$-Borel locale, for any $\kappa \le \infty$}; that is, the full inclusion $\!{\kappa Bor^+Loc} -> \!{\kappa BorPOLoc}$ is not essentially surjective, for any $\kappa \ge \omega_1$.
Indeed, for $\kappa < \infty$, the positive $\kappa$-Borel locale reflection $\eta : X -> X'$ of $X$, as described in the preceding remark, is such that $\eta^* : \@B_\kappa(X') -> \@B_\kappa(X)$ is not surjective (since its image does not contain $A$), i.e., $\eta$ is not a monomorphism.

Moreover, when $\kappa = \omega_1$, the same argument as above shows that the $\sigma$-Boolean subalgebra of $\@B_\sigma(X)$ generated by the upper sets consists of all $B$ as in ($*$) such that $C \triangle D$ is a \emph{countable} union of points.
This algebra is not countably generated, since countably many such $B$ can only generate other such $B$ for which $C \triangle D$ is contained in a fixed countable set (the union of $C \triangle D$ for the generating $B$).
Thus, \emph{there is a standard $\sigma$-Borel polocale whose positive $\sigma$-Borel locale reflection is not standard}; that is, the free functor $\!{\sigma BorPOLoc} -> \!{\sigma Bor^+Loc}$ does not preserve standard objects.
It follows easily from this that \emph{the forgetful functor $\!{\sigma Bor^+Loc}_\sigma -> \!{\sigma BorPOLoc}_\sigma$ does not have a left adjoint $F$}: if it did, letting $X -> X'$ be the reflection of $X$ in $\!{\sigma Bor^+Loc}$ (so that $\@B^+_\sigma(X') \subseteq \@B_\sigma(X)$ is the non-countably generated $\sigma$-Boolean subalgebra generated by the upper sets as above), $X'$ would be a countably cofiltered limit of standard positive $\sigma$-Borel locales $X'_i$, and each $X -> X' -> X'_i$ would have to factor uniquely through $X -> F(X)$; these factorizations would yield a retraction $F(X) -> X'$ of the comparison map $X' -> F(X)$, whence $X'$ must already be standard.

A similar argument shows that the forgetful functor $\!{\kappa Bor^+Loc}_\kappa -> \!{\kappa BorPOLoc}_\kappa$ does not have a left adjoint, for any $\kappa$ which is \emph{not inaccessible}: replace $[0, 1]$ above with a $\kappa$-standard linear order with $\ge \kappa$-many points, e.g., $\#S^\lambda$ lexicographically ordered for some $\lambda < \kappa$ with $2^\lambda \ge \kappa$.
However, we do not know what happens for $\kappa$ inaccessible; in particular, we do not know if the forgetful functor $\!{\infty Bor^+Loc} -> \!{\infty BorPOLoc}$ has a left adjoint.

For more information on standard Borel posets in the classical context, especially the special role played by lexicographical orders such as $X$ above, see \cite{HMSborord}, \cite{Kborord}.
\end{example}

\subsection{Baire category}
\label{sec:loc-baire}

The classical \defn{Baire category theorem} states that in sufficiently nice topological spaces $X$, such as locally compact sober spaces (see \cite[I-3.40.9]{GHKLMSdom}), completely metrizable spaces (see \cite[8.4]{Kcdst}), or quasi-Polish spaces (see \cite[Cor.~52]{dBqpol}, \cite[7.3]{Cqpol}), the intersection of countably many dense open sets is dense.
It follows that the dense open sets generate a $\sigma$-filter $\COMGR(X)$, consisting of the \defn{comeager} sets, all of which are dense, hence nonempty whenever the space is.
This $\sigma$-filter is equivalently generated by the dense $\*\Pi^0_2$-sets (in the sense of Selivanov; see \cref{sec:loc-bor}), due to

\begin{lemma}
In a topological space or locale $X$, if a union $U \cup F$ of an open set $U$ and closed set $F$ is dense, then so is the open subset $U \cup F^\circ$.
\end{lemma}
\begin{proof}
We have $X = \-{U \cup F} = \-U \cup \-F = \-U \cup F$, i.e., $\neg \-U \subseteq F$, whence $\neg \-U \subseteq F^\circ$ since $\neg \-U$ is open, whence $X = \-U \cup F^\circ \subseteq \-{U \cup F^\circ}$.
\end{proof}

A major result of locale theory, due to Isbell~\cite[1.5]{Iloc}, yields the localic analog of the Baire category theorem:

\begin{theorem}[Isbell]
\label{thm:loc-baire}
In every locale $X$, the intersection of all dense $\infty\Pi^0_2$ sets (i.e., sublocales, by \cref{thm:loc-sub-im}) is dense, i.e., there is a smallest dense $\infty\Pi^0_2$ set.
\end{theorem}

\begin{remark}
Unlike most localic analogs of classical results that we consider, the proof of the above localic Baire category theorem has nothing to do with the classical proof.
This is because the classical proof instead underlies the entire correspondence between locales and spaces in the countably presented setting (see \cref{sec:loc-ctbpres}).
\end{remark}

For a locale $X$, we denote the smallest dense $\infty\Pi^0_2$ set by $\Comgr(X) \subseteq X$, and call any other set $B \subseteq X$ \defn{comeager} if $\Comgr(X) \subseteq B$, and \defn{meager} if $\Comgr(X) \cap B = \emptyset$; thus $\Comgr(X)$ is the smallest comeager set, an improvement over the $\sigma$-filter $\COMGR(X)$ in the spatial context.
Explicitly, $\Comgr(X)$ is given by the equivalent formulas
\begin{align*}
\Comgr(X)
&= \neg \bigcup_{U \in \@O(X)} \partial U \quad\text{where $\partial U := \-U \setminus U$} \\
&= \bigcap_{U \in \@O(X)} (\-U -> U) \\
&= \bigcap_{U \in \@O(X)} (\-U^\circ -> U).
\end{align*}
The last formula means that $\@O(X) ->> \@O(\Comgr(X))$ is the quotient of $\@O(X)$ identifying each open set $U$ with its \defn{regularization} $\-U^\circ$.
Hence, $\@O(\Comgr(X))$ is isomorphic to the complete Boolean algebra of regular open sets $\{U \in \@O(X) \mid U = \-U^\circ\}$, sometimes called the \defn{(Baire) category algebra} of $X$ (see e.g., \cite[8.32]{Kcdst}).

\begin{example}
\label{ex:loc-baire-perfect}
For $X = \#R, 2^\#N$, or more generally any nonempty perfect (i.e., without isolated points) Polish space, $\@O(\Comgr(X))$ is the usual Cohen algebra (the unique nontrivial atomless complete Boolean algebra countably generated under joins; see \cite[8.32]{Kcdst}).
$\Comgr(X) \subseteq X$ is a dense, pointless sublocale; but its complement contains much more than just points, e.g., $\neg \Comgr(\#R)$ contains the usual Cantor set in $\#R$ (since it is nowhere dense), which is a perfect closed set, hence contains its own dense, pointless smallest comeager set which is disjoint from $\Comgr(\#R)$.
\end{example}

\begin{example}
For $X = [0, 1]$ with the \emph{Scott} topology (see \cref{ex:loc-ctbpres-dirint}), every nonempty open set is dense; thus $\Comgr(X) = \{1\}$.
\end{example}

Since $\@O(\Comgr(X))$ is Boolean, $\@B_\infty(\Comgr(X)) = \@O(\Comgr(X))$, i.e., every $\infty$-Borel set $B \subseteq \Comgr(X)$ is open, i.e., of the form $U \cap \Comgr(X)$ for some open $U \subseteq X$.
Thus, we have the following analog of the \defn{property of Baire} (see \cite[8.F]{Kcdst}) for $\infty$-Borel sets:

\begin{proposition}
\label{thm:loc-bor-baireprop}
Every $\infty$-Borel set in a locale differs from an open set by a meager set.
\qed
\end{proposition}

We also have the following strengthening of the localic Baire category \cref{thm:loc-baire}, which follows from it by taking the $X_i$ to consist of all dense open sets in a fixed locale $X$.
For the classical analog (as well as a different ``geometric'' proof using \cref{thm:loc-baire}), see \cite[A.1]{Cbermd}.

\begin{theorem}
\label{thm:loc-baire-invlim}
Let $(X_i)_{i \in I}$ be a codirected diagram of locales, for a directed poset $I$, with morphisms $f_{ij} : X_i -> X_j$ for $i \ge j$.
If each $f_{ij}$ has dense image (i.e., its closed image $\-\im(f_{ij}) \in \infty\Pi^0_1(X_j)$ is all of $X_j$), then so does each limit projection $\pi_i : \projlim_j X_j -> X_i$.
\end{theorem}
\begin{proof}
Each $f_{ij}$ having dense image means $\bot \in \@O(X_j)$ is the only element of its $\ker(f_{ij}^*)$-congruence class; by \cref{thm:frm-colim-dir-ker}, the same holds for each $\pi_i$.
\end{proof}

We conclude by giving, for the sake of completeness, a self-contained proof of a well-known application of localic Baire category: a localic ``cardinal collapse'', which says that any two infinite sets have a nonempty locale ``parametrizing bijections between them''.
See e.g., \cite[C1.2.8--9]{Jeleph}.

\begin{theorem}[folklore]
\label{thm:collapse}
Let $X, Y$ be $\kappa$-ary sets, with $Y$ infinite, regarded as $\kappa$-discrete locales.
\begin{enumerate}
\item[(a)]  There is a nonempty standard $\kappa$-locale $Z$ together with a $\kappa$-locale embedding $f : Z \times X `-> Z \times Y$ over $Z$ (i.e., commuting with the projections to $Z$).
\item[(b)]  If $X \ne \emptyset$, then $f$ in (a) has a $\kappa$-Borel retraction $g : Z \times Y ->> Z \times X$ over $Z$ with $g \circ f = 1_{Z \times X}$; in particular, there is an $\infty$-Borel surjective $\kappa$-Borel map $g : Z \times Y ->> Z \times X$ over $Z$.
\item[(c)]  If $X$ is infinite, then $f$ in (a) can be taken to be a $\kappa$-locale isomorphism.
\end{enumerate}
\end{theorem}
\begin{proof}
Consider the standard $\kappa$-locale $\#S^{X \times Y}$.
For $x \in X$ and $y \in Y$, let $[x |-> y] \in \@O(\#S^{X \times Y})$ denote the subbasic open set at the $(x, y)$th coordinate.
Let
\begin{align*}
P := \bigwedge_{x \in X; y \ne y' \in Y} \neg ([x |-> y] \wedge [x |-> y']) \wedge \bigwedge_{x \ne x' \in X; y \in Y} \neg ([x |-> y] \wedge [x' |-> y]) \in \kappa\Pi^0_1(\#S^{X \times Y})
\end{align*}
be the closed set ``of graphs of partial injections $X \lhook\joinrel\rightharpoonup Y$''.
Clearly $P$ has a point, namely the empty function; hence $P \ne \emptyset$.
Let
\begin{align*}
D_x &:= \bigvee_{y \in Y} [x |-> y] \in \@O(\#S^{X \times Y}), &
E_y &:= \bigvee_{x \in X} [x |-> y] \in \@O(\#S^{X \times Y}).
\end{align*}
Then each $P \wedge D_x \subseteq P$ is dense, since for a basic open set $U = [x_1 |-> y_1] \wedge \dotsb \wedge [x_n |-> y_n] \subseteq \#S^{X \times Y}$ which intersects $P$, we must have $x_i = x_j \iff y_i = y_j$ for all $1 \le i, j \le n$ by definition of $P$; taking $y = y_i$ if $x = x_i$ for some $i$ and $y \in Y \setminus \{y_1, \dotsc, y_n\}$ otherwise, the partial injection mapping $x_i |-> y_i$ and $x |-> y$ is a point of $P \wedge D_x \wedge U$.
Similarly, if $X$ is infinite, then each $P \wedge E_y \subseteq P$ is dense.

To prove (a), take $Z := P \wedge \bigwedge_{x \in X} D_x \subseteq P$, a dense and hence nonempty standard $\kappa$-sublocale.
Note that for each $x \in X$, the join defining $D_x$, when intersected with $P$, is pairwise disjoint, i.e., $Z$ is the disjoint union of open sublocales
\begin{align*}
Z = \bigsqcup_{y \in Y} (Z \wedge [x |-> y]).
\end{align*}
Thus $Z \times X$ is the disjoint union of open sublocales
\begin{align*}
Z \times X
&= \bigsqcup_{x \in X} (Z \times \{x\})
= \bigsqcup_{x \in X} \bigsqcup_{y \in Y} ((Z \wedge [x |-> y]) \times \{x\}).
\end{align*}
Define $f : Z \times X -> Z \times Y$ by taking its restriction to each $U_{x,y} := (Z \wedge [x |-> y]) \times \{x\}$ to be the inclusion $U_{x,y} \cong V_{x,y} := (Z \cap [x |-> y]) \times \{y\} \subseteq Z \times Y$.
Then $f$ restricted to each $U_{x,y}$ is an embedding; while each $U_{x,y}$ is itself equal to $f^*(V_{x,y})$: $\subseteq$ is by definition of $f$, while for $(x', y') \ne (x, y)$, $f(U_{x',y'}) = V_{x',y'}$ is disjoint from $V_{x,y}$, since either $y \ne y'$ in which case $\{y\} \cap \{y'\} = \emptyset$, or else $y = y'$ and $x \ne x'$ in which case $Z \cap [x |-> y], Z \cap [x' |-> y']$ are disjoint by definition of $P$; thus $\neg U_{x,y} = \bigsqcup_{(x',y') \ne (x,y)} U_{x',y'} \subseteq f^*(\neg V_{x,y})$.
Thus $f$ is an embedding, and clearly commutes with the projections to $Z$.

To prove (c), take $Z := P \wedge \bigwedge_{x \in X} D_x \wedge \bigwedge_{y \in Y} E_y$, define $f : Z \times X -> Z \times Y$ as before, and define $g : Z \times Y -> Z \times X$ in a symmetric manner, so that $f = g^{-1}$.

To prove (b), fix some $x \in X$, define $g$ on the $\infty$-Borel image $f(Z \times X) \subseteq Z \times Y$ (which is $\kappa$-Borel, since $Z \times X$ is standard $\kappa$-Borel) to be the inverse of the $\kappa$-Borel isomorphism $f : Z \times X \cong f(Z \times X)$, and define $g$ on $\neg f(Z \times X)$ to be the projection $\neg f(Z \times X) \subseteq Z \times Y -> Z$ followed by the inclusion $Z \cong Z \times \{x\} \subseteq Z \times X$.
\end{proof}

\section{Analytic sets and locales}
\label{sec:sigma11}

In this section, we study $\kappa$-analytic (i.e., $\kappa\Sigma^1_1$) sets in standard $\kappa$-Borel locales.
These are defined by ``completing'' the category $\!{\kappa BorLoc}_\kappa$ of standard $\kappa$-Borel locales under non-existing $\kappa$-Borel images.
There are two equivalent ways of performing this ``completion'': externally, by closing under images in the regular supercategory $\!{\kappa BorLoc}$ of possibly nonstandard $\kappa$-Borel locales, or internally, by freely ``adjoining'' an image for each morphism in $\!{\kappa BorLoc}_\kappa$.
In \cref{sec:sigma11-cat}, we give both definitions and derive some basic properties of the resulting category $\!{\kappa\Sigma^1_1BorLoc}$ of ``analytic $\kappa$-Borel locales''.

In the remaining subsections, we prove the analogs of some key classical results about analytic sets in standard Borel spaces.
In \cref{sec:sigma11-sep}, we state the Lusin--Novikov separation theorem for $\kappa$-analytic sets, as well as an ordered analog; these are simply translations of the algebraic interpolation theorems from \cref{sec:frm-interp}.  In \cref{sec:sigma11-invlim}, we discuss the connection between $\kappa$-analytic sets and well-foundedness; in particular, we prove that ``every $\kappa$-analytic set can be reduced to ill-foundedness of (generalized) trees'' (\cref{thm:sigma11-invlim-k}), that ``every $\kappa$-analytic set is a $\kappa$-length decreasing intersection of $\kappa$-Borel sets'', and that ``every $\kappa$-analytic well-founded (generalized) forest has bounded rank $<\kappa$'' (\cref{thm:sigma11-invlim-prune}).
Finally, in \cref{sec:sigma11-proper} we apply this boundedness theorem to prove that indeed, there exist $\kappa$-analytic, non-$\kappa$-Borel sets, thus justifying the careful treatment of $\kappa$-Borel images we have given up to this point.

\subsection{The main categories}
\label{sec:sigma11-cat}

Classically, an \defn{analytic set} or \defn{$\*\Sigma^1_1$ set} in a standard Borel space $X$ is a set which is the image of a Borel map $f : Y -> X$ from some standard Borel space $Y$.
An \defn{analytic Borel space} is an analytic set regarded as an abstract Borel space, equipped with the subspace Borel structure; these form a well-behaved category $\!{\*\Sigma^1_1Bor}$, containing the category $\!{SBor}$ of standard Borel spaces as a full subcategory.
See \cite[III]{Kcdst}.

In the localic setting, for $\kappa < \infty$, recall from \cref{thm:loc-bor-reg} that the category $\!{\kappa BorLoc}$ of all $\kappa$-Borel locales is regular, hence has a well-behaved notion of image (i.e., $\infty(\@B_\kappa)_\delta$-image).
We define an \defn{analytic $\kappa$-Borel locale} to mean a $\kappa$-Borel locale $X$ which is the image of a $\kappa$-Borel map $p : \~X -> X' \in \!{\kappa BorLoc}_\kappa$ between standard $\kappa$-Borel locales.
Equivalently, $X$ admits both a (regular) epimorphism ($\kappa$-Borel surjection) $p : \~X ->> X$ from a standard $\kappa$-Borel locale, as well as a monomorphism ($\kappa$-Borel embedding) $X `-> X'$ into a standard $\kappa$-Borel locale.
Dually, this means $\@B_\kappa(X)$ is a $\kappa$-generated $\kappa$-Boolean algebra which embeds into a $\kappa$-presented algebra.
We let $\!{\kappa\Sigma^1_1BorLoc} \subseteq \!{\kappa BorLoc}$ denote the full subcategory of analytic $\kappa$-Borel locales.

The following result summarizes the key properties of the category $\!{\kappa\Sigma^1_1BorLoc}$:

\begin{theorem}
\label{thm:sigma11-cat}
\leavevmode
\begin{enumerate}

\item[(a)]  $\!{\kappa\Sigma^1_1BorLoc} \subseteq \!{\kappa BorLoc}$ is closed under $\kappa$-ary limits, $\kappa$-ary coproducts, and images of morphisms.

\item[(b)]  $\!{\kappa\Sigma^1_1BorLoc}$ is a $\kappa$-extensive, $\kappa$-coherent category..

\item[(c)]  Let $(X_i)_{i \in I}$ be a $\kappa$-ary codirected diagram in $\!{\kappa\Sigma^1_1BorLoc}$, for a directed poset $I$, with morphisms $f_{ij} : X_i -> X_j$ for $i \ge j$.
If each $f_{ij}$ is a (regular) epimorphism, then so is each limit projection $\pi_i : \projlim_j X_j -> X_i$.

In particular, $\kappa$-ary products of epimorphisms in $\!{\kappa\Sigma^1_1BorLoc}$ are epimorphisms, i.e., $\!{\kappa\Sigma^1_1BorLoc}$ is a \defn{$\kappa$-completely regular} category.

\item[(d)]  The subobject lattices in $\!{\kappa\Sigma^1_1BorLoc}$ are $(\kappa, \kappa)$-frames.

\item[(e)]  All epimorphisms in $\!{\kappa\Sigma^1_1BorLoc}$ are regular.

\end{enumerate}
\end{theorem}
\begin{proof}
(a)
Closure under $\kappa$-ary products follows from the fact that product in $\!{\kappa BorLoc}$ preserves regular epimorphisms (i.e., complete regularity; see \cref{rmk:loc-bor-creg}), whence the $\kappa$-ary product of $X_i = \im(p_i : \~X_i -> X_i') \in \!{\kappa\Sigma^1_1BorLoc}$, where $p_i \in \!{\kappa BorLoc}_\kappa$, is the image of $\prod_i p_i : \prod_i \~X_i -> \prod_i X_i'$.
Likewise, closure under $\kappa$-ary coproducts follows from the fact that $\kappa$-ary coproduct in $\!{\kappa BorLoc}$ preserves monomorphisms (which follows from $\kappa$-extensivity, or more concretely from the dual fact that product in $\!{\kappa Bool}$ preserves surjective homomorphisms).
Closure under images is obvious.
Closure under pullback also follows easily from regularity of $\!{\kappa BorLoc}$: the pullback $X \times_Z Y$ of $f : X -> Z \in \!{\kappa\Sigma^1_1BorLoc}$ and $g : Y -> Z \in \!{\kappa\Sigma^1_1BorLoc}$ admits a regular epimorphism from a standard $\kappa$-Borel locale, namely $\~X \times_{Z'} \~Y$ for any $\~X ->> X$, $\~Y ->> Y$, and $Z' \supseteq Z$ with $\~X, \~Y, Z' \in \!{\kappa BorLoc}_\kappa$, and also embeds into $X \times Y$ which embeds into a standard $\kappa$-Borel locale since $X, Y$ do.
This is enough, since $\kappa$-products and (binary) pullbacks generate $\kappa$-ary limits (see e.g., \cite[2.8.1--2]{Bcat}).

(b) and (c) now follow from the corresponding properties of $\!{\kappa BorLoc}$ (\cref{thm:loc-bor-ext,thm:loc-bor-epi-invlim}).
We record the following consequence of (c):

\begin{lemma}
\label{thm:sigma11-meet}
For $<\kappa$-many morphisms $f_i : X_i -> Y \in \!{\kappa\Sigma^1_1BorLoc}$, the meet of their images is given by the image of their wide pullback
\begin{align*}
\bigwedge_i \im(f_i) = \im(\prod_Y (f_i)_i : \prod_Y (X_i)_i -> Y).
\end{align*}
\end{lemma}
\begin{proof}
For two morphisms, hence finitely many, this follows from pullback-stability of image (cf.\ \cref{thm:loc-bor-im-pullback}).
In general, the $\ge$ inequality is trivial; so it suffices to prove $\le$.
By considering the pullback of all the $f_i$ to $\bigwedge_i \im(f_i) \subseteq Y$, it is enough to assume that each $f_i$ is an epimorphism, and prove that $\prod_Y (X_i)_i -> Y$ is still an epimorphism.
For finite sets of indices $K \subseteq K'$, the projection $\prod_Y (X_i)_{i \in K'} -> \prod_Y (X_i)_{i \in K}$ is an epi by pullback-stability.
Thus the claim follows from (c).
\end{proof}

(d)
Let $X \in \!{\kappa\Sigma^1_1BorLoc}$.
That finite meets in $\Sub(X)$ distribute over $\kappa$-ary joins follows from pullback-stability of $\kappa$-ary joins by (b).
Let $Y, Z_i \subseteq X$ be $<\kappa$-many subobjects; we must show $\bigwedge_i (Y \vee Z_i) \le Y \vee \bigwedge_i Z_i$.
Again by replacing $X$ with $\bigwedge_i (Y \vee Z_i)$, we may assume each $Y \vee Z_i = X$, and prove $Y \vee \bigwedge_i Z_i = X$.
Now $X$ is an image of some $\~X \in \!{\kappa BorLoc}_\kappa$; by pulling everything back to $\~X$, we may assume $X$ was already standard $\kappa$-Borel.
And $Y, Z_i \subseteq X$ are the images of some $f : \~Y -> X$ and $g_i : \~Z_i -> X$ with $\~Y, \~Z_i$ standard $\kappa$-Borel.
By \cref{thm:sigma11-meet}, $\bigwedge_i Z_i = \im(\prod_X (g_i)_i : \prod_X (\~Z_i)_i -> X)$; while by the general construction of joins of subobjects in extensive categories (see before \cref{thm:loc-bor-ext}), each $Y \vee Z_i \subseteq X$ is the image of the map $[f, g_i] : \~Y \sqcup \~Z_i -> X$ induced by $f, g_i$, with a similar description of $Y \vee \bigwedge_i Z_i \subseteq X$.
All told, we are given $\kappa$-Borel maps $f : \~Y -> X$ and $g_i : \~Z_i -> X$ between standard $\kappa$-Borel locales, such that each $[f, g_i] : \~Y \sqcup \~Z_i -> X$ is $\kappa$-Borel surjective; and we must show that $[f, \prod_X (g_i)_i] : \~Y \sqcup \prod_X (\~Z_i)_i -> X$ is $\kappa$-Borel surjective.

Let $B \in \@B_\kappa(X)$ such that $[f, \prod_X (g_i)_i]^*(B) = \top$, i.e., $f^*(B) = \top$ and also $(\prod_X (g_i)_i)^*(B) = \top$.
The latter preimage map $(\prod_X (g_i)_i)^* : \@B_\kappa(X) -> \@B_\kappa(\prod_X (\~Z_i)_i)$ is the wide pushout of the $g_i^* : \@B_\kappa(X) -> \@B_\kappa(\~Z_i) \in \!{\kappa Bool}$, so by \cref{thm:bool-minterp} (applied to $b_i := \top$ and $a := \neg (\prod_X (g_i)_i)^*(B)$), there are $C_i \in \@B_\kappa(X)$ such that $B \wedge \bigwedge_i C_i = \emptyset$ and $g_i^*(C_i) = \~Z_i$ for each $i$; in particular, $g_i^*(B) \le g_i^*(C_i)$ for each $i$.
But since each $[f, g_i]$ is $\kappa$-Borel surjective, and we also have $f^*(B) = \emptyset \le f^*(C_i)$ for each $i$, we get $B \le C_i$ for each $i$, whence $B = B \wedge \bigwedge_i C_i = \emptyset$.
Thus $[f, \prod_X (g_i)_i]$ is $\kappa$-Borel surjective, as desired.

(e)
Given an epimorphism $f : X -> Y \in \!{\kappa\Sigma^1_1BorLoc}$, $f$ must also be an epimorphism in $\!{\kappa BorLoc}$, since the discrete $\kappa$-Borel locale $2$ is in $\!{\kappa\Sigma^1_1BorLoc}$, whence by considering the composite of $f$ with maps $Y -> 2$ we get that $f^* : \@B_\kappa(Y) -> \@B_\kappa(X)$ is injective.
Thus $f$ must be a regular epimorphism in $\!{\kappa BorLoc}$ (by \cref{thm:loc-mono-epi}), i.e., the coequalizer of its kernel which is in $\!{\kappa\Sigma^1_1BorLoc}$ by (a).
\end{proof}

\begin{remark}
\label{rmk:sigma11-lim-std}
We record the following useful fact about limits in $\!{\kappa\Sigma^1_1BorLoc}$, which was implicitly used in the proof of \cref{thm:sigma11-cat}(a) for pullbacks.
Given a $\kappa$-ary diagram $F : \!I -> \!{\kappa\Sigma^1_1BorLoc}$, for some $\kappa$-ary category $\!I$, if every $I \in \!I$ admits a morphism from some $\~I \in \!I$ such that $F(\~I)$ is a standard $\kappa$-Borel locale (i.e., the set of such $\~I$ is cofinal in $\!I^\op$ \emph{in the weaker order-theoretic sense}, not necessarily the categorical sense of a final functor), then $\projlim F$ is also standard $\kappa$-Borel.

To see this, first note that $F$ has the same limit as its restriction $F'$ to the subcategory $\!I' \subseteq \!I$ consisting of all nonidentity morphisms $f : I -> J$ such that $F(I)$ is standard $\kappa$-Borel (this subcategory \emph{is} initial in the categorical sense).
Now $F'$ is a diagram consisting only of (nonidentity) morphisms whose domain is standard $\kappa$-Borel; each such morphism can be composed with an embedding from its codomain to a standard $\kappa$-Borel locale, again without changing the limit.
\end{remark}

We call a subobject $A$ of an analytic $\kappa$-Borel locale $X \in \!{\kappa\Sigma^1_1BorLoc}$ a \defn{$\kappa$-analytic set} or \defn{$\kappa\Sigma^1_1$ set} of $X$, and denote the $(\kappa, \kappa)$-frame of all such by
\begin{align*}
\kappa\Sigma^1_1(X) := \Sub_\!{\kappa\Sigma^1_1BorLoc}(X).
\end{align*}
When $X$ is a standard $\kappa$-Borel locale, $\@B_\kappa(X) \subseteq \kappa\Sigma^1_1(X)$ is a $\kappa$-Boolean subalgebra.
We call images in $\!{\kappa\Sigma^1_1BorLoc}$ \defn{$\kappa\Sigma^1_1$-images}, denoted $\im^\kappa(f)$ for the image of a $\kappa$-Borel map $f$ or $f^\kappa(A)$ for the image of a $\kappa$-analytic set $A$ under a map $f$.
Note that this is the same as the $\infty(\@B_\kappa)_\delta$-image from \cref{sec:loc-im}, for a $\kappa$-Borel map that happens to be between analytic $\kappa$-Borel locales.
In particular, it agrees with existing $\kappa$-Borel images.

\begin{remark}
When $\kappa = \omega_1$, the spatialization functor $\Sp : \!{\sigma BorLoc} -> \!{\sigma Bor}$ restricts to an equivalence of categories
\begin{align*}
\!{\sigma\Sigma^1_1BorLoc} \simeq \!{\*\Sigma^1_1Bor}
\end{align*}
between analytic $\sigma$-Borel locales and analytic Borel spaces.
It follows that the localic notion of $\sigma$-analytic set also corresponds to the classical notion.

To see this, recall (\cref{thm:loc-ctbpres}) that $\Sp$ restricts to an equivalence $\!{\sigma BorLoc}_\sigma \simeq \!{SBor}$ between standard $\sigma$-Borel locales and standard Borel spaces.
It follows that analytic $\sigma$-Borel locales are spatial, since if $\~X ->> X$ is an epimorphism from a standard $\sigma$-Borel locale, then $\@B_\sigma(X) \subseteq \@B_\sigma(\~X)$ is a $\sigma$-Boolean subalgebra, hence admits enough homomorphisms to $2$ since $\@B_\sigma(\~X)$ does; thus $\Sp$ restricts to a full and faithful functor on $\!{\sigma\Sigma^1_1BorLoc}$.
Now given a Borel map $f : X -> Y$ between standard $\sigma$-Borel locales, we just need to check that the regular epi--mono factorization $X ->> \im(f) `-> Y$ of $f$ in $\!{\sigma BorLoc}$ is preserved by $\Sp$.
Since $\Sp$ is right adjoint, it preserves the mono.
To check that it preserves the epi $X ->> \im(f)$, we need to check that the induced Borel map $\Sp(f) : X -> \Sp(\im(f))$ is surjective.
Indeed, for a point $y$ in the standard Borel space $Y$, if $y \not\in \im(\Sp(f))$, then the singleton $\{y\} \in \@B_\sigma(Y)$ witnesses that the $\sigma$-Boolean homomorphism $y^* : \@B_\sigma(Y) -> 2$ does not descend to the image $\@B_\sigma(\im(f))$ of $f^* : \@B_\sigma(Y) -> \@B_\sigma(X)$, i.e., that $y \not\in \Sp(\im(f))$, as desired.
\end{remark}

We will show below (\cref{thm:sigma11-proper}) that for every $\kappa$, there are $\kappa$-analytic sets in standard $\kappa$-Borel locales which are not $\kappa$-Borel.
In particular, there are analytic $\kappa$-Borel locales which are not standard $\kappa$-Borel, i.e., the inclusion $\!{\kappa BorLoc}_\kappa \subseteq \!{\kappa\Sigma^1_1BorLoc}$ is proper.
This generalizes the classical fact when $\kappa = \omega_1$ (see \cite[14.2]{Kcdst}).

\begin{remark}
In contrast to \cref{thm:loc-mono-epi}(b), \emph{not every monomorphism in $\!{\kappa\Sigma^1_1BorLoc}$ is regular}.
This follows from the existence of $\kappa$-analytic, non-$\kappa$-Borel sets in standard $\kappa$-Borel locales $X$, since every equalizer of some $X \rightrightarrows Y \in \!{\kappa\Sigma^1_1BorLoc}$ must be a $\kappa$-Borel set, by \cref{rmk:sigma11-lim-std}.
\end{remark}

\begin{remark}
\label{rmk:sigma11-inf}
The only reason why the above definition of $\!{\kappa\Sigma^1_1BorLoc}$ requires $\kappa < \infty$ is that we defined (in \cref{sec:loc-cat}) $\infty$-Borel locales $X$ to always be standard, i.e., $\@B_\infty(X)$ must be a \emph{small-presented} large complete Boolean algebra.
As a result, the category $\!{\infty BorLoc}$ is not regular, so we could not take $\!{\infty\Sigma^1_1BorLoc}$ to be the closure under existing images.
By relaxing this restriction, we could define an \defn{analytic $\infty$-Borel locale} $X$ to mean the dual of a \emph{small-generated} large complete Boolean algebra $\@B_\infty(X)$ which embeds into a small-presented algebra $\@B_\infty(\~X)$ (note that small-generatability is still enough to ensure only a small set of $\infty$-Borel maps $Y -> X$ for any $Y$).
However, some foundational care would be needed to work with $\!{\infty\Sigma^1_1BorLoc}$ in this way.
\end{remark}

Using generalities about regular categories, we now give an equivalent ``synthetic'' description of $\!{\kappa\Sigma^1_1BorLoc}$, as a ``completion'' of the subcategory $\!{\kappa BorLoc}_\kappa$, which does not require the ambient category $\!{\kappa BorLoc}$ in which to take images.
In particular, this description works equally well when $\kappa = \infty$, thereby circumventing some of the aforementioned foundational issues.
\begin{itemize}

\item  An object $X \in \!{\kappa\Sigma^1_1BorLoc}$ is the $\kappa\Sigma^1_1$-image of some $p : \~X -> X' \in \!{\kappa BorLoc}_\kappa$, thus may be represented by such $p$.

In other words, we can \emph{define} an \defn{analytic $\kappa$-Borel locale} to mean a $\kappa$-Borel map $p : \~X -> X'$ between standard $\kappa$-Borel locales, thought of formally as representing its ``$\kappa\Sigma^1_1$-image''.

\item  Given analytic $\kappa$-Borel locales $X, Y \in \!{\kappa\Sigma^1_1BorLoc}$, represented as the ``formal $\kappa\Sigma^1_1$-images'' of $p : \~X -> X'$ and $q : \~Y -> Y' \in \!{\kappa BorLoc}_\kappa$ respectively, a morphism $f : X -> Y$ can be represented as the composite $f' : \~X --->>{p} X --->{f} Y \subseteq Y'$, which uniquely determines $f$ since $p$ is an epimorphism and $Y `-> Y'$ is a monomorphism.
\begin{equation}
\label{diag:sigma11-mor}
\begin{tikzcd}
\~X \times_{Y'} \~Y \dar[two heads, "q'"'] \drar[dashed, "f''"] \\
\~X \dar[two heads, "p"'] \drar["f'"{inner sep=1pt}] & \~Y \dar[two heads, "q"] \\
\mathllap{X' \supseteq{}} X \rar[dashed, "f"{pos=.3}] & Y \mathrlap{{} \subseteq Y'}
\end{tikzcd}
\end{equation}
Moreover, any $f' : \~X -> Y'$ descends to such $f : X -> Y$ iff
\begin{itemize}
\item[(i)]  $\ker(p) \subseteq \ker(f')$ (which ensures $f'$ descends along $p$ to $X -> Y'$, by the universal property of the coequalizer $\ker(p) \rightrightarrows \~X ->> X$), and
\item[(ii)]  the pullback of $q$ along $f'$ is a regular epi, i.e., $\kappa$-Borel surjection, $q' : \~X \times_{Y'} \~Y ->> \~X$ (which ensures that $f'$ lands in $Y \subseteq Y'$, since $q'$ lands in the pullback $f^{\prime*}(Y) \subseteq \~X$ of $Y \subseteq Y'$, which must hence be all of $\~X$ since $q'$ is an extremal epi).
\end{itemize}
Thus, we can \emph{define} a $\kappa$-Borel map $f : X -> Y$ to mean an $f' : \~X -> Y'$ obeying (i) and (ii).

\item  The identity at $X = \im^\kappa(p : \~X -> X')$ is represented by $p$.

\item  To compose $f : X -> Y$ represented by $f' : \~X -> Y'$ as above with $g : Y -> Z$ represented by $g' : \~Y -> Z'$, form the pullback $f'' : \~X \times_{Y'} \~Y -> \~Y$ of $f'$ as in \eqref{diag:sigma11-mor}, then compose it with $g'$; since $\ker(q') \subseteq \ker(f' \circ q') = \ker(q \circ f'') = f^{\prime\prime*}(\ker(q)) \subseteq f^{\prime\prime*}(\ker(g')) = \ker(g' \circ f'')$, the composite $g' \circ f'' : \~X \times_{Y'} \~Y -> Z'$ descends to $\~X -> Z'$ which represents $g \circ f$.

\end{itemize}
We may also describe subobjects in $\!{\kappa\Sigma^1_1BorLoc}$, i.e., $\kappa$-analytic sets, purely in terms of $\!{\kappa BorLoc}_\kappa$:
\begin{itemize}

\item  $\kappa$-analytic sets $A \subseteq X$ in a standard $\kappa$-Borel locale $X \in \!{\kappa BorLoc}_\kappa$ are represented by $\kappa$-Borel maps $f : \~A -> X$ from some other $\~A \in \!{\kappa BorLoc}_\kappa$, thought of as their ``formal $\kappa\Sigma^1_1$-images'', preordered by $A \subseteq B = \im^\kappa(g : \~B -> X)$ iff the pullback $\~A \times_X \~B -> \~A$ of $g$ along $f$ is $\kappa$-Borel surjective (by pullback-stability of regular epis).

$\kappa$-analytic sets in an analytic $\kappa$-Borel locale $X$ are the $\kappa$-analytic sets in a standard $\kappa$-Borel $X' \supseteq X$ which are contained in $X$.

\item  $\kappa$-ary meets are given by $\bigwedge_i \im^\kappa(f_i) = \im^\kappa(\prod_X (f_i)_i)$ (by \cref{thm:sigma11-meet}).

\item  $\kappa$-ary joins are given by image of disjoint union (by the general construction of joins of subobjects in extensive regular categories; see before \cref{thm:loc-bor-ext}).

\item  Preimage is given by pullback (by pullback-stability of image).

\item  $\kappa$-\emph{Borel} sets in an analytic $\kappa$-Borel locale $X = \im^\kappa(p : \~X -> X')$ are the restrictions of $\kappa$-Borel sets $B' \subseteq X'$, with two such sets having the same restriction iff they have the same $p$-preimage.

Alternatively, a $\kappa$-Borel set in $X$ may be uniquely represented by its $p$-preimage, which is a $\kappa$-Borel set $\~B \subseteq \~X$ which is invariant with respect to the internal equivalence relation $\ker(p) \subseteq \~B^2$ (by \cref{thm:bor-blackwell} below).

\end{itemize}

\begin{remark}
\label{rmk:sigma11-site}
The above construction of $\!{\kappa\Sigma^1_1BorLoc}$ from $\!{\kappa BorLoc}_\kappa$ is an instance of the general technique of \emph{regular completion of unary sites}; see \cite{Sksite}, \cite{CVregex}.
A \defn{unary site} is a category $\!C$ with finite limits%
\footnote{In \cite{Sksite}, a more general theory is developed where $\!C$ needs only have ``weak finite limits with respect to $\!E$''.}
equipped with a subcategory $\!E \subseteq \!C$ containing all objects in $\!C$ and whose morphisms are closed under pullback and right-cancellation with arbitrary morphisms in $\!C$; one thinks of $\!E$ as morphisms which are ``declared to be regular epimorphisms''.
The \defn{regular completion} of a unary site $(\!C, \!E)$ is the free regular category $\!C'$ generated by $\!C$ qua category with finite limits, subject to every morphism in $\!E$ becoming a regular epi in $\!C'$.
If every morphism in $\!E$ is already a regular epi in $\!C$, then the canonical unit functor $\eta : \!C -> \!C'$ is full and faithful, and $\!C'$ can be constructed explicitly as the category of ``formal images'' of morphisms in $\!C$, exactly as above (see \cite[8.22, 10.19]{Sksite}).
Moreover, $\!C'$ inherits various other existing limit and colimit properties from $\!C$, such as $\kappa$-ary limits and $\kappa$-extensivity (see \cite[11.13]{Sksite}, \cite[\S4.5]{CVregex}).

By applying this to $\!C = \!{\kappa BorLoc}_\kappa$ and $\!E =$ $\kappa$-Borel surjections (i.e., all regular epimorphisms in $\!C$), we obtain $\!{\kappa\Sigma^1_1BorLoc}$ as the free regular category obeying \cref{thm:sigma11-cat}(b--e) generated by $\!{\kappa BorLoc}_\kappa$ while remembering $\kappa$-ary limits, $\kappa$-ary coproducts, and regular epimorphisms.
In particular, \cref{thm:sigma11-cat}(b--e) can all be proved without referring to $\!{\kappa BorLoc}$, hence remaining in the small-presented context when $\kappa = \infty$.
\end{remark}

\begin{remark}
We can also describe $\!{\kappa\Sigma^1_1BorLoc}$ as the free regular category obeying \cref{thm:sigma11-cat}(b--e) generated by $\!{\kappa Loc}_\kappa$ while remembering $\kappa$-ary limits, $\kappa$-ary coproducts, and that $\kappa$-Borel surjections in $\!{\kappa Loc}_\kappa$ (equivalently by \cref{thm:loc-epi-pullback}, pullback-stable epimorphisms) become regular epimorphisms in $\!{\kappa\Sigma^1_1BorLoc}$.
This follows from combining the universal property of the regular completion as in \cref{rmk:sigma11-site} with the description of $\!{\kappa BorLoc}_\kappa$ as the localization of $\!{\kappa Loc}$ inverting the $\kappa$-continuous bijections (\cref{rmk:loc-bor-localization}).
\end{remark}

\begin{remark}
\emph{For $\kappa \le \lambda < \infty$, the forgetful functor $\!{\kappa BorLoc} -> \!{\lambda BorLoc}$ does not restrict to $\!{\kappa\Sigma^1_1BorLoc} -> \!{\lambda\Sigma^1_1BorLoc}$.}
This follows from the existence of $\sigma$-Borel maps with no $\infty$-Borel image (see \cref{thm:loc-im-bad} below, and also \cref{rmk:loc-im-bad}).

We can interpret this as saying that the embedding $\!{\kappa\Sigma^1_1BorLoc} \subseteq \!{\kappa BorLoc}$, despite offering a convenient first definition of $\!{\kappa\Sigma^1_1BorLoc}$, is in fact ``unnatural''.
In other words, a $\kappa$-analytic set $A \subseteq X$ should not really be identified with an $\infty(\@B_\kappa)_\delta$-set, i.e., subobject in $\!{\kappa BorLoc}$.
Indeed, \cref{thm:sigma11-proper} below says that there is a ``$\sigma$-analytic, non-$\infty$-Borel'' set; this only makes sense if we do not regard $\sigma\Sigma^1_1$ as a subset of $\infty(\@B_\sigma)_\delta$ (hence of $\@B_\infty$).
We will therefore mostly avoid using the embedding $\!{\kappa\Sigma^1_1BorLoc} \subseteq \!{\kappa BorLoc}$ henceforth.
\end{remark}

\begin{remark}
\label{rmk:sigma11-absolute}
Despite the preceding remark, \emph{we do not know if for $\kappa \le \lambda$, the forgetful functor $\!{\kappa BorLoc}_\kappa -> \!{\lambda BorLoc}_\lambda$ between the categories of \emph{standard} Borel locales preserves regular epimorphisms.}
If this were so, we would nonetheless have a natural forgetful functor
\begin{align*}
\!{\kappa\Sigma^1_1BorLoc} --> \!{\lambda\Sigma^1_1BorLoc},
\end{align*}
namely the extension of the forgetful functor $\!{\kappa BorLoc}_\kappa -> \!{\lambda BorLoc}_\lambda$ to the regular completions of the respective unary sites as in \cref{rmk:sigma11-site}.

In particular, for a standard Borel space $X$, \emph{we do not know if every classical analytic set $A \subseteq X$ yields an $\infty$-analytic set}.
We can take the $\infty$-analytic image of any Borel $f : \~A -> X$ with classical image $A$; but we do not know if the result depends on $f$ (even if $A = X$).
\end{remark}

\subsection{Separation theorems}
\label{sec:sigma11-sep}

The classical \defn{Novikov separation theorem} \cite[28.5]{Kcdst} states that countably many analytic sets $A_i \subseteq X$ in a standard Borel space (or analytic Borel space \cite[28.6]{Kcdst}) with empty intersection $\emptyset = \bigcap_i A_i$ must be contained in Borel sets $A_i \subseteq B_i \subseteq X$ with empty intersection $\emptyset = \bigcap_i B_i$ (the \defn{Lusin separation theorem} \cite[14.7]{Kcdst} is the special case of two analytic sets).
The localic analog is given by the interpolation theorems from \cref{sec:frm-interp}:

\begin{theorem}
\label{thm:sigma11-sep}
Let $X$ be an analytic $\kappa$-Borel locale, $A_i \subseteq X$ be $<\kappa$-many $\kappa$-analytic sets.
If $\bigcap_i A_i = \emptyset$, then there are $\kappa$-Borel $A_i \subseteq B_i \subseteq X$ with $\bigcap_i B_i = \emptyset$.

In particular, if $A_1, A_2 \subseteq X$ are disjoint $\kappa$-analytic sets, then there is a $\kappa$-Borel set $B \subseteq X$ with $A_1 \subseteq B$ and $A_2 \cap B = \emptyset$.
\end{theorem}
\begin{proof}
Let $X \subseteq X'$ with $X'$ standard $\kappa$-Borel, and
let $f_i : \~A_i ->> A_i$ be $\kappa$-Borel surjections from standard $\kappa$-Borel locales $\~A_i$.
Then (by \cref{thm:sigma11-meet}) $\bigcap_i \~A_i$ is the $\kappa$-analytic image of the wide pullback $\prod_{X'} (f_i)_i$, which must hence be empty.
By \cref{thm:bool-minterp} applied to $a := \bot \in \@B_\kappa(X')$ and $b_i := \top \in \@B_\kappa(\~A_i)$, we get that there are $B_i \in \@B_\kappa(X')$ with $\bigcap_i B_i \subseteq \emptyset$ and $\top \le f_i^*(B_i)$, i.e., $A_i = f^\kappa(\~A_i) \subseteq B_i$.
\end{proof}

\begin{remark}
In particular, this gives, modulo the Baire category-based spatiality results from \cref{sec:loc-ctbpres}, a completely lattice-theoretic proof of the classical Lusin--Novikov separation theorems.
(However, the classical proof from \cite[28.5]{Kcdst} and especially \cite[28.1, proof~II]{Kcdst} is somewhat analogous to the proof of \cref{thm:bifrm-minterp}, with the tree $\#N^{<\#N}$, which forms a posite presenting $\@O(\#N^\#N)$, playing a role analogous to that of the distributive polyposets in \cref{thm:bifrm-minterp}.)
\end{remark}

\begin{corollary}[Suslin's theorem for locales]
\label{thm:bor-suslin}
Let $X$ be an analytic $\kappa$-Borel locale.
Then $\@B_\kappa(X) = \kappa\Sigma^1_1(X)_\neg$, i.e., the $\kappa$-Borel sets in $X$ are precisely the complemented $\kappa$-analytic sets.
\qed
\end{corollary}

\begin{corollary}
\label{thm:bor-blackwell}
Let $f : X -> Y \in \!{\kappa\Sigma^1_1BorLoc}$ be a $\kappa$-Borel map between analytic $\kappa$-Borel locales.
Then a $\kappa$-Borel set $B \subseteq X$ is $f^*(C)$ for some $\kappa$-Borel $C \subseteq Y$ iff $B$ is invariant with respect to the internal equivalence relation $\ker(f) \subseteq X^2$.
\end{corollary}
The classical analog of this is a theorem of Blackwell (see \cite[14.16]{Kcdst}).
\begin{proof}
This follows from regarding $\!{\kappa\Sigma^1_1BorLoc} \subseteq \!{\kappa BorLoc}$, then replacing $Y$ with $\im^\kappa(f)$, so that $f : X ->> Y$ is a regular epi, i.e., $\@B_\kappa(Y)$ is the equalizer of $\@B_\kappa(X) \rightrightarrows \@B_\kappa(\ker(f))$, which is easily seen to consist of precisely the $\ker(f)$-invariant sets $B \in \@B_\kappa(X)$.

We also give a different proof, avoiding mention of $\!{\kappa BorLoc}$.
($\Longrightarrow$) is a general fact about categories with finite limits (e.g., reason as in $\!{Set}$ using the Yoneda lemma).
For ($\Longleftarrow$): in any regular category, $\ker(f)$-invariance of $B$ implies $B = f^*(f^\kappa(B))$; this easily yields that $f^\kappa(B) \cap f^\kappa(\neg B) = \emptyset$ (by Frobenius reciprocity), whence there is a $\kappa$-Borel $C \subseteq Y$ separating them, whence $B = f^*(C)$.
\end{proof}

\begin{remark}
For many other classical consequences of the Lusin separation theorem, we have already seen their localic analog, which ``secretly'' used the localic Lusin separation theorem in the form of the various consequences of the interpolation theorem from \cref{sec:frm-interp}.
This includes, e.g., the Lusin--Suslin theorem that injective Borel maps have Borel image (see \cref{rmk:lusin-suslin}).
\end{remark}

We also have the following dual version of the general interpolation \cref{thm:bifrm-minterp}, which generalizes \cref{thm:sigma11-sep} to the positive setting.
Recall from \cref{sec:poloc} that a standard positive $\kappa$-Borel locale $X$ has a specialization order ${\le_X} \subseteq X^2$, which is a $\kappa$-Borel partial order on $X$.

\begin{theorem}
\label{thm:bor+-sep}
Let $X$ be a standard positive $\kappa$-Borel locale, and let $A_i, B_j \subseteq X$ be $<\kappa$-many $\kappa$-analytic sets such that ``${}\bigcap_i \up A_i \cap \bigcap_j \down B_j = \emptyset$'', i.e.,
\begin{align*}
\!{\kappa\Sigma^1_1BorLoc} |= \nexists x \in X,\, a_i \in A_i,\, b_j \in B_j\, (\bigwedge_i (a_i \le x) \wedge \bigwedge_j (x \le b_j)).
\end{align*}
Then there are positive $\kappa$-Borel $A_i \subseteq C_i \subseteq X$ and negative (= complement of positive) $\kappa$-Borel $B_j \subseteq D_j \subseteq X$ such that
$
\bigcap_i C_i \cap \bigcap_j D_j = \emptyset.
$
\end{theorem}
\begin{proof}
Note first that the above formula in the internal logic should be interpreted as ``$\exists \dotsb \implies \bot$'' (which makes sense in the coherent category $\!{\kappa\Sigma^1_1BorLoc}$), or equivalently
\begin{align*}
\{(x, a_i, b_j)_{i,j} \in X \times \prod_i A_i \times \prod_j B_j \mid \bigwedge_i (a_i \le x) \wedge \bigwedge_j (x \le b_j)\}_\!{\kappa\Sigma^1_1BorLoc} = \emptyset.
\end{align*}
Let $f_i : \~A_i ->> A_i$, $g_j : \~B_j ->> B_j$ be $\kappa$-Borel surjections from standard $\kappa$-Borel $\~A_i, \~B_j$.
Then using that $1_X \times \prod_i f_i \times \prod_j g_j$ is $\kappa$-Borel surjective (by $\kappa$-complete regularity, \cref{thm:sigma11-cat}(d)), the above is equivalent to
\begin{align*}
\tag{$*$}
\{(x, \~a_i, \~b_j)_{i,j} \in X \times \prod_i \~A_i \times \prod_j \~B_j \mid \bigwedge_i (f_i(a_i) \le x) \wedge \bigwedge_j (x \le g_j(b_j))\}_{\!{\kappa BorLoc}_\kappa} = \emptyset.
\end{align*}
Regard $\~A_i, \~B_j$ as ``discrete'' positive $\kappa$-Borel locales via the free functor $\!{\kappa BorLoc} -> \!{\kappa Bor^+Loc}$, as in \cref{sec:poloc} (there denoted $D$).
Then the left-hand side of ($*$), but interpreted instead in $\!{\kappa Bor^+Loc}$, yields an order-embedded subobject of $X \times \prod_i \~A_i \times \prod_j \~B_j \in \!{\kappa Bor^+Loc}$ which the forgetful functor takes to the left-hand side of ($*$), since the specialization order $\le_X$ is by definition the forgetful functor applied to the internal order on $X \in \!{\kappa Bor^+Loc}$.
But this subobject interpreted in $\!{\kappa Bor^+Loc}$ is easily seen (e.g., as in $\!{Pos}$, using the enriched Yoneda lemma) to be the \defn{bilax pullback}
\begin{equation*}
\begin{tikzcd}
\~A_i \dar["f_i"'] \drar[phantom, "\le"{pos=.2,sloped}, "\le"{pos=.8,sloped}] &
\mathrlap{\{(x, \~a_i, \~b_j)_{i,j} \mid \bigwedge_i (f_i(a_i) \le x) \wedge \bigwedge_j (x \le g_j(b_j))\}_{\!{\kappa Bor^+Loc}}} \hphantom{D} \lar["\pi_i"'] \dar["\pi_j"] \dlar["\pi"{pos=.8,inner sep=1pt}] \\
X & \~B_j \lar["g_j"]
\end{tikzcd}
\hphantom{\bigwedge_i (f_i(a_i) \le x) \wedge \bigwedge_j (x \le g_j(b_j))}
\end{equation*}
which is dual to the bilax pushout from \cref{thm:bifrm-minterp}.
Thus, ($*$) implies that this bilax pullback is empty, i.e., $\bigwedge_i \pi_i^*(\top) \le \bigvee_j \pi_j^*(\bot)$.
By \cref{thm:bifrm-minterp}, there are $C_i, \neg D_j \in \@B^+_\kappa(X)$ such that $\bigwedge_i C_i \le \bigvee_j \neg D_j$, $\top \le f_i^*(C_i)$, and $g_j^*(\neg D_j) \le \bot$, which exactly yields the conclusion.
\end{proof}

The classical analog is a simultaneous generalization of the Novikov separation theorem and \cite[28.12]{Kcdst} (see \cref{ex:loc-sbor+}), the latter of whose localic analog (the binary case of \cref{thm:bor+-sep}, or dually \cref{thm:bifrm-interp}) we already used (via \cref{thm:loc-mono-epi}) to prove \cref{thm:loc-pos-upper}.

\subsection{Inverse limit representations}
\label{sec:sigma11-invlim}

Classically, a basic representation for a standard Borel space is as the space of infinite branches through a countably branching tree of height $\omega$; see \cite[7.8--9, 13.9]{Kcdst}.
A suitable ``fiberwise'' version of this yields a representation of any analytic set $A \subseteq X$ via a Borel family of trees $(T_x)_{x \in X}$ such that $x \in A$ iff $T_x$ has an infinite branch; see \cite[25.2--13]{Kcdst}.
In this subsection, we consider localic generalizations of such tree representations.

First, note that given a tree $T$, if we let $T_n$ be the nodes on level $n$, equipped with the predecessor maps $T_{n+1} -> T_n$ for each $n$, then the inverse limit $\projlim_n T_n$ is precisely the space of infinite branches.
Thus an arbitrary codirected diagram of sets $(T_i)_{i \in I}$, for a directed poset $I$ and ``predecessor'' functions $p_{ij} : T_i -> T_j$ for each $i \ge j \in I$, might be viewed as a generalization of a tree, with $\projlim_i T_i$ as the space of ``branches''.
Note that if $I$ is countable, then it has a cofinal monotone sequence $(i_n)_{n \in \#N}$, whence the subdiagram $(T_{i_n})_{n \in \#N}$ is a genuine tree with the same inverse limit as $(T_i)_{i \in I}$.
For $\kappa > \omega_1$, general codirected diagrams are more convenient than trees in the localic setting.

The generalization of the ``tree representation'' for a single standard $\kappa$-Borel locale is given by \cref{thm:loc-bor-upkz}, which we recall here:

\begin{proposition}
\label{thm:loc-bor-invlim}
Every standard $\kappa$-Borel locale $X$ is a limit, in $\!{\kappa BorLoc}_\kappa$, of a $\kappa$-ary codirected diagram of $\kappa$-ary sets.
\qed
\end{proposition}

Now given a family of trees $(T_x)_{x \in X}$, we can let each $T_n$ be the set of pairs $(x, t)$ where $x \in X$ and $t$ is a node at level $n$ in $T_x$.
Then each projection $p_n : T_n -> X$ is a countable-to-one function whose fiber over each $x \in X$ is the $n$th level of $T_n$, and the system of all such $T_n, p_n$, with the projections $p_{(n+1)n} : T_{n+1} -> T_n$, forms a codirected \defn{diagram over $X$} (i.e., $p_n \circ p_{(n+1)n} = p_{n+1}$), such that $\projlim_n T_n$ has fibers over each $x \in X$ consisting of the branches through $T_x$.
Thus, we can regard the system of $T_n, p_n$ as a diagrammatic representation of the family of trees $(T_x)_{x \in X}$.
\begin{equation*}
\begin{tikzcd}
\projlim_n T_n \rar["\pi_n"] \ar[drrrr, "\pi"'] & \dotsb \rar["p_{32}"] & T_2 \rar["p_{21}"] \ar[drr, "p_2"{inner sep=1pt,pos=.4}] & T_1 \rar["p_{10}"] \drar["p_1"{inner sep=2pt}] & T_0 \dar["p_0"] \\
&&&& X
\end{tikzcd}
\end{equation*}

For the localic generalization, we say that a $\kappa$-Borel map $f : X -> Y$ between $\kappa$-Borel locales is \defn{$\kappa$-ary-to-one}, or that $X$ is \defn{fiberwise $\kappa$-ary over $Y$}, if $X$ has a $\kappa$-ary cover by $\kappa$-Borel sets to which the restriction of $f$ is a monomorphism (this condition in the classical setting follows from the Lusin--Novikov uniformization theorem; see \cite[18.10]{Kcdst}).
We now have

\begin{theorem}
\label{thm:sigma11-invlim-k}
Let $f : X -> Y$ be a $\kappa$-Borel map between standard $\kappa$-Borel locales.
Then there is a $\kappa$-ary codirected diagram $(f_i : X_i -> Y)_{i \in I}$ over $Y$, for a $\kappa$-directed poset $I$ and $\kappa$-Borel maps $f_{ij} : X_i -> X_j$ for each $i \ge j \in I$, such that each $f_i$ is $\kappa$-ary-to-one and $X \cong \projlim_i X_i$ over $Y$.

In particular, every $\kappa$-analytic set $A \subseteq Y$ is the $\kappa$-analytic image of a $\kappa$-ary codirected limit of fiberwise $\kappa$-ary standard $\kappa$-Borel locales over $Y$.
\end{theorem}
\begin{proof}
By \cref{thm:loc-bor-upkz,thm:loc-mor-dissolv}, we may assume $X, Y$ are ultraparacompact zero-dimensional standard $\kappa$-locales and $f$ is ($\kappa$-)continuous.
Write $X = \projlim_i A_i$ and $Y = \projlim_j B_j$ as $\kappa$-ary codirected limits of $\kappa$-ary sets, with projections $p_i : X -> A_i$ and $q_j : Y -> B_j$.

For each $j$, we claim that the maps $(p_i, q_j \circ f) : X -> A_i \times B_j$ exhibit $X$ as the limit over $i$, in $\!{\kappa Loc}$, of the images $\im(p_i, q_j \circ f) \subseteq A_i \times B_j$ (where the image can be either closed or $\infty$-Borel, since $A_i \times B_j$ is a $\kappa$-ary set).
Indeed, the composite $(p_i, q_j \circ f)_i : X -> \projlim_i \im(p_i, q_j \circ f) \subseteq \projlim_i (A_i \times B_j) \cong X \times B_j$ is the usual embedding $(1_X, q_j \circ f) : X `-> X \times B_j$ of $X$ as the \defn{graph} of $q_j \circ f$.
The graph is a closed set in $X \times B_j$, since it is the continuous preimage, under $(q_j \circ f) \times 1_{B_j}$, of the diagonal in $B_j$ which is closed (since $B_j$ is a set).
Hence, the graph is an intersection of complements of basic clopen rectangles $\neg (p_i^*(\{a\}) \times \{b\}) = (p_i \times 1_{B_j})^*(\neg \{(a, b)\})$ for $a \in A_i$ and $b \in B_j$.
Each such basic closed set, in order to contain the graph of $q_j \circ f$, must be such that $(a, b) \not\in \im(p_i, q_j \circ f)$, and so must contain $(p_i \times 1_{B_j})^*(\im(p_i, q_j \circ f)) \supseteq \projlim_i \im(p_i, q_j \circ f)$.
Thus the graph of $q_j \circ f$ must be equal to $\projlim_i \im(p_i, q_j \circ f) \subseteq Y \times A_i$, which is hence isomorphic to $X$.

We thus have
\begin{align*}
X \cong \projlim_i \projlim_j \im(p_i, q_j \circ f)
\subseteq \projlim_i \projlim_j (A_i \times B_j)
= \projlim_i (A_i \times Y)
= X \times Y.
\end{align*}
Let $X_i := \projlim_j \im(p_i, q_j \circ f) \subseteq \projlim_j (A_i \times B_j) = A_i \times Y$, and let $f_i : X_i \subseteq A_i \times Y -> Y$ be the projection.
Since $A_i$ is a $\kappa$-ary set, clearly each $f_i$ is $\kappa$-ary-to-one.
And clearly, the limit projections $X = \projlim_{i'} X_{i'} -> X_i$, as well as the diagram maps $X_i -> X_{i'}$ for $i \ge j'$, commute with the projections to $Y$.
\end{proof}

\begin{remark}
\label{rmk:loc-im-k-bor}
Note that $\kappa$-ary-to-one maps in $\!{\kappa BorLoc}_\kappa$ have the nice property that (unlike for general $\kappa$-Borel maps) their $\kappa$-Borel images always exist and are \defn{absolute}, in the sense that they are preserved by any functor $\!{\kappa BorLoc}_\kappa -> \!C$ that preserves monomorphisms, in particular by the forgetful functors to $\!{\lambda BorLoc}_\lambda$ for $\lambda \ge \kappa$.
Equivalently, $\kappa$-ary-to-one epimorphisms $f : X ->> Y$ in $\!{\kappa BorLoc}_\kappa$ are preserved by any functor whatsoever.
This is because, given a $\kappa$-ary cover $X = \bigcup_i X_i$ such that each $f|X_i$ is a monomorphism, disjointifying the $f(X_i) \subseteq Y$ and then pulling back along the $f|X_i$ yields a $\kappa$-Borel $X' \subseteq X$ such that $f|X' : X' -> Y$ is a $\kappa$-Borel isomorphism, whose inverse is a $\kappa$-Borel section $Y `-> X$ of $f$.
\end{remark}

Let $(X_i)_{i \in I}$, with $\kappa$-Borel maps $f_{ij} : X_i -> X_j$ for $i \ge j$, be an arbitrary $\kappa$-ary codirected diagram of analytic $\kappa$-Borel spaces.
We say that the diagram is \defn{pruned} if each $f_{ij}$ is $\kappa$-Borel surjective.
By \cref{thm:sigma11-cat}(c), this implies that each projection $\pi_i : \projlim_j X_j -> X_i$ is $\kappa$-Borel surjective.
Thus if the diagram is over $Y \in \!{\kappa\Sigma^1_1BorLoc}$, then each $X_i -> Y$ already has the same $\kappa\Sigma^1_1$-image as $\projlim_i X_i$.
There is a unique pruned levelwise subdiagram $(X_i')_i$ of $(X_i)_i$, i.e., each $X_i' \subseteq X_i$, having the same inverse limit, namely $(\im^\kappa(\pi_i))_i$.

Let us also assume $\prec$ is a ``successor'' relation on $I$ which generates $\le$ as a preorder and obeys
\begin{align}
\label{eq:sigma11-invlim-succ}
j \ge i \prec k \implies \exists l\, (j \prec l \ge k);
\end{align}
motivating examples are the successor relation on $\#N$, or the relation of ``having exactly one more element'' on a finite powerset $\@P_\omega(A)$.
We define the \defn{canonical pruning} (with respect to $\prec$) of the codirected diagram $(X_i)_i$ to be the levelwise subdiagram $(X_i')_i$ given by
\begin{align*}
X_i' := \bigwedge_{j \succ i} f_{ji}^\kappa(X_j).
\end{align*}
Intuitively, we remove all ``leaves'' in the original diagram $(X_i)_i$.

\begin{lemma}
Each $f_{ij} : X_i -> X_j$ for $i \ge j$ restricts to $X_i' -> X_j'$.
\end{lemma}
\begin{proof}
By definition of $X_j'$, this means that for each $k \succ j$, we must show $f_{ij}^\kappa(X_i') \subseteq f_{kj}^\kappa(X_k)$.
By \eqref{eq:sigma11-invlim-succ}, there is some $i \prec l \ge k$, whence
$f_{ij}^\kappa(X_i^{\alpha+1})
\subseteq f_{ij}^\kappa(f_{li}^\kappa(X_l^\alpha))
= f_{kj}^\kappa(f_{lk}^\kappa(X_l^\alpha))
\subseteq f_{kj}^\kappa(X_k^\alpha)$.
\end{proof}

Clearly, $(X_i')_i = (X_i)_i$ iff $f_{ji}$ is $\kappa$-Borel surjective for each $j \succ i$; but since $\prec$ generates $\le$, this means the same for all $j \ge i$, i.e., the original diagram $(X_i)_i$ is pruned.
Also, the canonical pruning $(X_i')_i$ has the same limit as the original diagram $(X_i)_i$, since each projection $\pi_i : \projlim_j X_j -> X_i$ lands in $X_i' \subseteq X_i$, i.e., in $f_{ji}^\kappa(X_j)$ for each $j \succ i$, since $\pi_i$ factors through $\pi_j$.

We define the \defn{canonical pruning sequence} of $(X_i)_i$ to be the decreasing $\kappa$-length sequence of diagrams given by transfinitely iterating the canonical pruning:
\begin{align*}
X_i^{(0)} &:= X_i, \\
X_i^{(\alpha+1)} &:= X_i^{(\alpha)\prime}, \\
X_i^{(\alpha)} &:= \bigwedge_{\beta < \alpha} X_i^{(\beta)} \quad\text{for $\alpha$ limit}.
\end{align*}
By induction, each diagram $(X_i^{(\alpha)})_i$ in the canonical pruning sequence has the same limit as the original diagram $(X_i)_i$.
If the sequence stabilizes (at all $i$) past some $\alpha < \kappa$, then the resulting diagram $(X_i^{(\alpha)})_i$ is a pruned diagram having the same limit as $(X_i)_i$, hence must be $(\im^\kappa(\pi_i : \projlim_j X_j -> X_i))_i$.

We also say that a \defn{pruning} of $(X_i)_i$ is any levelwise subdiagram containing the canonical pruning, and that a \defn{pruning sequence} is any decreasing $\kappa$-length sequence of subdiagrams starting with $(X_i)_i$ and whose limit stages are given by meet, while a successor stage may be any pruning of the previous stage.
Intuitively, this means we remove \emph{some} ``leaves'' at each stage.
Every diagram in a pruning sequence has the same limit as the original diagram (since it is sandwiched between the limits of the original diagram and the canonical pruning sequence).
A non-canonical pruning sequence may stabilize at a non-pruned diagram; but if it stabilizes at a pruned diagram, then so must the canonical sequence (by the same or an earlier stage).

We now have the main result connecting pruning with $\kappa$-analytic sets, which intuitively says that the failure of pruning to stabilize before $\kappa$ is due to ``essential $\kappa$-analyticity'', i.e., non-$\kappa$-Borelness.

\begin{theorem}
\label{thm:sigma11-invlim-prune}
Let $(p_i : X_i -> Y)_{i \in I}$, with $\kappa$-Borel maps $f_{ij} : X_i -> X_j$ for $i \ge j$, be a $\kappa$-ary codirected diagram of analytic $\kappa$-Borel spaces over an analytic $\kappa$-Borel space $Y$, and let $\prec$ be a relation on $I$ generating $\le$ and obeying \eqref{eq:sigma11-invlim-succ} as above.
Let $X := \projlim_i X_i$ and $p := \projlim_i p_i : X -> Y$.
\begin{equation*}
\begin{tikzcd}
&& \vdots \dar[phantom, "\subseteq"{sloped}] & \vdots \dar[phantom, "\subseteq"{sloped}] \\[-1ex]
& \dotsb \rar & X_i' \dar[phantom, "\subseteq"{sloped}] \rar & X_j' \dar[phantom, "\subseteq"{sloped}] \\[-1ex]
X = \projlim_i X_i \rar["\pi_i"] \urar \ar[drrr, "p"'] &
\dotsb \rar & X_i \rar["f_{ij}"] \drar["p_i"] & X_j \dar["p_j"] \\
&&& Y
\end{tikzcd}
\end{equation*}
\begin{enumerate}

\item[(a)]  Each projection $\pi_i : X -> X_i$ has $\im^\kappa(\pi_i) \subseteq X_i$ given by the $\kappa^+$-ary meet, in the poset $\kappa\Sigma^1_1(X_i)$, of the stages $(X_i^{(\alpha)})_{\alpha < \kappa}$ of the canonical pruning sequence.

Similarly, $\im^\kappa(p) \subseteq Y$ is the $\kappa^+$-ary meet, in $\kappa\Sigma^1_1(Y)$, of the $(p_i^\kappa(X_i^{(\alpha)}))_{\alpha < \kappa}$, for any fixed $i$.

\item[(b)]  If $X = \emptyset$ (equivalently, $\im^\kappa(p) = \emptyset \subseteq Y$), then the canonical pruning sequence stabilizes (at the levelwise empty subdiagram) past some $\alpha < \kappa$.

\item[(c)]  If $\im^\kappa(p) \subseteq Y$ is $\kappa$-Borel (in $Y$), then for each $i$, the sequence of images $(p_i^\kappa(X_i^{(\alpha)}))_{\alpha < \kappa}$ stabilizes (at $\im^\kappa(p)$) past some $\alpha < \kappa$.

\end{enumerate}
\end{theorem}
\begin{proof}
(a) is most easily seen via the embedding $\!{\kappa\Sigma^1_1BorLoc} \subseteq \!{\kappa BorLoc}$ into the category of \emph{all} $\kappa$-Borel locales from \cref{sec:sigma11-cat}.
Since the dual algebraic category $\!{\kappa Bool}$ is locally $\kappa$-presentable, all of the $\kappa$-ary limit and colimit operations in $\!{\kappa BorLoc}$ used to compute the canonical pruning $(X_i')_i$ commute with $\kappa$-codirected limits (see \cite[1.59]{ARlpac}).
Thus, if we continue the canonical pruning sequence (in $\!{\kappa BorLoc}$) past $\kappa$ by defining $X_i^{(\kappa)} := \bigwedge_{\alpha < \kappa} X_i^{(\alpha)}$, then $X_i^{(\kappa)\prime} = \bigwedge_{\alpha < \kappa} X_i^{(\alpha)\prime} = X_i^{(\kappa)}$, i.e., the sequence necessarily stabilizes past $\kappa$.
Now similarly to above, \cref{thm:loc-bor-epi-invlim} implies that the stabilized value $X_i^{(\kappa)}$ must be $\im^\kappa(\pi_i)$.

For the last part of (a), again using that images commute with $\kappa$-codirected limits, the image $\im^\kappa(p) = p_i^\kappa(\im^\kappa(\pi_i)) = p_i^\kappa(X_i^{(\kappa)})$ must be the limit of the images of the stages $p_i^\kappa(X_i^{(\alpha)})$.

(For $\kappa = \infty$, one can either work with large $\infty$-Boolean algebras in a bigger universe, as in \cref{rmk:sigma11-inf}, or unravel the above argument to remove the need for $\!{\kappa BorLoc}$.)

(b)  If $\kappa = \omega$ then the diagram is finite and has an initial object which is the same as the limit $X$; so assume $\kappa \ge \omega_1$.
By replacing $Y$ with some standard $\kappa$-Borel locale $Y' \supseteq Y$, we may assume $Y$ is standard $\kappa$-Borel to begin with.
We first complete the proof of (b) assuming that each $X_i$ is also standard $\kappa$-Borel.
By \cref{thm:loc-bor-loc}, we may assume each $X_i$ is a standard $\kappa$-locale.
By \cref{thm:loc-mor-dissolv}, we may find for each $X_i$ a partial dissolution making all the $f_{ij}$ for $j \le i$ continuous.
By replacing $X_i$ with this partial dissolution, repeating $\omega$ many times, then taking the limit (using $\kappa \ge \omega_1$) of the resulting sequence of dissolutions, we get that each $f_{ij}$ may be assumed to be continuous.
Now define a pruning sequence $(\-X_i^{(\alpha)})_{i,\alpha}$ the same way as the canonical pruning sequence, except replacing $\kappa$-analytic images with closed images:
\begin{align*}
\-X_i^{(0)} &:= X_i, \\
\-X_i^{(\alpha+1)} &:= \bigwedge_{j \succ i} \-f_{ji}(\-X_j^{(\alpha)}), \\
\-X_i^{(\alpha)} &:= \bigwedge_{\beta < \alpha} \-X_i^{(\beta)} \quad\text{for $\alpha$ limit}.
\end{align*}

\begin{lemma}
Let $(a_\alpha)_{\alpha < \gamma}$ be a strictly increasing or decreasing transfinite sequence in a $\kappa$-generated $\bigvee$-lattice or $\bigwedge$-lattice $A$.
Then $\gamma < \kappa$.
\end{lemma}
\begin{proof}
Let $B \subseteq A$ be a $\kappa$-ary generating set.
If $A$ is a $\bigvee$-lattice, then $(\{b \in B \mid b \le a_\alpha\})_\alpha$ is a strictly increasing or decreasing sequence in $\@P(B)$, hence has length $\gamma<\kappa$.
Dually if $A$ is a $\bigwedge$-lattice.
\end{proof}

Since each $X_i$ is a standard $\kappa$-locale, and each $\-X_i^{(\alpha)}$ is closed, it follows from this lemma that the pruning sequence $(\-X_i^{(\alpha)})_{i,\alpha}$ stabilizes past some $\alpha < \kappa$.
By definition of $\-X_i^{(\alpha+1)}$, this means each $f_{ij} : \-X_i^{(\alpha)} -> \-X_j^{(\alpha)}$ has dense image.
By \cref{thm:loc-baire-invlim}, so does each projection $\pi_i : X = \projlim_j \-X_j^{(\alpha)} -> \-X_i^{(\alpha)}$.
Since $X = \emptyset$, this means each $\-X_i^{(\alpha)} = \emptyset$.
Hence the canonical pruning sequence also stabilizes by or before $\alpha$ at $\emptyset$.

To prove (b) in the general case where each $X_i$ is an analytic $\kappa$-Borel locale:
for each $i$, let $q_i : \~X_i ->> X_i$ be a $\kappa$-Borel surjection from a standard $\kappa$-Borel $\~X_i$.
Let $\^X_i$ be the limit of the diagram consisting of the full subdiagram of $(X_j)_j$ consisting of $X_j$ for all $j \le i$, as well as the $\~X_j$ and $q_j$ for all $j \ge i$, i.e.,
\begin{align*}
\^X_i := \set*{(\~x_j)_j \in \prod_{j \le i} \~X_i}{\bigwedge_{j \le i} (q_j(\~x_j) = f_{ij}(q_i(\~x_i)))}_\!{\kappa\Sigma^1_1BorLoc}.
\end{align*}
Since each $\~X_i$ is standard $\kappa$-Borel, so is $\^X_i$ (see \cref{rmk:sigma11-lim-std}).
We have obvious projections $\^f_{ij} : \^X_i -> \^X_j$ for $i \ge j$, as well as $r_i : \^X_i -> \~X_i$.
Let $\^q_i := q_i \circ r_i : \^X_i -> X_i$.
Then the $\^X_i$ together with the maps $\^f_{ij}$ form a codirected diagram of standard $\kappa$-Borel locales; and the $\^q_i$ form a natural transformation $(\^X_i)_i -> (X_i)_i$.
Let $\^X := \projlim_i \^X_i$.
\begin{equation*}
\tag{$*$}
\begin{tikzcd}
\^X \dar \rar["\^\pi_i"] & \dotsb \rar & \^X_i \dar["\^q_i"'] \rar["\^f_{ij}"] & \^X_j \dar["\^q_j"] \\
X \rar["\pi_i"] & \dotsb \rar & X_i \rar["f_{ij}"] & X_j
\end{tikzcd}
\end{equation*}

We claim that each $\^q_i$ is $\kappa$-Borel surjective, and moreover, for each $i \ge j$,
\begin{align*}
\tag{$\dagger$}
(\^q_i, \^f_{ij}) : \^X_i ->> X_i \times_{X_j} \^X_j \text{ is $\kappa$-Borel surjective.}
\end{align*}
To see this, for each $J \subseteq I$ with a greatest element $i \in J$, define $\^X_J$ in a similar manner to $\^X_i$ above, except restricting to $j \in J$.
Thus $\^X_i = \^X_{\down i}$, $\~X_i = \^X_{\{i\}}$, and $r_i$ is the projection $\^X_{\down i} -> \^X_{\{i\}}$.
We now claim that for $K \subseteq J \subseteq I$ with the same greatest element $i$, the projection $\^X_J -> \^X_K$ is $\kappa$-Borel surjective; this is because, by \cref{thm:loc-bor-epi-invlim}, it suffices to consider the case where $J = K \cup \{j\}$ for some $j \le i$, in which case the projection is the pullback of $q_i : \~X_j ->> X_j$ along $\~X_J -> X_i --->{f_{ij}} X_j$.
It follows that $r_i$ is $\kappa$-Borel surjective, whence so is $\^q_i = q_i \circ r_i$.
It also follows that for $j < i$, the projection $\^X_i = \^X_{\down i} -> \^X_{\{i\} \cup \down j} = \~X_i \times_{X_j} \^X_{\down j} = \~X_i \times_{X_j} \^X_j$ is $\kappa$-Borel surjective; composing with the pullback of $q_i : \~X_i ->> X_i$ yields ($\dagger$).

It now follows that each of the squares in ($*$) for $i \ge j$ obeys the ``Beck--Chevalley equation''
\begin{align*}
\^f_{ij}^\kappa(\^q_i^*(A)) &= \^q_j^*(f_{ij}^\kappa(A))
\quad \forall A \in \kappa\Sigma^1_1(X_i).
\end{align*}
Indeed, the pullback square for $f_{ij}, \^q_j$ obeys the Beck--Chevalley equation by regularity of $\!{\kappa\Sigma^1_1BorLoc}$ (see \cref{thm:loc-bor-im-pullback}); and by ($\dagger$), we may insert the preimage from the pullback to $\^X_i$ followed by the $\kappa\Sigma^1_1$-image into the left-hand side with no effect.
This implies that the computation of the canonical pruning of $(X_i)_i$, which only involves $\kappa\Sigma^1_1$-image and $\bigwedge$, may be pulled back to $(\^X_i)_i$, i.e.,
\begin{align*}
\^X_i^{(\alpha)} = \^q_i^*(X_i^{(\alpha)})
\end{align*}
for all $i, \alpha$.
Since $\^q_i$ is $\kappa$-Borel surjective, this implies
\begin{align*}
\^q_i^\kappa(\^X_i^{(\alpha)}) = X_i^{(\alpha)}.
\end{align*}
Now since $X = \emptyset$ and there is a map $\^X -> X$, also $\^X = \emptyset$, whence by the proof of (b) for the standard $\kappa$-Borel case, there is some $\alpha < \kappa$ by which all the $\^X_i^{(\alpha)}$ become empty, whence so do all the $X_i^{(\alpha)}$, completing the proof of (b).

(c) follows by pulling everything back along $\neg \im^\kappa(p) \subseteq Y$, noting that the definition of the canonical pruning is pullback-stable (by regularity of $\!{\kappa\Sigma^1_1BorLoc}$), then applying (b).
\end{proof}

\begin{corollary}
Every $\kappa$-analytic set in a standard $\kappa$-Borel locale $Y$ is a meet, in $\kappa\Sigma^1_1(Y)$, of a $\kappa$-length decreasing sequence of $\kappa$-Borel sets.
\end{corollary}
\begin{proof}
Apply \cref{thm:sigma11-invlim-prune}(a) to a diagram obtained from \cref{thm:sigma11-invlim-k}, where each of the $X_i$ is a fiberwise $\kappa$-ary standard $\kappa$-Borel locale over $Y$, whence all of the stages of the canonical pruning sequence are $\kappa$-Borel, as are their images in $Y$ (see \cref{rmk:loc-im-k-bor}).
\end{proof}

The above corollary can be seen as a generalization of the classical \defn{Lusin--Sierpinski theorem} that every analytic set in a standard Borel space is an $\omega_1$-length decreasing intersection of Borel sets; see \cite[25.16]{Kcdst}.

\begin{remark}
\label{rmk:sigma11-invlim-bounded}
\Cref{thm:sigma11-invlim-prune}(b) can be seen as a generalization of the classical \defn{boundedness theorems for $\*\Sigma^1_1$}; see \cite[31.A]{Kcdst}.
For example, one version says that any analytic well-founded binary relation has rank $<\omega_1$.
For $\kappa \ge \omega_1$, we have the following generalization: take a $\kappa$-analytic binary relation $R \subseteq Y^2$ in an analytic $\kappa$-Borel locale $Y$, which is internally well-founded in that
\begin{align*}
\!{\kappa\Sigma^1_1BorLoc} |= \forall (y_0, y_1, \dotsc) \in Y^\#N\, (R(y_1, y_0) \wedge R(y_2, y_1) \wedge \dotsb \implies \bot).
\end{align*}
The set
$\{(x_0, x_1, \dotsc) \mid R(y_1, y_0) \wedge R(y_2, y_1) \wedge \dotsb\}_\!{\kappa\Sigma^1_1BorLoc} \subseteq Y^\#N$
is the limit of the inverse sequence of sets of finite sequences
$X_n := \{(x_0, x_1, \dotsc, x_n) \mid R(y_1, y_0) \wedge \dotsb \wedge R(y_n, y_{n-1})\}_\!{\kappa\Sigma^1_1BorLoc} \subseteq Y^{n+1}$
under the projection maps $X_{n+1} -> X_n$, whence by \cref{thm:sigma11-invlim-prune}(b), we get that the pruning of this inverse sequence must stop before $\kappa$, which can be thought of as an internal way of saying that ``$R$ has rank $<\kappa$''.
\end{remark}

\subsection{Ill-founded relations}
\label{sec:sigma11-proper}

We now prove that there exist $\kappa$-analytic, non-$\infty$-Borel sets.
Classically, the existence of analytic, non-Borel sets may be shown using (one version of) the boundedness theorem for $\*\Sigma^1_1$: if the set of ill-founded binary relations $R \subseteq \#N^2$ were Borel, then the pruning of minimal elements for all $R \subseteq \#N^2$ must stop by some fixed stage $\alpha < \omega_1$, contradicting that there are well-founded $R$ with arbitrarily high rank $<\omega_1$.
Our goal is to generalize a version of this proof.

Let $X$ be an infinite set.
For each $n \in \#N$, let
\begin{align*}
\DESC(X, n) &:= \{(R, x_0, \dotsc, x_{n-1}) \in \#S^{X^2} \times X^n \mid x_{n-1} \mathrel{R} \dotsb \mathrel{R} x_0\}
\end{align*}
be the space of binary relations on $X$ together with a descending sequence of length $n$, an open subspace of $\#S^{X^2} \times X^n$ (where $X$ is discrete).
For $n \ge 1$, let $p_n : \DESC(X, n) -> \DESC(X, n-1)$ be the projection omitting the last coordinate.
If $X$ is $\kappa$-ary, then clearly (the underlying locale of) each $\DESC(X, n)$ is a standard $\kappa$-locale, and $p_n$ is $\kappa$-ary-to-one as witnessed by the clopen partition
\begin{align*}
\DESC(X, n) = \bigsqcup_{x \in X} \{(R, x_0, \dotsc, x_{n-1}) \in \DESC(X, n) \mid x_{n-1} = x\}
\end{align*}
each piece to which the restriction of $p_n$ is an open embedding.
In particular, every $\kappa$-Borel $B \subseteq \DESC(X, n)$ has a $\kappa$-Borel image under $p_n$ which is also the $\infty$-Borel image (see \cref{rmk:loc-im-k-bor}).
It follows that the $\kappa$-length canonical pruning sequence, as defined in the preceding subsection, of $(\DESC(X, n))_n$, regarded as a diagram in $\!{\kappa\Sigma^1_1BorLoc}$ with the maps $p_n$, consists of standard $\kappa$-Borel locales at each stage $<\kappa$, and forms an initial segment of the $\infty$-length canonical pruning sequence when $(\DESC(X, n))_n$ is instead regarded as a diagram in $\!{\infty\Sigma^1_1BorLoc}$.

\begin{lemma}
\label{thm:if-collapse}
For any two infinite $\kappa$-ary sets $X, Y$, there is a nonempty standard $\kappa$-locale $Z$ together with a natural isomorphism in $\!{\kappa Loc}$ of diagrams $(f_n : Z \times \DESC(X, n) \cong Z \times \DESC(Y, n))_{n \in \#N}$ over $Z$.
\begin{equation*}
\tag{$*$}
\begin{tikzcd}
\dotsb \rar & Z \times \DESC(X, 2) \dar["\cong"', "f_2"] \rar["p_2"] & Z \times \DESC(X, 1) \dar["\cong"', "f_1"] \rar["p_1"] & Z \times \DESC(X, 0) \dar["\cong"', "f_0"] \\
\dotsb \rar & Z \times \DESC(Y, 2) \drar \rar["p_2"] & Z \times \DESC(Y, 1) \dar \rar["p_1"] & Z \times \DESC(Y, 0) \dlar \\
&& |[xshift=3em]| Z
\end{tikzcd}
\end{equation*}
\end{lemma}
\begin{proof}
By \cref{thm:collapse}, there is a nonempty standard $\kappa$-locale $Z$ ``collapsing $\abs{X}, \abs{Y}$'', so that there is an isomorphism $g : Z \times X \cong Z \times Y \in \!{Loc}$ over $Z$.
The result now follows by internalizing the definitions of $\DESC(X, n), p_n$ over $Z$.
Namely, let $g^n_Z : Z \times X^n \cong Z \times Y^n$ be the $n$-fold fiber product of $g$ over $Z$, and let
\begin{align*}
\#S^{g^n_Z}_Z : Z \times \#S^{X^n} &--> Z \times \#S^{Y^n} \in \!{Loc}
\end{align*}
be the ``fiberwise exponential of $g^n_Z$ over $Z$'', i.e., $\#S^{g^n_Z}_Z$ commutes with the projections to $Z$, while for the $\vec{y}$th subbasic open set in $\#S^{Y^n}$ where $\vec{y} \in Y^n$, i.e., the $\vec{y}$th free generator of $\ang{Y^n}_\!{Frm}$, we put
\begin{align*}
(\#S^{g^n_Z}_Z)^*(\{(z, S) \in Z \times \#S^{Y^n} \mid S(\vec{y})\})
&:= \{(z, R) \in Z \times \#S^{X^n} \mid \bigvee_{\vec{x} \in X^n} ((g^n_Z(z, \vec{x}) = (z, \vec{y})) \wedge R(\vec{x}))\}.
\end{align*}
(Here we are using the internal logic in $\!{Loc}$, where $R(\vec{x})$ is really an abbreviation for the atomic formula $[\vec{x}](R)$ where $[\vec{x}] \subseteq \#S^{X^n}$ is the $\vec{x}$th subbasic open set.
Note that the formula $g^n_Z(z, \vec{x}) = (z, \vec{y})$ defines an open set of $z$, since $\{\vec{y}\} \subseteq Y^n$ is open.)
Similarly, let
\begin{align*}
\#S^{(g^n_Z)^{-1}}_Z : Z \times \#S^{Y^n} &--> Z \times \#S^{X^n} \in \!{Loc}
\end{align*}
be the ``fiberwise exponential of $(g^n_Z)^{-1}$ over $Z$''.
Using the internal logic, it is straightforward to check that these two maps are inverses:
\begin{align*}
\MoveEqLeft
(\#S^{(g^n_Z)^{-1}}_Z)^*((\#S^{g^n_Z}_Z)^*(\{(z, S) \mid S(\vec{y})\})) \\
&= \set[\Big]{(z, S)}{\bigvee_{\vec{x} \in X^n} \paren[\Big]{\paren[\big]{g^n_Z(z, \vec{x}) = (z, \vec{y})} \wedge \paren[\big]{(z, S) \in (\#S^{(g^n_Z)^{-1}}_Z)^*(\{(z, R) \mid R(\vec{x})\})}}} \\
&= \set[\big]{(z, S)}{\bigvee_{\vec{x} \in X^n} \paren[\big]{\paren{g^n_Z(z, \vec{x}) = (z, \vec{y})} \wedge \bigvee_{\vec{y}' \in Y^n} \paren{((g^n_Z)^{-1}(z, \vec{y}') = (z, \vec{x})) \wedge S(\vec{y'})}}} \\
&= \set{(z, S)}{\bigvee_{\vec{x} \in X^n} \paren{\paren{g^n_Z(z, \vec{x}) = (z, \vec{y})} \wedge S(\vec{y})}}
= \{(z, S) \mid S(\vec{y})\}
\end{align*}
whence $\#S^{g^n_Z}_Z \circ \#S^{(g^n_Z)^{-1}}_Z = 1$, and similarly $\#S^{(g^n_Z)^{-1}}_Z \circ \#S^{g^n_Z}_Z = 1$.
Now let
\begin{align*}
f_n := \#S^{g^2_Z}_Z \times_Z g^n_Z : Z \times \#S^{X^2} \times X^n \cong Z \times \#S^{Y^2} \times Y^n.
\end{align*}
It is again straightforward to check using the internal logic that these maps commute with the projections $p_n : Z \times \#S^{X^2} \times (-)^n -> Z \times \#S^{X^2} \times (-)^{n-1}$ as in the diagram ($*$), and that they take $Z \times \DESC(X, n)$ to $Z \times \DESC(Y, n)$:
\begin{align*}
\MoveEqLeft
f_n^*(Z \times \DESC(Y, n)) \\
&= f_n^*(\{(z, S, y_0, \dotsc, y_{n-1}) \mid y_{n-1} \mathrel{S} \dotsb \mathrel{S} y_0\}) \\
&= \set[\Big]{(z, R, \vec{x})}{\bigvee_{\vec{y} \in Y^n} \paren[\Big]{\paren[\big]{g^n_Z(z, \vec{x}) = (z, \vec{y})} \wedge \paren[\big]{(z, R) \in (\#S^{g^2_Z}_Z)^*(\{(z, S) \mid y_{n-1} \mathrel{S} \dotsb \mathrel{S} y_0\})}}} \\
&= \set[\big]{(z, R, \vec{x})}{\bigvee_{\vec{y} \in Y^n} \paren[\big]{(g^n_Z(z, \vec{x}) = (z, \vec{y})) \wedge \bigvee_{\vec{x}' \in X^n} ((g^n_Z(z, \vec{x'}) = (z, \vec{y})) \wedge (x_{n-1}' \mathrel{R} \dotsb \mathrel{R} x_0'))}} \hspace{-1em} \\
&= \set{(z, R, \vec{x})}{\bigvee_{\vec{y} \in Y^n} \paren{(g^n_Z(z, \vec{x}) = (z, \vec{y})) \wedge (x_{n-1} \mathrel{R} \dotsb \mathrel{R} x_0)}} \\
&= \set{(z, R, \vec{x})}{x_{n-1} \mathrel{R} \dotsb \mathrel{R} x_0}
= Z \times \DESC(X, n).
\qedhere
\end{align*}
\end{proof}

\begin{theorem}
\label{thm:if-strict}
For any infinite set $X$, the $\infty$-length canonical pruning sequence of $(\DESC(X, n))_n$ is strictly decreasing at each level $n$: for each $n$ and ordinal $\alpha < \infty$, we have
\begin{align*}
\DESC(X, n)^{(\alpha)} \supsetneq \DESC(X, n)^{(\alpha+1)}.
\end{align*}
\end{theorem}
\begin{proof}
Let $Y$ be an infinite set of cardinality $\ge \abs{\alpha}$, and let $Z, f_n$ be given by \cref{thm:if-collapse}.
By pullback-stability along $Z -> 1$, the canonical pruning sequence of $(Z \times \DESC(X, n))_n$ is the product of $Z$ with the canonical pruning sequence of $(\DESC(X, n))_n$.
By the natural isomorphism $(f_n)_n$, it is also isomorphic (in $\!{\infty BorLoc}$) to the canonical pruning sequence of $(Z \times \DESC(Y, n))_n$, which again by pullback-stability, is the product of $Z$ with that of $(\DESC(Y, n))_n$.
So it suffices to show that $\DESC(Y, n)^{(\alpha)} \supsetneq \DESC(Y, n)^{(\alpha+1)}$ for each $n, \alpha$, since that implies $Z \times \DESC(Y, n)^{(\alpha)} \supsetneq Z \times \DESC(Y, n)^{(\alpha+1)}$ because $Z \ne \emptyset$ (see \cref{thm:bool-coprod}), which in turn implies $\DESC(X, n)^{(\alpha)} \supsetneq \DESC(X, n)^{(\alpha+1)}$.

Since each $p_n$ is $\infty$-ary-to-one, by \cref{rmk:loc-im-k-bor}, the construction of the canonical pruning sequence of $(\DESC(Y, n))_n$ is preserved by the spatialization functor $\Sp : \!{\infty BorLoc} -> \!{Set} \subseteq \!{\infty Bor}$.
An easy induction shows that if $S \subseteq Y^2$ is the predecessor relation of a forest of height $\omega$, and $y_{n-1} \mathrel{S} \dotsb \mathrel{S} y_0$ is a finite path through $S$, then $(S, y_0, \dotsc, y_{n-1}) \in \Sp(\DESC(Y, n)^{(\alpha)})$ iff the subtree below $y_{n-1}$ (or the entire forest, if $n = 0$) has rank $\ge \alpha$.
Thus if we let $S$ be a forest of rank $\alpha$, together with a finite terminal segment $y_{n-1} \mathrel{S} \dotsb \mathrel{S} y_0$ attached if $n \ge 1$, then $(S, y_0, \dotsc, y_{n-1})$ is a point in $\DESC(Y, n)^{(\alpha)} \setminus \DESC(Y, n)^{(\alpha+1)}$.
\end{proof}

Let $\DESC(X, \omega) := \projlim_n \DESC(X, n)$, with projections $\pi_n : \DESC(X, \omega) -> \DESC(X, n)$.
Clearly
\begin{align*}
\DESC(X, \omega) \cong \{(R, x_0, x_1, \dotsc) \in \#S^{X^2} \times X^\#N \mid \bigwedge_n (x_{n+1} \mathrel{R} x_n)\}_\!{Loc}
\end{align*}
is the locale ``of binary relations on $X$ with an infinite descending sequence''.

\begin{theorem}
\label{thm:sigma11-proper}
There exists a continuous map $\pi_0 : \DESC(\#N, \omega) -> \DESC(\#N, 0)$ between quasi-Polish spaces which does not have a $\kappa$-Borel image for any $\omega_1 \le \kappa \le \infty$.

In particular, for all $\omega_1 \le \kappa \le \infty$, we have $\@B_\kappa(\#S^\#N) \subsetneq \kappa\Sigma^1_1(\#S^\#N)$.
\end{theorem}
\begin{proof}
By \cref{thm:if-strict}, $(p_1(\DESC(\#N, 1)^{(\alpha)}))_\alpha = (\DESC(\#N, 0)^{(\alpha+1)})_\alpha$ does not stabilize below any $\kappa$, whence by \cref{thm:sigma11-invlim-prune}(c), the $\kappa\Sigma^1_1$-image $\im^\kappa(\pi_0) \subseteq \DESC(\#N, 0)$ is not $\kappa$-Borel.

The last statement follows because $\DESC(\#N, 0) = \#S^{\#N^2} \cong \#S^\#N$.
\end{proof}

\begin{corollary}[Gaifman--Hales]
\label{cor:gaifman-hales-2}
$\ang{\#N}_\!{CBOOL}$ is a proper class.
\end{corollary}
\begin{proof}
If $\ang{\#N}_\!{CBOOL} = \@B_\infty(\#S^\#N)$ were a set, it would be a small complete lattice, whence every $\infty$-Borel map $X -> \#S^\#N$ would have an $\infty$-Borel image.
\end{proof}

The following consequence was mentioned in \cref{rmk:loc-im-bad}:

\begin{corollary}
\label{thm:loc-im-bad}
There exist $\kappa = \omega_1 < \lambda = (2^{\aleph_0})^+$ and a $\sigma$-Borel map $\pi_0 : \DESC(\#N, \omega) -> \DESC(\#N, 0) \in \!{\kappa BorLoc}_\kappa$ between standard $\kappa$-Borel locales, such that the (regular) epi--mono factorization of $\pi_0$ in $\!{\kappa BorLoc}$ is not preserved by the forgetful functor $\!{\kappa BorLoc} -> \!{\lambda BorLoc}$.

In particular, the forgetful functor $\!{\kappa BorLoc} -> \!{\lambda BorLoc}$ does not preserve (regular) epimorphisms, and does not map the embedding $\!{\kappa\Sigma^1_1BorLoc} \subseteq \!{\kappa BorLoc}$ into $\!{\lambda\Sigma^1_1BorLoc} \subseteq \!{\lambda BorLoc}$.
\end{corollary}
\begin{proof}
Since $\abs{\@B_\sigma(\DESC(\#N, 0))} = 2^{\aleph_0}$ (because $\@B_\sigma(\DESC(\#N, 0))$ is standard Borel), the $\infty(\@B_\sigma)_\delta$-image of $\pi_0$ is the $\lambda$-ary intersection of all $B \in \@B_\sigma(\DESC(\#N, 0))$ such that $\pi_0^*(B) = \top$, hence is $\lambda$-Borel.
Thus it cannot also be the $\infty(\@B_\lambda)_\delta$-image, or else it would be the $\lambda$-Borel image, which means the $\infty(\@B_\lambda)_\delta$-image must be strictly smaller.
\end{proof}

We also state the purely Boolean-algebraic dual result, which is new as far as we know, and somewhat surprising given the good behavior of monomorphisms in $\!{\kappa Bool}$ in other respects (e.g., \cref{thm:bool-pushout-inj,thm:bool-mono-reg}):

\begin{corollary}
\label{thm:bool-free-mono-bad}
There exist $\kappa = \omega_1 < \lambda = (2^{\aleph_0})^+$ such that the free functor $\!{\kappa Bool} -> \!{\lambda Bool}$ does not preserve injective homomorphisms.
\qed
\end{corollary}

\begin{remark}
The easiest classical proof of the existence of analytic, non-Borel sets is by diagonalizing a \defn{universal analytic set} $A \subseteq 2^\#N \times 2^\#N$, i.e., an analytic set whose vertical fibers over each $x \in 2^\#N$ yield all analytic sets in $2^\#N$; see \cite[14.2]{Kcdst}.
Such a technique cannot directly yield the full strength of \cref{thm:sigma11-proper}, since every locale has only a small set of points, whereas $2^\#N$ has a proper class of $\infty$-analytic sets (since it has a proper class of $\infty$-Borel sets), so there cannot be a universal $\infty$-analytic set.
However, for each \emph{fixed} non-inaccessible $\kappa \ge \omega_1$, say with $\lambda < \kappa$ such that $2^\lambda \ge \kappa$, a universal $\kappa$-analytic set in $2^\lambda \times 2^\lambda$ can probably be used to show the weaker statement that $\@B_\kappa(X) \subsetneq \kappa\Sigma^1_1(X)$ for some $\kappa$-locale $X = 2^\lambda$ depending on $\kappa$.
\end{remark}

\appendix
\section{Appendix: nice categories of structures}
\label{sec:cat}

In this appendix, we briefly review the basic theory of locally presentable and monadic categories, as well as their ordered analogs.
These are well-behaved categories of (infinitary) first-order structures which are sufficiently ``algebraic'', meaning they are defined by axioms of certain restricted forms, such that the usual universal categorical constructions are available: products, coproducts, quotients, presentations via generators and relations, etc.

See \cref{sec:frm} for our general conventions on categories and cardinalities.

\subsection{Limit theories and locally presentable categories}
\label{sec:cat-lim}

Let $\@L$ be an infinitary first-order signature, consisting of function and relation symbols, each with an associated arity which may be an arbitrary set (usually a cardinal).%
\footnote{The general theory also allows multi-sorted signatures; see below.}
If every symbol in $\@L$ has arity of size $<\kappa$, then we say that $\@L$ is $\kappa$-ary.
For now, we require $\@L$ to be a set, rather than a proper class (but see the end of this subsection).

\begin{example}
The $\omega$-ary signature of distributive lattices may be taken as $\{\bot, \vee, \top, \wedge, \le\}$ where $\bot, \top$ are nullary function symbols, $\vee, \wedge$ are binary relation symbols, and $\le$ is a binary relation symbol.
We may also omit $\le$, since it can be defined from $\vee$ or $\wedge$.
\end{example}

\begin{example}
The $\omega_1$-ary signature of $\sigma$-frames may be taken as $\{\bot, \bigvee, \top, \wedge, \le\}$, where $\bigvee$ has arity $\omega$ and the other symbols are as before.
\end{example}

For $\kappa < \infty$ and a $\kappa$-ary signature $\@L$, a \defn{$\kappa$-limit axiom} over $\@L$ is an axiom (in the infinitary first-order logic $\@L_{\kappa\kappa}$) of the form
\begin{align*}
\forall \vec{x} \left(\bigwedge_i \phi_i(\vec{x}) -> \exists! \vec{y}\, \bigwedge_j \psi_j(\vec{x}, \vec{y})\right)
\tag{$*$}
\end{align*}
where $\phi_i, \psi_j$ are atomic formulas (i.e., equations or atomic relations between terms), $\bigwedge_i, \bigwedge_j$ are $\kappa$-ary conjunctions, and $\vec{x}, \vec{y}$ are $\kappa$-ary families of variables.
A \defn{$\kappa$-limit theory} is a set of $\kappa$-limit axioms.
The category of all (set-based) models of a $\kappa$-limit theory, together with all homomorphisms between them, is called a \defn{locally $\kappa$-presentable category}.  See \cite{ARlpac} for a comprehensive reference.

\begin{example}
\label{ex:cat-lim-alg}
An \defn{algebraic theory} is a universal-equational theory, i.e., $\@L$ has only function symbols, and each axiom as in ($*$) has empty left-hand side (no $\phi_i$'s) and empty $\vec{y}$.
Examples include the theories of $\kappa$-frames, $\kappa$-$\bigvee$-lattices, $\kappa$-Boolean algebras, or groups, rings, etc.
\end{example}

\begin{example}
\label{ex:cat-lim-horn}
More generally, a \defn{universal Horn theory} is a limit theory with no existentials ($\vec{y}$ in ($*$) is empty), e.g., the theory of posets.
\end{example}
 
\begin{example}
\label{ex:cat-lim-essalg}
More generally, we may allow structures with partial operations, where the domain of each operation must be defined by a conjunction of atomic formulas; such theories are sometimes called \defn{essentially algebraic}.
To axiomatize such structures via a limit theory, replace each partial operation with its graph relation, and use an axiom with $\exists!$ to say that it is a function.
For example, $\wedge$-lattices with pairwise disjoint $\kappa$-ary joins (disjunctive $\kappa$-frames, as defined in \cref{sec:upkzfrm}; see \cref{ex:cat-ord-djfrm} below) may be axiomatized in this way.
\end{example}

Let $\!C$ be a locally $\kappa$-presentable category of structures, axiomatized by some $\kappa$-limit theory.
Then $\!C$ has arbitrary small limits and colimits.
Limits of structures are constructed in the usual way, by taking the limit of the underlying sets in $\!{Set}$, equipped with the coordinatewise operations and relations.
Colimits are generally not constructed as in $\!{Set}$; however, $\kappa$-directed colimits (meaning colimits of diagrams indexed by a $\kappa$-directed preordered set) are constructed as in $\!{Set}$, with a relation holding in the colimit iff it already holds at some stage of the diagram.

Structures in $\!C$ may be presented via generators and relations: for any set $G$ of generators, and any set $R$ of relations between the elements of $G$ (i.e., atomic formulas with variables from $G$),%
\footnote{More generally, $R$ can contain existentials which are already known to be unique; see \cref{sec:loc-intlog}, \cite[V~1.12]{Jstone}.}
there is a \defn{structure presented by $G, R$}, denoted
\begin{align*}
\ang{G \mid R} = \ang{G \mid R}_\!C \in \!C,
\end{align*}
which is the universal structure obeying the axioms and equipped with a map $\eta : G -> \ang{G \mid R}$ making all of the relations in $R$ hold.
That is, for any $A \in \!C$, composition with $\eta$ gives a bijection
\begin{align*}
\!C(\ang{G \mid R}, A) &\cong \{f : G -> A \mid \text{every relation in $R$ holds in $A$ after substituting $f$}\}.
\end{align*}
The structure $\ang{G \mid R}$ may be constructed via a standard transfinite iteration of length $<\kappa^+$, starting with the algebra of terms generated from $G$ by the function symbols in $\@L$, and then at each stage, enforcing the relations in $R$ as well as the $\kappa$-limit axioms in the theory relative to the previous stage, by adjoining new elements (to satisfy $\exists$), taking a quotient (to satisfy $\exists!$ and equations), and enlarging the relations (to satisfy other atomic formulas).
Every $A \in \!C$ has a presentation, i.e., $A \cong \ang{G \mid R}$ for some $G, R$, for example $G :=$ underlying set of $A$ and $R :=$ all relations which hold in $A$.

Colimits in $\!C$ may be constructed via presentations: to construct a colimit of a diagram, in which each object is given by a presentation, take the disjoint union of the presentations (which presents the coproduct of the objects), then add equations identifying each $a \in A$ with $f(a) \in B$ for each morphism $f : A -> B$ in the diagram.

If $A \in \!C$ is presented by some $G, R$ which are $\lambda$-ary (respectively, just $G$ is $\lambda$-ary), then $A$ is called \defn{$\lambda$-presented} (\defn{$\lambda$-generated}).
Let $\!C_\lambda \subseteq \!C$ denote the full \defn{subcategory of $\lambda$-presented structures}.
By the construction of colimits described above, $\!C_\lambda$ is closed under $\lambda$-ary colimits.
The notion of $\lambda$-presentability in locally $\kappa$-presentable $\!C$ is most robust when $\lambda \ge \kappa$, in which case $\!C_\lambda \subseteq \!C$ can be intrinsically characterized as all objects $A$ for which the representable functor $\!C(A, -) : \!C -> \!{Set}$ preserves $\lambda$-directed colimits.

There is a \defn{(special) adjoint functor theorem} for locally presentable categories: any functor $\!D -> \!C$ between locally presentable categories which preserves limits as well as $\lambda$-directed colimits for sufficiently large $\lambda$ has a left adjoint $\!C -> \!D$.  In particular, forgetful functors between $\!D, \!C$ axiomatized by $\kappa$-limit theories (which preserve limits and $\kappa$-directed colimits) have left adjoints, which construct the \defn{free $\!D$-structure generated by a $\!C$-structure}:
\begin{align*}
\!C &--> \!D \\
A &|--> \ang{A \mid \text{all $\!C$-structure relations which hold in } A}_\!D =: \ang{A \qua \!C}_\!D = \ang{A}_\!D.
\end{align*}
We adopt this ``\defn{qua}'' notation from \cite{JVpfrm}, but omit it when context makes clear which category $\!C$ we are starting from.  More generally, we may add further generators or impose further $\!D$-structure relations, via
\begin{align*}
\ang{A \qua \!C,\, G \mid R}_\!D := \ang{A \sqcup G \mid \{\text{all $\!C$-relations which hold in } A\} \cup R}_\!D.
\end{align*}
Any left adjoint $F : \!C -> \!D$ between locally $\kappa$-presentable categories whose right adjoint $G : \!D -> \!C$ preserves $\kappa$-directed colimits (in particular, any free functor $F$) preserves $\kappa$-presented objects, i.e., restricts to a functor $F : \!C_\kappa -> \!D_\kappa$.
There is also an adjoint functor theorem for colimit-preserving functors between locally presentable categories: they all have right adjoints.

Two special types of adjunctions $F : \!C -> \!D \dashv G : \!D -> \!C$ are the \defn{reflective} and \defn{coreflective} adjunctions, where $G$, respectively $F$, is full and faithful, hence can be regarded (up to equivalence of categories) as a full subcategory inclusion.
When $\!C, \!D$ are locally $\kappa$-presentable and $G$ preserves $\kappa$-directed colimits (e.g., $G$ is forgetful), these interact with presentations as follows.
In the coreflective case, as mentioned above, $F$ preserves $\kappa$-presented objects; but since $F$ is full and faithful, it also reflects them, i.e.,
\begin{align*}
\!C_\kappa \simeq F^{-1}(\!D_\kappa).
\end{align*}
In the reflective case, we instead have

\begin{proposition}
\label{thm:cat-lim-refl}
Given a reflective adjunction $F \dashv G : \!D -> \!C$ between locally $\kappa$-presentable categories as above, where $G$ preserves $\kappa$-directed colimits, every $B \in \!D_\kappa$ is a retract of some $F(A)$ for $A \in \!C_\kappa$.
If furthermore $\kappa$ is uncountable, then every $B$ is isomorphic to some $F(A)$, i.e.,
\begin{align*}
F(\!C_\kappa) \simeq \!D_\kappa.
\end{align*}
\end{proposition}

\begin{remark}
More generally, the conclusion above holds for the functor induced by $F$ between coslice categories $A \down \!C_\kappa -> F(A) \down \!D_\kappa$ for every $A \in \!C_\kappa$.
In fact, this generalized statement completely characterizes when a $\kappa$-directed-colimit-preserving right adjoint between locally $\kappa$-presentable categories is full and faithful; see \cite[2.6]{MPlfp}.
\end{remark}

\Cref{thm:cat-lim-refl} is a consequence of

\begin{lemma}
\label{thm:cat-lim-dircolim-retr}
Let $\!D$ be a locally $\kappa$-presentable category.
If a $\kappa$-presented object $B \in \!D_\kappa$ is a $\kappa$-directed colimit $B = \injlim_{i \in I} A_i$ of $\kappa$-presented objects $A_i \in \!D_\kappa$ with morphisms $f_{ij} : A_i -> A_j$ for $i \le j \in I$, for some $\kappa$-directed preorder $I$, then $B$ is a retract of some $A_{i_0}$, as well as the colimit of some subsequence $A_{i_0} -> A_{i_1} -> \dotsb$ of the $A_i$, where $i_0 \le i_1 \le \dotsb \in I$.
\end{lemma}
\begin{proof}
Since $B$ is $\kappa$-presented, the identity $1_B : B -> \injlim_i A_i$ factors through the cocone map $\iota_{i_0} : A_{i_0} -> \injlim_i A_i$ for some $i_0 \in I$, say as $s : B -> A_{i_0}$, so that $\iota_{i_0} \circ s = 1_B$, whence $B$ is a retract of $A_{i_0}$.
Then the composite $A_{i_0} --->{\iota_{i_0}} B --->{s} A_{i_0} --->{\iota_{i_0}} B = \injlim_i A_i$ is equal to $\iota_{i_0}$, so since $A_{i_0}$ is $\kappa$-presented, the two maps $s \circ \iota_{i_0}, 1_{A_{i_0}} : A_{i_0} -> A_{i_0}$ are already equal when composed with the diagram map $f_{i_0i_1} : A_{i_0} -> A_{i_1}$ for some $i_1 \ge i_0$.
Then we have $\iota_{i_1} \circ f_{i_0i_1} \circ s = \iota_{i_0} \circ s = 1_B$, so the composite $A_{i_1} --->{\iota_{i_1}} B --->{s} A_{i_0} --->{f_{i_0i_1}} A_{i_1} --->{\iota_{i_1}} B = \injlim_i A_i$ is equal to $\iota_{i_1}$, so as before, $f_{i_0i_1} \circ s \circ \iota_{i_1}, 1_{A_{i_1}} : A_{i_1} -> A_{i_1}$ are already equal when composed with $f_{i_1i_2} : A_{i_1} -> A_{i_2}$, i.e., we have $f_{i_0i_2} \circ s \circ \iota_{i_1} = f_{i_1i_2} : A_{i_1} -> A_{i_2}$, for some $i_2 \ge i_1$.
Continue finding $i_2 \le i_3 \le \dotsb$ in this manner, such that for each $j$, we have
\begin{align*}
f_{i_0i_{j+1}} \circ s \circ \iota_{i_j} = f_{i_ji_{j+1}} : A_{i_j} -> A_{i_{j+1}}.
\end{align*}
Let $\iota'_j : A_{i_j} -> \injlim_j A_{i_j}$ be the cocone maps.
We have a colimit comparison map $g : \injlim_j A_{i_j} -> \injlim_i A_i = B$, such that $g \circ \iota'_j = \iota_{i_j}$, as well as a map $h := \iota'_0 \circ s : B -> \injlim_j A_{i_j}$.
We have $g \circ h = g \circ \iota'_0 \circ s = \iota_{i_0} \circ s = 1_B$, and conversely, $h \circ g = 1_{\injlim_j A_{i_j}}$ because for each $j$, we have $h \circ g \circ \iota'_j = \iota'_0 \circ s \circ \iota_{i_j} = \iota'_{j+1} \circ f_{i_0i_{j+1}} \circ s \circ \iota_{i_j} = \iota'_{j+1} \circ f_{i_ji_{j+1}} = \iota'_j$.
Thus $g^{-1} = h : B \cong \injlim_j A_{i_j}$.
\end{proof}

\begin{proof}[Proof of \cref{thm:cat-lim-refl}]
Write $G(B) \in \!C$ as a $\kappa$-directed colimit $\injlim_{i \in I} A_i$, for a $\kappa$-directed poset $I$, of $\kappa$-presented objects $A_i \in \!C_\kappa$ and morphisms $f_{ij} : A_i -> A_j$ for $i \le j$ (e.g., $A_i = \ang{G_i \mid R_i}$, for all $\kappa$-ary $G_i \subseteq G(A)$ and $\kappa$-ary sets $R_i$ of relations which hold between them in $G(A)$).
Then the adjunction counit $F(G(B)) \cong \injlim_{i \in I} F(A_i) -> B$ is an isomorphism.
Since $G$ preserves $\kappa$-directed colimits, each $F(A_i) \in \!D_\kappa$.
By \cref{thm:cat-lim-dircolim-retr}, $B$ is a retract of some $F(A_{i_0})$, as well as isomorphic to $\injlim_j F(A_{i_j}) \cong F(\injlim_j A_{i_j})$ for some $i_0 \le i_1 \le \dotsb$; if $\kappa$ is uncountable, then $\injlim_j A_{i_j} \in \!C_\kappa$.
\end{proof}

Many naturally occurring adjunctions between locally $\kappa$-presentable categories are not manifestly free/forgetful.
For example, the usual forgetful functor $\!{Bool} -> \!{DLat}$ also has a \emph{right} adjoint, taking each distributive lattice $A$ to the Boolean subalgebra $A_\neg$ of complemented elements.
However, up to modifying the theories axiomatizing the categories involved, it is always possible to regard such adjunctions as free/forgetful.
In this example, we may regard a distributive lattice $A$ as a \defn{two-sorted structure}, with its usual underlying set $A$ as well as $A_\neg$ as an additional ``underlying set'', with the inclusion $A_\neg `-> A$ regarded as a unary operation.
The resulting class of structures is still $\omega$-limit axiomatizable;
the key point is that complements are uniquely defined by an $\omega$-ary conjunction of atomic formulas $(x \wedge y = \bot) \wedge (x \vee y = \top)$.
We can then regard $A |-> A_\neg$ as forgetting the usual underlying set $A$; the usual forgetful functor $\!{Bool} -> \!{DLat}$ then takes $B \in \!{Bool}$ to the ``free distributive lattice $A$ generated by $B \subseteq A_\neg$''.
Any adjunction between locally presentable categories can be described as free/forgetful in a similar manner, so that the concrete constructions in terms of presentations described above can in fact be applied to arbitrary adjunctions.

For limit theories over a \emph{large} signature, one can try to develop the above notions in the same way; the only issue is that the transfinite construction of presented structures may result in a proper class, hence small-presented structures need not be small.
If it is known (by some other means) that small-presented structures are always small, e.g., in the case of $\!{{\bigvee}Lat}$ or $\!{Frm}$ (by \cref{thm:frm-small}), then the above facts all go through.

On the other hand, for example, the result of Gaifman--Hales (\cref{thm:gaifman-hales}) implies that the forgetful functor $\!{CBool} -> \!{Set}$ does not have a left adjoint.
In this case, we denote the category of large models in all-caps, e.g., $\!{CBOOL}$ is the category of large complete Boolean algebras, and let $\!{CBOOL}_\infty \subseteq \!{CBOOL}$ denote the full subcategory of small-presented structures; these are all the large structures we will ever need in this paper.
Note that $\!{CBOOL}$, being a collection of possibly proper classes, is an ``extra-large'' category (hence would require extra foundational care to work with rigorously); however, $\!{CBOOL}_\infty \subseteq \!{CBOOL}$ is only a large category, since we may equivalently define it to consist of small presentations, together with homomorphisms between the presented structures which are determined by their values on generators.

\subsection{Algebraic theories and monadic categories}
\label{sec:cat-alg}

Categories of structures axiomatized by algebraic theories (in the sense of \cref{ex:cat-lim-alg}) over a small signature are \defn{monadic over $\!{Set}$}.%
\footnote{Multi-sorted algebraic theories are monadic over $\!{Set}^S$ where $S$ is the set of sorts.}
See \cite[II~Ch.~2--4]{Bcat} for a detailed reference.
In such categories, in addition to the notions from the preceding subsection, the following universal-algebraic notions can be categorically formulated.

Recall that a \defn{monomorphism} in an arbitrary category $\!C$ is a morphism $f$ such that $f \circ g = f \circ h \implies g = h$; the dual notion is an \defn{epimorphism}.
A \defn{subobject} of an object $X \in \!C$ is an equivalence class of monomorphisms $A -> X$ with respect to the preorder
\begin{align*}
(f : A `-> X) \subseteq (g : B `-> X) \coloniff \exists h : A -> B\, (f = g \circ h)
\end{align*}
(such an $h$ is necessarily unique); this preorder descends to a partial order on the \defn{poset of subobjects} $\Sub(X) = \Sub_\!C(X)$.
It is common to abuse notation regarding monomorphisms: we often treat single monomorphisms $A -> X$ as if they were subobjects, denoted $A \subseteq X$, and refer to the representing monomorphism as the \defn{inclusion}.
In $\!C$ monadic over $\!{Set}$, monomorphisms are precisely the embeddings of subalgebras; hence, $\Sub(X) \cong \{\text{subalgebras of $X$}\}$.%
\footnote{In categories of models of arbitrary limit theories, monomorphisms are injective homomorphisms, not necessarily embeddings (which are not an intrinsic categorical notion, since they depend on the choice of signature).}

For a morphism $f : X -> Y$ in an arbitrary category $\!C$ with finite limits, its \defn{kernel} is the pullback $\ker(f) := X \times_Y X$ of $f$ with itself.
It is an instance of a \defn{congruence} (also known as \defn{equivalence relation}) on $X$, meaning a subobject ${\sim} \subseteq X^2$ which internally satisfies the axioms of an equivalence relation.
In $\!C$ monadic over $\!{Set}$, these correspond to congruences in the usual universal-algebraic sense (i.e., equivalence relations which are also subalgebras).

Given a congruence ${\sim} \subseteq X^2$, we may take the coequalizer of the projections $\pi_1, \pi_2 : {\sim} \rightrightarrows X$.
In $\!C$ monadic over $\!{Set}$, this is the usual \defn{quotient algebra} $X/{\sim}$.
The quotient map $q : X ->> X/{\sim}$ has kernel $\sim$, and is a \defn{regular epimorphism}, meaning that it is the coequalizer of its kernel (which implies that it is an epimorphism).
So we have a bijection
\begin{align*}
\{\text{congruences on } X\} &\overunderset{\coeq}{\ker}\rightleftarrows \{\text{equiv.\ classes of regular epimorphisms from $X$}\} =: \RSub_{\!C^\op}(X) \subseteq \Sub_{\!C^\op}(X).
\end{align*}
Any morphism $f : X -> Y \in \!C$ factors uniquely-up-to-unique-isomorphism into a regular epimorphism $X ->> X/\ker(f)$ followed by a monomorphism $X/\ker(f) `-> Y$; the subobject of $Y$ represented by this monomorphism is called the \defn{image} of $f$.
Moreover, regular epimorphisms are closed under arbitrary products, as well as under pullback across arbitrary morphisms; thus these same operations also preserve image factorizations of arbitrary morphisms.

Categories satisfying the above compatibility conditions between subobjects, congruences, and regular epimorphisms are called \defn{(Barr-)exact completely regular} \cite{CVregex}.
Thus, categories monadic over $\!{Set}$ are exact completely regular.
Moreover, in addition to limits, also images and coequalizers of congruences are computed as for the underlying sets (unlike in arbitrary locally presentable categories, e.g., the quotient of $\{0 < 1 < 2\}$ by $0 \sim 2$ in $\!{Pos}$).

Presented algebras $\ang{G \mid R}$ may be computed in a single step, as the quotient of the free algebra $\ang{G}$ by the congruence generated by $R$ (where each equation between terms in $R$ is identified with a pair of elements in $\ang{G}$); there is no need for the transfinite iteration described in the preceding subsection (except in order to generate the congruence).
It follows that an algebraic theory over a large signature defines a well-behaved category $\!C$ as soon as we know that free algebras $\ang{G}$ over small generating sets $G$ are small, i.e., that the left adjoint to the forgetful functor $\!C -> \!{Set}$ exists; this then implies that small-presented algebras are small and $\!C$ is monadic over $\!{Set}$, so all of the facts in this and the preceding section go through.
For example, this applies to the categories $\!{{\bigvee}Lat}, \!{Frm}$ (\cref{thm:frm-small}).

We will need the following standard technical facts:

\begin{lemma}
\label{thm:cat-alg-cong-pres}
Let $A$ be a $\kappa$-generated algebra for some algebraic theory, and let $\sim$ be a congruence on $A$ such that the quotient algebra $A/{\sim}$ has a presentation with $<\kappa$-many relations.
Then $\sim$ is a $\kappa$-generated congruence.
\end{lemma}
\begin{proof}
Let $G \subseteq A$ be a $\kappa$-ary generating set, let $p : A ->> A/{\sim}$ be the quotient map, and let $A/{\sim} = \ang{H \mid S}$ be some presentation with $\abs{S} < \kappa$, so that $S \subseteq \ang{H}^2$ generates the kernel of the quotient map $q : \ang{H} ->> A/{\sim}$.
For each $g \in G$, pick some $u(g) \in A$ such that $p(g) = q(u(g)) \in A/{\sim}$, and for each $h \in H$, pick some $v(h) \in A$ such that $q(h) = p(v(h)) \in A/{\sim}$.
Extend $v$ to a homomorphism $v : \ang{H} -> A$; by definition of $v$ on generators, $v$ lifts $q$ along $p$:
\begin{equation*}
\begin{tikzcd}
\ang{H} \drar[two heads, "q"'] \rar["v"] & A \dar[two heads, "p"] \\
& A/{\sim}
\end{tikzcd}
\end{equation*}
Then $\sim$ is generated by
\begin{align*}
R := v(S) \cup \{(g, v(u(g))) \mid g \in G\}.
\end{align*}
Indeed, $R$ is contained in ${\sim} = \ker(p)$: $v(S) \subseteq v(\ker(q))$ is, because $q = p \circ v$; and for $g \in G$, we have $p(v(u(g))) = q(u(g)) = p(g)$.
Conversely, it follows from $q = p \circ v$ that $A/{\sim}$ is the quotient of the image $v(\ang{H}) \subseteq A$ by $v(\ker(q))$ which is the congruence generated by $v(S)$; while the extra pairs $(g, v(u(g)))$ for $g \in G$ ensure that every element of $A$ is identified with some element of $v(\ang{H})$.
\end{proof}

\begin{corollary}
\label{thm:cat-alg-pres}
Let $A$ be an algebra for some algebraic theory.
Suppose $A$ has $<\kappa$-many generators $G \subseteq A$, as well as a presentation with $<\kappa$-many relations.
Then $A$ has a $\kappa$-ary presentation using the generators $G$ and $<\kappa$-many relations.
\end{corollary}
\begin{proof}
Apply \cref{thm:cat-alg-cong-pres} to the kernel of $\ang{G} ->> A$.
\end{proof}

\begin{lemma}
\label{thm:cat-alg-mono-epi-pres}
Let $\!C$ be the category of algebras for some $\kappa$-algebraic theory.
Then a morphism $f : X -> Y \in \!C_\kappa$ between $\kappa$-presented algebras is a monomorphism, epimorphism, or regular epimorphism, respectively, in $\!C_\kappa$ iff it is so in $\!C$.
\end{lemma}
\begin{proof}
Since the inclusion $\!C_\kappa \subseteq \!C$ is faithful, it reflects monos and epis.
Since $\!C_\kappa \subseteq \!C$ is closed under finite colimits, hence cokernels and coequalizers, epis and regular epis in $\!C_\kappa$ are also so in $\!C$.
If $f : X -> Y \in \!C_\kappa$ is a mono, then since the free algebra $\ang{1}_\!C$ on one generator is in $\!C_\kappa$, we get that the underlying function of sets $\!C_\kappa(\ang{1}_\!C, f)$ is injective, whence $f$ is a mono in $\!C$.
Finally, if $f : X -> Y \in \!C_\kappa$ is a regular epi in $\!C$, i.e., a surjective homomorphism, then by \cref{thm:cat-alg-cong-pres} its kernel is a $\kappa$-generated congruence ${\sim} \subseteq X^2$, whence letting $(g, h) : \ang{Z}_\!C -> X^2$ for some $\kappa$-ary $Z$ map generators to generators of $\sim$, we have that $f$ is the coequalizer of $g, h : \ang{Z}_\!C \rightrightarrows X \in \!C_\kappa$.
\end{proof}

\begin{lemma}
\label{thm:cat-funct-sub}
Let $F : \!C -> \!D$ be a finite-limit-preserving functor between categories with finite limits.
\begin{enumerate}
\item[(a)]
$F$ is faithful iff for each $X \in \!C$, $F$ induces an embedding $\RSub_\!C(X) -> \RSub_\!D(F(X))$.
\item[(b)]
$F$ is conservative iff for each $X \in \!C$, $F$ induces an embedding $\Sub_\!C(X) -> \Sub_\!D(F(X))$.
\end{enumerate}
\end{lemma}
\begin{proof}
(a) ($\Longrightarrow$)
Let $A, B \subseteq X$ be regular subobjects, the equalizers of parallel pairs $X \rightrightarrows Y$ and $X \rightrightarrows Z$ respectively, such that $F(A) \subseteq F(B) \subseteq F(X)$.
Then $F(A)$ equalizes $F(X) \rightrightarrows F(Z)$, whence by faithfulness, $A$ equalizes $X \rightrightarrows Z$, whence $A \subseteq B$.

(a) ($\Longleftarrow$)
Let $f, g : X \rightrightarrows Y \in \!C$ with $F(f) = F(g) \in \!D$.
Then $F(\eq(f, g)) = \eq(F(f), F(g)) \subseteq F(X)$ is all of $F(X)$, whence since $F : \RSub_\!C(X) -> \RSub_\!D(F(X))$ is injective, $\eq(f, g)$ is all of $X$, whence $f = g$.

(b) ($\Longrightarrow$)
Since $F$ preserves finite limits, $F$ induces a $\wedge$-lattice homomorphism $F : \Sub_\!C(X) -> \Sub_\!D(F(X))$ for each $X \in \!C$.
Thus by taking meets, it suffices to check that whenever $A \subseteq B \subseteq X$ with $F(A) = F(B) \subseteq F(X)$, then $A = B \subseteq X$.
This follows from conservativity.

(b) ($\Longleftarrow$)
In any category with finite limits, morphisms $f : X -> Y$ are in bijection with their \defn{graphs} $G \subseteq X \times Y$, which are arbitrary subobjects such that the first projection $G -> X$ is an isomorphism.
So for any $f : X -> Y \in \!C$ such that $F(f)$ is invertible, letting $G \subseteq X \times Y$ be the graph of $f$, the second projection $F(G) -> F(Y)$ must be an isomorphism.
In particular, it is monic, i.e., its kernel is the diagonal $F(G) `-> F(G) \times F(G) \cong F(G \times G)$, whence since $F : \Sub_\!C(G \times G) -> \Sub_\!D(F(G \times G))$ is injective, so must be the kernel of the second projection $G -> Y$, which is thus monic.
Now regarding $G$ as a subobject of $Y$, since $F : \Sub_\!C(Y) -> \Sub_\!D(F(Y))$ is injective and $F(G) \cong F(Y)$, we have $G \cong Y$, i.e., $f$ is invertible.
\end{proof}

\subsection{Ordered theories and categories}
\label{sec:cat-ord}

A \defn{locally ordered category} is a category $\!C$, each of whose hom-sets is equipped with a partial order, such that composition is order-preserving on both sides.
In other words, it is a category \defn{enriched} over the category $\!{Pos}$ of posets.
See \cite{Kvcat} for the standard reference on general enriched category theory.
There are also enriched analogs of locally presentable and monadic categories; see \cite{Kvlfp}, \cite{Pvlaw}.
In this subsection, we give an elementary description of the locally ordered case, in syntactic terms largely analogous to the previous two subsections.
The basic idea will be to ``replace $=$ with $\le$ throughout''.

Let $\@L$ be an infinitary first-order signature in the usual sense.
An (infinitary first-order) \defn{ordered $\@L$-formula} will mean an $(\@L \sqcup \{\le\})$-formula in the usual sense, where $\le$ is a binary relation symbol.
Equivalently, we may think of $\le$ replacing $=$ as the only extra allowed relation symbol, by treating $x = y$ as an abbreviation for $(x \le y) \wedge (y \le x)$.
An \defn{ordered $\@L$-structure} will mean an $\@L$-structure in the usual sense, but with the underlying set replaced by an underlying poset, and with all function symbols required to be interpreted as order-preserving maps (from the product partial order).
We define the interpretation of ordered $\@L$-formulas in an ordered $\@L$-structure in the expected manner, with $\le$ always interpreted as the partial order in the underlying poset.

To define ordered $\kappa$-limit axioms, we need to replace the $\exists!$ quantifier in ordinary $\kappa$-limit axioms, which involves a hidden $=$ to assert uniqueness, with an ordered analog.
Note that a $\kappa$-limit axiom
\begin{align*}
\forall \vec{x} \left(\bigwedge_i \phi_i(\vec{x}) -> \exists! \vec{y}\, \bigwedge_j \psi_j(\vec{x}, \vec{y})\right)
\end{align*}
is equivalent to the combination of the two axioms
\begin{gather*}
\forall \vec{x}, \vec{x}', \vec{y}, \vec{y}' \left(\bigwedge_i \phi_i(\vec{x}) \wedge \bigwedge_i \phi_i(\vec{x}') \wedge \bigwedge_j \psi_j(\vec{x}, \vec{y}) \wedge \bigwedge_j \psi_j(\vec{x}', \vec{y}') \wedge (\vec{x} = \vec{x}') -> (\vec{y} = \vec{y}')\right), \\
\forall \vec{x} \left(\bigwedge_i \phi_i(\vec{x}) -> \exists \vec{y}\, \bigwedge_j \psi_j(\vec{x}, \vec{y})\right).
\end{gather*}
By an \defn{ordered $\kappa$-limit theory}, we mean a set of pairs of axioms of the form
\begin{gather*}
\forall \vec{x}, \vec{x}', \vec{y}, \vec{y}' \left(\bigwedge_i \phi_i(\vec{x}) \wedge \bigwedge_i \phi_i(\vec{x}') \wedge \bigwedge_j \psi_j(\vec{x}, \vec{y}) \wedge \bigwedge_j \psi_j(\vec{x}', \vec{y}') \wedge (\vec{x} \le \vec{x}') -> (\vec{y} \le \vec{y}')\right), \\
\forall \vec{x} \left(\bigwedge_i \phi_i(\vec{x}) -> \exists \vec{y}\, \bigwedge_j \psi_j(\vec{x}, \vec{y})\right)
\end{gather*}
(where the $\phi_i$ and $\psi_j$ are atomic and the variable tuples and $\bigwedge$'s are $\kappa$-ary).
Clearly, these imply the previous two axioms with $=$.
We will refer to the latter pair of axioms as a single \defn{ordered $\kappa$-limit axiom}, and abbreviate it by
\begin{align*}
\forall \vec{x} \left(\bigwedge_i \phi_i(\vec{x}) -> \exists!_\le^{\vec{x}} \vec{y}\, \bigwedge_j \psi_j(\vec{x}, \vec{y})\right)
\end{align*}
where $\exists!_\le^{\vec{x}} \vec{y}$ is read ``there exists unique $\vec{y}$, monotone in $\vec{x}$, s.t.''; but note that $\exists!_\le^{\vec{x}}$ has no independent meaning in a formula.
The category of all models of an ordered $\kappa$-limit theory, together with all homomorphisms (which are required to be monotone), is called a \defn{locally $\kappa$-presentable locally ordered category}.
These are in particular locally $\kappa$-presentable categories in the ordinary sense, since any ordered $\kappa$-limit theory over $\@L$ is an ordinary $\kappa$-limit theory over $\@L \sqcup \{\le\}$.

\begin{example}
An ordered $\kappa$-limit axiom with empty $\vec{y}$ is just an ordinary $\kappa$-Horn axiom.
\end{example}

\begin{example}
\label{ex:cat-ord-djfrm}
The axiom $\forall \vec{x}\, \exists!_\le^{\vec{x}} \vec{y}\, \bigwedge_j \psi_j(\vec{x}, \vec{y})$ says that $\bigwedge_j \psi_j$ is the graph of a monotone function.
One can similarly axiomatize ordered structures with partial monotone operations, as in \cref{ex:cat-lim-essalg}.
For example, the axioms
\begin{gather*}
\forall x, y, z\, (J(x, y, z) -> (x \le z) \wedge (y \le z) \wedge (x \wedge y = \bot)), \\
\forall x, y, z, w\, (J(x, y, z) \wedge (x \le w) \wedge (y \le w) -> (z \le w)), \\
\forall x, y\, ((x \wedge y = \bot) -> \exists!_\le^{x,y} z\, J(x, y, z))
\end{gather*}
axiomatize disjunctive $\omega$-frames, where $J$ is the graph of the partial join operation.
\end{example}

\begin{example}
\label{ex:cat-ord-bool}
Boolean algebras \emph{cannot} be axiomatized by an ordered $\omega$-limit theory.
We cannot include $\neg$ in the signature, since $\neg$ is not a monotone function in Boolean algebras.
If we were allowed ordinary limit axioms, we could say $\forall x\, \exists! y\, ((x \wedge y = \bot) \wedge (x \vee y = \top))$; but the corresponding ordered limit axiom is false in Boolean algebras, again because $\neg$ is not monotone.
\end{example}

In locally ordered categories $\!C$, one has \defn{weighted limits and colimits} (called \defn{indexed} in older literature), which are universal objects equipped with morphisms to/from a diagram satisfying specified inequalities rather than (or in addition to) equations; see \cite[Ch.~3]{Kvcat}, \cite[B1.1.3]{Jeleph}.
We will not need the general notion, but will use the following particular types of weighted limits:
\begin{itemize}

\item  The \defn{inserter} $\ins(f, g)$ of a parallel pair of morphisms $f, g : A \rightrightarrows B$ is (if it exists) a universal object equipped with a morphism $\pi : \ins(f, g) -> A$ such that $f \circ \pi \le g \circ \pi$.
Thus, inserters are ordered analogs of equalizers.

If inserters exist, then so do equalizers, via $\eq(f, g) = \ins(g \circ \pi, f \circ \pi : \ins(f, g) \rightrightarrows B)$.

In $\!{Pos}$, inserters are given by $\ins(f, g) = \{a \in A \mid f(a) \le g(a)\}$.

\item  The \defn{comma object} $A \downarrow_C B$ of two objects $A, B$ equipped with morphisms $f : A -> C$ and $g : B -> C$ to a third object is (if it exists) a universal object equipped with morphisms $\pi_1 : A \downarrow_C B -> A$ and $\pi_2 : A \downarrow_C B -> B$ such that $f \circ \pi_1 \le g \circ \pi_2$.
These are ordered analogs of fiber products.

In $\!{Pos}$, comma objects are given by $A \downarrow_C B = \{(a, b) \in A \times B \mid f(a) \le g(b)\}$.

\item  The \defn{power} $A^I$ of an object $A$ by a poset $I$ is (if it exists) a universal object equipped with a monotone $I$-indexed family of morphisms $(\pi_i)_{i \in I} : I -> \!C(A^I, A)$.  These are ordered analogs of ordinary powers (iterated products).

In $\!{Pos}$, powers are given by $A^I = \!{Pos}(I, A)$ with the pointwise partial ordering.

In particular, when $I = \#S = \{0 < 1\}$, we call $A^\#S =: {\le_A}$ the \defn{(internal) (partial) order on $A$}; it plays a role analogous to the diagonal $A `-> A^2$ in ordinary categories.

\end{itemize}
Duals of these are called \defn{coinserters}, denoted $\coins(f, g)$, \defn{cocomma objects}, and \defn{copowers} (or sometimes \defn{tensors}).
Much as ordinary limits may be constructed from products and equalizers, so too may weighted limits in locally ordered categories be constructed from products and inserters (or from products and comma objects, or from products, equalizers, and powers by $\#S$); dually for weighted colimits.
Locally presentable locally ordered categories have arbitrary small weighted limits and colimits;
weighted limits of models of ordered $\kappa$-limit theories are constructed as in $\!{Pos}$.

\begin{example}
If we regard $\!{Bool}$ as locally discretely ordered, then it has inserters coinciding with equalizers; note that these are \emph{not} computed as in $\!{Pos}$.
This is another reflection of the fact (\cref{ex:cat-ord-bool}) that $\!{Bool}$ is not axiomatizable by an ordered limit theory.
\end{example}

Presentations of models of ordered $\kappa$-limit theories are the same as in the unordered context (of course, relations may now use $\le$), and may be used to construct colimits as well as left adjoints of forgetful functors (which are locally order-preserving on morphisms).
The notion of $\lambda$-presentability is also the same as before, and is preserved by coinserters, cocomma objects, and copowers by $\lambda$-ary posets.
The same caveats as before regarding large signatures continue to apply.

By an \defn{ordered algebraic theory}, we will mean an ordered limit theory whose signature has no relation symbols and axioms have empty left-hand side and no existentials, in analogy with \cref{ex:cat-lim-alg}.
Note that $\le$ may still be used; hence, these do not reduce to unordered algebraic theories
(unless there are operations like $\wedge, \vee$ which may be used to define $\le$ in terms of $=$).
Locally ordered categories axiomatized by small ordered algebraic theories admit the following ordered universal-algebraic notions (see \cite{St2sh}, where they are defined more generally for 2-categories).%
\footnote{These categories are strongly (i.e., $\!{Pos}$-enriched) monadic over $\!{Pos}$; see \cite{Pvlaw}.  However, unlike in the unordered setting, not every category strongly monadic over $\!{Pos}$ is axiomatized by an ordered algebraic theory (in our sense): in general, one would need to allow operations whose ``arity'' is a poset, i.e., partial operations with domain defined by a system of inequalities.
Categories axiomatized by such theories are less well-behaved: they do not admit a good notion of order-congruence which corresponds to surjective homomorphisms.}

An \defn{order-monomorphism} will mean a morphism $f$ such that $f \circ g \le f \circ h \implies g \le h$.
In $\!{Pos}$, these are order-embeddings, whereas monomorphisms are injective monotone maps.
An \defn{order-embedded subobject} will mean a subobject which is represented by an order-monomorphism.
Let $\OSub(X) = \OSub_\!C(X) \subseteq \Sub(X)$ denote the poset of order-embedded subobjects.
Then $\OSub_\!C(X) \cong \{\text{subalgebras of $X$}\}$ (which is $\cong \Sub_\!C(X)$ if $\!C$ is also monadic over $\!{Set}$, e.g., $\!C = \!{{\bigvee}Lat}$; but $\Sub_\!{Pos}(X)$ consists of subsets of $X$ equipped with a finer partial order, for example).

The \defn{order-kernel} of $f : X -> Y$ will mean the comma object
\begin{align*}
\oker(f) := (X \downarrow_Y X) = \{(x, x') \in X^2 \mid f(x) \le f(x')\}.
\end{align*}
Thus $\oker(f)$ is an order-embedded subobject of $X^2$ containing $\ker(f)$, and is an \defn{order-congruence} on $X$, meaning a subobject ${\lesim} \subseteq X^2$ which is transitive and contains the partial order $\le_X$ on $X$ (which implies reflexivity).

Conversely, given any order-congruence $\lesim$, let ${\sim} := {\lesim} \cap {\lesim}^{-1}$ denote its symmetric part, which is a congruence (in the unordered sense) on $X$.
Then $\lesim$ descends to a partial order on the quotient algebra $X/{\sim}$, which is characterized as the coinserter of the projections $\pi_1, \pi_2 : {\lesim} \rightrightarrows X$, meaning that any homomorphism $f : X -> Y$ such that $x \lesim x' \implies f(x) \le f(x')$, i.e., ${\lesim} \subseteq \oker(f)$, descends to a monotone homomorphism $X/{\sim} -> Y$.
The quotient map $q : X ->> X/{\sim}$ is an \defn{order-regular epimorphism}, meaning the coinserter of its order-kernel.
So we have a bijection
\begin{align*}
\{\text{order-congruences on } X\} &\overunderset{\coins}{\oker}\rightleftarrows \{\text{equiv.\ classes of order-regular epimorphisms from $X$}\}.
\end{align*}
Any morphism $f : X -> Y \in \!C$ factors uniquely-up-to-unique-isomorphism into an order-regular epimorphism $X ->> X/\ker(f)$ followed by an order-monomorphism $X/\ker(f) `-> Y$; the order-embedded subobject of $Y$ represented by this monomorphism is called the \defn{image} of $f$.
Moreover, order-regular epimorphisms are closed under products and pullback across arbitrary morphisms.

\bigskip\noindent
Department of Mathematics \\
University of Illinois at Urbana--Champaign \\
Urbana, IL 61801

\medskip\noindent
\nolinkurl{ruiyuan@illinois.edu}

\end{document}